\documentclass[10.5pt]{article}

\usepackage{amsmath, latexsym, amsfonts, amssymb, amsthm}
\usepackage{graphicx, color,hyperref,dsfont,epsfig,caption,wrapfig,subfig}
\usepackage{pstricks}
\usepackage{movie15}
\usepackage{bm}

\hypersetup{
    linktoc=page,
    linkcolor=red,          
    citecolor=blue,        
    filecolor=blue,      
    urlcolor=cyan,
    colorlinks=true           
}

\numberwithin{equation}{section}

\newtheorem{theorem}{Theorem}[section]
\newtheorem{lemma}[theorem]{Lemma}
\newtheorem{proposition}[theorem]{Proposition}
\newtheorem{corollary}[theorem]{Corollary}

\newtheorem{remark}[theorem]{Remark}
\newtheorem{definition}[theorem]{Definition}

\newcounter{conj}
\newtheorem{conjecture}[conj]{Conjecture}

\theoremstyle{definition}

\usepackage{listings}
\usepackage{color} 
\definecolor{mygreen}{RGB}{34,139,34} 
\definecolor{myblue}{RGB}{0,0,205}
\definecolor{myred}{RGB}{178,34,34}
\definecolor{myorange}{RGB}{255,127,80}
\definecolor{mylilas}{RGB}{170,55,241}

\renewcommand{\P}{\mathbb{P}}
\newcommand{\E}{\mathbb{E}}
\renewcommand{\H}{\mathbb{H}}

\newcommand{\intr}{\int_{\mathbb{R}}}

\newcommand{\hg}{\hat{g}}

\newcommand{\s}{\mathbf{s}}

\numberwithin{equation}{section}
\hypersetup{
    colorlinks,
    linkcolor={red!50!black},
    citecolor={blue!50!black},
    urlcolor={blue!80!black}
}

\title{Integrability of boundary Liouville conformal field theory}

\author{
\begin{tabular}{c} Guillaume Remy\footnote{Department of Mathematics, Columbia University, 2990 Broadway, New York, NY 10027, USA. remy@math.columbia.edu} \end{tabular} 
\begin{tabular}{c} Tunan Zhu\footnote{D\'epartement de Math\'ematiques et Applications, \'Ecole Normale Sup\'erieure de Paris, 45 Rue d'Ulm, 75005 Paris, France. tunan.zhu@gmail.com}  \end{tabular}
}

\begin{document}
\maketitle

\begin{abstract}
Liouville conformal field theory (LCFT) is considered on a simply connected domain with boundary, specializing to the case where the Liouville potential is integrated only over the boundary of the domain. We work in the probabilistic framework of boundary LCFT introduced by Huang-Rhodes-Vargas (2015). Building upon the known proof of the bulk one-point function by the first author, exact formulas are rigorously derived for the remaining basic correlation functions of the theory, i.e., the bulk-boundary correlator, the boundary two-point and the boundary three-point functions. These four correlations should be seen as the fundamental building blocks of boundary Liouville theory, playing the analogous role of the DOZZ formula in the case of the Riemann sphere. Our study of boundary LCFT also provides the general framework to understand the integrability of one-dimensional Gaussian multiplicative chaos measures as well as their tail expansions. Finally these results have applications to studying the conformal blocks of CFT and set the stage for the more general case of boundary LCFT with both bulk and boundary Liouville potentials.
\end{abstract}

\tableofcontents

\section{Introduction and main results}

Liouville conformal field theory - LCFT henceforth - first appeared in Polyakov's seminal 1981 paper \cite{Pol} where he introduces a theory of summation over the space of Riemannian metrics on a given two-dimensional surface and sets LCFT as a fundamental building block of non-critical string theory. The necessity to solve Liouville theory led Belavin, Polyakov, and Zamolodchikov (BPZ) to introduce in \cite{BPZ} conformal field theory (CFT), a powerful framework to study quantum field theories possessing conformal symmetry. On surfaces without boundary, solving Liouville theory amounts to computing the three-point function on the sphere - which is given by the DOZZ formula proposed in \cite{DO,ZZ} - and arguing that correlation functions of higher order or in higher genus can be obtained from it using the conformal bootstrap method of \cite{BPZ}. A similar program can be pursued for surfaces with boundary, where the basic correlations have been derived in the physics literature in \cite{FZZ, bulk_boundary, three_point} and the conformal bootstrap is also applicable.

We work here in the probabilistic framework of LCFT first introduced by David-Kupiainen-Rhodes-Vargas on the Riemann sphere in \cite{Sphere}, and later followed by companion works for the boundary case \cite{Disk} and in higher genus \cite{Tori, Genus, Annulus}. The strength of this framework lies in the fact it allows to put Liouville theory on solid mathematical grounds and to rigorously carry out the program of solving the theory as described above. Indeed, in the case of the Riemann sphere, the BPZ differential equations expressing the constraints of the local conformal invariance of CFT were shown to hold in \cite{DOZZ1}. Building on this work a proof of the DOZZ formula was then given in \cite{DOZZ2}. Very shortly after, the same procedure was implemented by the first author \cite{remy} in the case of boundary LCFT to prove the Fyodorov-Bouchaud formula proposed in \cite{FyBo} that can also be interpreted as a bulk one-point function of boundary LCFT.

The purpose of the present work is to pursue solving Liouville theory on a simply connected domain with boundary, in the special case where the Liouville potential is only present on the boundary, see the Liouville action \eqref{liouville_action} below. In the study of boundary LCFT there are four basic correlation functions that must be computed: the bulk one-point function, the bulk-boundary correlator, and the boundary two-point and three-point functions. For the last two correlations we allow the freedom to choose different cosmological constants for each connected component of the boundary, see equation \eqref{def:mu_boundary} in Definition \ref{def_four_correls} below. See also Figure \ref{figure} for an illustration of the four cases. Taking as an input our previous works \cite{remy, interval}, we will thus compute all the basic correlations of boundary LCFT. In a future work we plan to address the same problem in the more general setting where there is also a bulk Liouville potential in the action. Lastly for finding higher order correlations or correlations in higher genus one needs in principle to apply the conformal bootstrap method of \cite{BPZ}. At the level of mathematics the case of boundaryless surfaces has been solved in the recent breakthroughs \cite{GKRV, GKRV2} and the study of the boundary case has been very recently initiated in \cite{Wu}. See Sections \ref{app:bulk_boundary} and \ref{app:bootstrap} for more on these two outlooks.

The key probabilistic object required to define LCFT probabilistically is the Gaussian multiplicative chaos (GMC) measure, which formally corresponds to exponentiating a log-correlated Gaussian field. Since the pioneering work of Kahane \cite{Kah}, it is well understood how to define this object using a suitable regularization procedure \cite{Ber, review}. GMC measures are now an extremely well studied object in probability theory and appear in many apparently unrelated problems such as 3d turbulence, mathematical finance, statistical physics, two-dimensional random geometry and probabilistic LCFT (see the review \cite{review} and references therein). One illustration is the Fyodorov-Bouchaud formula giving the law of the total mass of the GMC measure on the unit circle that was first proposed in statistical physics \cite{FyBo} in the context of random energy models. It was proved in \cite{remy} by viewing it as the bulk one-point function of boundary LCFT - a quantity derived in theoretical physics in \cite{FZZ} - and by implementing the BPZ differential equations of CFT in a probabilistic framework. This connection between \cite{FZZ} and \cite{FyBo} was unknown to physicists. Along the same lines, our previous work \cite{interval} gives the distribution of the mass of GMC on the unit interval making again rigorous predictions of statistical physicists \cite{FLeR, Ostro1} using once more the BPZ equations (see also the related works \cite{Ostro2, Ostro_review}). Finally there is a link between these results on GMC and the behavior of the maximum of log-correlated fields and random matrix theory, see the discussions in \cite{remy, interval} and the references therein.

In the present paper we further uncover these connections between the theory of GMC measures and Liouville CFT. We show how the law of the total mass of GMC on the unit interval with insertions at the boundaries proved in \cite{interval} can be viewed as a special case of the boundary three-point function \eqref{main_th2_H}. Similarly the exact formula we derived for the bulk-boundary correlator \eqref{eq:main_th1} gives the law of the total mass of GMC on the unit circle with an insertion, solving a conjecture of Ostrovsky \cite{Ostro2}. Our third formula, the boundary two-point function \eqref{main_th2_R}, gives a very general result on the tail expansion of one-dimensional GMC measures. These connections are detailed in Sections \ref{app:link_int} and \ref{app:laws_GMC}. The study of boundary LCFT with boundary Liouville potential is thus the most general framework to understand the integrability of one-dimensional GMC measures. On another note, these results are connected to the study of Liouville conformal blocks, the most fundamental functions in the conformal bootstrap program of \cite{BPZ}. The formula \eqref{main_th2_R} for the boundary two-point function was a crucial input in the recent work \cite{Blocks} studying the one-point conformal block on the torus. In a follow-up work studying the modular invariance of these same conformal blocks, the formula \eqref{main_th2_H} for the boundary three-point will be required. See Section \ref{app:blocks} for more details on this connection.

Let us now introduce the framework of our paper. By conformal invariance we can work equivalently on the upper half plane $\mathbb{H} = \{z \in \mathbb{C} \:| \: \mathrm{Im}(z) >0 \}$ or on the unit disk $\mathbb{D} = \{z \in \mathbb{C} \:| \: |z| <1 \}$ but for almost all of this paper we will work on $\mathbb{H}$. We use the notations $\overline{\mathbb{H}} = \mathbb{H} \cup \mathbb{R}$,  $\partial \mathbb{D}$ for the unit circle and similarly $\overline{\mathbb{D}} = \mathbb{D} \cup \partial \mathbb{D}$. In theoretical physics Liouville theory is defined using the path integral formalism. Let us fix $N$ bulk insertion points $z_i \in \mathbb{H}$ of associated weights $\alpha_i \in \mathbb{R}$ and $M$ boundary insertions points $s_j \in \mathbb{R}$ with weight $\beta_j \in \mathbb{R}$. In physics the correlation function of LCFT at these points is defined using the following infinite dimensional integral on the space of maps $X: \mathbb{H} \mapsto \mathbb{R}$,
\begin{equation}\label{path_integral}
\left \langle \prod_{i=1}^N e^{\alpha_i \phi(z_i)} \prod_{j=1}^M e^{\frac{\beta_j}{2} \phi(s_j)} \right \rangle = \int_{X: \mathbb{H} \mapsto \mathbb{R}} DX \prod_{i=1}^N  e^{\alpha_i X(z_i)} \prod_{j=1}^M e^{\frac{\beta_j}{2} X(s_j)} e^{-S_L(X)},
\end{equation}
where $DX$ is a formal uniform measure on the maps $X$ and $S_L(X)$ is the Liouville action given by:
\begin{equation}\label{liouville_action}
S_L(X) = \frac{1}{4 \pi} \int_{\mathbb{H}} \left( \vert \partial^g X \vert^2 + Q R_g X \right) d \lambda_g  + \frac{1}{2 \pi} \int_{ \mathbb{R}} \left( Q K_g X + 2 \pi \mu_{B} e^{\frac{\gamma}{2} X} \right) d \lambda_{\partial g}.
\end{equation}
The most fundamental parameter is $\gamma \in (0,2)$, which is related to $Q$ and to the central charge $c_L$ by: 
\begin{equation}\label{def_gamma_Q}
Q = \frac{\gamma}{2} + \frac{2}{\gamma}, \quad \quad  c_L = 1 + 6 Q^2.
\end{equation}
For a choice $g$ of background Riemannian metric on $\mathbb{H}$, $\partial^g$, $R_g$, $K_g$, $d \lambda_g $, $d \lambda_{\partial g}$ respectively stand for the gradient, Ricci curvature, geodesic curvature of the boundary, volume form and line element in the metric $g$. The precise choice of $g$ is irrelevant thanks to the Weyl anomaly proven in \cite[Proposition 3.7]{Disk}, see also the proof of Lemma \ref{link_D_H} for a concrete change of metrics.  $\mu_{B}$ is the boundary cosmological constant tuning the interaction strength of the Liouville potential $e^{\frac{\gamma}{2} X}$ integrated over the boundary. It will be chosen either to be a fixed positive number or more generally a function $ \mu_{B}: \mathbb{R} \mapsto \mathbb{C}$ constraint to be constant in between two consecutive insertion points $s_j$ on $\mathbb{R}$, see equation \eqref{def:mu_boundary}. In a more general setup there would also be a bulk cosmological constant but we set it here to zero, see Section \ref{app:bulk_boundary} for more details.
Of course since the path integral \eqref{path_integral} does not make rigorous sense we will rely on the construction of \cite{Disk} to obtain a valid probabilistic definition for these correlation functions. A sufficient requirement for a correlation to be well-defined is that the following Seiberg bounds must hold,
\begin{align}\label{sieberg}
\sum_{i=1}^N \alpha_i + \sum_{j=1}^M \frac{\beta_j}{2} > Q, \quad \forall j, \: \beta_j < Q,
\end{align}
although this can be lifted in some sense by analytic continuation, see the explanations below Definition \ref{def_four_correls}. Notice here that we do not have the condition $\alpha_i < Q$ present in \cite{Disk} as we do not have a bulk Liouville potential.
One of the key properties of a CFT is that its correlations behave as conformal tensors under conformal automorphism. This has indeed been checked for the probabilistic LCFT in \cite[Theorem~3.5]{Sphere} and \cite[Theorem~3.5]{Disk}. Since our correlation functions of interest contain at most one bulk and one boundary or three boundary points, this behavior under conformal maps immediately determines the dependence on the position of the marked points of the correlations. This is also precisely the reason why these are the basic correlations of the theory, as in a case with more marked points the conformal automorphisms would not suffice to pin down the dependence on the position of the points. We thus perform this reduction for our four basic correlations and reduce each of their expressions to a single constant known as a structure constant.
\begin{itemize}
\item \textbf{Bulk one-point function.} For $z \in \mathbb{H}$:
\begin{equation}\label{c1}
\left \langle e^{\alpha \phi(z)} \right \rangle = \frac{U(\alpha)}{\vert z - \overline{z} \vert^{2 \Delta_{\alpha}}}.
\end{equation}
\item \textbf{Bulk-boundary correlator.} For $ z \in \mathbb{H}$, $s \in \mathbb{R}$:
\begin{equation}\label{c2}
\left \langle e^{\alpha \phi(z)} e^{\frac{\beta}{2} \phi(s)} \right \rangle = \frac{G(\alpha, \beta)}{\vert z - \overline{z} \vert^{2 \Delta_{\alpha} - \Delta_{\beta}}\vert z -s \vert^{2 \Delta_{\beta}}}.
\end{equation}
\item \textbf{Boundary two-point function.} For $s_1, s_2 \in \mathbb{R}$:
\begin{equation}\label{c3}
\left \langle e^{\frac{\beta}{2} \phi(s_1)} e^{\frac{\beta}{2} \phi(s_2)} \right \rangle = \frac{R(\beta, \mu_1, \mu_2)}{\vert s_1 - s_2 \vert^{2 \Delta_{\beta}}}.
\end{equation}  
\item \textbf{Boundary three-point function.} For  $s_1, s_2, s_3 \in \mathbb{R}$:
\begin{equation}\label{c4}
\left \langle e^{\frac{\beta_1}{2} \phi(s_1)} e^{\frac{\beta_2}{2} \phi(s_2)} e^{\frac{\beta_3}{2} \phi(s_3)} \right \rangle = \frac{H^{(\beta_1, \beta_2, \beta_3)}_{(\mu_1, \mu_2, \mu_3)}}{\vert s_1 - s_2 \vert^{\Delta_1 + \Delta_2 - \Delta_3} \vert s_1 - s_3 \vert^{\Delta_1 + \Delta_3 - \Delta_2} \vert s_2 - s_3 \vert^{\Delta_2 + \Delta_3 - \Delta_1} }.
\end{equation}  
\end{itemize}

We have used the notations $\Delta_{\alpha} = \frac{\alpha}{2}(Q - \frac{\alpha}{2})$, $\Delta_{\beta} = \frac{\beta}{2}(Q - \frac{\beta}{2})$, and $\Delta_i = \frac{\beta_i}{2}(Q - \frac{\beta_i}{2})$. The parameters $\mu_1, \mu_2, \mu_3$ will correspond to the values taken by $\mu_B$ in between boundary insertions, see \eqref{def:mu_boundary}. Each of the four structure constants $U, G, R, H$ will then have a definition involving GMC given in Definition \ref{def_four_correls}. Our main results Theorem \ref{main_th1} and Theorem \ref{main_th2} state that these probabilistic definitions using GMC match the exact formulas predicted in physics in \cite{FZZ, Ostro2, bulk_boundary, three_point}. The picture below summarizes these four cases. We have drawn it on the disk for more clarity.

\begin{figure}[!htp]
\centering
\includegraphics[width=0.31\linewidth]{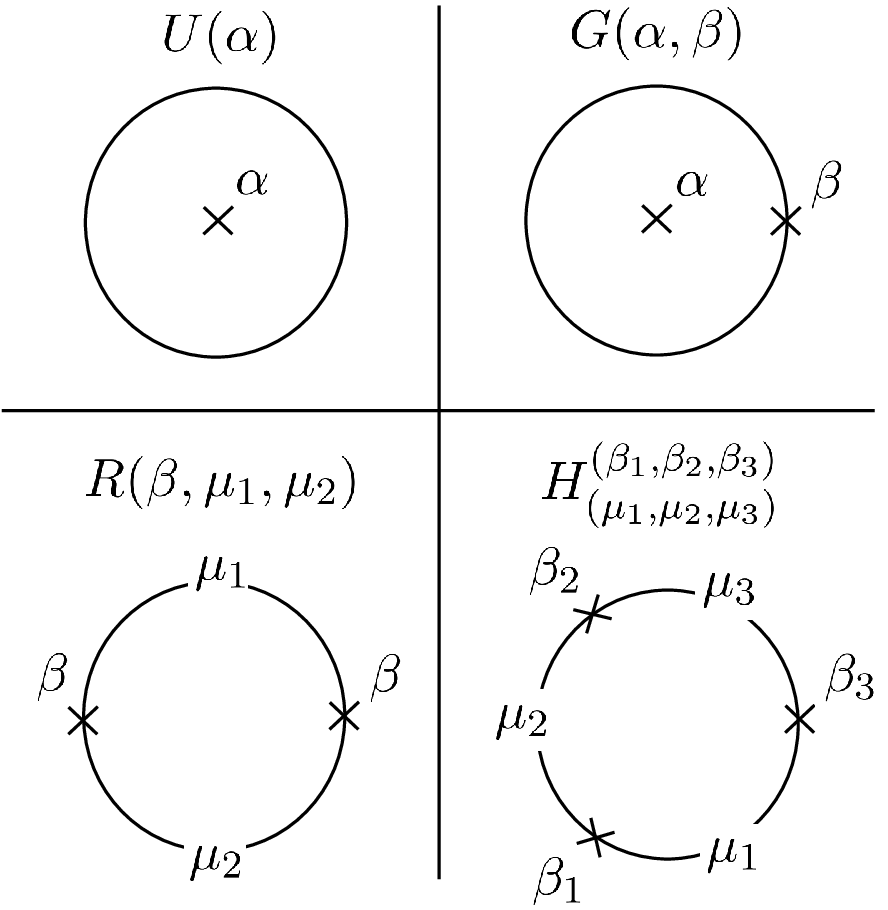}
\caption{Structure constants for boundary Liouville theory}
\label{figure}
\end{figure}

\subsection{Probabilistic definitions}
We will now introduce the two probabilistic objects required to rigorously define the four structure constants $U, G, R, H$, namely the Gaussian free field (GFF) and the Gaussian multiplicative chaos (GMC).
We work on the domain $\mathbb{H}$, viewing $\mathbb{H}$ as being equipped with the following background metric $g$, written here in diagonal form $g = g(x) dx^2$ with
\begin{equation}\label{def_metric_g}
g(x) = \frac{1}{|x|_+^4}, \quad \text{where} \quad |x|_+ := \max(|x|,1).
\end{equation}
This choice is convenient to work with because it will make the computations work in the same way as in \cite{DOZZ2} and \cite{interval}. The next definition introduces our GFF.
\begin{definition}\label{def_GFF} (Gaussian free field on $\mathbb{H}$) The Gaussian free field $X$ is the centered Gaussian process on $\mathbb{H}$ with covariance given by, for $x,y \in \mathbb{H}$:
\begin{equation}\label{covariance}
\E[X(x)X(y)] = \ln \frac{1}{|x-y||x-\bar{y}|} -\frac{1}{2}\ln g(x) - \frac{1}{2} \ln g(y).
\end{equation}
Since the variance at each point is infinite, $X$ is not defined pointwise and exists as a random distribution. It also satisfies:
\begin{equation}
\int_0^{\pi} X(e^{i \theta}) d \theta = 0.
\end{equation}
\end{definition}

See Section \ref{sec:useful_facts} for how to construct this GFF $X$ from the standard Neumann boundary (or free boundary) GFF on the disk $\mathbb{D}$. We now define the GMC measure on $\mathbb{R}$, the boundary of our domain $\mathbb{H}$.
\begin{definition}\label{def_GMC} (Gaussian multiplicative chaos)
Fix a $\gamma \in (0,2)$. The Gaussian multiplicative chaos measure associated to the field $X$ is defined by the following limit,
\begin{equation}
e^{\frac{\gamma}{2} X(x)} dx = \lim_{\epsilon \rightarrow 0} e^{\frac{\gamma}{2} X_{\epsilon}(x) - \frac{\gamma^2}{8} \mathbb{E}[X_{\epsilon}(x)^2]} dx,
\end{equation}
where the convergence is in probability and in the sense of weak convergence of measures on $\mathbb{R}$. Here $X_{\epsilon}$ is a suitable regularization of the field. More precisely, for a continuous compactly supported function $f$ on $\mathbb{R}$, the following convergence holds in probability:
\begin{equation}
\int_{\mathbb{R}} f(x) e^{\frac{\gamma}{2} X(x)} dx = \lim_{\epsilon \rightarrow 0}  \int_{\mathbb{R}} f(x) e^{\frac{\gamma}{2} X_{\epsilon}(x) - \frac{\gamma^2}{8} \mathbb{E}[X_{\epsilon}(x)^2]} dx.
\end{equation}
\end{definition}
For an elementary proof of this convergence and examples of smoothings of the field $X_{\epsilon}$, see for instance \cite{Ber} and references therein.

In order to define the boundary two-point and three-point functions we will consider parameters $\mu_1, \mu_2, \mu_3$ in $ \mathbb{C}$ corresponding to the values taken by $\mu_B$ in the Liouville action \eqref{liouville_action} on the different arcs in between insertion points (see also Figure \ref{figure}). To be able to choose a suitable branch cut of the logarithm to define the moment of GMC below, we introduce the following condition on the parameters $\mu_1, \mu_2, \mu_3$ which we will refer to as the half-space condition.

\begin{definition}(Half-space condition for the $\mu_i$)\label{half-space}
Consider $\mu_1, \mu_2, \mu_3 \in \mathbb{C}$. We say that $(\mu_i)_{i=1,2,3}$ satisfies the half-space condition if there exists a half-space $\mathcal{H}$ of $\mathbb{C}$ whose boundary is a line passing through the origin not equal to the real axis and satisfying the following. The half-space $\mathcal{H}$ does not contain the half-line $(-\infty,0)$. Each $\mu_i$ is contained in $\overline{\mathcal{H}}$ (the half-space with its boundary included) and the sum $\mu_1 + \mu_2 + \mu_3$ is strictly contained in $\mathcal{H}$. We will also refer to the half-space condition for a pair $\mu_1,\mu_2 \in \mathbb{C}$ which will be the condition above with $\mu_3$ set to $0$.
\end{definition}
The point of having this condition is two-fold. First we can always choose unambiguously the argument of each $\mu_i$ by choosing as branch cut the line $(-\infty,0)$, since the half-space $\mathcal{H}$ avoids this line. Second, any positive linear combination of the $\mu_i$ will always be contained in $\mathcal{H}$, and therefore such a linear combination is a complex number whose argument can also be defined using the branch cut $(-\infty,0)$. It will be very convenient to introduce variables $\sigma_i$ corresponding to the argument of the $\mu_i$ defined in the following way.

\begin{definition}\label{def_sigma}
Consider $\mu_1, \mu_2, \mu_3 \in \mathbb{C}$ obeying the half-space condition of Definition \ref{half-space}. We introduce the variables $\sigma_1, \sigma_2, \sigma_3 \in \mathbb{C}$ defined by the relations
\begin{equation}
\mu_{1} = e^{i \pi \gamma (\sigma_1-\frac{Q}{2})}, \quad \mu_{2} = e^{i \pi \gamma (\sigma_2-\frac{Q}{2})}, \quad \mu_{3} = e^{i \pi \gamma (\sigma_3-\frac{Q}{2})},
\end{equation}
with the convention that for positive $\mu_i$ one has $\mathrm{Re}(\sigma_i) = \frac{Q}{2}$. 
\end{definition}
With this definition at hand, remark that if $\sigma_i \in (-\frac{1}{2\gamma} +\frac{Q}{2}, \frac{1}{2\gamma} +\frac{Q}{2} ) $ for $i =1,2,3$, then $\mathrm{Re}( \mu_i) > 0$ and thus the half-space condition is satisfied with $\mathcal{H}$ being the right half-space of $\mathbb{C}$. We can now introduce the probabilistic definition of the four structure constants $U, G, R, H$ using moments of GMC on $\mathbb{R}$. Following the notations of \cite{DOZZ2}, it is convenient to work with the four quantities $\overline{U}, \overline{G}, \overline{R}, \overline{H}$ which will be purely defined as moments of GMC on $\mathbb{R}$ and be each related to the corresponding  $U, G, R, H$ by an explicit prefactor.

\begin{definition}\label{def_four_correls} (Correlation functions of Liouville theory on $\mathbb{H}$)
Fix $\gamma \in (0,2)$. Consider parameters $\alpha, \beta, \beta_1, \beta_2, \beta_3 \in \mathbb{R}$, $\mu_{B} \in (0,+\infty)$, and $\mu_1, \mu_2, \mu_3  \in  \mathbb{C}$. The four correlation functions $U, G, R, H$ have the following probabilistic definitions.
\begin{itemize}
\item $U(\alpha) = \frac{2}{\gamma} \Gamma(\frac{2(\alpha - Q)}{\gamma}) \left( \mu_{B}^{\frac{2(Q-\alpha)}{\gamma}} \right)  \overline{U}(\alpha)$  where for $\alpha > \frac{\gamma}{2}$:
\begin{equation}\label{cc1}
\overline{U}(\alpha) = \mathbb{E} \left[ \left( \int_{\mathbb{R}}  \frac{g(x)^{\frac{\gamma}{4}(\frac{2}{\gamma}-\alpha)}}{|x-i|^{\gamma\alpha  } } e^{\frac{\gamma}{2} X(x)}  dx  \right)^{\frac{2(Q -\alpha)}{\gamma}} \right].
\end{equation}
\item $G(\alpha, \beta) = \frac{2}{\gamma} \Gamma(\frac{2\alpha + \beta - 2Q}{\gamma}) \left( \mu_{B}^{\frac{2Q-2\alpha-\beta}{\gamma}} \right) \overline{G}(\alpha, \beta)$  where for $\beta < Q$, $\frac{\gamma}{2} - \alpha < \frac{\beta}{2} < \alpha $: 
\begin{equation}\label{cc2}
\overline{G}(\alpha, \beta )  =  \mathbb{E} \left[ \left(\int_{\mathbb{R}} \frac{g(x)^{\frac{\gamma}{4}(\frac{2}{\gamma}-\alpha-\frac{\beta}{2})}}{|x-i|^{\gamma\alpha  } }  e^{\frac{\gamma}{2} X(x) }  d x \right)^{\frac{2}{\gamma}(Q-\alpha-\frac{\beta}{2})} \right].
\end{equation}
\item $H^{(\beta_1, \beta_2, \beta_3)}_{(\mu_1, \mu_2, \mu_3)} = \frac{2}{\gamma} \Gamma(\frac{\beta_1+ \beta_2 + \beta_3-2Q}{\gamma}) \overline{H}^{(\beta_1, \beta_2, \beta_3)}_{(\mu_1, \mu_2, \mu_3)}$  where in the following range of parameters,
\begin{align}\label{para_H}
 &(\mu_i)_{i=1,2,3} \: \: \: \text{satisfies Definition \ref{half-space},} \quad \beta_i < Q,  \quad   {\frac{1}{\gamma}(2 Q  - \sum_{i=1}^3 \beta_i)} < \frac{4}{\gamma^2} \wedge \min_i \frac{2}{\gamma}(Q - \beta_i),
\end{align}
one can define:
\begin{equation}\label{cc3}
\overline{H}^{(\beta_1, \beta_2, \beta_3)}_{(\mu_1, \mu_2, \mu_3)}  = \mathbb{E} \left[ \left( \intr \frac{g(x)^{\frac{\gamma}{8}(\frac{4}{\gamma}- \sum_{i=1}^3 \beta_i )} }{|x|^{\frac{\gamma \beta_1 }{2}}|x-1|^{\frac{\gamma \beta_2 }{2}}}  e^{\frac{\gamma}{2} X(x)} d \mu(x) \right)^{\frac{1}{\gamma}(2 Q  - \sum_{i=1}^3 \beta_i)} \right].
\end{equation}
The dependence on the parameters $\mu_1, \mu_2, \mu_3$  appears through the measure:
\begin{equation}\label{def:mu_boundary}
d\mu(x) = \mu_1\mathbf{1}_{(-\infty,0)}(x)dx + \mu_2\mathbf{1}_{(0,1)}(x)dx + \mu_3 \mathbf{1}_{(1,\infty)}(x)dx.
\end{equation}
The GMC integral inside the expectation is a complex number avoiding $(-\infty,0)$. To define its fractional power we choose its argument in $(-\pi,\pi)$.
\item $R(\beta, \mu_1, \mu_2) = -\Gamma(1-\frac{2(Q-\beta)}{\gamma}) \overline{R}(\beta, \mu_1, \mu_2)$, where $\overline{R}(\beta, \mu_1, \mu_2)$ is defined for $\beta \in (\frac{\gamma}{2},Q)$ and  $\mu_1,\mu_2$ obeying the constraint of Definition \ref{half-space} by the following limiting procedure.
Consider $\frac{\gamma}{2} < \beta_2 < \beta <Q$ and $ \beta - \beta_2 < \beta_3 < Q$. Then the following limit exists and we set:
\begin{equation}\label{cc4}
\overline{R}(\beta, \mu_1, \mu_2) := \frac{1}{2 (Q-\beta)} \lim_{\beta_3 \downarrow \beta - \beta_2} (\beta_2 + \beta_3 - \beta) \overline{H}^{(\beta, \beta_2, \beta_3)}_{(\mu_1, \mu_2, \mu_3)}.
\end{equation}
\end{itemize}
\end{definition}

To obtain these definitions from \cite{Disk} one needs to set $\mu =0$ in the equation of \cite[Proposition~3.2]{Disk} and also use $\mathbb{H}$ instead of $\mathbb{D}$ as base domain. Let us recall the main steps given in \cite{Disk} to justify why these moments of the GMC measure are the correct probabilistic interpretation of \eqref{path_integral}. The path integral \eqref{path_integral} is interpreted in the probabilistic context as a functional of the GFF $X$ on $\mathbb{H}$. The formal measure $DX$ combined with the gradient squared of the Liouville action is replaced by an expectation over $X + c$, where $c$ is distributed according to the Lebesgue measure on $\mathbb{R}$. This constant $c$ is the so-called zero mode in physics and must be integrated over to obtain a conformally invariant theory. Performing this procedure gives:
\begin{equation*}
\left \langle \prod_{i=1}^N e^{\alpha_i \phi(z_i)} \prod_{j=1}^M e^{\frac{\beta_j}{2} \phi(s_j)} \right \rangle = \int_{\mathbb{R}} dc e^{- Q c} \E \left[  \prod_{i=1}^N  e^{\alpha_i (X(z_i) +c) } \prod_{j=1}^M e^{\frac{\beta_j}{2} (X(s_j) +c)} e^{- e^{\frac{\gamma c}{2}} \int_{ \mathbb{R}}  \mu_{B} e^{\frac{\gamma}{2} X(x)}g^{1/2}(x)dx } \right].
\end{equation*}
The right hand side is now a well-defined quantity, provided that one defines all exponentials of $X$ by a limiting procedure as performed in Definition \ref{def_GMC}. Then by a simple change of variable one can compute the integral over $c$ and obtain a Gamma function integral times a moment of GMC on $\mathbb{R}$. Using the notation $s = \frac{2}{\gamma}(\sum_{i=1}^N \alpha_i + \sum_{j=1}^M \frac{\beta_j}{2} -Q )$ one gets:
\begin{align}
&\left \langle \prod_{i=1}^N e^{\alpha_i \phi(z_i)} \prod_{j=1}^M e^{\frac{\beta_j}{2} \phi(s_j)} \right \rangle = \frac{2}{\gamma}\int_0^{\infty} du u^{s - 1} e^{-u} \E \left[  \prod_{i=1}^N  e^{\alpha_i X(z_i) } \prod_{j=1}^M e^{\frac{\beta_j}{2} X(s_j) } \left(\int_{ \mathbb{R}}  \mu_{B} e^{\frac{\gamma}{2} X(x)}g^{1/2}(x)dx \right)^{-s} \right] \nonumber \\
& = \frac{2}{\gamma} \Gamma(s) \E \left[  \left(\int_{ \mathbb{R}}   \mu_{B} \prod_{i=1}^N \left( \frac{|x|_+ |z_i|_+}{|x-z_i|} \right)^{\alpha_i \gamma}  \prod_{j=1}^M \left( \frac{|x|_+ |s_j|_+}{|x-s_j|} \right)^{\frac{\beta_j \gamma}{2}}   e^{\frac{\gamma}{2} X(x)}g^{1/2}(x)dx \right)^{-s} \right]. \label{eq:intro_moment}
\end{align}

To obtain the second line from the first we have applied the Girsanov Theorem \ref{girsanov} to the insertions $\prod_{i=1}^N  e^{\alpha_i X(z_i) } \prod_{j=1}^M e^{\frac{\beta_j}{2} X(s_j) }$. Let us take a look at the bounds on the parameters $\alpha_i, \beta_j$ for  \eqref{eq:intro_moment} to be finite. A sufficient condition is given by the Seiberg bounds \eqref{sieberg}. But now that we have an expression involving a Gamma function times a moment of GMC, we can analytically continue the Gamma function and look only at the bounds for GMC moments. This leads to the extended Seiberg bounds or unit volume bounds:
\begin{align}\label{bounds_unit_volume}
 {\frac{1}{\gamma}(2 Q - 2 \sum_{i=1}^N \alpha_i - \sum_{j=1}^M \beta_i)} < \frac{4}{\gamma^2} \wedge \min_i \frac{2}{\gamma}(Q - \beta_i), \quad \forall j, \: \beta_j < Q.
\end{align}
The proof of finiteness of \eqref{eq:intro_moment} under these bounds has been performed in \cite[Corollary 3.10]{Disk}, except that as for $\overline{H}$ we will sometimes need to choose the $\mu_B$ to be piecewise complex valued on $\mathbb{R}$ which gives the extra constraint that the $\mu_i$ in $\mathbb{C}$ need to obey the half-space condition of Definition \ref{half-space}. We give a short adaptation of the proof of \cite{Disk} for this case in Proposition \ref{lem:GMC-moment}.

We can now specialize \eqref{eq:intro_moment} and \eqref{bounds_unit_volume} to define $U$, $G$ and $H$. In our expressions the locations of the insertions have been chosen as follows. For $\overline{U}$ and $ \overline{G}$ the bulk insertion is at $i$ and the boundary insertion of $\overline{G}$ is at infinity. For $\overline{H}$ the three boundary insertions have been placed at $0$, $1$ and infinity. An alternative choice would have been to write everything on the disk $\mathbb{D}$ as shown on Figure 1. In this case it is natural to place the bulk insertion of $\overline{U}$ and $ \overline{G}$ at $0$ and the boundary insertion of $\overline{G}$ at $1$. See the statement of Lemma \ref{link_D_H} for the expressions of $\overline{U}$ and $ \overline{G}$ written as moments of GMC on the unit circle. 

For $\overline{H}$ we have allowed the freedom to choose different cosmological constants $\mu_i$ on each arc of the boundary $\mathbb{R}$ in between insertions, namely the segments $(-\infty,0)$, $(0,1)$ and $(1, +\infty)$. The precise definition is given by equation \eqref{def:mu_boundary} above. This extra degree of freedom is standard in the physics literature \cite{FZZ, three_point} but was not considered in \cite{Disk}. In a similar way $\overline{R}$ will be dependent in the general case on a $\mu_1$ and $\mu_2$. For the correlations $\overline{U}$ and $ \overline{G}$ since the boundary $\mathbb{R}$ is not separated in disjoint components by boundary insertions, we fix $\mu_B$ to be constant on all $\mathbb{R}$. It then turns out that the dependence of these two correlations on $\mu_B$ is simply a fractional power of $\mu_B$ and we directly include it in the definition of $U$ and $G$. On the other hand the dependence of $R$ and $H$ on the $\mu_i$ will be non-trivial.

 Lastly let us discuss the case of the function $\overline{R}$ which cannot be defined directly using \eqref{eq:intro_moment}. $\overline{R}$ needs to be constructed either by a limiting procedure \eqref{cc4} starting from $\overline{H}$ or directly by the probabilistic expression given by \eqref{def_R2}. Lemma \ref{lim_H_R} then asserts these two definitions match. This phenomenon was first studied in \cite{DOZZ2} for the case of the two-point function on the Riemann sphere. See Section \ref{sec_def_reflection} for a detailed explanation in our case. Heuristically the limit \eqref{cc4} corresponds to zooming around a boundary point of $\overline{H}$ of weight $\beta$ with parameters $\mu_1$ to the left and $\mu_2$ to the right. This also gives a heuristic explanation of why the boundary two-point function $R$ only depends on one parameter $\beta$ instead of two.

\subsection{Main theorems}
In order to state our main results, we need to introduce the following special functions. For all $\gamma \in (0,2) $ and for $\mathrm{Re}(x) >0$, $\Gamma_{\frac{\gamma}{2}}(x)$ is defined by the following integral formula:
\begin{equation}
\ln \Gamma_{\frac{\gamma}{2}}(x) = \int_0^{\infty} \frac{dt}{t} \left[ \frac{ e^{-xt} -e^{- \frac{Qt}{2}}   }{(1 - e^{- \frac{\gamma t}{2}})(1 - e^{- \frac{2t}{\gamma}})} - \frac{( \frac{Q}{2} -x)^2 }{2}e^{-t} + \frac{ x -\frac{Q}{2}  }{t} \right].
\end{equation} 
This function is a natural generalization of the standard Gamma function. It admits a meromorphic extension to $\mathbb{C}$ with simple poles on the lattice $\{ - n \frac{\gamma}{2} - m \frac{2}{\gamma} | n,m \in \mathbb{N} \}$ where here and throughout this paper $\mathbb{N}$ denotes non-negative integers. Consider similarly the function $S_{\frac{\gamma}{2}}(x)$ defined for $\gamma \in (0,2)$ and $\mathrm{Re}(x) \in (0, Q)$ by
\begin{equation}
S_{\frac{\gamma}{2}}(x) = \frac{\Gamma_{\frac{\gamma}{2}}(x)}{\Gamma_{\frac{\gamma}{2}}(Q -x)},
\end{equation}
which also admits a meromorphic extension to all of $ \mathbb{C}$. See Section \ref{sec_special_func} for more details on $\Gamma_{\frac{\gamma}{2}}$ and $S_{\frac{\gamma}{2}}$. Now let $\overline{\beta} = \beta_1 + \beta_2 + \beta_3$ and consider the six parameters $\beta_1, \beta_2, \beta_3, \sigma_1, \sigma_2, \sigma_3 \in \mathbb{C}^6$ constrained to satisfy the condition 
\begin{equation}\label{eq:cond_H}
\mathrm{Re}\left( Q - \sigma_3 + \sigma_2 -\frac{\beta_2}{2} \right) > 0.
\end{equation}
Under the condition \eqref{eq:cond_H} one can define the function:
\begin{align}\label{formule_PT}
&\mathcal{I}
\begin{pmatrix}
\beta_1 , \beta_2, \beta_3 \\
\sigma_1,  \sigma_2,   \sigma_3 
\end{pmatrix}\\
   & =\frac{(2\pi)^{\frac{2Q-\overline{\beta}}{\gamma}+1}(\frac{2}{\gamma})^{(\frac{\gamma}{2}-\frac{2}{\gamma})(Q-\frac{\overline{\beta}}{2})-1}}{\Gamma(1-\frac{\gamma^2}{4})^{\frac{2Q-\overline{\beta}}{\gamma}}\Gamma(\frac{\overline{\beta}-2Q}{\gamma})} 
 \frac{\Gamma_{\frac{\gamma}{2}}(2Q-\frac{\overline{\beta}}{2})\Gamma_{\frac{\gamma}{2}}(\frac{\beta_1+\beta_3-\beta_2}{2})\Gamma_{\frac{\gamma}{2}}(Q-\frac{\beta_1+\beta_2-\beta_3}{2})\Gamma_{\frac{\gamma}{2}}(Q-\frac{\beta_2+\beta_3-\beta_1}{2})}{\Gamma_{\frac{\gamma}{2}}(Q) \Gamma_{\frac{\gamma}{2}}(Q-\beta_1) \Gamma_{\frac{\gamma}{2}}(Q-\beta_2) \Gamma_{\frac{\gamma}{2}}(Q-\beta_3)} \nonumber \\ 
&\times \frac{e^{i\frac{\pi}{2}(-(2Q-\frac{\beta_1}{2}-\sigma_1-\sigma_2)(Q-\frac{\beta_1}{2}-\sigma_1-\sigma_2) + (Q+\frac{\beta_2}{2}-\sigma_2-\sigma_3)(\frac{\beta_2}{2}-\sigma_2-\sigma_3)+(Q+\frac{\beta_3}{2}-\sigma_1-\sigma_3)(\frac{\beta_3}{2}-\sigma_1-\sigma_3) -2\sigma_3(2\sigma_3-Q)        )}}{S_{\frac{\gamma}{2}}(\frac{\beta_1}{2}+\sigma_1-\sigma_2)  S_{\frac{\gamma}{2}}(\frac{\beta_3}{2}+\sigma_3-\sigma_1) } \nonumber \\
&\times \int_{\mathcal{C}} \frac{S_{\frac{\gamma}{2}}(Q-\frac{\beta_2}{2}+\sigma_3-\sigma_2+r) S_{\frac{\gamma}{2}}(\frac{\beta_3}{2}+\sigma_3-\sigma_1+r) S_{\frac{\gamma}{2}}(Q-\frac{\beta_3}{2}+\sigma_3-\sigma_1+r)}{S_{\frac{\gamma}{2}}(Q+\frac{\beta_1}{2}-\frac{\beta_2}{2}+\sigma_3-\sigma_1+r) S_{\frac{\gamma}{2}}(2Q-\frac{\beta_1}{2}-\frac{\beta_2}{2}+\sigma_3-\sigma_1+r) S_{\frac{\gamma}{2}}(Q+r)}e^{i\pi(-\frac{\beta_2}{2}+\sigma_2-\sigma_3)r} \frac{dr}{i}. \nonumber
\end{align}
In the integral appearing above the contour $\mathcal{C}$ goes from $-i \infty$ to $ i \infty$ passing to the right of the poles at $r = -(Q-\frac{\beta_2}{2}+\sigma_3-\sigma_2) -n\frac{\gamma}{2}-m\frac{2}{\gamma}$, $r = -(\frac{\beta_3}{2}+\sigma_3-\sigma_1)-n\frac{\gamma}{2}-m\frac{2}{\gamma}$, $r= -(Q-\frac{\beta_3}{2}+\sigma_3-\sigma_1)-n\frac{\gamma}{2}-m\frac{2}{\gamma}$ and to the left of the poles at $r=-(\frac{\beta_1}{2}-\frac{\beta_2}{2}+\sigma_3-\sigma_1)+n\frac{\gamma}{2}+m\frac{2}{\gamma}$, $r = -(Q-\frac{\beta_1}{2}-\frac{\beta_2}{2}+\sigma_3-\sigma_1)+n\frac{\gamma}{2}+m\frac{2}{\gamma} $, $r=n\frac{\gamma}{2}+m\frac{2}{\gamma} $ with $m,n \in \mathbb{N}$. See Appendix \ref{formula_I} for an in depth study of this formula including the check that the integral over $\mathcal{C}$ is converging and an analytic continuation to $\mathbb{C}^6$ removing the constraint of \eqref{eq:cond_H}.

We can now state our main results. For the sake of completeness we first recall the result of \cite{remy}.

\begin{theorem}\label{FyBo}(Bulk one-point function, R. 2017 \cite{remy}) For $\gamma \in (0,2)$,  $\alpha > \frac{\gamma}{2}$, one has: 
\begin{equation}
\overline{U}(\alpha) =  \left(\frac{2^{-\frac{\gamma\alpha}{2}} 2\pi}{\Gamma(1-\frac{\gamma^2}{4})} \right)^{\frac{2}{\gamma}(Q-\alpha)}\Gamma(\frac{\gamma \alpha}{2} - \frac{\gamma^2}{4}).
\end{equation}
\end{theorem}
Now the main result of the present work is to provide expressions for the remaining three structure constants. We will prove the following theorems.
\begin{theorem}\label{main_th1} (Bulk-boundary correlator) For $\gamma \in (0,2)$, $\beta < Q $, $ \frac{\gamma}{2} - \alpha < \frac{\beta}{2} < \alpha$, one has:
\begin{equation}\label{eq:main_th1}
\overline{G}(\alpha, \beta)=  \left( \frac{2^{\frac{\gamma}{2}(\frac{\beta}{2} - \alpha)} 2\pi}{\Gamma(1-\frac{\gamma^2}{4})} \right)^{\frac{2}{\gamma}(Q-\alpha-\frac{\beta}{2})} \frac{ \Gamma(\frac{\gamma\alpha}{2}+\frac{\gamma\beta}{4}-\frac{\gamma^2}{4}) \Gamma_{\frac{\gamma}{2}}(\alpha-\frac{\beta}{2} )  \Gamma_{\frac{\gamma}{2}}(\alpha+\frac{\beta}{2}  ) \Gamma_{\frac{\gamma}{2}}(Q - \frac{\beta}{2}  )^2}{ \Gamma_{\frac{\gamma}{2}}(Q - \beta ) \Gamma_{\frac{\gamma}{2}}(\alpha )^2\Gamma_{\frac{\gamma}{2}}(Q) }.
\end{equation}
\end{theorem}

\begin{theorem}\label{main_th2} (Boundary two-point and three-point functions)
Consider $\gamma \in (0,2)$, $\beta \in (\frac{\gamma}{2}, Q)$, and  $\mu_1,\mu_2$ obeying the condition of Definition \ref{half-space}. Then one has: 
\begin{align}\label{main_th2_R}
\overline{R}(\beta, \mu_1, \mu_2) = \frac{ (2 \pi)^{ \frac{2}{\gamma}(Q -\beta ) -\frac{1}{2}} (\frac{2}{\gamma})^{ \frac{\gamma}{2}(Q -\beta ) -\frac{1}{2} }  }{(Q-\beta) \Gamma(1 -\frac{\gamma^2}{4}  )^{ \frac{2}{\gamma}(Q -\beta ) } } \frac{ \Gamma_{\frac{\gamma}{2}}(\beta - \frac{\gamma}{2}  ) e^{i \pi(\sigma_1+\sigma_2 -Q)(Q-\beta)}}{\Gamma_{\frac{\gamma}{2}}(Q- \beta ) S_{\frac{\gamma}{2}}(\frac{\beta}{2} + \sigma_2- \sigma_1)S_{\frac{\gamma}{2}}(\frac{\beta}{2} + \sigma_1- \sigma_2) }.
\end{align} 
Similarly, for $\beta_1, \beta_2, \beta_3$ and $\mu_1, \mu_2, \mu_3 $ satisfying the set of conditions \eqref{para_H},
\begin{align}\label{main_th2_H}
\overline{H}^{(\beta_1 , \beta_2, \beta_3)}_{(  \mu_1,  \mu_2,   \mu_3 )} =  \mathcal{I}
\begin{pmatrix}
\beta_1 , \beta_2, \beta_3 \\
\sigma_1,  \sigma_2,   \sigma_3 
\end{pmatrix}.
\end{align}
\end{theorem}

For a statement of these results as giving the law of a random variable involving GMC, see Section \ref{app:laws_GMC}. Let us now mention the physics references where these exact formulas have been proposed. The formulas for $\overline{U}$ and $\overline{G}$ appeared in statistical physics respectively in the works \cite{FyBo} and \cite{Ostro2}. These formulas also appeared in the works on the Liouville CFT side, respectively in \cite{FZZ} and \cite{bulk_boundary}, provided that one takes a suitable limit to set the bulk cosmological constant $\mu$ to $0$ in the expressions of  \cite{FZZ, bulk_boundary}. To the best of our knowledge the connection between the two sets of works \cite{FyBo, Ostro2} and \cite{FZZ, bulk_boundary} was unknown to physicists. Lastly the formulas for $\overline{R}$ and $\overline{H}$ were found respectively in \cite{FZZ} and \cite{three_point}, taking again the limit $\mu \rightarrow 0$. 

An important remark is that these exact formulas now provide an analytic continuation of the probabilistic definitions to the whole complex plane in all of the parameters $\alpha, \beta, \beta_i, \sigma_i$. Notice also this shows that the variables $\sigma_i$ are the correct parametrization of the boundary cosmological constants $\mu_i$ in order to obtain a meromorphic function (as in the original $\mu_i$ parameters the functions $\overline{R}$ and $\overline{H}$ are multivalued).

Before moving on to the proof of these results we will explain in the next three subsections how the boundary two-point function $\overline{R}$ can be viewed as a reflection coefficient, detail the applications and outlooks of our work, and lastly  present an outline of our proof strategy.

\subsection{The reflection coefficient}\label{sec_def_reflection}

We explain here how the boundary two-point function $\overline{R}(\beta, \mu_1, \mu_2)$, also known as the reflection coefficient, provides a tail expansion result for one-dimensional Gaussian multiplicative chaos measures. A more detailed discussion of this phenomenon is provided in \cite{interval}, see also \cite{DOZZ2} for the first probabilistic analysis of the reflection coefficient in the sphere case.
We start by explaining how we can give a direct probabilistic definition to $\overline{R}(\beta, \mu_1, \mu_2)$ without using the limit of \eqref{cc4}. Following \cite{Mating} we use the standard radial decomposition around the point $0$ of the covariance \eqref{covariance} of $X$ restricted to $(-1,1)$, i.e. we write for $ s \geq 0 $,
\begin{equation}\label{radial}
X( e^{-s/2}) = B_s + Y( e^{-s/2}), \quad X(-e^{-s/2}) = B_s + Y(-e^{-s/2}),
\end{equation}
where $B_s$ is a standard Brownian motion and $Y$ is an independent Gaussian process that can be defined on the whole plane with covariance given for $x, y \in \mathbb{C}$ by:
\begin{equation}
\mathbb{E}[ Y(x) Y(y)] = 2 \ln \frac {|x| \vee |y|}{\vert x -y \vert}.
\end{equation}
We introduce for $\lambda>0$ the process that will be used in the definition below,
\begin{equation}\label{mathcal B}
\mathcal{B}^{\lambda}_s := 
\begin{cases}
\hat{B}_s-\lambda s  \quad s\ge 0\\
\bar{B}_{-s}+\lambda s \quad s<0,
\end{cases}
\end{equation}
where $(\hat{B}_s-\lambda s)_{s\ge 0}$ and $(\bar{B}_{s}-\lambda s)_{s\ge 0}$ are two independent Brownian motions with negative drift conditioned to stay negative. Now for $\beta \in (\frac{\gamma}{2},Q)$ and $\mu_1,\mu_2$ satisfying the constraint of Definition \ref{half-space} we can give an alternative definition of $\overline{R}(\beta, \mu_1, \mu_2)$:
\begin{equation}\label{def_R2}
\overline{R}(\beta, \mu_1, \mu_2) = \mathbb{E} \left[ \left( \frac{1}{2} \int_{- \infty}^{\infty} e^{ \frac{\gamma}{2} \mathcal{B}_s^{\frac{Q -\beta}{2}} } \left( \mu_2 e^{\frac{\gamma}{2} Y(e^{-s/2})} + \mu_1 e^{\frac{\gamma}{2} Y(-e^{-s/2})}  \right) ds \right)^{\frac{2}{\gamma}(Q - \beta)} \right].
\end{equation}

We now provide a lemma proven in Section \ref{sec_H_R} that shows that both definitions \eqref{cc4} and \eqref{def_R2} of $\overline{R}(\beta, \mu_1, \mu_2)$ are equivalent.
\begin{lemma}\label{lim_H_R} Assume that $\mu_1,\mu_2, \mu_3 \in \mathbb{C} $ obey the constraint of Definition \ref{half-space}. Consider $\beta_1, \beta_2, \beta_3$ satisfying $\frac{\gamma}{2}\lor \beta_2 < \beta_1 <Q$, $\beta_2 >0$, and $ \beta_1 - \beta_2 < \beta_3 < Q$. Taking \eqref{def_R2} as the definition of $\overline{R}(\beta, \mu_1, \mu_2)$ the following limit holds:
\begin{equation}
\lim_{\beta_3 \downarrow \beta_1 - \beta_2} (\beta_2 + \beta_3 - \beta_1) \overline{H}^{(\beta_1, \beta_2, \beta_3)}_{(\mu_1, \mu_2, \mu_3)}  = 2 (Q-\beta_1) \overline{R}(\beta_1, \mu_1, \mu_2).
\end{equation}
A similar result holds when $\beta_1 = \beta_2$ and $0<\beta_3 < Q$:
\begin{equation}
\lim_{\beta_3 \downarrow 0} \beta_3 \overline{H}^{(\beta_1, \beta_2, \beta_3)}_{(\mu_1, \mu_2, \mu_3)}  = 2 (Q-\beta_1)\left(\overline{R}(\beta_1, \mu_1, \mu_2) + \overline{R}(\beta_1, \mu_2, \mu_3) \right).
\end{equation}
\end{lemma}
Let us now state how the value of $\overline{R}(\beta, \mu_1, \mu_2)$ provides a very general first order tail expansion for the probability of a one-dimensional GMC measure to be large. For this discussion we choose $\mu_1, \mu_2 \in[0,\infty)$ with at most one of the two parameters being $0$, and we introduce the notation:
\begin{equation}
I_{\eta_1, \eta_2}(\beta) := \int_{-\eta_1}^{\eta_2} \frac{1}{|x|^{\frac{\beta \gamma}{2}}}  e^{\frac{\gamma}{2}X(x)} \left( \mu_1 \mathbf{1}_{ \{x< 0 \}} + \mu_2 \mathbf{1}_{\{ x>0 \} }  \right)dx.
\end{equation}
In the above $\eta_1,\eta_2 \in (0,1)$. Now the tail expansion result is the following:
\begin{proposition}\label{proposition reflection}
For $\beta \in ( \frac{\gamma}{2}, Q) $ and any $\eta_1,\eta_2 \in (0,1)$, we have the following tail expansion for $I_{\eta_1, \eta_2}(\beta)$ as $u \rightarrow \infty$ and for some $\nu >0$:
\begin{equation}\label{tail_result1}
\mathbb{P} ( I_{\eta_1, \eta_2}(\beta) >u )  = \frac{ \overline{R}(\beta, \mu_1, \mu_2)}{u^{ \frac{2}{\gamma}(Q - \beta) }} + O ( \frac{1}{u^{ \frac{2}{\gamma}(Q - \beta)+ \nu }  }   ).
\end{equation}
\end{proposition}
The proof of this proposition follows exactly the same steps as for the case of $\mu_1 =0, \mu_2 >0$ considered in \cite{interval}. Notice that we impose the condition $ \beta \in (\frac{\gamma}{2}, Q)$. This is crucial for the tail behavior of $I_{\eta_1, \eta_2}(\beta)$ to be dominated by the insertion and this is precisely why the asymptotic expansion is independent of the choice of $\eta_1$ and $\eta_2$. It also explains why the radial decomposition \eqref{radial} is natural as it is well suited to study $X$ around a particular point. If one is interested in the case where $\beta < \frac{\gamma}{2}$ (or simply $\beta =0$), a different argument known as the localization trick is required to obtain the tail expansion, see \cite{tail} for more details.

\begin{figure}[!htp]
\centering
\includegraphics[width=0.35\linewidth]{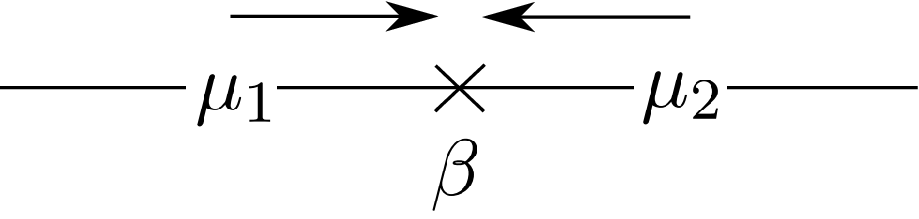}
\caption{$\overline{R}(\beta,\mu_1,\mu_2)$}
\end{figure}
The above picture summarizes what the reflection coefficient computes. In the range $\beta \in (\frac{\gamma}{2},Q)$, the tail expansion of the GMC is dominated by the insertion. The parameters $\mu_1, \mu_2 $ tune the weights of both sides as we approach the insertion. For more details and results on tail expansions of GMC measures with the reflection coefficients see the works \cite{lacoin, tail, Wong-tail}.

\subsection{Discussion and perspectives}\label{sec_discussion}

In the following subsections we provide details about applications of our work as well as future directions of research.

\subsubsection{Link with our previous work \cite{interval} on the unit interval}\label{app:link_int}
We first detail here the precise correspondence between our previous work \cite{interval} and the results of this paper. In \cite{interval} we derived an exact formula for the following quantity
\begin{equation}\label{result_int}
M(\gamma, p, a, b) := \mathbb{E}\left[ \left( \int_0^1 x^a (1-x)^b e^{\frac{\gamma}{2}X(x)} dx \right)^p \right],
\end{equation}
using a method involving observables and differential equations similar to the one of the present paper. After writing the paper \cite{interval}, it remained a mystery to us how this result and its proof strategy fitted into the framework of Liouville CFT. It turns out that the formula for \eqref{result_int} is actually a special case of equation \eqref{main_th2_H} of Theorem \ref{main_th2}, under the following choice of parameters
\begin{equation}\label{eq:para_reduce}
a = - \frac{\gamma \beta_1}{2}, \quad b = - \frac{\gamma \beta_2}{2}, \quad p = \frac{1}{\gamma}(2Q - \sum_{i=1}^3 \beta_i), \quad \mu_1 = \mu_3 =0, \quad \mu_2 =1.
\end{equation}

Indeed by setting $\mu_1 = \mu_3 =0, \: \mu_2 =1$, one only keeps the portion of the GMC measure on the unit interval $[0,1]$ (notice also $g(x) =1$ for $x\in [0,1]$).  The precise computational check that the exact formula $\mathcal{I}$ for $\overline{H}$ reduces to the formula of \cite{interval} for \eqref{result_int} is performed in Section \ref{sec:consistency_int}. The result of \cite{interval} is therefore giving a special case of the boundary three-point function of boundary Liouville CFT. This connection was to the best of our knowledge unknown to physicists.

An additional mystery of \cite{interval} concerns the observables introduced in the proof that satisfy the hypergeometric equation. In the present paper, as explained in Section \ref{sec:outline}, we introduce the observables $G_{\chi}(t)$ and $H_{\chi}(t)$ that correspond to adding in the integrand of the GMC integral the fractional power $(t-x)^{\frac{\gamma \chi }{2}}$, where $\chi =\frac{\gamma}{2}$ or $\frac{2}{\gamma}$. The function $H_{\chi}(t)$ reduces to the observable used in \cite{interval} under a similar choice of parameters as \eqref{eq:para_reduce}. Using the understanding of the present paper, we can now see that these observables correspond to adding a degenerate boundary insertion to a correlation function of boundary LCFT. The term degenerate means the weight of the added boundary insertion needs to be equal to $-\chi$ for $\chi =\frac{\gamma}{2}$ or $\frac{2}{\gamma}$. A similar procedure has been used in the works \cite{DOZZ1, DOZZ2, remy} up to one subtlety, which is that in our present case there are no absolute values around $(t-x)^{\frac{\gamma \chi }{2}}$.

Indeed, if one adds an insertion on the boundary at the location $t$ of weight $-\chi$, thanks to the Girsanov Theorem \ref{girsanov}, one should be getting a term $|t-x|^{\frac{\gamma \chi}{2}}$ inside the GMC integral. This is the situation of the works \cite{DOZZ1, DOZZ2, remy}, but with the difference that in those references the degenerate insertion is added in the bulk of the domain. The reason why in our case with the boundary degenerate insertion one needs to remove the absolute values is actually present in the physics literature on LCFT, although again the link to the GMC models was unknown to physicists and unknown to us at the time we wrote \cite{interval}. In \cite{FZZ} it is claimed that in order for the second order BPZ equation to hold for boundary LCFT in the case where the degenerate insertion is placed on the boundary, a relation needs to hold between the two cosmological constants $\mu_i, \mu_{i+1}$ to the left and right of the degenerate insertion. In the case where the bulk cosmological constant is not present, meaning $\mu=0$, this condition simply reduces to $\mu_{i+1} = e^{\pm i \pi \frac{\gamma \chi}{2} } \mu_i$. Therefore it is easy to see that an equivalent way to encode this condition is simply to remove the absolute values around $|t-x|^{\frac{\gamma \chi}{2}}$, which is precisely what is done in \cite{interval} and in the present paper. Let us mentioned we have also checked that if one does add the absolute values the differential equations of Section \ref{sec_BPZ} do not hold.

Lastly a similar link can be established for the reflection coefficient $\overline{R}(\beta, \mu_1, \mu_2)$. By setting $\mu_1 =0, \: \mu_2=1$, our formula \eqref{main_th2_R} reduces to the formula given in \cite[Proposition~1.5.]{interval}. The notation used in \cite{interval} for $\overline{R}(\beta, 0, 1)$ is $\overline{R}_1^{\partial}(\alpha)$.

\subsubsection{Laws for one-dimensional GMC}\label{app:laws_GMC}
It is illustrative to state our results not as exact formulas for moments of GMC but as giving the law of a random variable involving GMC. We give here a summary of all the laws of one-dimensional GMC that can be derived from our results. Using Lemma \ref{link_D_H}, we  first write $\overline{G}$ as a moment of GMC on the unit circle, which will then give the law of the total mass of GMC on the unit circle with an insertion at $1$ of weight $\beta$. The formula for $\overline{U}$ then corresponds to the special case with no insertion. To extract a law of GMC from the exact formula for $\overline{H}$, as explained in the previous Section \ref{app:link_int}, we need to restrict ourselves to the special case $\mu_1 = \mu_3 =0$ and $\mu_2 =1$. We will then obtain the law of the total mass of GMC on the unit interval $[0,1]$ with insertions at both endpoints. Let us start by stating the laws on the circle.

\begin{corollary}\label{cor_law1} (Result of \cite{remy}) The following equality in law holds
\begin{align}
\frac{1}{2 \pi} \int_0^{2 \pi} e^{ \frac{\gamma}{2} X_{\mathbb{D}}(e^{i \theta})} d \theta = \frac{1}{\Gamma(1 - \frac{\gamma^2}{4})} \mathcal{E}(1)^{- \frac{\gamma^2}{4}}.
\end{align}
Here  $X_{\mathbb{D}}$ is the GFF on $\mathbb{D}$ with covariance given by \eqref{GFF_D} and  $ \mathcal{E}(1)$ is an exponential law of parameter $1$.
\end{corollary}

\begin{corollary}\label{cor_law2} The following equality in law holds
\begin{equation}
\frac{1}{2\pi}  \int_0^{2 \pi} \frac{1}{|e^{i \theta} -1|^{\frac{\gamma \beta}{2}}} e^{ \frac{\gamma}{2} X_{\mathbb{D}}(e^{i \theta})} d \theta = (\frac{4}{\gamma^2})^{\frac{\gamma^2}{4}} \frac{1}{\Gamma(1 - \frac{\gamma^2}{4})} X_1 X_2,
\end{equation}
where $X_1, X_2$ are two independent random variables in $\mathbb{R_+}$ with the following laws:

\begin{align*}
X_1 &= \beta_{2,2}^{-1}(1, \frac{4}{\gamma^2} ; \frac{4}{\gamma^2},  1 + \frac{\beta}{\gamma}, 1 + \frac{\beta}{\gamma}), \\
X_2 &= \beta_{1,0}^{-1}(\frac{4}{\gamma^2}; \frac{2 \beta}{\gamma} + \frac{4}{\gamma^2}+1).
\end{align*}
Here $X_{\mathbb{D}}$ is again the GFF on $\mathbb{D}$ with covariance given by \eqref{GFF_D}, and $\beta_{1,0}$ and  $\beta_{2,2}$ are special beta laws. $\beta_{2,2}$ is defined in Section \ref{sec:double_gamma} and $\beta_{1,0}$ is a law with moments given by $\E[\beta_{1,0}(a;b)^q] = a^{\frac{q}{a}} \frac{\Gamma(\frac{q+b}{a})}{\Gamma(\frac{b}{a})}$.
\end{corollary}

\begin{proof}
It was shown in \cite[Theorem 4.1]{Ostro2} that if a random variable admits moments given by the right hand side of \eqref{eq:main_th1}, then it is equal in law to the decomposition given by the above corollary. Since our Theorem  \eqref{main_th1} establishes precisely this fact for $\frac{1}{2\pi}  \int_0^{2 \pi} |e^{i \theta} -1|^{-\frac{\gamma \beta}{2}} e^{ \frac{\gamma}{2} X_{\mathbb{D}}(e^{i \theta})} d \theta $, it follows that the corollary holds. See also the review \cite{Ostro_review} for more details.
\end{proof}

Corollary \ref{cor_law1} is an equivalent statement of Theorem \ref{FyBo}. This law was conjectured in \cite{FyBo} and proven in \cite{remy}. Similarly Corollary \ref{cor_law2} is an equivalent statement of Theorem \ref{main_th1}. This law was conjectured by Ostrovsky in \cite{Ostro2}. Lastly we finish with the law on the unit interval which is equivalent to equation \eqref{main_th2_H} of Theorem \ref{main_th2} in the special case $\mu_1 = \mu_3 =0$ and $\mu_2 =1$. This law was conjectured in \cite{FLeR, Ostro1} and first proved in \cite{interval}, see also the discussions in \cite{interval} for more details.

\begin{corollary}\label{cor_law3} (Result of \cite{interval}) The following equality in law holds
\begin{equation}
\int_{0}^{1}  x^{a}(1-x)^{b} e^{\frac{\gamma}{2} X( x) } d x = 2 \pi 2^{-(3(1 + \frac{\gamma^2}{4})+ 2( a + b))   }  L Y_{\gamma} X_1 X_2 X_3 ,
\end{equation}
where $L, Y_{\gamma}, X_1, X_2, X_3$ are five independent random variables in $\mathbb{R_+}$ with the following laws:
\begin{align*}
L &= \exp (\mathcal{N}(0, \gamma^2 \ln 2 )), \\
Y_{\gamma} &= \frac{1}{\Gamma(1 - \frac{\gamma^2}{4})} \mathcal{E}(1)^{-\frac{\gamma^2}{4}}, \\
X_1 &= \beta_{2,2}^{-1}(1, \frac{4}{\gamma^2} ;1 +  \frac{4}{\gamma^2}(1 + a),  \frac{2(b-a)}{\gamma^2}, \frac{2(b-a)}{\gamma^2}), \\
X_2 &= \beta_{2,2}^{-1}( 1,\frac{4}{\gamma^2} ;1 +  \frac{2}{\gamma^2}(2 + a + b),  \frac{1}{2},  \frac{2}{\gamma^2}  ), \\
X_3 &= \beta_{2,2}^{-1}( 1,\frac{4}{\gamma^2} ; 1 + \frac{4}{\gamma^2},  \frac{1}{2} + \frac{2}{\gamma^2}(1 + a +b ) ,  \frac{1}{2} + \frac{2}{\gamma^2}(1 + a +b )  ).
\end{align*}
Here again $\mathcal{E}(1)$ is an exponential law of parameter $1$ and $\beta_{2,2}$ is a special beta law defined in Section \ref{sec:double_gamma}.
\end{corollary}

One can convince oneself that these are the most general laws that can be extracted from the formulas we have proved. In order to obtain laws with more insertion points, or to obtain the joint law of the mass of GMC on several arcs, one needs to compute correlations with more insertion points. This will then require implementing the conformal bootstrap procedure, see Section \ref{app:bootstrap}.

\subsubsection{Application to toric conformal blocks and the modular kernel}\label{app:blocks}

Very recently in \cite{Blocks}, a GMC expression has been proposed for the one-point conformal block for LCFT on the torus. The main result of \cite{Blocks} is that this probabilistic definition matches the formal power series given in physics by Zamolodchikov's recursion and shows convergence of this series. More precisely, the expression of \cite{Blocks} for conformal blocks is given by, for parameters $\beta \in (-\frac{4}{\gamma},Q)$,\footnote{In \cite{Blocks} this parameter $\beta$ is called $\alpha$, but we use here the notation $\beta$ in order to keep the convention of this paper for insertions on the boundary.} $P \in \mathbb{R}$, $q \in (0,1)$,
\begin{equation}\label{eq_blocks}
\mathcal{G}^{\beta}_{\gamma, P}(q) := \frac{q^{\frac{1}{12}(1 - \frac{\beta \gamma}{2} - \frac{\beta^2}{2})}
\eta(q)^{ \frac{\beta \gamma}{2} - \beta^2 - 1}}{Z}
\E\left[\left(\int_0^1 \vert \Theta_\tau(x)\vert ^{-\frac{\beta \gamma}{2}} e^{\pi \gamma P x} e^{\frac{\gamma}{2} Y_\tau(x)} dx \right)^{-\frac{\beta}{\gamma}} \right]
\end{equation}
where $Z$ is a normalization constant, $\Theta_{\tau}$ is the Jacobi theta function, $\eta(q)$ the Dedekind eta function and $Y_{\tau}$ is a log-correlated field which can be thought of as the restriction of a 2d GFF on the torus to one of the loops of the torus (see \cite{Blocks} for more details). Both $\Theta_{\tau}$ and $Y_{\tau}$ depend on the parameter $q $ related to the moduli $\tau$ of the torus by $ q = e^{i \pi \tau}$. The proof strategy of \cite{Blocks} contains steps similar to those detailed in Section \ref{sec:outline} to prove the results of the current paper. In particular one needs again to perform the operator product expansion with reflection and in order to obtain an explicit answer the formula of Theorem \ref{main_th2} for the boundary two-point function is required in its full generality. Indeed \cite{Blocks} uses the following fact coming from Theorem \ref{main_th2}:
\begin{align*}
&\frac{\overline{R}(\beta , 1,  e^{ - i \pi + \pi \gamma P})}{\overline{R}(\beta + \frac{2}{\gamma} -\frac{\gamma}{2}, 1, e^{- i \pi \frac{ \gamma^2}{4} + \pi \gamma P})} =   \frac{2(2 \pi)^{\frac{4}{\gamma^2} -1} \Gamma(\frac{2\beta }{\gamma}) \Gamma(1 - \frac{2 \beta }{\gamma})}{\gamma(Q -\beta) \Gamma(1 - \frac{\gamma^2}{4})^{\frac{4}{\gamma^2} -1 } \Gamma(\frac{\gamma \beta}{2} - \frac{\gamma^2}{2}) \Gamma( 1 - \frac{\gamma \beta}{2} + \frac{\gamma^2}{4}) } \frac{1 - e^{\frac{4 \pi P}{\gamma} - \frac{4 i \pi}{\gamma^2} + i \pi \frac{2 \beta}{\gamma}} }{1 + e^{ \pi \gamma  P - \frac{ i \pi \gamma^2}{2} + i \pi \frac{\gamma \beta}{2}} }.
\end{align*}
Furthermore, the normalization $Z$ of the conformal block $\mathcal{G}^{\beta}_{\gamma, P}(q)$ is explicitly given by
\begin{align}\label{eq:Z_blocks}
Z &:=   \E\left[\left( \frac{1}{2 \pi} \int_0^{2\pi} \vert e^{i \theta} - 1 \vert^{-\frac{\beta \gamma}{2}} e^{ \frac{ \gamma P \theta }{2}} e^{\frac{\gamma}{2} X_{\mathbb{D}}(e^{ i \theta})} d \theta \right)^{-\frac{\beta}{\gamma}}\right]\\
&= \left( \frac{\gamma}{2} \right)^{  \frac{\gamma \beta}{4} } e^{-\frac{\pi \beta P}{2}} \Gamma(1 - \frac{\gamma^2}{4})^{\frac{\beta}{\gamma}}\frac{\Gamma_{\frac{\gamma}{2}}(Q - \frac{\beta}{2}) \Gamma_{\frac{\gamma}{2}}(\frac{2}{\gamma} + \frac{\beta}{2}) \Gamma_{\frac{\gamma}{2}}(Q - \frac{\beta}{2} - i P) \Gamma_{\frac{\gamma}{2}}(Q - \frac{\beta}{2} + i P) }{\Gamma_{\frac{\gamma}{2}}(\frac{2}{\gamma}) \Gamma_{\frac{\gamma}{2}}(Q  - i P) \Gamma_{\frac{\gamma}{2}}(Q  + i P) \Gamma_{\frac{\gamma}{2}}(Q - \beta) }, \nonumber
\end{align}
where in the GMC expression $X_{\mathbb{D}}$ is the GFF on $\mathbb{D}$ with covariance given by \eqref{GFF_D}. As an output of the proof of \cite{Blocks}, the GMC expectation above is explicitly evaluated by the formula given in the second line, requiring  again the exact formula for the boundary two-point function \eqref{main_th2_R}. It is enlightening to compare \eqref{eq:Z_blocks} to the result of Theorem \ref{main_th1}, which using the GFF $X_{\mathbb{D}}$ on the disk $\mathbb{D}$ can be restated as:

\begin{equation}\label{eq:G_GMC}
\E\left[ \left( \frac{1}{2 \pi} \int_0^{2\pi} \frac{1}{|e^{i \theta} -1|^{\frac{\gamma \beta}{2}}} e^{\frac{\gamma}{2} X_{\mathbb{D}}(e^{i\theta})}  d\theta\right)^{\frac{2Q-2\alpha - \beta}{\gamma}}\right] =   \frac{ \Gamma(\frac{\gamma\alpha}{2}+\frac{\gamma\beta}{4}-\frac{\gamma^2}{4}) \Gamma_{\frac{\gamma}{2}}(\alpha-\frac{\beta}{2} )  \Gamma_{\frac{\gamma}{2}}(\alpha+\frac{\beta}{2}  ) \Gamma_{\frac{\gamma}{2}}(Q - \frac{\beta}{2}  )^2}{\Gamma(1-\frac{\gamma^2}{4})^{\frac{2}{\gamma}(Q-\alpha-\frac{\beta}{2})} \Gamma_{\frac{\gamma}{2}}(Q - \beta ) \Gamma_{\frac{\gamma}{2}}(\alpha )^2\Gamma_{\frac{\gamma}{2}}(Q) }.
\end{equation}
Using \eqref{eq:shift_S} one can show that both \eqref{eq:G_GMC} and \eqref{eq:Z_blocks} degenerate to the same formula if one choose $\alpha = Q$ in \eqref{eq:G_GMC} and $P=0$ in \eqref{eq:Z_blocks}. Both \eqref{eq:G_GMC} and \eqref{eq:Z_blocks} are thus a special case of the following conjectured formula, for $P \in \mathbb{R}$, $\beta < Q$, $\frac{\gamma}{2} - \alpha < \frac{\beta}{2} < \alpha $:
\begin{align}\label{eq:Z_blocks_gen}
& \E\left[\left( \frac{1}{2 \pi} \int_0^{2 \pi}  | e^{ i \theta} -1 |^{-\frac{\beta \gamma}{2}} e^{\frac{\gamma P \theta}{2} } e^{\frac{\gamma}{2} X_{\mathbb{D}}(e^{ i \theta})} d \theta \right)^{\frac{2}{\gamma}(Q - \alpha - \frac{\beta}{2})}\right]\\
& = e^{\pi P(Q -\alpha -\frac{\beta}{2})} \frac{\Gamma( \frac{\gamma \alpha}{2} + \frac{\gamma \beta}{4} - \frac{\gamma^2}{4} ) \Gamma_{\frac{\gamma}{2}}(\alpha  - \frac{\beta}{2}  ) \Gamma_{\frac{\gamma}{2}}(Q - \frac{\beta}{2} - iP ) \Gamma_{\frac{\gamma}{2}}(Q - \frac{\beta}{2} + iP  ) \Gamma_{\frac{\gamma}{2}}(\alpha + \frac{\beta}{2} ) }{\Gamma(1 - \frac{\gamma^2}{4} )^{\frac{2}{\gamma}(Q -\alpha - \frac{\beta}{2})} \Gamma_{\frac{\gamma}{2}}(\alpha -iP  ) \Gamma_{\frac{\gamma}{2}}(\alpha + iP) \Gamma_{\frac{\gamma}{2}}(Q - \beta ) \Gamma_{\frac{\gamma}{2}}(Q) }. \nonumber
\end{align}
This formula is compatible with both \eqref{eq:G_GMC} and \eqref{eq:Z_blocks} as well as with the expression coming from Selberg type integrals if one choose $\frac{2}{\gamma}(Q -\alpha - \frac{\beta}{2})$ to be a positive integer, see \cite[Equation (1.17)]{Selberg}. We leave repeating our methods to prove \eqref{eq:Z_blocks_gen} for a future work.

Lastly, as a follow-up to their work, the authors of \cite{Blocks} are looking into the problem of the modular transformation of the one-point toric conformal block. For $\tau \in \mathbb{H}$, consider $q=e^{i \pi \tau}$ and $\tilde q = e^{-i  \pi \tau^{-1}}$. The conjecture transformation rule states that
\begin{equation}\label{modular_eq}
	{\tilde q}^{\frac{P^2}{2}} \mathcal{G}^{\beta}_{\gamma, P}(\tilde q) = \int_{\mathbb{R}} \mathcal{M}_{\gamma,\beta}(P, P')q^{\frac{P'^2}{2}} \mathcal{G}^{\beta}_{\gamma, P'}(q) dP'
\end{equation}
for a certain explicit modular kernel $\mathcal{M}_{\gamma,\beta}(P, P')$. It turns out that proving equation \eqref{modular_eq} will crucially require the exact formulas for $\overline{G}$ and $\overline{H}$ given by Theorems \ref{main_th1} and \ref{main_th2}.

\subsubsection{Applications to SLE and the mating-of-trees framework}\label{sec:sle}
Very recently important progress has been made to connect three types of integrability in conformally invariant probability: the one coming from LCFT which is the subject of the present paper, the one of the Schramm-Loewner evolutions (SLE), and the one of the celebrated mating-of-trees framework of \cite{Mating}. The SLE curves are a family of canonical conformally invariant random planar curves that describe the scaling limit of many models of 2D statistical physics at criticality. The mating-of-trees framework is an encoding of the so-called quantum surfaces - an alternative but equivalent description of the random surface described by LCFT - and of SLE on these surfaces in terms of Brownian motion. It is also instrumental in understanding the scaling limit of random planar maps. See \cite{AHS} and references therein for more details.

In the paper \cite{AHS} the authors precisely obtain integrability results for SLE using the conformal welding of SLE decorated quantum surfaces, and the connection between these quantum surfaces and LCFT. The exact formulas for the correlation functions of LCFT are thus a necessary ingredient. More precisely \cite{AHS} proves an exact formula for the law of a conformal derivative of a variant of SLE called $\mathrm{SLE}_{\kappa}(\rho_-; \rho_+)$, which uses our formulas \eqref{main_th2_R} and \eqref{main_th2_H}.

Furthermore, the mating-of-trees framework encodes certain quantum surfaces in terms of a 2D Brownian motion. In \cite{Mating} the covariance of this Brownian motion was given up to an unknown constant, which has now been explicitly computed in \cite{ARS} using as input the formula \eqref{main_th2_R}. Then \cite{ARS} takes this as an input and applies conformal welding techniques to prove Conjecture \ref{conj1} below. See the next subsection for details.

\subsubsection{Boundary LCFT with bulk and boundary Liouville potentials.}\label{app:bulk_boundary}

As done in the physics literature \cite{FZZ, three_point}, it is actually possible to work with both a bulk and a boundary Liouville potential. The Liouville action will then have the more general form:
\begin{align}\label{liouville_action_bulk}
S_L(X) &= \frac{1}{4 \pi} \int_{\mathbb{H}} \left( \vert \partial^g X \vert^2 + Q R_g X + 4 \pi \mu e^{\gamma X} \right)d \lambda_g + \frac{1}{2 \pi} \int_{ \mathbb{R}} \left( Q K_g X + 2 \pi \mu_B e^{\frac{\gamma}{2} X} \right)d \lambda_{\partial g}.
\end{align}

The basic correlation functions $U, G, R, H$ can then also be expressed using the GMC measure following the framework of \cite{Disk}. In this case the expressions will not reduce to a simple moment of GMC but involve a more complicated functional containing both a GMC measure integrated over the area and over the boundary of the domain. Physicists have then proposed exact formulas for this general case \cite{FZZ, bulk_boundary, three_point}. In order to state the expected results, we must redefine the variables $\sigma_i$ in this generalized case where $\mu >0$, the definition of the $\mu_i$ remaining unchanged. The new relation defining the $\sigma_i$ is given by
\begin{equation}\label{general_sigma}
\mu_i = \sqrt{\frac{\mu}{\sin(\pi\frac{\gamma^2}{4})}}\cos(\pi\gamma(\sigma_i-\frac{Q}{2})).
\end{equation}
We state here as conjectures the two simplest formulas in this more general case predicted in \cite{FZZ}.

\begin{conjecture}\label{conj1} (Bulk one-point function with $\mu, \mu_B >0$, has now been proved in \cite{ARS}). Consider parameters $\alpha \in (\frac{\gamma}{2}, Q)$, $\mu, \mu_B >0$. The probabilistic expression for $U(\alpha)$ is equal to the following exact formula
\begin{equation}
U(\alpha) = \frac{4}{\gamma} \left( \pi \mu \frac{\Gamma(\frac{\gamma^2}{4})}{\Gamma(1 -\frac{\gamma^2}{4})} \right)^{\frac{Q - \alpha}{\gamma} } \Gamma(\frac{\alpha \gamma}{2} - \frac{\gamma^2}{4} ) \Gamma(\frac{2 \alpha}{\gamma} - \frac{4}{\gamma^2} -1 ) \cos \left(2 \pi (\alpha - Q) (\sigma - \frac{Q}{2}) \right),
\end{equation}
where the parameter $\sigma$ is defined through \eqref{general_sigma} with $\mu_i = \mu_B$ and $\sigma_i = \sigma$.
\end{conjecture}

\begin{conjecture}\label{conj2}(Boundary two-point function with $\mu, \mu_1, \mu_2 > 0$, has now been proved in \cite{ARSZ}). Consider parameters $\beta \in (\frac{\gamma}{2}, Q)$, $\mu, \mu_1, \mu_2 >0$, and let $\sigma_i$ be defined by \eqref{general_sigma}. The probabilistic expression for $R(\beta, \sigma_1, \sigma_2)$ is equal to the following exact formula
\begin{align}
R(\beta, \sigma_1, \sigma_2) = \left(\frac{\pi \mu (\frac{\gamma}{2})^{2-\frac{\gamma^2}{2}} \Gamma(\frac{\gamma^2}{4})}{\Gamma(1-\frac{\gamma^2}{4})} \right)^{\frac{Q-\beta}{\gamma}} \frac{\Gamma_{\frac{\gamma}{2}}(\beta - Q ) \Gamma_{\frac{\gamma}{2}}(Q- \beta )^{-1}}{S_{\frac{\gamma}{2}}(\frac{\beta}{2}+\sigma_1 +\sigma_2-Q) S_{\frac{\gamma}{2}}(\frac{\beta}{2}-\sigma_1-\sigma_2+Q)} \\ \nonumber
\times \frac{1}{ S_{\frac{\gamma}{2}}(\frac{\beta}{2} + \sigma_2- \sigma_1)S_{\frac{\gamma}{2}}(\frac{\beta}{2} + \sigma_1- \sigma_2) }.
\end{align}
\end{conjecture}

These formulas as well as the modified condition \eqref{general_sigma} can be reduced to the case of the present paper by taking in a suitable way the limit $\mu \rightarrow 0 $. In the recent papers \cite{ARS} and \cite{ARSZ}, both Conjecture \ref{conj1} and \ref{conj2} have now been verified. As explained in Section \ref{sec:sle}, the proof of \cite{ARS} uses the conformal welding of SLE curves and quantum surfaces, but requires in one step our formula \eqref{main_th2_R}. In the work \cite{ARSZ} with M. Ang and X. Sun, we have proven Conjecture \ref{conj2} as well as the generalization of the formulas for $H$ and $G$ using again a blend of the method of the present paper combined with conformal welding of quantum surfaces.

 \subsubsection{Conformal bootstrap for boundary LCFT}\label{app:bootstrap} 
 
With the four basic correlation functions on $\mathbb{H}$ computed, to completely solve boundary LCFT one must then compute correlations on $\mathbb{H}$ with more insertion points or in higher genus (such as on an annulus). This requires the conformal bootstrap method first proposed in \cite{BPZ}, which claims that correlations with more points or in higher genus can be expressed in terms of the basic correlations and of the conformal blocks, a universal function completely specified by the representation theory of the Virasoro algebra. In the simpler setup of boundaryless surfaces, an example of a conformal block is given by \eqref{eq_blocks}. This one-point toric conformal block allows to compute the one-point function of LCFT on the torus in terms of the three-point function on the sphere, see the discussions in \cite{Blocks} for more details. For an example in physics of a bootstrap decomposition in the case of a domain with boundary, see \cite{Martinec} for the case of the annulus.
 
At the level of mathematics the bootstrap problem has recently been solved in the groundbreaking work \cite{GKRV} in the case of the $N$-point function on the Riemann sphere and then in \cite{GKRV2} for the case of arbitrary correlations on any boundaryless Riemann surface. In the boundary case the method of \cite{GKRV, GKRV2} has recently been adapted to the case of the annulus in \cite{Wu} and the more general boundary cases are a work in progress by the authors of \cite{GKRV, Wu}. As an illustration we give the statement \cite[Theorem 1.3]{Wu} on the annulus $\mathbb{A} := \{ z \: \vert \: |q| \leq |z| \leq 1 \} $ giving the bootstrap formula for the boundary one-point LCFT correlation on $\mathbb{A}$:
\begin{align}
&\langle e^{\frac{\beta}{2} \phi(1)} \rangle_{\mathbb{A}} = c_1 \int_{\mathbb{R}} G(Q + iP, \beta) U(Q -iP) q^{\frac{P^2}{2}} \mathcal{F}_{\mathbb{A}}(q, \beta, P) dP.
\end{align}
Similarly for the boundary 4-point function or the 1-bulk-2-boundary point function on $\mathbb{H}$ one expects the following bootstrap formulas:
\begin{align}
\langle e^{\frac{\beta_1}{2} \phi(0)} e^{\frac{\beta_2}{2} \phi(s)} e^{\frac{\beta_3}{2} \phi(1)} e^{\frac{\beta_4}{2} \phi(\infty)} \rangle &= c_2 \int_{\mathbb{R}} H^{(\beta_1, \beta_2, Q + i P)}_{(\mu_1, \mu_2, \mu_3)} H^{(Q -iP, \beta_3, \beta_4)}_{(\mu_1, \mu_3, \mu_4)} s^{\frac{P^2}{2}} \mathcal{F}_{\mathbb{H},1}(s, \beta_i, P) dP,\\
\langle e^{\frac{\beta_1}{2} \phi(0)} e^{\frac{\beta_2}{2} \phi(s)} e^{\alpha \phi(i)}  \rangle &= c_3 \int_{\mathbb{R}} H^{(\beta_1, \beta_2, Q + i P)}_{(\mu_1, \mu_2, \mu_3)} (G(\alpha, Q-iP) )_{\vert \mu_B = \mu_1} s^{\frac{P^2}{2}} \mathcal{F}_{\mathbb{H},2}(s,\alpha, \beta_i, P) dP
\end{align}
Here $\mathcal{F}_{\mathbb{A}}(q, \beta, P), \mathcal{F}_{\mathbb{H},1}(s, \beta_i, P)$, and $\mathcal{F}_{\mathbb{H},2}(s,\alpha, \beta_i, P)$ are instances of conformal blocks and $c_1, c_2, c_3$ are simple constants. $\mathcal{F}_{\mathbb{A}}(q, \beta, P) $ can actually be related to \eqref{eq_blocks}, see \cite{Wu}. Notice how these statements involve the functions $U, G, H$ that we have computed. For another discussion in the probability literature of the boundary bootstrap see also \cite[Section 5]{BaWong} on fusion in boundary LCFT.

\subsection{Outline of the proof}\label{sec:outline}

We summarize here the main steps of the proof and the intermediate results that will lead us to Theorems \ref{main_th1} and \ref{main_th2}. Our proof strategy follows the one of the previous works \cite{DOZZ1, remy, interval} but there are many novel difficulties that must be resolved due to the fact that we are forced to work with complex valued quantities (instead of positive as in the cited works). The computations are also much more involved.

\begin{itemize}
\item \textbf{BPZ differential equations.} Since LCFT is a conformal field theory, correlation functions containing a field with a degenerate insertion are predicted to obey a differential equation known as the BPZ equation. Therefore if one considers a correlation function where one of the boundary insertion points has a weight  $\beta = - \frac{\gamma}{2} $ or $ -\frac{2}{\gamma}$, then the whole correlation will obey the BPZ equation.\footnote{It is also possible to consider degenerate insertions in the bulk but they will not be used in the present paper. See the discussion in Section \ref{app:link_int} for more details.}  More precisely, for $\chi = \frac{\gamma}{2}$ or $\frac{2}{\gamma}$ and $t \in \mathbb{H}$, we will consider the following observables,
\begin{align*}
G_{\chi}(t) & =  \mathbb{E} \left[ \left(\int_{\mathbb{R}} \frac{(t-x)^{\frac{\gamma\chi}{2}}}{|x-i|^{\gamma\alpha  } }   g(x)^{\frac{\gamma^2}{8}(p-1)} e^{\frac{\gamma}{2} X(x) } d x \right)^{p} \right] \: \: \text{where} \: \: p = \frac{2}{\gamma}(Q-\alpha-\frac{\beta}{2}+\frac{\chi}{2}),
 \\
H_{\chi}(t)  &= \mathbb{E} \left[ \left( \intr \frac{(t-x)^{\frac{\gamma\chi}{2}}}{|x|^{\frac{\gamma \beta_1 }{2}} |x-1|^{\frac{\gamma \beta_2 }{2}}}  g(x)^{\frac{\gamma^2}{8}(q-1)} e^{\frac{\gamma}{2} X(x)} d \mu(x) \right)^{q} \right] \: \: \text{where} \: \: q = \frac{1}{\gamma}(2 Q  - \beta_1 - \beta_2 - \beta_3 + \chi).
\end{align*}
The functions $G_{\chi}(t) $ and $H_{\chi}(t)$ will be used respectively to prove Theorem \ref{main_th1} and Theorem \ref{main_th2}. In Section \ref{sec_BPZ}  we show that $H_{\chi}(t)$ obeys a hypergeometric equation and similarly for $G_{\chi}(t)$ after an extra change of variable. It is then possible to write down explicitly the solution space, writing it here to illustrate the discussion for $H_{\chi}(t)$,
\begin{align*}
H_{\chi}(t) &= C_1 F(A,B,C,t) + C_2 t^{1 - C} F(1 + A-C, 1 +B - C, 2 -C, t) \\
 &= B_1 F(A,B,1+A+B- C, 1 -t) + B_2 (1-t)^{C -A -B} F(C- A, C- B, 1 + C - A - B , 1 -t),
\end{align*}
where $A,B,C$ are known parameters depending on $\gamma, \beta_1, \beta_2, \beta_3$ and the $C_1,C_2,B_1,B_2$ are parameters that parametrize the solution space of the hypergeometric equation. These last four parameters are unknown at this stage of the proof.

\item \textbf{Operator product expansion (OPE).}  The next step is to perform an asymptotic analysis directly on the probabilistic definition of $H_{\chi}(t)  $ (and similarly for $G_{\chi}(t) $) to identify the constants $C_1,C_2,B_1,B_2$  in terms of $ \overline{H}$, the quantity we are interested in computing. For instance by sending $t$ to $0$, one immediately obtains the result
\begin{equation}
C_1 = \overline{H}^{(\beta_1 - \frac{\gamma}{2}, \beta_2, \beta_3)}_{( \mu_1, e^{\frac{i \pi  \gamma^2}{4}} \mu_2, e^{\frac{i \pi  \gamma^2}{4}} \mu_3)}.
\end{equation}
In the case where $\chi = \frac{\gamma}{2}$ and for a suitable range of $\beta_i$ in which $\beta_1 \in (\frac{\gamma}{2}, \frac{2}{\gamma})$, one can obtain by a straightforward analysis that:
\begin{equation}
C_2  = q \frac{\Gamma(-1+\frac{\gamma \beta_1}{2} - \frac{\gamma^2}{4} ) \Gamma(1 - \frac{\gamma \beta_1}{2})}{\Gamma(- \frac{\gamma^2}{4})} \left( \mu_1 - \mu_2 e^{i \pi \frac{\gamma \beta_1}{2}} \right)  \overline{H}^{(\beta_1 + \frac{\gamma}{2}, \beta_2, \beta_3)}_{( \mu_1, e^{\frac{i \pi  \gamma^2}{4}} \mu_2,  e^{\frac{i \pi  \gamma^2}{4}}\mu_3)}.
\end{equation}

\item \textbf{OPE with reflection.} The method described above only works for the first degenerate weight $\chi =\frac{\gamma}{2}$, and only in a very specific domain of parameters. In the case of $\chi = \frac{2}{\gamma}$, or for $\chi =\frac{\gamma}{2}$ but with $\beta_1$ chosen close to $Q$, the asymptotic analysis required to identified $C_2$ will be much more involved. It is called the OPE with reflection as the boundary two-point function - also called the reflection coefficient - will always appear in the asymptotic. Carrying this out one finds the answer:
\begin{equation}
C_2 = \frac{2(Q - \beta_1)}{\gamma} \frac{\Gamma(\frac{2}{\gamma}(\beta_1 -Q)) \Gamma(\frac{2}{\gamma}(Q -\beta_1) -q ) }{\Gamma(-q)} \overline{R}(\beta_1, \mu_1, \mu_2) \overline{H}^{(2 Q - \beta_1 - \frac{\gamma}{2} , \beta_2 , \beta_3)}_{( \mu_1, e^{i\pi\frac{\gamma^2}{4}}\mu_2, e^{\frac{i \pi  \gamma^2}{4}} \mu_3)}.
\end{equation} 
This phenomenon was known to physicists and its probabilistic description is one of the major achievements of \cite{DOZZ2}. Repeating this in our case requires non-trivial work as we are dealing with complex valued observables and many of the inequalities in \cite{DOZZ2} do not work in our case. We overcome this difficulty in Lemmas \ref{lem_reflection_ope2} and \ref{lem_reflection_ope1}.
\item \textbf{Shift equations and analytic continuation.}  Once we have derived expressions for the coefficients $C_1,C_2,B_1,B_2$, the theory of hypergeometric equations will imply a non trivial relation on our quantities of interest. For instance one has the following relation between $C_1, C_2, B_1$ in the case of the hypergeometric equation satisfied by the function $H_{\chi}(t)$:
\begin{equation}
B_1 = \frac{\Gamma(\chi(\beta_1-\chi)) \Gamma(1 - \chi\beta_2 + \chi^2)}{\Gamma(\chi(\beta_1-\chi+q\frac{\gamma}{2}) \Gamma(1 - \chi\beta_2 + \chi^2 -q\frac{\gamma\chi}{2}) } C_1 + \frac{\Gamma(2  - \chi\beta_1+\chi^2) \Gamma(1 - \chi\beta_2 + \chi^2)}{\Gamma(1 + \frac{q \gamma\chi}{2}) \Gamma(2 - \chi(\beta_1 + \beta_2-2\chi+q\frac{\gamma}{2}))   } C_2.
\end{equation}
These equations will then translate to functional equations on $\overline{G}$ and $\overline{H}$ that we will refer to as shift equations because they will involve our functions of interest at shifted values of the insertion weight, the shift being $ \pm \chi$ for $\chi = \frac{\gamma}{2}$ or $\frac{2}{\gamma}$. A key observation is that the shift equation obtained for $\chi = \frac{\gamma}{2}$ allows to analytically continue our probabilistic definitions of $\overline{H}$ and $\overline{G}$ to meromorphic functions defined in a suitable domain. The procedure is analogous to the well-known example of the Gamma function where the functional equation $\Gamma(x+1) = x \Gamma(x)$ can be used to extend the Gamma function to a meromorphic function of $\mathbb{C}$ with prescribed poles. In our case the poles will also be prescribed by the shift equations. Once established this analytic continuation will then be used to derive a second shift equation corresponding to $\chi = \frac{2}{\gamma}$.

\item \textbf{Shift equations imply the result.} The final step is simply to check that the two shift equations obtained for a specific correlation function completely specify its value. Let us explain this for $\overline{G}(\alpha, \beta)$. Assume $\gamma^2 \notin \mathbb{Q}$. The shift equations imply a relation between the correlation at $\beta $ and $ \beta + \gamma$ and between the correlation at $\beta$ and $\beta + \frac{4}{\gamma}$. Since the ratio of the two periods is not in $\mathbb{Q}$, the shift equations uniquely specify the function up to the knowledge of one value which can be computed when $\beta =0$. One then has $\overline{G}(\alpha, 0) = \overline{U}(\alpha)$ which is known from the previous work \cite{remy}. By using the special functions $\Gamma_{\frac{\gamma}{2}}$, $S_{\frac{\gamma}{2}}$ introduced in Section \ref{sec_special_func}, it is also possible to explicitly construct an analytic function satisfying the same shift equations. Therefore the correlation function must be equal to this analytic function, and we can extend the result to the case where $\gamma^2 \in \mathbb{Q}$ by continuity in $\gamma$.
\end{itemize}

The rest of the paper is organized as follows. Section \ref{sec_bulk_boundary} gives the proof of Theorem \ref{main_th1}, taking as an input the value of the boundary two-point function $\overline{R}$ plus the fact that the observable $G_{\chi}(t)$ obeys the BPZ equation. Section \ref{sec_3pt} proves Theorem \ref{main_th2} using similarly the fact that $H_{\chi}(t)$ satisfies the BPZ equation. Section \ref{sec_BPZ} proves that $G_{\chi}(t)$ and $H_{\chi}(t)$ indeed satisfy the BPZ equations. Finally the Appendix \ref{sec_appendix} collects probabilistic facts used in the main text, technical estimates on the GMC measures, the relation between moments of GMC on $\mathbb{R}$ and $\partial \mathbb{D}$, and finally the required facts about the special functions $\Gamma_{\frac{\gamma}{2}}, S_{\frac{\gamma}{2}}, \mathcal{I}$.

\textbf{Acknowledgements.} The authors would like to thank R\'emi Rhodes and Vincent Vargas for making us discover Liouville CFT. We would also like to thank Morris Ang, Guillaume Baverez, and Xin Sun for many helpful discussions, as well as Lorenz Eberhardt who helped us fix the uniqueness argument for the boundary three-point function. Finally we would like to thank the two anonymous referees for their careful reading of this manuscript and for their numerous comments that helped improve this paper. G.R. was supported by an NSF mathematical sciences postdoctoral research fellowship, NSF Grant DMS-1902804. T.Z. was supported by a grant from R\'egion Ile-de-France.

\section{The bulk-boundary correlator}\label{sec_bulk_boundary}

In this section we will prove Theorem \ref{main_th1}. To compute our quantity of interest $\overline{G}(\alpha, \beta)$ we will show it obeys two functional equations that will completely specify its value. We thus need to show:
\begin{proposition}\label{shift_bulk_boundary} (Shift equations for $\overline{G}(\alpha, \beta)$)
For every fixed $\alpha > Q$, the function $\beta \rightarrow \overline{G}(\alpha, \beta)$ originally defined for $\beta \in (\gamma- 2 \alpha, Q )$ admits a meromorphic extension in a complex neighborhood of the real line and this extension satisfies the following two equations,
\begin{align}
\overline{G}(\alpha,\beta+\gamma)&= \frac{\Gamma(1-\frac{\gamma^2}{4})}{2^{\frac{\gamma\beta}{2}} \pi}\frac{ \Gamma(\frac{\gamma\alpha}{2}-\frac{\gamma\beta}{4}-\frac{\gamma^2}{4}) \Gamma(1-\frac{\gamma\beta}{4} )^2}{\Gamma(\frac{\gamma\alpha}{2}+\frac{
\gamma\beta}{4}-\frac{\gamma^2}{4}) \Gamma(1-\frac{\gamma\beta}{2} ) \Gamma(1-\frac{\gamma\beta}{2}-\frac{\gamma^2}{4})} \overline{G}(\alpha, \beta), \label{G_shift1} \\
\overline{G}(\alpha,\beta+\frac{4}{\gamma})  &= \frac{\gamma^2 \Gamma(1-\frac{\gamma^2}{4})^{\frac{4}{\gamma^2}}}{2^{\frac{2\beta}{\gamma}+1}(2\pi)^{\frac{4}{\gamma^2}}}\frac{\Gamma(\frac{2 \alpha}{\gamma} -\frac{\beta}{\gamma} -\frac{4}{\gamma^2}) \Gamma(1-\frac{\beta}{\gamma})^2}{ \Gamma(-1+\frac{2\alpha}{\gamma}+\frac{\beta}{\gamma}) \Gamma(1-\frac{2\beta}{\gamma})\Gamma(1-\frac{2\beta}{\gamma} -\frac{4}{\gamma^2}) } \overline{G}(\alpha,\beta), \label{G_shift2}
\end{align}
viewed as equalities of meromorphic functions.
\end{proposition}
Using Proposition \ref{shift_bulk_boundary} and the fact that $\overline{U}(\alpha)$ is known from the previous work \cite{remy}, it is easy to prove the value of $\overline{G}(\alpha, \beta)$.

\begin{proof}[Proof of Theorem \ref{main_th1}] Let $r(\alpha,\beta)$ be defined as $\overline{G}(\alpha,\beta)$ divided by its expected expression, namely the right hand side of equation \eqref{eq:main_th1}. Assume first that $\gamma^2 \notin \mathbb{Q}$. Using the two shift equations of Proposition \ref{shift_bulk_boundary} and the shift equations satisfied by $\Gamma_{\frac{\gamma}{2}}$ given in Section \ref{sec:double_gamma}, one can check that $r(\alpha,\beta + \gamma) = r(\alpha,\beta)$ and $r(\alpha,\beta + \frac{4}{\gamma}) = r(\alpha,\beta)$. Since $\gamma^2 \notin \mathbb{Q}$, these two periodicity relations imply that $r$ is constant in $\beta$. Since the value of $\overline{G}(\alpha,0)$  is given by $\overline{U}(\alpha)$ which is known, we can check that $r(\alpha, 0) =1$. Hence $r(\alpha,\beta) =1$ for all $\beta$ and the constraint on $\gamma$ can be lifted by a simple continuity argument. Therefore we have proved:
\begin{equation}
\overline{G}(\alpha, \beta)=  \left( \frac{2^{\frac{\gamma}{2}(-\alpha+\frac{\beta}{2})} 2\pi}{\Gamma(1-\frac{\gamma^2}{4})} \right)^{\frac{2}{\gamma}(Q-\alpha-\frac{\beta}{2})} \frac{ \Gamma(\frac{\gamma\alpha}{2}+\frac{\gamma\beta}{4}-\frac{\gamma^2}{4}) \Gamma_{\frac{\gamma}{2}}(\alpha-\frac{\beta}{2} )  \Gamma_{\frac{\gamma}{2}}(\alpha+\frac{\beta}{2}  ) \Gamma_{\frac{\gamma}{2}}(Q - \frac{\beta}{2}  )^2}{ \Gamma_{\frac{\gamma}{2}}(Q - \beta ) \Gamma_{\frac{\gamma}{2}}(\alpha )^2\Gamma_{\frac{\gamma}{2}}(Q) }.
\end{equation}
\end{proof}
To show Proposition \ref{shift_bulk_boundary}, we will use the solvability coming from the BPZ equations of Liouville theory. For $\chi = \frac{\gamma}{2} \text{ or }\frac{2}{\gamma}$, we denote:
\begin{equation}
p = \frac{2}{\gamma}(Q-\alpha-\frac{\beta}{2}+\frac{\chi}{2}).
\end{equation}
We now introduce two auxiliary functions corresponding to the two values of $\chi = \frac{\gamma}{2} \text{ or }\frac{2}{\gamma}$, for $t\in \H$:
\begin{equation}
G_{\chi}(t)  =  \mathbb{E} \left[ \left(\int_{\mathbb{R}} \frac{(t-x)^{\frac{\gamma\chi}{2}}}{|x-i|^{\gamma\alpha  } } g(x)^{\frac{\gamma^2}{8}(p-1)} e^{\frac{\gamma}{2} X(x) }  d x \right)^{p} \right].
\end{equation}
The parameter range where $G_{\chi}(t)$ is well-defined is:
\begin{equation}
\beta < Q, \quad \text{and} \quad p < \frac{4}{\gamma^2} \wedge \frac{2}{\gamma}(Q - \beta).
\end{equation}
Let us justify why $G_{\chi}(t)$ is well-defined under these conditions. First for $\chi = \frac{\gamma}{2}$ since $(t-x) $ is always contained in the upper-half plane we can define $(t-x)^{\frac{\gamma^2}{4}}$ by choosing the argument to be in $[0,\pi]$.  This means for $t \in \H$ and for either value of $\chi$, the GMC integral
\begin{equation}\label{eq2.7}
\int_{\mathbb{R}} \frac{(t-x)^{\frac{\gamma\chi}{2}}}{|x-i|^{\gamma\alpha  } }  g(x)^{\frac{\gamma^2}{8}(p-1)} e^{\frac{\gamma}{2} X(x) }  d x
\end{equation}
is a random complex number almost surely contained in $\overline{\H}$. We can thus define its $p$ power again by choosing an argument in $[0, \pi]$. Finally Proposition \ref{lem:GMC-moment} given in appendix states the moment $p$ of \eqref{eq2.7} is finite.

Assume now $t \in \{r e^{i \theta} \: | \: r> 0, \:  \theta \in (0 , \frac{\pi}{2})\} $ and perform the change of variable $s = \frac{1}{1+t^2}$. The variable $s$ then belongs to the set $s \in - \mathbb {H} $. We choose the argument of $s$ to be in $(-\pi, 0)$ and define $\sqrt{1-s} = t \sqrt{s}$. Now set:
\begin{equation}
\tilde{G}_{\chi}(s) = s^{p\frac{\gamma \chi}{4}}G_{\chi}(t).
\end{equation}
Then one has
\begin{equation}\label{eq:def_G_tilde}
\tilde{G}_{\chi}(s)  =  \mathbb{E} \left[ \left( \int_{\mathbb{R}} \frac{(\sqrt{1-s}-\sqrt{s}x)^{\frac{\gamma\chi}{2}}}{|x-i|^{\gamma\alpha  } }  e^{\frac{\gamma}{2} X(x) } g(x)^{\frac{\gamma^2}{8}(p-1)} d x \right)^{p} \right],
\end{equation}
where the argument of the GMC integral can be chosen in $(-\pi, \pi)$. We will introduce a dual set of auxiliary functions corresponding to, for $t \in \{r e^{i \theta} \: | \: r> 0, \:  \theta \in (-\frac{\pi}{2},0)\} $, $s = \frac{1}{1+t^2}$ with argument this time in $(0,\pi)$, $\sqrt{1-s} = t \sqrt{s}$, and:
\begin{equation}
 \hat{G}_{\chi}(s) = s^{p\frac{\gamma \chi}{4}}G_{\chi}(-t).
\end{equation}
One lands on the expression
\begin{equation}
\hat{G}_{\chi}(s)  =  \mathbb{E} \left[ \left(\int_{\mathbb{R}} \frac{(-\sqrt{1-s}-\sqrt{s}x)^{\frac{\gamma\chi}{2}}}{|x-i|^{\gamma\alpha  } }  e^{\frac{\gamma}{2} X(x) } g(x)^{\frac{\gamma^2}{8}(p-1)} d x \right)^{p} \right].
\end{equation}
The above GMC integral in the expectation avoids the branch cut $(0,\infty)$ and its argument is chosen to be in $(0, 2\pi)$.
We prove in Section \ref{sec_BPZ1} that $\tilde{G}_{\chi}(s)$ obeys the following hypergeometric equation,
\begin{align}
s(1-s)\partial_s^2 \tilde{G}_{\chi}(s) +(C-(A+B+1)s)\partial_s \tilde{G}_{\chi}(s) - AB\tilde{G}_{\chi}(s) = 0,
\end{align}
with parameters given by:
\begin{align}
A = -p\frac{\gamma \chi}{4}, B = 1+\chi(\chi - \alpha -p\frac{\gamma}{4}), C = \frac{3}{2}+\chi(\chi -\alpha - p\frac{\gamma}{2}).
\end{align}
The exact same equation also holds for $\hat{G}_{\chi}(s)$. As detailed in Section \ref{sec_special_func}, one can explicitly write the solution space of the equation around $s=0$ and $s=1$, under the assumption that $C$ and  $ C - A - B $ are not integers:\footnote{The values excluded here are recovered by an easy continuity argument.}
\begin{align}
\tilde{G}_{\chi}(s) &= \tilde{C}_1 F(A,B,C,s) + \tilde{C}_2 s^{1 - C} F(1 + A-C, 1 +B - C, 2 -C, s)\\
 &= \tilde{B}_1 F(A,B,1+A+B- C, 1 -s) + \tilde{B}_2 (1-s)^{C -A -B} F(C- A, C- B, 1 + C - A - B , 1 -s), \nonumber
\end{align}
\begin{align}
\hat{G}_{\chi}(s) &= \hat{C}_1 F(A,B,C,s)+ \hat{C}_2 s^{1 - C} F(1 + A-C, 1 +B - C, 2 -C, s) \\
 &= \hat{B}_1 F(A,B,1+A+B- C, 1 -s) + \hat{B}_2 (1-s)^{C -A -B} F(C- A, C- B, 1 + C - A - B , 1 -s). \nonumber
\end{align}
Here $ \tilde{C}_1, \tilde{C}_2 , \tilde{B}_1, \tilde{B}_2, \hat{C}_1, \hat{C}_2, \hat{B}_1, \hat{B}_2 $ are all real constants that parametrize the different basis of solutions. Since the solution space is two-dimensional, there is a change of basis formula \eqref{connection1} that relates $\tilde{C}_1, \tilde{C}_2$ with $ \tilde{B}_1, \tilde{B}_2$ and similarly for $\hat{C}_1, \hat{C}_2, \hat{B}_1, \hat{B}_2 $. In the following we will relate several of these coefficients to $\overline{G}$ and it is precisely the change of basis that will lead us to the shift equations of Proposition \ref{shift_bulk_boundary}.

\subsection{First shift equation}

In this section we prove the first shift equation \eqref{G_shift1} in the range of parameters where $\overline{G}(\alpha, \beta)$ is well-defined probabilistically, namely the range $\beta < Q$, $\frac{\gamma}{2} - \alpha < \frac{\beta}{2} < \alpha $ given in Definition \ref{def_four_correls}.
\begin{lemma}\label{bb_shift1}
For $\alpha, \beta \in \mathbb{R}$ satisfying $\beta < \frac{2}{\gamma} - \frac{\gamma}{2}$ and $ \frac{\gamma}{2} - \alpha < \frac{\beta}{2} < \alpha - \frac{\gamma}{2}$, the following equation holds:
\begin{align}
\overline{G}(\alpha,\beta+\gamma)= \frac{1}{2^{\frac{\gamma\beta}{2}} \pi} \frac{\Gamma(1-\frac{\gamma^2}{4}) \Gamma(\frac{\gamma\alpha}{2}-\frac{\gamma\beta}{4}-\frac{\gamma^2}{4}) \Gamma(1-\frac{\gamma\beta}{4} )^2}{\Gamma(\frac{\gamma\alpha}{2}+\frac{
\gamma\beta}{4}-\frac{\gamma^2}{4}) \Gamma(1-\frac{\gamma\beta}{2} ) \Gamma(1-\frac{\gamma\beta}{2}-\frac{\gamma^2}{4})} \overline{G}(\alpha, \beta).
\end{align}
\end{lemma}
\begin{proof}
We start off with the following parameter choices:
\begin{equation}\label{para_choice_1}
\chi = \frac{\gamma}{2}, \quad  \alpha > Q, \quad \frac{\gamma}{2} <  \beta < \frac{2}{\gamma}.
\end{equation}
In the case of $\chi = \frac{\gamma}{2}$ we can actually assume $t \in (0, +\infty)$ which means that $s \in (0,1)$.
By sending $s$ to $0$ one automatically gets that:
\begin{equation}
\tilde{C}_1 = \overline{G}(\alpha, \beta-\frac{\gamma}{2}), \quad \hat{C}_1 = e^{i\pi p\frac{\gamma^2}{4}}\overline{G}(\alpha, \beta-\frac{\gamma}{2}).
\end{equation}
Although we cannot express $\tilde{B}_1$ and $\hat{B}_1$ in terms of the bulk-boundary correlator $\overline{G}$, by setting $s=1$ one obtains the equality:
\begin{equation}
\tilde{B}_1 = \hat{B}_1.
\end{equation}
In order to derive an expression for $\tilde{C}_2$, we have to expand $\tilde{G}_{\frac{\gamma}{2}}(s)$ up to the order $s^{1-C}$. In this case $C = \frac{1}{2} - \frac{\gamma^2}{8} + \frac{\gamma \beta}{4}$.  The parameter choice \eqref{para_choice_1} implies that $ 0 < 1 -C <1$. Following the analysis of \cite{interval}, we expand the difference  $\tilde{G}_{\frac{\gamma}{2}}(s) - \tilde{G}_{\frac{\gamma}{2}}(0)$ as $s \to 0$ using the simple formula
\begin{align}
A(s)^p - A(0)^p = p (A(s) - A(0)) A(0)^{p-1} + o(\vert A(s) - A(0) \vert),
\end{align}
where in our case $A(s)$ is chosen to be the random variable:
\begin{align}
A(s) := \int_{\mathbb{R}} \frac{(\sqrt{1-s}-\sqrt{s}x)^{\frac{\gamma^2}{4}}}{|x-i|^{\gamma\alpha  } }  e^{\frac{\gamma}{2} X(x) } g(x)^{\frac{\gamma^2}{8}(p-1)} d x.
\end{align}
We can first compute the expectation of the error term:
\begin{align}
\E [\vert A(s) - A(0) \vert ] &= \E \left[ \left \vert \int_{\mathbb{R}} \frac{(\sqrt{1-s}-\sqrt{s}x)^{\frac{\gamma^2}{4}} -1 }{|x-i|^{\gamma\alpha  } }  e^{\frac{\gamma}{2} X(x) } g(x)^{\frac{\gamma^2}{8}(p-1)} d x \right \vert \right] \\
& \leq   \int_{\mathbb{R}} \frac{\vert (\sqrt{1-s}-\sqrt{s}x)^{\frac{\gamma^2}{4}} -1 \vert }{ |x-i|^{\gamma\alpha  } }  g(x)^{\frac{\gamma^2}{8}(p-1)} d x \leq c s^{\frac{1}{2}(\gamma \alpha -1 + \frac{\gamma^2}{2}(p-1) )} \left(  \intr du \frac{\vert (1-u)^{\frac{\gamma^2}{4}}-1 \vert}{|u|^{\gamma\alpha + (p-1)\frac{\gamma^2}{2} } } \right). \nonumber
\end{align}
To obtain the last inequality we have performed the change of variable $u = \sqrt{s} x$. The integral over $u$ one obtains is converging both at $u=0$ and as $u \rightarrow \pm \infty $ thanks to the condition $\beta \in(\frac{\gamma}{2}, \frac{2}{\gamma})$ given by \eqref{para_choice_1}. Since the exponent of $s$ is equal to  $1 - C$, this implies that $\E [o(\vert A(s) - A(0) \vert) ] = o(s^{1-C})$.
  
Now moving to the other piece and applying Theorem \ref{girsanov} we obtain:
\begin{align}\label{OPE_no_reflection}
&\tilde{G}_{\frac{\gamma}{2}}(s) - \tilde{G}_{\frac{\gamma}{2}}(0)  \\
& = p\intr dx_1\frac{(\sqrt{1-s}-\sqrt{s}x_1)^{\frac{\gamma^2}{4}}-1}{|x_1-i|^{\gamma\alpha  } } g(x_1)^{\frac{\gamma^2}{8}(p-1)}  \mathbb{E} \left[ e^{\frac{\gamma}{2} X(x_1) } \left(\int_{\mathbb{R}} \frac{g(x)^{\frac{\gamma^2}{8}(p-1)}}{|x-i|^{\gamma\alpha  } }  e^{\frac{\gamma}{2} X(x) }  d x \right)^{p-1} \right] + o(s^{1-C}) \nonumber \\ \nonumber
& = p\intr dx_1\frac{(\sqrt{1-s}-\sqrt{s}x_1)^{\frac{\gamma^2}{4}}-1}{|x_1-i|^{\gamma\alpha  } }  \mathbb{E} \left[ \left(\int_{\mathbb{R}} \frac{ g(x)^{\frac{\gamma^2}{8}(p-2)} }{|x-i|^{\gamma\alpha  } |x-x_1|^{\frac{\gamma^2}{2}} }  e^{\frac{\gamma}{2} X(x) } d x \right)^{p-1} \right] + o(s^{1-C}) \\ \nonumber
& \overset{x_1 = \frac{u}{\sqrt{s}}}{=} p\intr  \frac{du}{\sqrt{s}}  \frac{(\sqrt{1-s}- u)^{\frac{\gamma^2}{4}}-1}{|\frac{u}{\sqrt{s}}-i|^{\gamma\alpha  } }  \mathbb{E} \left[ \left(\int_{\mathbb{R}} \frac{ g(x)^{\frac{\gamma^2}{8}(p-2)} }{|x-i|^{\gamma\alpha  } |x-\frac{u}{\sqrt{s}}|^{\frac{\gamma^2}{2}} }  e^{\frac{\gamma}{2} X(x) } d x \right)^{p-1} \right] + o(s^{1-C}) \\ \nonumber
& = s^{1-C} p \left(  \intr du \frac{(1-u)^{\frac{\gamma^2}{4}}-1}{|u|^{\gamma\alpha + (p-1)\frac{\gamma^2}{2} } } \right) \overline{G}(\alpha, \beta+\frac{\gamma}{2}) + o(s^{1-C}). \nonumber
\end{align} 
The correct way to interpret the above integral over $\mathbb{R}$ is by writing
\begin{equation}
 \intr du \frac{(1-u)^{\frac{\gamma^2}{4}}-1}{|u|^{\gamma\alpha + (p-1)\frac{\gamma^2}{2} } }  = \int_{\mathbb{R}_+ } du \frac{(1+u)^{\frac{\gamma^2}{4}}-1}{u^{\gamma\alpha  + (p-1)\frac{\gamma^2}{2}} } -e^{i \pi (\gamma \alpha + (p-1)\frac{\gamma^2}{2})} \int_{\mathbb{R}_+ e^{i\pi}} du \frac{(1+u)^{\frac{\gamma^2}{4}}-1}{u^{\gamma\alpha + (p-1)\frac{\gamma^2}{2} } },
\end{equation}
where here $\mathbb{R}_+e^{i \pi} $ means that the integral should be understood as a contour integral on $(-\infty,0)$ passing just above the point $u = -1$.
In Section \ref{section useful integrals} we give the exact value of this integral in terms of the Gamma function \eqref{equation integral 1}. 
Putting everything together we have shown,
\begin{align}
& \tilde{G}_{\frac{\gamma}{2}}(s) - \tilde{G}_{\frac{\gamma}{2}}(0) = s^{1-C} p\frac{\Gamma(2C-2)\Gamma(2-2C-\frac{\gamma^2}{4})}{\Gamma(-\frac{\gamma^2}{4})}(1-e^{i \pi (3-2C)}) \overline{G}(\alpha, \beta+\frac{\gamma}{2}) + o(s^{1-C}) \nonumber \\
 & \Rightarrow \: \tilde{C}_2 = p\frac{\Gamma(2C-2)\Gamma(2-2C-\frac{\gamma^2}{4})}{\Gamma(-\frac{\gamma^2}{4})}(1-e^{i \pi (3-2C)}) \overline{G}(\alpha, \beta+\frac{\gamma}{2}).
\end{align}
$\hat{C}_2$ can be calculated in a similar manner, using this time:
\begin{align*}
 \intr du \frac{(-1-u)^{\frac{\gamma^2}{4}}-(-1)^{\frac{\gamma^2}{4}}}{|u|^{\gamma\alpha + (p-1)\frac{\gamma^2}{2}}} &= e^{i\pi \frac{\gamma^2}{4}} \left( \int_{\mathbb{R}_+ } du \frac{(1+u)^{\frac{\gamma^2}{4}}-1}{u^{\gamma\alpha+ (p-1)\frac{\gamma^2}{2}  } } -  e^{- i \pi (\gamma \alpha + (p-1)\frac{\gamma^2}{2})}\int_{\mathbb{R}_+ e^{-i\pi} } du \frac{(1+u)^{\frac{\gamma^2}{4}}-1}{u^{\gamma\alpha+ (p-1)\frac{\gamma^2}{2}  } }\right)\\
& = p\frac{\Gamma(2C-2)\Gamma(2-2C-\frac{\gamma^2}{4})}{\Gamma(-\frac{\gamma^2}{4})}e^{i \pi p\frac{\gamma^2}{4}}(1-e^{-i \pi (3-2C)}).
\end{align*}
Hence:
\begin{equation}
\hat{C}_2 = p\frac{\Gamma(2C-2)\Gamma(2-2C -\frac{\gamma^2}{4})}{\Gamma(-\frac{\gamma^2}{4})} e^{i \pi p\frac{\gamma^2}{4}}(1-e^{-i \pi (3-2C)}) \overline{G}(\alpha, \beta+\frac{\gamma}{2}).
\end{equation}
Now we have the connection formula \eqref{connection1} expressing $\tilde{B}_1$ in terms of $\tilde{C}_1, \tilde{C}_2$ and similarly for $\hat{B}_1,\hat{C}_1,\hat{C}_2$. 
Using the fact that $\tilde{B}_1 = \hat{B}_1$ and our expressions for $\tilde{C}_1, \tilde{C}_2,\hat{C}_1,\hat{C}_2$ in terms of $\overline{G}$ we deduce:
\begin{align*}
&\frac{\Gamma(C)\Gamma(C-A-B)}{\Gamma(C-A)\Gamma(C-B)}\overline{G}(\alpha, \beta-\frac{\gamma}{2})(1-e^{-i 2\pi A})\\ 
&+ p \frac{\Gamma(2-C)\Gamma(C-A-B)}{\Gamma(1-A)\Gamma(1-B)} \frac{\Gamma(2C-2)\Gamma(2-2C -\frac{\gamma^2}{4})}{\Gamma(-\frac{\gamma^2}{4})} (1-e^{i \pi (3-2C)})(1+e^{i \pi (2B-2)}) \overline{G}(\alpha, \beta+\frac{\gamma}{2})=0.
\end{align*}
We thus land on the following shift equation, which is simplified using \eqref{prop_gamma}:
\begin{align*}
\frac{\overline{G}(\alpha,\beta+\frac{\gamma}{2})}{\overline{G}(\alpha, \beta-\frac{\gamma}{2})} &= \frac{\Gamma(C)\Gamma(1-A)\Gamma(1-B)}{p\Gamma(2-C)\Gamma(C-A)\Gamma(C-B) } \frac{\Gamma(-\frac{\gamma^2}{4})}{\Gamma(2C-2)\Gamma(2-2C -\frac{\gamma^2}{4})} \frac{\sin(\pi A)}{2 \cos(\pi C) \cos(\pi B)}\\
&= \frac{\Gamma(C)\Gamma(1-A)\Gamma(1-B)}{p\Gamma(2-C)\Gamma(\frac{1}{2}+A)\Gamma(\frac{1}{2}+B) } \frac{\Gamma(-\frac{\gamma^2}{4})}{\Gamma(2C-2)\Gamma(2-2C -\frac{\gamma^2}{4})} \frac{\sin(\pi C)\sin(\pi A)}{\sin(2\pi C) \sin(\pi (\frac{1}{2}+B))}\\
& = \frac{\Gamma(\frac{1}{2}-B)\Gamma(1-B)}{p\Gamma(1-C)\Gamma(2-C)\Gamma(A)\Gamma(\frac{1}{2}+A) } \frac{\Gamma(3-2C)\Gamma(-\frac{\gamma^2}{4}) }{\Gamma(2-2C -\frac{\gamma^2}{4})} \\
&= \frac{\Gamma(1-\frac{\gamma^2}{4}) \Gamma(1-2B) \Gamma(\frac{3}{2}-C)^2}{\Gamma(1+2A) \Gamma(2-2C) \Gamma(2-2C-\frac{\gamma^2}{4})}  \frac{1}{2^{2C-1} \pi}\\
&=\frac{\Gamma(1-\frac{\gamma^2}{4}) \Gamma(\frac{\gamma\alpha}{2}-\frac{\gamma\beta}{4}-\frac{\gamma^2}{8}) \Gamma(1-\frac{\gamma\beta}{4} +\frac{\gamma^2}{8})^2}{\Gamma(\frac{\gamma\alpha}{2}+\frac{
\gamma\beta}{4}-\frac{3\gamma^2}{8}) \Gamma(1-\frac{\gamma\beta}{2} +\frac{\gamma^2}{4}) \Gamma(1-\frac{\gamma\beta}{2})}\frac{1}{2^{\frac{\gamma\beta}{2}-\frac{\gamma^2}{4}} \pi}.
\end{align*}
Then by replacing $\beta$ by $\beta + \frac{\gamma}{2}$ one lands on the equation of Lemma \ref{bb_shift1}. To extend it to the wider range of validity in $\alpha$ and $\beta$ one uses the analyticity of $\overline{G}$ with respect to these parameters shown in Lemma \ref{lem_analycity}.
\end{proof}
One consequence of Lemma \ref{bb_shift1} is that it allows to analytically continue $\overline{G}$ as a meromorphic function of $\beta$ defined in a complex neighborhood of the real line.
\begin{lemma}\label{bb_continuation}
Fix $\alpha >Q$. The function $\beta \mapsto \overline{G}(\alpha, \beta)$ originally defined for $ \beta < Q$, $ \frac{\gamma}{2} - \alpha < \frac{\beta}{2} < \alpha$ admits a meromorphic extension in a complex neighborhood of the real line.
\end{lemma}
\begin{proof}
Lemma \ref{lem_analycity} shows that $\overline{G}(\alpha, \beta)$ is complex analytic in a complex neighborhood of the real domain $\{ (\alpha, \beta) \in \mathbb{R}^2 \: | \: \beta < Q,  \frac{\gamma}{2} - \alpha < \frac{\beta}{2} < \alpha  \}$ where it is defined probabilistically using GMC. For a fixed $\alpha >Q$, the shift equation of Lemma \ref{bb_shift1} then shows $\beta \mapsto \overline{G}(\alpha, \beta)$ can be meromorphically continued in $\beta$ to a complex neighborhood of the whole real line, the pole structure being prescribed by the Gamma functions in the shift equation.
\end{proof}

\subsection{Second shift equation}\label{section ope reflection 1}

We will now derive an expression of $\tilde{C}_2$ in a different manner corresponding to the so-called operator product expansion (OPE) with reflection. This computation will be valid for $ \frac{2}{\gamma} <  \beta < Q$. For $\chi = \frac{\gamma}{2}$ this will give us the reflection principle and for $\chi = \frac{2}{\gamma}$ it will allow us to obtain the second shift equation on $\beta$. A complete proof of the following steps can be found in \cite{DOZZ2}. We first perform a change of variable $x \to \frac{1}{x}$ on the GMC integral in the expression of $\tilde{G}_{\chi}(s)$:

\begin{equation}
\tilde{G}_{\chi}(s)  =  \mathbb{E} \left[ \left(\int_{\mathbb{R}}\frac{\left(\sqrt{1-s}\,x-\sqrt{s}\right)^{\frac{\gamma\chi}{2}}}{|x-i|^{\gamma\alpha  }|x|^{\frac{\gamma\beta}{2}} }  e^{\frac{\gamma}{2} X(x) } g(x)^{\frac{\gamma^2}{8}(p-1)} dx \right)^{p} \right].
\end{equation}

Note that $C = \frac{1}{2} - \frac{\chi^2}{2} + \frac{\chi \beta}{2}$. For all $\beta \in (Q-\beta_0, Q)$ where $\beta_0$ is a small positive number, and for $\alpha > \frac{1}{\gamma} + \frac{\gamma}{2}$, the following asymptotic is then shown in Lemma \ref{lem_reflection_ope2} for the case $\chi =\frac{2}{\gamma}$ and in Lemma  \ref{lem_reflection_ope1} for the case of $\chi =\frac{\gamma}{2}$:

\begin{align}
\tilde{G}_{\chi}(s) - \tilde{G}_{\chi}(0) = -s^{1-C}\frac{\Gamma(1-\frac{2(Q-\beta)}{\gamma}) \Gamma(-p+\frac{2}{\gamma}(Q-\beta))}{\Gamma(-p)}\overline{R}(\beta, 1, e^{i\pi\frac{\gamma\chi}{2}}) \overline{G}(\alpha, 2Q-\beta -\chi) + o(s^{1-C}).
\end{align}
In the above $s$ is chosen in $(0,1) $ for $\chi = \frac{\gamma}{2}$ and in $ (-1,0)$ for $ \chi =\frac{2}{\gamma}$. From this we can deduce the expression of $\tilde{C}_2$, still for $\beta \in (Q - \beta_0,Q)$,
\begin{equation}
\tilde{C}_2 = -\frac{\Gamma(1-\frac{2(Q-\beta)}{\gamma}) \Gamma(-p+\frac{2}{\gamma}(Q-\beta))}{\Gamma(-p)}\overline{R}(\beta, 1, e^{i\pi\frac{\gamma\chi}{2}}) \overline{G}(\alpha, 2Q-\beta -\chi).
\end{equation}
The range of validity of the above expression can be extended from $\beta \in (Q- \beta_0, Q)$ to the range $\beta \in (\frac{2}{\gamma}, Q) $ by using analyticity in the parameter $\beta$. Indeed, Lemma \ref{lem_analycity} implies the analyticity in $\beta$ in a complex neighborhood of $(\frac{2}{\gamma}, Q)$ of both $\tilde{G}_{\chi}(s)$ and $\overline{G}(\alpha, 2Q-\beta -\chi)$. The analyticity of $\tilde{G}_{\chi}(s)$ then implies the analyticity of $\tilde{C}_2$ and the analyticity of $\overline{R}(\beta, 1, e^{i\pi\frac{\gamma\chi}{2}})$ is known from the exact formula for $\overline{R}$ proved in Section \ref{sec_3pt}. Thus we extend the equality to $\beta \in (\frac{2}{\gamma}, Q)$. From this we can deduce:
\begin{lemma}[Reflection principle for $\overline{G}(\alpha,\beta)$]\label{reflection1} We can analytically continue the definition of $\overline{G}(\alpha, \beta) $ in $\beta$ beyond the point $\beta = Q$ by using the following formula
\begin{align} \label{eq:reflection1}
 \overline{G}(\alpha,\beta) = -\frac{\Gamma(\frac{2 \beta}{\gamma}-\frac{4}{\gamma^2})\Gamma(\frac{2\alpha}{\gamma}-\frac{\beta}{\gamma})}{ \Gamma(-1+\frac{2\alpha}{\gamma}+\frac{\beta}{\gamma}-\frac{4}{\gamma^2})} \overline{R}(\beta,1,1)\overline{G}(\alpha, 2Q-\beta ).
\end{align}
This equation should be viewed as an equality of meromorphic functions and it gives a definition of $\overline{G}(\alpha, 2Q-\beta )$ valid for $\alpha, \beta$ satisfying $\beta \in (\frac{\gamma}{2}, Q)$ and $\frac{\gamma}{2} - \alpha < \frac{\beta}{2} < \alpha $.
\end{lemma}

\begin{proof}
We work with $\chi = \frac{\gamma}{2}$. We have seen two ways of calculating $\tilde{C}_2$ based on the value of $\beta$:
\begin{equation}
\tilde{C}_2 = 
\begin{cases}
p\frac{\Gamma(2C-2)\Gamma(2-2C-\frac{\gamma^2}{4})}{\Gamma(-\frac{\gamma^2}{4})}(1-e^{i \pi (3-2C)}) \overline{G}(\alpha, \beta+\frac{\gamma}{2}), & \beta<\frac{2}{\gamma},\\
-\frac{\Gamma(1-\frac{2(Q-\beta)}{\gamma}) \Gamma(-p+\frac{2}{\gamma}(Q-\beta))}{\Gamma(-p)}\overline{R}(\beta, 1, e^{i\pi\frac{\gamma^2}{4}}) \overline{G}(\alpha, 2Q-\beta -\frac{\gamma}{2}), &  \frac{2}{\gamma}<\beta <Q.
\end{cases}
\end{equation}
Since $\tilde{G}_{\chi}(s)$ is complex analytic in $\beta$ in a complex neighborhood of $\frac{2}{\gamma}$, this implies the analyticity of $\tilde{C}_2$ around $ \beta = \frac{2}{\gamma}$. This implies that there is an equality between the two expressions for $\tilde{C}_2$ viewed as meromorphic functions of $\beta$ in a neighborhood of $\frac{2}{\gamma}$. Lastly from equation \eqref{equation R2} of Section \ref{sec_3pt} we have a shift equation that relates  $\overline{R}(\beta+\frac{\gamma}{2}, 1, 1)$ and $ \overline{R}(\beta, 1, e^{i\pi\frac{\gamma^2}{4}})$. Therefore we can rewrite the relation in the desired way claimed in the lemma. \end{proof} 

With both Lemma \ref{bb_continuation} and Lemma \ref{reflection1} we can now finish the proof of Proposition \ref{shift_bulk_boundary}. 

\begin{proof}[Proof of Proposition \ref{shift_bulk_boundary}] 
We switch to $\chi = \frac{2}{\gamma}$ to deduce the second shift equation. Thanks to Lemma \ref{bb_continuation} we can view $\beta \rightarrow \overline{G}(\alpha, \beta)$ as meromorphic function defined in a complex neighborhood of $\mathbb{R}$. We start from:
\begin{equation}
\tilde{C}_2 = -\frac{\Gamma(1-\frac{2(Q-\beta)}{\gamma}) \Gamma(-p+\frac{2}{\gamma}(Q-\beta))}{\Gamma(-p)}\overline{R}(\beta, 1, e^{i\pi }) \overline{G}(\alpha, 2Q-\beta - \frac{2}{\gamma}).
\end{equation}
By applying equation \eqref{eq:reflection1} of Lemma \ref{reflection1} with $\beta$ replaced by $\beta+\frac{2}{\gamma}$, we obtain:
\begin{equation}
\tilde{C}_2 =  \frac{\Gamma(-1+\frac{2\alpha}{\gamma}+\frac{\beta}{\gamma}-\frac{2}{\gamma^2}) \Gamma(1-\frac{2(Q-\beta)}{\gamma}) }{\Gamma(\frac{2 \beta}{\gamma} ) \Gamma(-1 + \frac{2\alpha}{\gamma} +\frac{\beta}{2} - \frac{6}{\gamma^2} )} \frac{ \overline{R}(\beta, 1, e^{i\pi })}{\overline{R}(\beta + \frac{2}{\gamma},1,1)} \overline{G}(\alpha, \beta + \frac{2}{\gamma}).
\end{equation}

Using the shift equation \eqref{eq:shift_r_inproof2} on $\overline{R}$ with $\mu_1= \mu_2 =1$ and $\beta $ replaced by $\beta - \frac{2}{\gamma}$ we can write:
\begin{align}
\frac{\overline{R}(\beta, 1, e^{i\pi})}{\overline{R}(\beta+\frac{2}{\gamma}, 1, 1)} 
=- \frac{2}{\gamma(Q-\beta)} (2\pi)^{\frac{4}{\gamma^2}-1}  \frac{\Gamma(\frac{2\beta}{\gamma}) \Gamma(1-\frac{2\beta}{\gamma})}{\Gamma(1-\frac{\gamma^2}{4})^{\frac{4}{\gamma^2}} } (1-e^{-i\pi (\frac{2\beta}{\gamma} - \frac{4}{\gamma^2})}).
\end{align}
Plugging this formula into the expression of $\tilde{C}_2$ one obtains:
\begin{align*}
\tilde{C}_2&=   \frac{\frac{4}{\gamma^2} (2\pi)^{\frac{4}{\gamma^2}-1}}{\Gamma(1-\frac{\gamma^2}{4})^{\frac{4}{\gamma^2}}}  \frac{\Gamma(-1+\frac{2\beta}{\gamma}-\frac{4}{\gamma^2}) \Gamma(1-\frac{2\beta}{\gamma}) \Gamma(-1+\frac{2\alpha}{\gamma}+\frac{\beta}{\gamma}-\frac{2}{\gamma^2})}{\Gamma(-1+\frac{2\alpha}{\gamma}+\frac{\beta}{\gamma}-\frac{6}{\gamma^2})} (1-e^{-i\pi (\frac{2\beta}{\gamma} - \frac{4}{\gamma^2})}) \overline{G}(\alpha,\beta+\frac{2}{\gamma}),\\
&= \frac{\frac{4}{\gamma^2} (2\pi)^{\frac{4}{\gamma^2}-1}}{\Gamma(1-\frac{\gamma^2}{4})^{\frac{4}{\gamma^2}}}  \frac{\Gamma(2C-2) \Gamma(1-\frac{2\beta}{\gamma})\Gamma(-1+\frac{2\alpha}{\gamma}+\frac{\beta}{\gamma}-\frac{2}{\gamma^2}) }{ \Gamma(2A)} (1+e^{i\pi \frac{2}{\gamma}(Q-\beta)}) \overline{G}(\alpha,\beta+\frac{2}{\gamma}).
\end{align*}
We can find easily the other coefficients:
\begin{align*}
\tilde{C}_1 = \overline{G}(\alpha, \beta-\frac{2}{\gamma}), \quad \hat{C}_1 = e^{i\pi p}\overline{G}(\alpha, \beta-\frac{2}{\gamma}), \quad  \hat{C_2} = e^{i\pi (p-\frac{2}{\gamma}(Q-\beta))} \tilde{C}_2, \quad \tilde{B}_1 = \hat{B}_1.
\end{align*}
As in the previous subsection we can thus write:
\begin{align*}
\frac{\Gamma(C)\Gamma(C-A-B)}{\Gamma(C-A)\Gamma(C-B)} \tilde{C}_1 (1-e^{i\pi p}) + \frac{\Gamma(2-C)\Gamma(C-A-B)}{\Gamma(1-A)\Gamma(1-B)} \tilde{C}_2 (1-e^{i\pi (p-\frac{2}{\gamma}(Q-\beta))})=0.
\end{align*}
Then we can deduce the shift equation,
\begin{align*}
\frac{\overline{G}(\alpha,\beta+\frac{2}{\gamma})}{\overline{G}(\alpha,\beta-\frac{2}{\gamma})} & = \frac{ \frac{\gamma^2}{4} (2\pi)^{1-\frac{4}{\gamma^2}} \Gamma(1-\frac{\gamma^2}{4})^{\frac{4}{\gamma^2}}}{\Gamma(2C-2) \Gamma(1-\frac{2\beta}{\gamma})\Gamma(-1+\frac{2\alpha}{\gamma}+\frac{\beta}{\gamma}-\frac{2}{\gamma^2})}
\frac{ \Gamma(2A) \Gamma(C)\Gamma(1-A)\Gamma(1-B)}{\Gamma(2-C)\Gamma(A+\frac{1}{2})\Gamma(B+\frac{1}{2})}\\
& \quad \times \frac{\sin( \pi A)}{2\sin(\pi (B+\frac{1}{2}))\sin(\pi (C-\frac{1}{2}))}\\
 &= \frac{\frac{\gamma^2}{4} \Gamma(1-\frac{\gamma^2}{4})^{\frac{4}{\gamma^2}}}{(2\pi)^{\frac{4}{\gamma^2}-1}2^{2C-1}\pi} \frac{\Gamma(1-2B)\Gamma(\frac{3}{2}-C)^2}{ \Gamma(2-2C) \Gamma(1-\frac{2\beta}{\gamma})\Gamma(-1+\frac{2\alpha}{\gamma}+\frac{\beta}{\gamma}-\frac{2}{\gamma^2})}\\
&= \frac{\frac{\gamma^2}{4} \Gamma(1-\frac{\gamma^2}{4})^{\frac{4}{\gamma^2}}}{(2\pi)^{\frac{4}{\gamma^2}}2^{-1+\frac{2\beta}{\gamma}-\frac{4}{\gamma^2}}}\frac{\Gamma(\frac{2 \alpha}{\gamma} -\frac{\beta}{\gamma} -\frac{2}{\gamma^2}) \Gamma(1-\frac{\beta}{\gamma}+\frac{2}{\gamma^2})^2}{\Gamma(1-\frac{2\beta}{\gamma}+\frac{4}{\gamma^2})\Gamma(1-\frac{2\beta}{\gamma})  \Gamma(-1+\frac{2\alpha}{\gamma}+\frac{\beta}{\gamma}-\frac{2}{\gamma^2})},
\end{align*}
and finally:
\begin{equation}
\frac{\overline{G}(\alpha,\beta+\frac{4}{\gamma})}{\overline{G}(\alpha,\beta)} = \frac{\frac{\gamma^2}{4} \Gamma(1-\frac{\gamma^2}{4})^{\frac{4}{\gamma^2}}}{(2\pi)^{\frac{4}{\gamma^2}}2^{-1+\frac{2\beta}{\gamma}}} \frac{\Gamma(\frac{2 \alpha}{\gamma} -\frac{\beta}{\gamma} -\frac{4}{\gamma^2}) \Gamma(1-\frac{\beta}{\gamma})^2}{\Gamma(1-\frac{2\beta}{\gamma})\Gamma(1-\frac{2\beta}{\gamma} -\frac{4}{\gamma^2}) \Gamma(-1+\frac{2\alpha}{\gamma}+\frac{\beta}{\gamma}) }.
\end{equation}
Hence we have proven Proposition \ref{shift_bulk_boundary}.
\end{proof}

\section{The boundary two-point and three-point functions}\label{sec_3pt}

The goal of this section is to prove Theorem \ref{main_th2}. We follow roughly the same steps as in the previous section, except we will derive explicitly the expression for the boundary two-point function $\overline{R}$ used in the proof of Theorem \ref{main_th1}. Again we will rely on the hypergeometric equations shown in Section \ref{sec_BPZ} to obtain shift equations on $\overline{H}$. The difference here is that the functional equation obtained will contain $3$ terms instead of $2$, see for instance equation \eqref{eq_3pt_1} below. Throughout this section we will use:
\begin{equation}
q = \frac{1}{\gamma}(2 Q  - \beta_1 - \beta_2 - \beta_3 + \chi).
\end{equation}
We introduce the auxiliary function for $\chi = \frac{\gamma}{2} \text{ or }\frac{2}{\gamma}$ and $t\in \H$,
\begin{equation}\label{eq:def_H_chi}
H_{\chi}(t)  = \mathbb{E} \left[ \left( \intr \frac{(t-x)^{\frac{\gamma\chi}{2}}}{|x|^{\frac{\gamma \beta_1 }{2}} |x-1|^{\frac{\gamma \beta_2 }{2}}}  g(x)^{\frac{\gamma^2}{8}(q-1)} e^{\frac{\gamma}{2} X(x)} d \mu(x) \right)^{q} \right],
\end{equation}
where:
\begin{equation}
d\mu(x) = \mu_1\mathbf{1}_{(-\infty,0)}(x)dx + \mu_2  \mathbf{1}_{(0,1)}(x)dx + \mu_3 \mathbf{1}_{(1,\infty)}(x)dx.
\end{equation}
To start the range of parameters we want to work with is:
\begin{equation}\label{para_Ht_1}
\beta_i < Q, \: \:   \mu_1 \in (0,\infty), \: \: \mu_2, \mu_3 \in - \overline{\H} \quad \text{and} \quad   q < \frac{4}{\gamma^2} \wedge \min_i \frac{2}{\gamma}(Q - \beta_i).
\end{equation}
By $\mu_2, \mu_3 \in - \overline{\H} $ we mean that their argument is chosen in $[-\pi,0]$. Furthermore, to define $(t-x)^{\frac{\gamma\chi}{2}}$ in the case $\chi = \frac{\gamma}{2}$, for $t \in \mathbb{H}$ and $x \in \mathbb{R}$ we choose the argument of $(t-x)$ in $(0, \pi)$. With this choice the GMC integral
\begin{equation}
\intr \frac{(t-x)^{\frac{\gamma\chi}{2}}}{|x|^{\frac{\gamma \beta_1 }{2}} |x-1|^{\frac{\gamma \beta_2 }{2}}}  g(x)^{\frac{\gamma^2}{8}(q-1)} e^{\frac{\gamma}{2} X(x)} d \mu(x),
\end{equation}
never hits the line $(-\infty,0)$ and so its argument can be chosen to be in $(-\pi, \pi)$ and its $q$ power is thus well-defined. The finiteness of this moment is then given by Proposition \ref{lem:GMC-moment}.
Now $t \mapsto H_{\chi}(t)$ is holomorphic in $\H$ and it is shown in Section \ref{sec_BPZ} that $H_{\chi}(t)$ satisfies the hypergeometric equation,
\begin{equation}
t(1-t)\partial_t^2 H_{\chi}(t) + (C-(A+B+1)t)\partial_t H_{\chi}(t)-AB H_{\chi}(t) = 0,
\end{equation}
with parameters:
\begin{equation}
A = -q\frac{\gamma\chi}{2}, \quad B = -1+ \chi(\beta_1+\beta_2-2\chi + q\frac{\gamma}{2}), \quad C = \chi(\beta_1 - \chi).
\end{equation}
We will also use the auxiliary function $\tilde{H}_{\chi}(t)$, 
\begin{equation}
\tilde{H}_{\chi}(t)  = \mathbb{E} \left[ \left( \intr \frac{(x-t)^{\frac{\gamma\chi}{2}}}{|x|^{\frac{\gamma \beta_1 }{2}} |x-1|^{\frac{\gamma \beta_2 }{2}}}  g(x)^{\frac{\gamma^2}{8}(q-1)} e^{\frac{\gamma}{2} X(x)} d {\mu}(x) \right)^q \right],
\end{equation}
which is defined with the following parameter choices:
\begin{equation}
t \in -\H, \: \: \beta_i < Q, \: \:   \mu_1, \mu_2 \in - \overline{\H}, \: \: \mu_3 \in  (0,\infty), \quad \text{and} \quad   q < \frac{4}{\gamma^2} \wedge \min_i \frac{2}{\gamma}(Q - \beta_i).
\end{equation}
We choose the argument of $\mu_1, \mu_2$ in $[-\pi,0]$ and the argument $(x-t)$ in $[0, \pi]$. With these choices of parameters the GMC integral is again a complex number which is avoiding the half-line $(-\infty,0)$ and whose argument can be chosen in $(-\pi, \pi)$. $\tilde{H}_{\chi}(t)$ obeys the exact same hypergeometric equation as $ H_{\chi}(t)$. 

\subsection{First shift equation for the three-point function}

We start again by proving the first shift equation on $\overline{H}$ by setting $\chi = \frac{\gamma}{2}$ and working with the functions $H_{\frac{\gamma}{2}}(t)$ and $\tilde{H}_{\frac{\gamma}{2}}(t)$. For this first lemma the parameter range on the $\beta_i$ and $\mu_i$ is such that each $\overline{H}$ appearing is defined probabilistically (without analytic continuation) meaning the bounds \eqref{para_H} are satisfied.
\begin{lemma}\label{Shift1_H_proba} ($\frac{\gamma}{2}$-shift equations for $\overline{H}$) The following two shift equations for $\overline{H}$ hold,
\begin{align}\label{eq_3pt_1}
&\overline{H}^{(\beta_1 , \beta_2 - \frac{\gamma}{2}, \beta_3)}_{( \mu_1, \mu_2,   \mu_3)} = \frac{\Gamma(\frac{\gamma \beta_1}{2} - \frac{\gamma^2}{4}) \Gamma(1 - \frac{\gamma \beta_2}{2} + \frac{\gamma^2}{4})}{\Gamma(\frac{\gamma \beta_1}{2} +(q-1) \frac{\gamma^2}{4}) \Gamma(1 - \frac{\gamma}{2}\beta_2 - (q-1)\frac{\gamma^2}{4})) } \overline{H}^{(\beta_1 - \frac{\gamma}{2}, \beta_2, \beta_3)}_{( \mu_1, e^{\frac{i \pi \gamma^2}{4}} \mu_2,  \mu_3)} \\
&+ \frac{ q \Gamma(2 + \frac{\gamma^2}{4} - \frac{\gamma \beta_1}{2}) \Gamma(1 - \frac{\gamma \beta_2}{2} + \frac{\gamma^2}{4})\Gamma(-1+\frac{\gamma \beta_1}{2} - \frac{\gamma^2}{4} ) \Gamma(1 - \frac{\gamma \beta_1}{2}) }{\Gamma(1 + \frac{q \gamma^2}{4}) \Gamma(2 - \frac{\gamma}{2}(\beta_1 + \beta_2)  - (q-2)\frac{\gamma^2}{4} )  ) \Gamma(- \frac{\gamma^2}{4}) }   \left( \mu_1 - \mu_2 e^{i \pi \frac{\gamma \beta_1}{2}} \right)  \overline{H}^{(\beta_1 + \frac{\gamma}{2}, \beta_2, \beta_3)}_{( \mu_1, e^{\frac{i \pi  \gamma^2}{4}} \mu_2,   \mu_3)}, \nonumber
\end{align}
and,
\begin{align}\label{eq_3pt_2}
& \Gamma(1- \frac{\gamma \beta_2  }{2}) (\mu_3 -\mu_2 e^{ \frac{i\pi \gamma \beta_2}{2}}) \overline{H}^{(\beta_1 , \beta_2 + \frac{\gamma}{2} , \beta_3)}_{(\mu_1, e^{i\pi\frac{\gamma^2}{4}}\mu_2,   \mu_3)} = - \frac{4}{q \gamma^2} \frac{ \Gamma(1-\frac{\gamma^2}{4}) \Gamma( \frac{\gamma \beta_1}{2} - \frac{\gamma^2}{4} ) }{\Gamma( -  \frac{q \gamma^2}{4} ) \Gamma(-1 + \frac{\gamma \beta_1}{2} + \frac{\gamma \beta_2}{2} - \frac{\gamma^2}{2} + q\frac{\gamma^2}{4} ) }  \overline{H}^{(\beta_1 -\frac{\gamma}{2} , \beta_2, \beta_3)}_{(  \mu_1,  \mu_2,   \mu_3)}  \nonumber\\
&  \quad + \frac{ \Gamma(2  -  \frac{\gamma \beta_1}{2} + \frac{\gamma^2}{4}) \Gamma(1-\frac{\gamma \beta_1}{2}) \Gamma(\frac{\gamma \beta_1}{2} - \frac{\gamma^2}{4} -1 )}{\Gamma(1- \frac{ \gamma \beta_1}{2} + \frac{\gamma^2}{4} - q\frac{\gamma^2}{4} ) \Gamma( \frac{ \gamma \beta_2 }{2} - \frac{\gamma^2}{4} +q\frac{\gamma^2}{4})}(\mu_1 -\mu_2 e^{i\pi (\frac{\gamma^2}{4} - \frac{ \gamma \beta_1}{2})}) \overline{H}^{(\beta_1 + \frac{\gamma}{2} , \beta_2 , \beta_3)}_{(\mu_1, \mu_2,   \mu_3)},
\end{align}
provided that for every function $\overline{H}$ appearing, its parameters obey the constraint \eqref{para_H} required for $\overline{H}$ to be defined probabilistically. 
\end{lemma}

\begin{proof}
We first choose the parameters $\beta_1, \beta_2, \beta_3$ and $\mu_1, \mu_2 , \mu_3$ so that they obey the constraint \eqref{para_Ht_1} plus the following extra constraint on $\beta_1$:
\begin{equation}\label{para_range_301}
\frac{\gamma}{2} <  \beta_1 < \frac{2}{\gamma}.
\end{equation}
The function $t \mapsto H_{\frac{\gamma}{2}}(t)$ is holomorphic on $\mathbb{H}$ and extends continuously on $\overline{\mathbb{H}}$. Using the basis of solutions of the hypergeometric equation recalled in Section \ref{sec_hypergeo}, we can write the following solutions around $t=0$, $t= 1$ and $t=\infty$, under the assumption that neither $C$, $C-A-B$, or $A-B$ are integers:\footnote{Again the values excluded here are recovered by a continuity argument.}
\begin{align}
H_{\frac{\gamma}{2}}(t) &= C_1 F(A,B,C,t) + C_2 t^{1 - C} F(1 + A-C, 1 +B - C, 2 -C, t) \\
 &= B_1 F(A,B,1+A+B- C, 1 -t) + B_2 (1-t)^{C -A -B} F(C- A, C- B, 1 + C - A - B , 1 -t) \nonumber \\ \nonumber
 &= D_1 e^{i\pi A} t^{-A} F(A,1+A-C,1+A-B,t^{-1}) + D_2 e^{i\pi B}t^{-B} F(B, 1 +B - C, 1 +B - A, t^{-1}).
\end{align}
The constants $C_1, C_2, B_1, B_2, D_1, D_2$ are again the real constants that parametrize the solution space around the different points. As was performed in Section \ref{sec_bulk_boundary} we will identify them by Taylor expansion. First we note that by setting $t=0$:
\begin{equation}
C_1 = H_{\frac{\gamma}{2}}(0) = \overline{H}^{(\beta_1 - \frac{\gamma}{2}, \beta_2, \beta_3)}_{( \mu_1, e^{\frac{i \pi  \gamma^2}{4}} \mu_2, e^{\frac{i \pi  \gamma^2}{4}} \mu_3)}.
\end{equation}
Next to find $C_2$ we go at higher order in the $t \rightarrow 0$ limit. We then expand the increment $H_{\frac{\gamma}{2}}(t) - H_{\frac{\gamma}{2}}(0)$ at first order following the same step as for \eqref{OPE_no_reflection}: 
\begin{align}
&H_{\frac{\gamma}{2}}(t) - H_{\frac{\gamma}{2}}(0) \\ \nonumber
&= q \int_{\mathbb{R}} d \mu(x_1) \frac{(t - x_1)^{\frac{\gamma^2}{4}}- (- x_1)^{\frac{\gamma^2}{4}}}{\vert x_1 \vert^{\frac{\gamma \beta_1}{2}} \vert x_1 - 1 \vert^{\frac{\gamma \beta_2}{2}}} g(x_1)^{\frac{\gamma^2}{8}(q-1)} \mathbb{E} \left[ e^{\frac{\gamma}{2} X(x_1)} \left( \intr \frac{(-x)^{\frac{\gamma^2}{4}} g(x)^{\frac{\gamma^2}{8}(q-1)}}{|x|^{\frac{\gamma \beta_1 }{2}} |x-1|^{\frac{\gamma \beta_2 }{2} } }   e^{\frac{\gamma}{2} X(x)} d \mu(x) \right)^{q-1} \right]\\
& \nonumber \quad + o(t^{1-C})\\ \nonumber
&= q \int_{\mathbb{R}} d \mu(x_1) \frac{(t - x_1)^{\frac{\gamma^2}{4}}- (- x_1)^{\frac{\gamma^2}{4}}}{\vert x_1 \vert^{\frac{\gamma \beta_1}{2}} \vert x_1 - 1 \vert^{\frac{\gamma \beta_2}{2}}} \mathbb{E} \left[ \left( \intr \frac{(-x)^{\frac{\gamma^2}{4}} g(x)^{\frac{\gamma^2}{8}(q-2)}}{|x|^{\frac{\gamma \beta_1 }{2}} |x-1|^{\frac{\gamma \beta_2 }{2} } |x-x_1|^{\frac{\gamma^2}{2}}}   e^{\frac{\gamma}{2} X(x)} d \mu(x) \right)^{q-1} \right] + o(t^{1-C})\\ \nonumber
&= q t^{1-C} \left(  \int_{\mathbb{R}} d u \left( \mu_1\mathbf{1}_{(-\infty,0)}(u) + \mu_2  \mathbf{1}_{(0,+\infty)}(u) \right) \frac{(1 - u)^{\frac{\gamma^2}{4}}- (- u)^{\frac{\gamma^2}{4}}}{\vert u \vert^{\frac{\gamma \beta_1}{2}} } \right) \overline{H}^{(\beta_1 + \frac{\gamma}{2}, \beta_2, \beta_3)}_{( \mu_1, e^{\frac{i \pi  \gamma^2}{4}} \mu_2, e^{\frac{i \pi  \gamma^2}{4}} \mu_3)} + o(t^{1-C}).
\end{align}

The integral above in front of $\overline{H}$ converges thanks to the condition \eqref{para_range_301} and can be evaluated using \eqref{equation integral 2}. Also notice with our conventions the argument of $(- x)$ is either $0$ or $\pi$. Hence one obtains:
\begin{equation}\label{eq_ope_gluing2}
C_2  = q \frac{\Gamma(-1+\frac{\gamma \beta_1}{2} - \frac{\gamma^2}{4} ) \Gamma(1 - \frac{\gamma \beta_1}{2})}{\Gamma(- \frac{\gamma^2}{4})} \left( \mu_1 - \mu_2 e^{i \pi \frac{\gamma \beta_1}{2}} \right)  \overline{H}^{(\beta_1 + \frac{\gamma}{2}, \beta_2, \beta_3)}_{( \mu_1, e^{\frac{i \pi  \gamma^2}{4}} \mu_2,  e^{\frac{i \pi  \gamma^2}{4}}\mu_3)}.
\end{equation}
Similarly by setting $t=1 $ we get:
\begin{align}
B_1 & = \overline{H}^{(\beta_1 , \beta_2 - \frac{\gamma}{2} , \beta_3)}_{( \mu_1, \mu_2, e^{\frac{i \pi  \gamma^2}{4}} \mu_3)}.
\end{align}
The connection formula \eqref{connection1} between $C_1$, $C_2$, and $B_1$ then implies the shift equation \eqref{eq_3pt_1} in the range of parameters constraint by \eqref{para_Ht_1} and \eqref{para_range_301}, after performing furthermore the replacement $\mu_3 \rightarrow e^{-\frac{i \pi \gamma^2}{4}} \mu_3 $ (which also rotates the domain where $\mu_3$ belongs).  To lift these constraint we then invoke the analyticity of $\overline{H}$ as a function of its parameters given by Lemma \ref{lem_analycity}. We have thus shown that \eqref{eq_3pt_1} holds whenever all three $\overline{H}$ appearing are well-defined as GMC quantities. Now we repeat these steps with $\tilde{H}_{\frac{\gamma}{2}}$ to obtain the shift equation with the opposite phase. We expand $\tilde{H}_{\frac{\gamma}{2}}(t)$, 
\begin{align}
\tilde{H}_{\frac{\gamma}{2}}(t) &= \tilde{C}_1 F(A,B,C,t) + \tilde{C}_2 t^{1 - C} F(1 + A-C, 1 +B - C, 2 -C, t) \\ \nonumber
 &= \tilde{B}_1 F(A,B,1+A+B- C, 1 -t) + \tilde{B}_2 (1-t)^{C -A -B} F(C- A, C- B, 1 + C - A - B , 1 -t) \\ \nonumber
 &= \tilde{D}_1 e^{i\pi A} t^{-A} F(A,1+A-C,1+A-B,t^{-1}) + \tilde{D}_2 e^{i\pi B}t^{-B} F(B, 1 +B - C, 1 +B - A, t^{-1}),
\end{align}
and compute in the same way the values of $ \tilde{C}_1, \tilde{C}_2,  \tilde{B}_2$: 
\begin{align}
\tilde{C}_1 &=   \overline{H}^{(\beta_1 - \frac{\gamma}{2} , \beta_2 , \beta_3)}_{(  e^{\frac{i \pi  \gamma^2}{4}} \mu_1,  \mu_2,   \mu_3)} , \\
\tilde{C}_2 &= q \frac{\Gamma(-1+\frac{\gamma \beta_1}{2} - \frac{\gamma^2}{4} ) \Gamma(1 - \frac{\gamma \beta_1}{2})}{\Gamma(- \frac{\gamma^2}{4})} \left( \mu_1 e^{i\pi\frac{\gamma^2}{4}} - \mu_2 e^{i \pi (\frac{\gamma^2}{4}-\frac{\gamma \beta_1}{2})} \right)  \overline{H}^{(\beta_1 + \frac{\gamma}{2} , \beta_2 , \beta_3)}_{(e^{\frac{i \pi  \gamma^2}{4}} \mu_1, \mu_2,   \mu_3)},\\
\tilde{B}_2 &= q \frac{\Gamma(-1+\frac{\gamma \beta_2}{2} - \frac{\gamma^2}{4} ) \Gamma(1 - \frac{\gamma \beta_2}{2})}{\Gamma(- \frac{\gamma^2}{4})} \left( \mu_3 - \mu_2 e^{i \pi \frac{\gamma \beta_1}{2}} \right)  \overline{H}^{(\beta_1 + \frac{\gamma}{2} , \beta_2 , \beta_3)}_{(e^{\frac{i \pi  \gamma^2}{4}} \mu_1, e^{\frac{i\pi\gamma^2}{4}}\mu_2 , \mu_3)}. 
\end{align}
Then the connection formula \eqref{connection1} implies the shift equation \eqref{eq_3pt_2}.
\end{proof}

\subsection{Solving the boundary two-point function}

At this point we will postpone computing the boundary three-point function $\overline{H}$ and focus on determining shift equations that will completely specify $\overline{R}$. Once we have proved the exact formula for $\overline{R}$, it will then be possible to finish computing $\overline{H}$. In a similar way the value of $\overline{R}$ was required in the proof of the value of $\overline{G}$ in Section \ref{sec_bulk_boundary}.

\subsubsection{First shift equation on the reflection coefficient}

We start again by proving a first shift equation for $\overline{R}(\beta_1, \mu_1, \mu_2)$ restricted to the case where $\overline{R}$ is defined probabilistically, meaning the parameters obey the bounds of Definition \ref{def_four_correls}.

\begin{lemma}\label{lem_shift_eq_R} Consider $\gamma \in (0,2)$, $\beta_1 \in (\frac{\gamma}{2}, \frac{2}{\gamma})$, $\mu_1, \mu_2 \in \mathbb{C}$ such that both pairs $(\mu_1, \mu_2)$ and $(\mu_1, e^{\frac{i \pi \gamma^2}{4}}\mu_2)$ obey the condition of Definition \ref{half-space}. Then $\overline{R}(\beta, \mu_1, \mu_2)$ obeys,
\begin{equation}\label{equation R1}
\overline{R}(\beta_1, \mu_1, \mu_2) = - \frac{\Gamma(-1+\frac{\gamma \beta_1}{2} - \frac{\gamma^2}{4} ) \Gamma(2 - \frac{\gamma \beta_1}{2}) }{ \Gamma(1- \frac{\gamma^2}{4}) }  \left( \mu_1 - \mu_2 e^{i \pi \frac{\gamma \beta_1}{2}} \right) \overline{R}(\beta_1 + \frac{\gamma}{2}, \mu_1,  e^{\frac{i \pi \gamma^2}{4}} \mu_2).
\end{equation}
Similarly for $\beta_1 \in (0, \frac{2}{\gamma} - \frac{\gamma}{2})$ and the same constraint on $\mu_1, \mu_2$ as before,
\begin{equation}\label{equation R2}
\overline{R}(\beta_1+\frac{\gamma}{2}, \mu_1, e^{\frac{i\pi\gamma^2}{4}}\mu_2) =-\frac{\Gamma(-1+\frac{\gamma \beta_1}{2} ) \Gamma(2 - \frac{\gamma \beta_1}{2}-\frac{\gamma^2}{4}) }{ \Gamma(1- \frac{\gamma^2}{4}) }  \left( \mu_1 - \mu_2 e^{-i \pi \frac{\gamma \beta_1}{2}} \right) \overline{R}(\beta_1 + \gamma, \mu_1, \mu_2).
\end{equation}
\end{lemma}

\begin{proof}

The key idea to derive the shift equations for $\overline{R}$ is to take suitable limits of the shift equations of Lemma \ref{Shift1_H_proba} to make $\overline{R}$ appear from $\overline{H}$. We will use extensively the Lemma \ref{lim_H_R} of Section \ref{sec_def_reflection} which provides this limit. Fix a $\beta_1 \in (\frac{\gamma}{2}, \frac{2}{\gamma}) $. Consider two small parameters $\epsilon, \eta >0$ and set $ \beta_2 = \beta_1  - \epsilon $, $\beta_3 = \beta_1 - \beta_2 +\frac{\gamma}{2} + \eta = \frac{\gamma}{2} + \epsilon +\eta$. Notice that for this parameter choice the three $\overline{H}$ functions appearing in the shift equation \eqref{eq_3pt_1} are well-defined. Now the idea is to match the poles of \eqref{eq_3pt_1} as $\eta$ goes to $0$ or in other words as $\beta_3$  goes to $ \beta_1 -\beta_2+\frac{\gamma}{2}$. By applying Lemma \ref{lim_H_R} we get:

\begin{align*}
&\lim_{\beta_3 \downarrow \beta_1 - \beta_2 + \frac{\gamma}{2}} (\beta_2 + \beta_3 - \beta_1 - \frac{\gamma}{2}) \overline{H}^{(\beta_1, \beta_2 - \frac{\gamma}{2}, \beta_3)}_{(\mu_1, \mu_2, \mu_3)}  = 2 (Q-\beta_1) \overline{R}(\beta_1, \mu_1, \mu_2) \\
&\lim_{\beta_3 \downarrow \beta_1 - \beta_2 + \frac{\gamma}{2}} (\beta_2 + \beta_3 - \beta_1 - \frac{\gamma}{2})  \frac{\Gamma(\frac{\gamma \beta_1}{2} - \frac{\gamma^2}{4}) \Gamma(1 - \frac{\gamma \beta_2}{2} + \frac{\gamma^2}{4})}{\Gamma(\frac{\gamma \beta_1}{2} +(q-1) \frac{\gamma^2}{4}) \Gamma(1 - \frac{\gamma}{2}\beta_2 - (q-1)\frac{\gamma^2}{4})) } \overline{H}^{(\beta_1 - \frac{\gamma}{2}, \beta_2, \beta_3)}_{( \mu_1, e^{\frac{i \pi \gamma^2}{4}} \mu_2,  \mu_3)} = 0  \\
& \lim_{\beta_3 \downarrow \beta_1 - \beta_2 + \frac{\gamma}{2}} (\beta_2 + \beta_3 - \beta_1 - \frac{\gamma}{2})  \Bigg[ \frac{ q \Gamma(2 + \frac{\gamma^2}{4} - \frac{\gamma \beta_1}{2}) \Gamma(1 - \frac{\gamma \beta_2}{2} + \frac{\gamma^2}{4})\Gamma(-1+\frac{\gamma \beta_1}{2} - \frac{\gamma^2}{4} ) \Gamma(1 - \frac{\gamma \beta_1}{2}) }{\Gamma(1 + \frac{q \gamma^2}{4}) \Gamma(2 - \frac{\gamma}{2}(\beta_1 + \beta_2)  - (q-2)\frac{\gamma^2}{4} )  ) \Gamma(- \frac{\gamma^2}{4}) } \\
& \quad \quad \quad  \quad \quad \quad  \quad \quad \quad \times  \left( \mu_1 - \mu_2 e^{i \pi \frac{\gamma \beta_1}{2}} \right)  \overline{H}^{(\beta_1 + \frac{\gamma}{2}, \beta_2, \beta_3)}_{( \mu_1, e^{\frac{i \pi  \gamma^2}{4}} \mu_2,   \mu_3)} \Bigg] \\
& =  \frac{\frac{2}{\gamma}(Q -\beta_1)   \Gamma(-1+\frac{\gamma \beta_1}{2} - \frac{\gamma^2}{4} ) \Gamma(1 - \frac{\gamma \beta_1}{2}) }{ \Gamma(- \frac{\gamma^2}{4}) }  \left( \mu_1 - \mu_2 e^{i \pi \frac{\gamma \beta_1}{2}} \right) 2(\frac{2}{\gamma} - \beta_1) \overline{R}(\beta_1 + \frac{\gamma}{2}, \mu_1,  e^{\frac{i \pi \gamma^2}{4}} \mu_2).\\
\end{align*}
This leads to the equation \eqref{equation R1} on the reflection coefficient.
By using the alternative auxiliary function $\tilde{H}_{\frac{\gamma}{2}}(t)$ along the same lines we obtain equation \eqref{equation R2}.
Indeed, equations \eqref{equation R1} (with $\beta = \beta_1$) and \eqref{equation R2} (with $\beta = \beta_1 + \frac{\gamma}{2}$) are both stated for $\beta \in (\frac{\gamma}{2}, \frac{2}{\gamma})$ and when viewed in terms of $\beta$ they differ only by a sign in each occurence of the complex unit $i$. Therefore their proofs are essentially identical and we omit the proof of \eqref{equation R2}. This completes the proof of the lemma.
\end{proof}

At this point in the proof we need to show $\overline{R}(\beta_1,\mu_1,\mu_2)$ is analytic in $\beta_1$ in the interval $(\frac{\gamma}{2},Q)$ and in $\mu_1, \mu_2$ in the complex domain where Definition \ref{half-space} is satisfied. For this we will again take a limit from the first shift equation. 

\begin{lemma}(Analyticity of $\overline{R}(\beta_1,\mu_1,\mu_2)$ in $\beta_1$ and $\mu_1, \mu_2$) \label{lem:analycity_R}
For all $\mu_1, \mu_2 $ obeying Definition \ref{half-space}, the function $\beta_1 \mapsto \overline{R}(\beta_1,\mu_1,\mu_2)$ is complex analytic on a complex neighborhood of any compact set $K \subset (\frac{\gamma}{2}, Q)$. For all $\beta_1 \in (\frac{\gamma}{2},Q)$, the function $(\mu_1, \mu_2) \mapsto \overline{R}(\beta_1,\mu_1,\mu_2) $ is complex analytic on any compact set $\tilde{K}$ contained in the open set of pairs $(\mu_1, \mu_2)$ obeying Definition \ref{half-space}.
\end{lemma}
\begin{proof}
In the shift equation \eqref{eq_3pt_1}, set $\beta_1 = \frac{\gamma}{2}+\eta$, $\frac{\gamma}{2}< \beta_2 = \beta_3<Q$. We multiply the shift equation \eqref{eq_3pt_1} by $\eta$, exchange $\mu_2$ and $\mu_3$, and let $\eta \to 0_+$. Thanks to Lemma \ref{lim_H_R} this yields:
\begin{align}
2(Q-\beta_2)\overline{R}(\beta_2, \mu_1,\mu_2)&= (Q-\beta_2)\left(\overline{R}(\beta_2, \mu_1,\mu_2)+ \overline{R}(\beta_2, \mu_1,e^{\frac{i \pi \gamma^2}{4}} \mu_3)\right) \nonumber \\
&+\frac{2}{\gamma}\left( \mu_1 - \mu_3 e^{i \pi \frac{\gamma^2}{4}} \right)  \overline{H}^{(\gamma, \beta_2, \beta_2)}_{( \mu_1, e^{\frac{i \pi  \gamma^2}{4}} \mu_3,   \mu_2)} \nonumber \\
\Rightarrow \overline{R}(\beta_2, \mu_1,\mu_2)=   \overline{R}(\beta_2, &\mu_1,e^{\frac{i \pi \gamma^2}{4}} \mu_3) +\frac{2}{\gamma(Q-\beta_2)}\left( \mu_1 - \mu_3 e^{i \pi \frac{\gamma^2}{4}} \right)  \overline{H}^{(\gamma, \beta_2, \beta_2)}_{( \mu_1, e^{\frac{i \pi  \gamma^2}{4}} \mu_3,   \mu_2)}.
\end{align}
Take $\mu_3= 0$ in the previous equation and fix a compact $K \subset (\frac{\gamma}{2}, Q)$. In our previous work \cite[Proposition 1.5]{interval} we have calculated the expression of $\overline{R}(\beta_2,\mu_1,0)$ and it is complex analytic in $\beta_2$ in a complex neighborhood of $K$. By the result of Lemma \ref{lem_analycity} we know the function $\overline{H}^{(\gamma, \beta_2, \beta_2)}_{( \mu_1, 0,   \mu_2)}$ is also complex analytic in $\beta_2$ in a complex neighborhood of $K$. Therefore  the above equation with $\mu_3= 0$ implies the claim of analyticity for $\beta_2 \mapsto \overline{R}(\beta_2,\mu_1,\mu_2)$. The exact same reasoning implies the analyticity of  $(\mu_1, \mu_2) \mapsto \overline{R}(\beta_2,\mu_1,\mu_2)$.
\end{proof}

\subsubsection{OPE with reflection and the reflection principle}\label{section ope reflection 2}
We now move to performing the OPE with reflection. We rely extensively on Lemma \ref{lem_reflection_ope2} and Lemma \ref{lem_reflection_ope1} giving the Taylor expansions using the reflection coefficient. As in Section \ref{section ope reflection 1} we first use OPE with reflection for $\chi = \frac{\gamma}{2}$ to obtain the reflection principle.

\begin{lemma}[Reflection principle for $\overline{H}$] \label{reflection_H} Consider parameters $\mu_1, \mu_2, \mu_3 $, $\beta_1, \beta_2, \beta_3$ such that $\beta_1 \in (\frac{\gamma}{2},Q)$ and satisfying the parameter range \eqref{para_H} for $ \overline{H}^{(\beta_1,\beta_2, \beta_3)}_{( \mu_1, \mu_2,  \mu_3)}$ and $\overline{R}(\beta_1,\mu_1,\mu_2)$ to be well-defined. Then one can meromorphically extend $\beta_1 \mapsto  \overline{H}^{(\beta_1,\beta_2, \beta_3)}_{( \mu_1, \mu_2,  \mu_3)}$ beyond the point $\beta_1 = Q$ by the following relation:
\begin{align}\label{eq:reflection_R_H}
  \overline{H}^{(\beta_1,\beta_2, \beta_3)}_{( \mu_1, \mu_2,  \mu_3)} = -\frac{\Gamma(\frac{2\beta_1}{\gamma}-\frac{4}{\gamma^2})\Gamma(\frac{\beta_2+\beta_3-\beta_1}{\gamma})}{\Gamma(\frac{\beta_1+\beta_2+\beta_3-2Q}{\gamma})} \overline{R}(\beta_1,\mu_1,\mu_2)\overline{H}^{(2Q-\beta_1,\beta_2 , \beta_3)}_{( \mu_1, \mu_2,  \mu_3)}.
\end{align}
The quantity $\overline{H}^{(2Q-\beta_1,\beta_2 , \beta_3)}_{( \mu_1, \mu_2,  \mu_3)}$ is thus well-defined as long as $ \overline{H}^{(\beta_1,\beta_2, \beta_3)}_{( \mu_1, \mu_2,  \mu_3)}  $ and $\overline{R}(\beta_1,\mu_1,\mu_2)$ are well-defined.
Similarly, for $(\mu_1, \mu_2)$ satisfying the constraint of Definition \ref{half-space}, we can analytically extend $\beta_1 \mapsto \overline{R}(\beta_1,\mu_1,\mu_2)$ to the range $(\frac{\gamma}{2}, Q + \frac{2}{\gamma})$ thanks to the relation:
\begin{align}\label{equation R reflect}
\overline{R}(\beta_1,\mu_1,\mu_2)\overline{R}(2Q-\beta_1,\mu_1,\mu_2) = \frac{1}{\Gamma(1-\frac{2(Q-\beta_1)}{\gamma}) \Gamma(1+\frac{2(Q-\beta_1)}{\gamma})}.
\end{align}
\end{lemma}

\begin{proof}
 Throughout the proof we keep the same notations as used in the proof of Lemma \ref{Shift1_H_proba} for the solution space of the hypergeometric equation satisfied by $H_{\frac{\gamma}{2}}(t)$. The first step is to assume $\beta_1 \in (Q - \beta_0, Q )$ so that we can apply the result of Lemma \ref{lem_reflection_ope1} and identify the value of $C_2$ to be:
\begin{equation}\label{eq_ope_gluing}
C_2 = \frac{2(Q - \beta_1)}{\gamma} \frac{\Gamma(\frac{2}{\gamma}(\beta_1 -Q)) \Gamma(\frac{2}{\gamma}(Q -\beta_1) -q ) }{\Gamma(-q)} \overline{R}(\beta_1, \mu_1, \mu_2) \overline{H}^{(2 Q - \beta_1 - \frac{\gamma}{2} , \beta_2 , \beta_3)}_{( \mu_1, e^{i\pi\frac{\gamma^2}{4}}\mu_2, e^{\frac{i \pi  \gamma^2}{4}} \mu_3)}.
\end{equation}
The key argument is to observe that since by Lemma \ref{lem_analycity} $ \beta_1 \mapsto H_{\frac{\gamma}{2}}(t)$ is complex analytic so is the coefficient $C_2$. By using this combined with the analyticity of $\overline{R}$ and $\overline{H}$, we can extend the range of validity of equation \eqref{eq_ope_gluing} from $\beta_1 \in (Q - \beta_0, Q )$ to $\beta_1 \in (\frac{2}{\gamma}, Q )$. Now equation \eqref{eq_ope_gluing2} derived in the the proof of Lemma \ref{Shift1_H_proba} gives us an alternative expression for $C_2$, which is valid for $\beta_1 \in (\frac{\gamma}{2}, \frac{2}{\gamma})$. The analyticity of $\beta_1 \mapsto C_2$ in a complex neighborhood of $\frac{2}{\gamma}$ then implies that one can ``glue" together the two expressions for $C_2$. More precisely the equality,
\begin{align}
&\frac{2(Q - \beta_1)}{\gamma} \frac{\Gamma(\frac{2}{\gamma}(\beta_1 -Q)) \Gamma(\frac{2}{\gamma}(Q -\beta_1) -q ) }{\Gamma(-q)} \overline{R}(\beta_1, \mu_1, \mu_2) \overline{H}^{(2 Q - \beta_1 - \frac{\gamma}{2} , \beta_2 , \beta_3)}_{( \mu_1, e^{i\pi\frac{\gamma^2}{4}}\mu_2, e^{\frac{i \pi  \gamma^2}{4}} \mu_3)}  \\ \nonumber
&=q \frac{\Gamma(-1+\frac{\gamma \beta_1}{2} - \frac{\gamma^2}{4} ) \Gamma(1 - \frac{\gamma \beta_1}{2})}{\Gamma(- \frac{\gamma^2}{4})} \left( \mu_1 - \mu_2 e^{i \pi \frac{\gamma \beta_1}{2}} \right)  \overline{H}^{(\beta_1 + \frac{\gamma}{2}, \beta_2, \beta_3)}_{( \mu_1, e^{\frac{i \pi  \gamma^2}{4}} \mu_2, e^{\frac{i \pi  \gamma^2}{4}} \mu_3)},
\end{align} 
provides the desired analytic continuation of $\overline{H}$. To land on the form of the reflection equation given in the lemma one needs to replace $\beta_1$ by $\beta_1 - \frac{\gamma}{2}$. This transforms $\overline{R}(\beta_1, \mu_1, \mu_2)$ into $\overline{R}(\beta_1 - \frac{\gamma}{2}, \mu_1, \mu_2)$ which we can shift back to $\overline{R}(\beta_1, \mu_1, e^{\frac{i \pi \gamma^2}{4}} \mu_2)$ using the shift equation \eqref{equation R1}. Lastly we perform the parameter replacement $ e^{\frac{i \pi \gamma^2}{4}} \mu_2$ to $\mu_2$ and $ e^{\frac{i \pi \gamma^2}{4}} \mu_3$ to $\mu_3$. Therefore this implies the claim of the reflection principle for $\overline{H}$. The claim for $\overline{R}$ is then an immediate consequence of applying twice \eqref{eq:reflection_R_H}, where once we replace $\beta_1$ by $2Q- \beta_1$.
\end{proof}

\subsubsection{Analytic continuation of $\overline{H}$ and $\overline{R}$}

At this stage we will use the shift equations we have derived to analytically continue $\overline{H}$ and $\overline{R}$ both in the parameters $\beta_i$ and $\mu_i$. The analytic continuations will be defined in a larger range of parameters than the one of Definition \ref{def_four_correls} required for the GMC expression to be well-defined. 

\begin{lemma}\label{analycity_R} (Analytic continuation of $ \overline{R}$)
For all $ \mu_1, \mu_2$ obeying the constraint of Definition \ref{half-space}, the function $\beta_1 \mapsto \overline{R}(\beta_1, \mu_1, \mu_2)$ originally defined on the interval $(\frac{\gamma}{2}, Q)$ extends to a meromorphic function defined in a complex neighborhood of $\mathbb{R}$ and satisfying the shift equation:
\begin{equation}\label{eq:shift_R_beta}
\overline{R}(\beta_1, \mu_1, \mu_2) = -\frac{\Gamma(-1+\frac{\gamma \beta_1}{2} - \frac{\gamma^2}{4} )\Gamma(2 - \frac{\gamma \beta_1}{2}-\frac{\gamma^2}{4}) }{ \Gamma(1- \frac{\gamma^2}{4})^2 } \frac{\pi}{\sin(\pi\frac{\gamma\beta_1}{2})} \left| \mu_1 - \mu_2 e^{i \pi \frac{\gamma \beta_1}{2}} \right|^2\overline{R}(\beta_1 + \gamma, \mu_1, \mu_2).
\end{equation}
Furthermore, for a fixed $\beta_1$ in the above complex neighborhood of $\mathbb{R}$, the function $ \overline{R}(\beta_1, e^{i \pi \gamma (\sigma_1-\frac{Q}{2})},e^{i \pi \gamma (\sigma_2-\frac{Q}{2})})$ extends to a meromorphic function of $(\sigma_1, \sigma_2)$ on $\mathbb{C}^2$.
\end{lemma}
\begin{proof} Fix $\mu_1$ and $\mu_2$ obeying the constraint of Definition \ref{half-space}. The function $\beta_1 \mapsto \overline{R}(\beta_1, \mu_1, \mu_2) $ is originally defined on $(\frac{\gamma}{2}, Q)$ which has length  $\frac{2}{\gamma}$, but using \eqref{equation R reflect} we can analytically extend its definition to an interval of length $\frac{4}{\gamma}$, i.e. the interval $\beta_1 \in (\frac{\gamma}{2}, Q + \frac{2}{\gamma})$. This gives us a large enough interval to successively apply both shift equations of Lemma \ref{lem_shift_eq_R} and derive the shift equation \eqref{eq:shift_R_beta} written above. The advantage of \eqref{eq:shift_R_beta} is it only shifts the parameter $\beta_1$ and thus combined with Lemma \ref{lem:analycity_R} it implies the analytic continuation of $\beta_1 \mapsto \overline{R}(\beta_1, \mu_1, \mu_2)$ to a meromorphic function defined in a complex neighborhood of the real line. Now that the function $\overline{R}$ has been analytically extended as a function of $\beta_1$, the analytic continuation in $(\sigma_1, \sigma_2)$ follows directly from the shift equations of Lemma \ref{lem_shift_eq_R}.
\end{proof}

\begin{lemma}\label{analycity_H}(Analytic continuation of $\overline{H}$)  Fix $\mu_1, \mu_2, \mu_3 $ obeying the condition of Definition \ref{half-space}. The function $(\beta_1, \beta_2, \beta_3) \mapsto \overline{H}^{(\beta_1 , \beta_2, \beta_3)}_{( \mu_1, \mu_2,   \mu_3)}$ originally defined in the parameter range given by \eqref{para_H} extends to a meromorphic function of the three variables in a small complex neighborhood of $\mathbb{R}^3$. Now fix $\beta_1, \beta_2, \beta_3$ in this complex neighborhood of $\mathbb{R}^3$ and write $\mu_{i} := e^{i \pi \gamma (\sigma_i-\frac{Q}{2})}$  with the convention that $\mathrm{Re}(\sigma_i) = \frac{Q}{2}$ when $\mu_i>0$. The function 
\begin{equation}
(\sigma_1, \sigma_2, \sigma_3) \mapsto \overline{H}^{(\beta_1 , \beta_2, \beta_3)}_{( e^{i \pi \gamma (\sigma_1-\frac{Q}{2})}, e^{i \pi \gamma (\sigma_2-\frac{Q}{2})}, e^{i \pi \gamma (\sigma_3-\frac{Q}{2})})}
\end{equation}
then extends to a meromorphic function of $\mathbb{C}^3$.
\end{lemma}
\begin{proof}
Our starting point will be a domain $\mathcal{E}_1$ of the parameters $\beta_i, \sigma_i$ where the condition \eqref{para_H} is satisfied for  $\overline{H}$ to be well-defined as a GMC moment. We choose the domain
\begin{equation}
\mathcal{E}_1 := \left \{  \beta_i \in (\frac{2}{\gamma} - \delta, Q) \times [-\nu, \nu], \: \sigma_i \in (-\frac{1}{2\gamma} + \frac{Q}{2}, \frac{1}{2\gamma} + \frac{Q}{2}) \times \mathbb{R} \right \}.
\end{equation}
By $\beta_i \in (\frac{2}{\gamma} - \delta, Q)\times [-\nu, \nu]$ we mean $\mathrm{Re}(\beta_i) \in (\frac{2}{\gamma} - \delta, Q) $ and $\mathrm{Im}(\beta_i) \in [-\nu, \nu] $, the same convention is used for the domain of $\sigma_i$. We have introduced $\delta, \nu >0$ chosen small enough so that \eqref{para_H} holds for $\beta_i \in (\frac{2}{\gamma} - \delta, Q) $ and one can apply Lemma \ref{lem_analycity} to show analyticity of $\overline{H}$ in all of its variables on $\mathcal{E}_1$. The condition on $\delta$ is $\delta < \frac{1}{\gamma} - \frac{\gamma}{4}$. The proof will now be divided into several steps each corresponding to analytically extending the definition of $\overline{H}$ on a larger domain until we finally construct a meromorphic function on the domain claimed in the lemma.\\

\noindent \textbf{Step 1.} We extend $\overline{H}$ to a meromorphic function defined on the domain
\begin{equation}
\mathcal{E}_2 := \left \{ \beta_i \in (\frac{2}{\gamma} - \delta, Q + \frac{\gamma}{2} + \delta) \times [-\nu, \nu] , \:\sigma_i \in (-\frac{1}{2\gamma} + \frac{Q}{2}, \frac{1}{2\gamma} + \frac{Q}{2}) \times \mathbb{R} \right \}.
\end{equation}
To do this we simply need to apply three times the equation \eqref{eq:reflection_R_H}, once for each variable. To define the value of $\overline{H}$ for $\beta_1, \beta_2 \in (Q, Q + \frac{2}{\gamma}) $, $ \beta_3 \in (\frac{2}{\gamma}, Q)$, we apply twice \eqref{eq:reflection_R_H} to obtain:
\begin{equation}
\overline{H}^{(\beta_1 , \beta_2, \beta_3)}_{( \mu_1, \mu_2,   \mu_3)} = \frac{\Gamma(\frac{2 \beta_1}{\gamma} - \frac{4}{\gamma^2} ) \Gamma(\frac{2 \beta_2}{\gamma} - \frac{4}{\gamma^2} ) \Gamma( \frac{2Q - \beta_1 -\beta_2 + \beta_3}{\gamma} )}{\Gamma( \frac{ \beta_1 +\beta_2 + \beta_3 - 2Q}{\gamma} )} \overline{R}(\beta_1, \mu_1, \mu_2) \overline{R}(\beta_2, \mu_1, \mu_2) \overline{H}^{(2Q - \beta_1 ,2Q - \beta_2, \beta_3)}_{( \mu_1, \mu_2,   \mu_3)}
\end{equation}
Notice this equation is symmetric with respect to $\beta_1, \beta_2$ as excepted as we need to get the same answer whether we first apply the reflection with respect to $\beta_1$ or $\beta_2$. The procedure works similarly for $\beta_3$. Hence we have extended $\overline{H}$ to $\mathcal{E}_2$.\\

\noindent \textbf{Step 2.} Here we fix a $\sigma_2$ such $\sigma_2$ and $\sigma_2 + \frac{\gamma}{4}$ are both in  $ (-\frac{1}{2\gamma} + \frac{Q}{2}, \frac{1}{2\gamma} + \frac{Q}{2})\times \mathbb{R}$, in other words $\sigma_2 \in (-\frac{1}{2\gamma} + \frac{Q}{2}, \frac{1}{2\gamma} +\frac{Q}{2} -\frac{\gamma}{4} ) \times \mathbb{R} $. Notice this set is not empty since $\frac{\gamma}{4} < \frac{1}{\gamma}$. We will now first extend $\overline{H}$ to the following domain
\begin{align}
\mathcal{E}_3 := \Bigg \{ \beta_1 \in &(Q - \gamma -\delta, Q + \gamma + \delta) \times [-\nu, \nu] , \: \beta_2, \beta_3 \in (\frac{2}{\gamma} - \delta , Q +\frac{\gamma}{2} +\delta) \times [-\nu, \nu] ,\\
&  \sigma_1, \sigma_3 \in (-\frac{1}{2\gamma} + \frac{Q}{2}, \frac{1}{2\gamma} + \frac{Q}{2}) \times \mathbb{R}, \: \sigma_2 \in (-\frac{1}{2\gamma} +\frac{Q}{2}, \frac{1}{2\gamma} +\frac{Q}{2} -\frac{\gamma}{4} ) \times \mathbb{R} \Bigg \}. \nonumber
\end{align}
To do this we will use the shift equations on $\overline{H}$ written in a more compact form as
\begin{align}\label{eq:anaH_shift1}
&\overline{H}^{(\beta_1 , \beta_2 - \frac{\gamma}{2}, \beta_3)}_{( \mu_1, \mu_2,   \mu_3)} = f_1(\beta_1, \beta_2) \overline{H}^{(\beta_1 - \frac{\gamma}{2}, \beta_2, \beta_3)}_{( \mu_1, e^{\frac{i \pi \gamma^2}{4}} \mu_2,  \mu_3)} + f_2(\beta_1, \beta_2) \overline{H}^{(\beta_1 + \frac{\gamma}{2}, \beta_2, \beta_3)}_{( \mu_1, e^{\frac{i \pi  \gamma^2}{4}} \mu_2,   \mu_3)}, \\ \label{eq:anaH_shift2}
 &\overline{H}^{(\beta_1 , \beta_2 + \frac{\gamma}{2} , \beta_3)}_{(\mu_1, e^{i\pi\frac{\gamma^2}{4}}\mu_2,   \mu_3)} = f_3(\beta_1, \beta_2) \overline{H}^{(\beta_1 -\frac{\gamma}{2} , \beta_2, \beta_3)}_{(  \mu_1,  \mu_2,   \mu_3)}   + f_4(\beta_1, \beta_2) \overline{H}^{(\beta_1 + \frac{\gamma}{2} , \beta_2 , \beta_3)}_{(\mu_1, \mu_2,   \mu_3)},
\end{align}
for some explicit functions $f_1, f_2, f_3, f_4$. The main idea is that if in the above shift equations two of three terms are well-defined - meaning the parameters of the two $\overline{H}$ functions belong to $\mathcal{E}_2$ - the third term will be defined by the shift equation. Fix $\beta_2 \in (\frac{2}{\gamma} - \delta, Q) \times [-\nu, \nu]$, $\sigma_2 \in (-\frac{1}{2\gamma} +\frac{Q}{2}, \frac{1}{2\gamma} +\frac{Q}{2} -\frac{\gamma}{4} ) \times \mathbb{R}$. Then for $\beta_1 \in (\frac{2 }{\gamma} - \delta, Q) \times [-\nu, \nu] $, the shift \eqref{eq:anaH_shift2} defines $\overline{H}$ for $\beta_1 \in (\frac{2}{\gamma} - \frac{\gamma}{2} - \delta, \frac{2}{\gamma} - \delta] \times [-\nu, \nu] $. Indeed the two terms 
\begin{equation}
\overline{H}^{(\beta_1 , \beta_2 + \frac{\gamma}{2} , \beta_3)}_{(\mu_1, e^{i\pi\frac{\gamma^2}{4}}\mu_2,   \mu_3)}, \quad \overline{H}^{(\beta_1 + \frac{\gamma}{2} , \beta_2 , \beta_3)}_{(\mu_1, \mu_2,   \mu_3)},
\end{equation}
 are well-defined and so the shift equation provides the definition of $\overline{H}^{(\beta_1 -\frac{\gamma}{2} , \beta_2, \beta_3)}_{(  \mu_1,  \mu_2,   \mu_3)}$. Then by applying equation \eqref{eq:reflection_R_H} as in the first step we can extend the range of $\beta_1$ and $\beta_2$ respectively to $(Q -\gamma - \delta, Q +\gamma + \delta) \times [-\nu, \nu]$ and $(\frac{2}{\gamma}-\delta, Q + \frac{\gamma}{2} + \delta) \times [-\nu, \nu]$. Hence we have extended $\overline{H}$ to the set $\mathcal{E}_3$. Next we are going to extend $\overline{H}$ to the set
\begin{align}
\mathcal{E}_4 := \Bigg \{ & \beta_1 \in \mathbb{R} \times [-\nu, \nu],\: \beta_2, \beta_3 \in (\frac{2}{\gamma} -\delta, Q +\frac{\gamma}{2} +\delta) \times [-\nu, \nu], \\
 &\sigma_1, \sigma_3 \in (-\frac{1}{2\gamma} + \frac{Q}{2}, \frac{1}{2\gamma} + \frac{Q}{2}) \times \mathbb{R}, \: \sigma_2 \in (-\frac{1}{2\gamma} + \frac{Q}{2}, \frac{1}{2\gamma} +\frac{Q}{2} -\frac{\gamma}{4} ) \times \mathbb{R} \Bigg \}. \nonumber
\end{align}
To do this let us write equation \eqref{eq:anaH_shift2} with the shift $\beta_1 \rightarrow \beta_1 + \gamma$
\begin{align}\label{eq:anaH_shift3}
 \overline{H}^{(\beta_1+\gamma , \beta_2 + \frac{\gamma}{2} , \beta_3)}_{(\mu_1, e^{i\pi\frac{\gamma^2}{4}}\mu_2,   \mu_3)} &= f_3(\beta_1 + \gamma, \beta_2) \overline{H}^{(\beta_1 +\frac{\gamma}{2} , \beta_2, \beta_3)}_{(  \mu_1,  \mu_2,   \mu_3)}   + f_4(\beta_1 +\gamma, \beta_2) \overline{H}^{(\beta_1 + \frac{3 \gamma}{2} , \beta_2 , \beta_3)}_{(\mu_1, \mu_2,   \mu_3)},
\end{align}
and equation \eqref{eq:anaH_shift1} with the shifts $\beta_1 \rightarrow \beta_1 + \frac{\gamma}{2}$,  $\beta_2 \rightarrow \beta_2 + \frac{\gamma}{2}$
\begin{align}\label{eq:anaH_shift4}
\overline{H}^{(\beta_1  + \frac{\gamma}{2}, \beta_2 , \beta_3)}_{( \mu_1, \mu_2,   \mu_3)} &= f_1(\beta_1 + \frac{\gamma}{2}, \beta_2 + \frac{\gamma}{2} ) \overline{H}^{(\beta_1, \beta_2 + \frac{\gamma}{2}, \beta_3)}_{( \mu_1, e^{\frac{i \pi \gamma^2}{4}} \mu_2,  \mu_3)} + f_2(\beta_1 + \frac{\gamma}{2}, \beta_2 + \frac{\gamma}{2}) \overline{H}^{(\beta_1 +\gamma, \beta_2 + \frac{\gamma}{2}, \beta_3)}_{( \mu_1, e^{\frac{i \pi  \gamma^2}{4}} \mu_2,   \mu_3)}.
\end{align}
Then by substituting equations \eqref{eq:anaH_shift2} and \eqref{eq:anaH_shift3} into \eqref{eq:anaH_shift4} one obtains
\begin{align}\label{eq:shift_only_x}
\overline{H}^{(\beta_1  + \frac{\gamma}{2}, \beta_2 , \beta_3)}_{( \mu_1, \mu_2,   \mu_3)} &= f_1(\beta_1 + \frac{\gamma}{2}, \beta_2 + \frac{\gamma}{2} ) \left( f_3(\beta_1, \beta_2) \overline{H}^{(\beta_1 -\frac{\gamma}{2} , \beta_2, \beta_3)}_{(  \mu_1,  \mu_2,   \mu_3)}   + f_4(\beta_1, \beta_2) \overline{H}^{(\beta_1 + \frac{\gamma}{2} , \beta_2 , \beta_3)}_{(\mu_1, \mu_2,   \mu_3)} \right)\\ \nonumber
& + f_2(\beta_1 + \frac{\gamma}{2}, \beta_2 + \frac{\gamma}{2}) \left( f_3(\beta_1 + \gamma, \beta_2) \overline{H}^{(\beta_1 +\frac{\gamma}{2} , \beta_2, \beta_3)}_{(  \mu_1,  \mu_2,   \mu_3)}   + f_4(\beta_1 +\gamma, \beta_2) \overline{H}^{(\beta_1 + \frac{3 \gamma}{2} , \beta_2 , \beta_3)}_{(\mu_1, \mu_2,   \mu_3)} \right)
\end{align}
which can be rewritten as
\begin{align}\label{eq:anaH_shift5}
0 &=   f_1(\beta_1 + \frac{\gamma}{2}, \beta_2 + \frac{\gamma}{2} )  f_3(\beta_1, \beta_2) \overline{H}^{(\beta_1 -\frac{\gamma}{2} , \beta_2, \beta_3)}_{(  \mu_1,  \mu_2,   \mu_3)} \\ \nonumber
&  + \left( 1+  f_1(\beta_1 + \frac{\gamma}{2}, \beta_2 + \frac{\gamma}{2} )  f_4(\beta_1, \beta_2) +  f_2(\beta_1 + \frac{\gamma}{2}, \beta_2 + \frac{\gamma}{2})  f_3(\beta_1 + \gamma, \beta_2) \right) \overline{H}^{(\beta_1 + \frac{\gamma}{2} , \beta_2 , \beta_3)}_{(\mu_1, \mu_2,   \mu_3)}   \\ \nonumber
& + f_2(\beta_1 + \frac{\gamma}{2}, \beta_2 + \frac{\gamma}{2})  f_4(\beta_1 +\gamma, \beta_2) \overline{H}^{(\beta_1 + \frac{3 \gamma}{2} , \beta_2 , \beta_3)}_{(\mu_1, \mu_2,   \mu_3)}.
\end{align}
This is a three term shift equation with the big advantage to contain only shifts on $\beta_1$. Notice there is a distance $2\gamma$ between $ \mathrm{Re}(\beta_1) - \frac{\gamma}{2} $ and $\mathrm{Re}(\beta_1) + \frac{3\gamma}{2}$ and that the set $\mathcal{E}_3$ contains an interval of length strictly larger than $2 \gamma$ for $\mathrm{Re}(\beta_1)$. This means we have enough room to apply the shift equation \eqref{eq:anaH_shift5} to analytically extend $\overline{H}$ to the set $\mathcal{E}_4$.
Finally using the shift equations again we can extend $\overline{H}$ to 
\begin{align*}
\mathcal{E}_5 := \Bigg \{ \beta_1,\beta_2 \in \mathbb{R} & \times [-\nu, \nu],\:  \beta_3 \in (\frac{2}{\gamma} - \delta, Q +\frac{\gamma}{2} + \delta) \times [-\nu, \nu], \\
& \sigma_1, \sigma_3 \in (-\frac{1}{2\gamma} + \frac{Q}{2}, \frac{1}{2\gamma} + \frac{Q}{2}) \times \mathbb{R}, \: \sigma_2 \in (-\frac{1}{2\gamma} + \frac{Q}{2}, \frac{1}{2\gamma} +\frac{Q}{2} -\frac{\gamma}{4} ) \times \mathbb{R} \Bigg \}. 
\end{align*}
We actually extend in $\sigma_2$ to both $\sigma_2 \in (-\frac{1}{2\gamma} + \frac{Q}{2}, \frac{1}{2\gamma} +\frac{Q}{2} -\frac{\gamma}{4} ) \times \mathbb{R}$ and $ \sigma_2 - \frac{\gamma}{4} \in (-\frac{1}{2\gamma} +\frac{Q}{2}, \frac{1}{2\gamma} +\frac{Q}{2} -\frac{\gamma}{4} ) \times \mathbb{R}$. The starting domain for $\beta_2$ is $(\frac{2}{\gamma} - \delta, Q + \frac{\gamma}{2} + \delta) \times [-\nu, \nu] $. To construct the extension to $\beta_2 \in (\frac{2}{\gamma} - \frac{\gamma}{2} - \delta, \frac{2}{\gamma} -\delta] \times [-\nu, \nu]$ and $\sigma_2 \in (-\frac{1}{2\gamma} +\frac{Q}{2}, \frac{1}{2\gamma} +\frac{Q}{2} -\frac{\gamma}{4} ) \times \mathbb{R}$ one can use the shift equation \eqref{eq:anaH_shift1}. Then by the reflection equation \eqref{eq:reflection_R_H} one obtains all values of $\beta_2$ in $(Q - \gamma - \delta, Q + \gamma +\delta) \times [-\nu, \nu] $. For the case of $ \sigma_2 - \frac{\gamma}{4} \in (-\frac{1}{2\gamma} + \frac{Q}{2}, \frac{1}{2\gamma} +\frac{Q}{2} -\frac{\gamma}{4} ) \times \mathbb{R}$, one first uses equation \eqref{eq:anaH_shift2} to extend the range of $\beta_2$ to the range $\beta_2 \in [Q + \frac{\gamma}{2} +\delta, Q +\gamma + \delta) \times [-\nu, \nu]$ then again by the reflection equation \eqref{eq:reflection_R_H} we obtain the range $\beta_2 \in (Q - \gamma -\delta, Q + \gamma +\delta) \times [-\nu, \nu]$. By iterating this procedure one can extend the range of $\beta_2$ to $\mathbb{R} \times [-\nu, \nu]$, hence we have obtain the extension of $\overline{H}$ to all of $\mathcal{E}_5$. \\ 

\noindent \textbf{Step 3.} In this step we start by writing the shift equations on $\overline{H}$ to shift the parameters $\beta_2, \beta_3, \sigma_3$, namely
\begin{align}\label{eq:anaH_shift6}
&\overline{H}^{(\beta_1 , \beta_2, \beta_3 - \frac{\gamma}{2})}_{( \mu_1, \mu_2,   \mu_3)} = f_1(\beta_2, \beta_3) \overline{H}^{(\beta_1, \beta_2 - \frac{\gamma}{2}, \beta_3)}_{( \mu_1,  \mu_2, e^{\frac{i \pi \gamma^2}{4}} \mu_3)} + f_2(\beta_2, \beta_3) \overline{H}^{(\beta_1, \beta_2 + \frac{\gamma}{2}, \beta_3)}_{( \mu_1,  \mu_2,  e^{\frac{i \pi  \gamma^2}{4}} \mu_3)}, \\ \label{eq:anaH_shift7}
 &\overline{H}^{(\beta_1 , \beta_2  , \beta_3+ \frac{\gamma}{2})}_{(\mu_1, \mu_2,  e^{i\pi\frac{\gamma^2}{4}}  \mu_3)} = f_3(\beta_2, \beta_3) \overline{H}^{(\beta_1  , \beta_2 -\frac{\gamma}{2}, \beta_3)}_{(  \mu_1,  \mu_2,   \mu_3)}   + f_4(\beta_2, \beta_3) \overline{H}^{(\beta_1  , \beta_2 + \frac{\gamma}{2}, \beta_3)}_{(\mu_1, \mu_2,   \mu_3)}.
\end{align}
Using these equations we will perform the meromorphic continuation to the set
\begin{equation}
\mathcal{E}_6 := \left \{ \beta_i \in \mathbb{R} \times [-\nu, \nu],  \: \sigma_1 \in (-\frac{1}{2\gamma} + \frac{Q}{2}, \frac{1}{2\gamma} + \frac{Q}{2}) \times \mathbb{R}, \: \sigma_2 \in (-\frac{1}{2\gamma} +\frac{Q}{2}, \frac{1}{2\gamma} +\frac{Q}{2} -\frac{\gamma}{4} ) \times \mathbb{R}, \:  \sigma_3 \in \mathbb{C} \right \}.
\end{equation}
Notice here that the big simplification compared to the second step is that the domain of the variable $\beta_2$ has already been extended to $\mathbb{R} \times [-\nu, \nu] $ previously, so we only need to worry about extending $\sigma_3$ and $\beta_3$. 
Notice that at $\sigma_3 + \frac{\gamma}{4}$ the function $\overline{H}$ is not necessarily yet well-defined. Fix $\beta_3 \in (\frac{2}{\gamma} - \delta, Q +\frac{\gamma}{2} + \delta) \times [-\nu, \nu]$. Using the reflection equation \eqref{eq:reflection_R_H} we can rewrite \eqref{eq:anaH_shift7} as 
\begin{equation}\label{eq:anaH_shift8}
f_5(\beta_i, \sigma_i) \overline{H}^{(\beta_1 , \beta_2  , 2 Q - \beta_3 - \frac{\gamma}{2})}_{(\mu_1, \mu_2,  e^{i\pi\frac{\gamma^2}{4}}  \mu_3)} = f_3(\beta_2, \beta_3) \overline{H}^{(\beta_1  , \beta_2 -\frac{\gamma}{2}, \beta_3)}_{(  \mu_1,  \mu_2,   \mu_3)}   + f_4(\beta_2, \beta_3) \overline{H}^{(\beta_1  , \beta_2 + \frac{\gamma}{2}, \beta_3)}_{(\mu_1, \mu_2,   \mu_3)}.
\end{equation}
for some function $f_5$ containing the Gamma function and $\overline{R}$ which is meromorphic in all of its parameters. This equation \eqref{eq:anaH_shift8} then allows to define the $\overline{H}$ function in the l.h.s. still for $\beta_3 \in (\frac{2}{\gamma} - \delta, Q +\frac{\gamma}{2} + \delta) \times [-\nu, \nu]$  and $\sigma_3 \in (-\frac{1}{2\gamma} + \frac{Q}{2}, \frac{1}{2\gamma} + \frac{Q}{2}) \times \mathbb{R} $. Plugging this into equation \eqref{eq:anaH_shift6} with $\beta_3$ replaced by $2Q - \beta_3 - \frac{\gamma}{2}$ it gives a definition of both terms on the r.h.s. of \eqref{eq:anaH_shift6}, which therefore implies the function $\overline{H}^{(\beta_1 , \beta_2, \frac{4}{\gamma}-\beta_3  )}_{( \mu_1, \mu_2,   \mu_3)}$ is well-defined for $\beta_3 \in (\frac{2}{\gamma} - \delta, Q +\frac{\gamma}{2} + \delta) \times [-\nu, \nu]$ and $\sigma_3 \in (-\frac{1}{2\gamma} + \frac{Q}{2}, \frac{1}{2\gamma} + \frac{Q}{2}) \times \mathbb{R} $. This means we have extended the range of $\beta_3$ for $\overline{H}$ to the interval $(\frac{2}{\gamma} -\gamma - \delta, Q + \frac{\gamma}{2} + \delta) \times [-\nu, \nu] $ which by using the reflection equation \eqref{eq:reflection_R_H} can be extended to $(\frac{2}{\gamma} -\gamma - \delta, Q + \frac{3\gamma}{2} + \delta)  \times [-\nu, \nu] $, still for $\sigma_3 \in (-\frac{1}{2\gamma} + \frac{Q}{2}, \frac{1}{2\gamma} + \frac{Q}{2}) \times \mathbb{R} $. By iterating this procedure we can extend indefinitely the interval of $\beta_3$ meaning that we get $\mathbb{R}\times [-\nu, \nu]$. Now with $\beta_2, \beta_3$ both extended to $\mathbb{R} \times [-\nu, \nu]$, the shift equations \eqref{eq:anaH_shift6} and \eqref{eq:anaH_shift7} immediately imply the range of $\sigma_3$ can be extended to $\mathbb{C}$.

\noindent \textbf{Step 4.} Writing now the shift equations of $\overline{H}$ on the parameters $\beta_3, \beta_1, \sigma_1$, and noting we have already extended the range of $\beta_1, \beta_3$ to $\mathbb{R}\times [-\nu, \nu]$, one can now perform the analytic continuation in $\sigma_1$ to all of $\mathbb{C}$. Hence we have extended $\overline{H}$ to a meromorphic function defined on the set

\begin{equation}
\mathcal{E}_7 := \left \{ \beta_i \in \mathbb{R} \times [-\nu, \nu],  \: (\sigma_1,\sigma_3) \in \mathbb{C}^2, \: \sigma_2 \in (-\frac{1}{2\gamma} + \frac{Q}{2}, \frac{1}{2\gamma} + \frac{Q}{2} -\frac{\gamma}{4} ) \times \mathbb{R} \right \}.
\end{equation}

\noindent \textbf{Step 5.} Lastly we will remove the constraint on $\sigma_2$, namely $\sigma_2 \in (-\frac{1}{2\gamma} + \frac{Q}{2}, \frac{1}{2\gamma} + \frac{Q}{2} -\frac{\gamma}{4} ) \times \mathbb{R} $ introduced in step 2. The problem is this interval for $\sigma_2$ has a width $\frac{1}{\gamma} -\frac{\gamma}{4} $ which is smaller than $\frac{\gamma}{4} $ when $\gamma > \sqrt{2}$. Therefore we do not always have enough room to apply the shift equations to analytically continue in $\sigma_2$. Instead we will notice there is an extra property of $\overline{H}$ that we have not used yet, which is that one can perform a global scaling of the cosmological constants $\mu_1, \mu_2, \mu_3$. Indeed notice that for $\mu_1, \mu_2, \mu_3$ in the probabilistic range and $e^{i \phi}$ a small angle one can easily show using the GMC expression that
\begin{equation}\label{eq:scalingH}
 \overline{H}^{(\beta_1 , \beta_2 , \beta_3)}_{(e^{i \phi} \mu_1,  e^{i \phi} \mu_2,  e^{i \phi}  \mu_3)} = e^{\frac{i \phi}{\gamma} (2Q - \beta_1 - \beta_2 - \beta_3) } \overline{H}^{(\beta_1 , \beta_2 , \beta_3)}_{(\mu_1, \mu_2,   \mu_3)}.
\end{equation}
In terms of the $\sigma_i$ variables, this means that if one adds a global shift $A$ to all three $\sigma_i$, then $\overline{H}$ is multiplied by a factor $e^{ i \pi A (2Q - \beta_1 - \beta_2 - \beta_3) }$. As a sanity check, this scaling property can also be verified on the exact formula for $\overline{H}$, see Lemma \ref{lem:scaling_I}. Thus since \eqref{eq:scalingH} must hold for the meromorphic extension of $\overline{H}$ to $\mathcal{E}_7$, it allow to extend $\sigma_2$ to all of $\mathbb{C}$. This completes the proof. 
 
\end{proof}

\subsubsection{Second shift equation on the reflection coefficient}

Finally we will derive the second shift equation on $ \overline{R}(\beta_1, \mu_1, \mu_2)$ that will completely specify its value.
\begin{lemma}(Second shift equation for $ \overline{R}$). For all $ \mu_1, \mu_2 $ obeying the constraint of Definition \ref{half-space}, the meromorphic function $\beta_1 \mapsto \overline{R}(\beta_1, \mu_1, \mu_2)$ defined in a complex neighborhood of $\mathbb{R}$ satisfies the following shift equation:
\begin{align}\label{R_second_shift} 
\overline{R}(\beta, \mu_1, \mu_2)  = \frac{(2\pi)^{\frac{8}{\gamma^2}}}{\gamma^2(Q-\beta)(\frac{\gamma}{2}-\beta)}  \frac{1}{\Gamma(1-\frac{\gamma^2}{4})^{\frac{8}{\gamma^2}} \sin(\pi\frac{2\beta}{\gamma}) \sin(\pi (\frac{2\beta}{\gamma}+\frac{4}{\gamma^2}))}\left|\mu_1^{\frac{4}{\gamma^2}}-\mu_2^{\frac{4}{\gamma^2}}e^{i\pi \frac{2\beta}{\gamma}}\right|^2 \overline{R}(\beta+\frac{4}{\gamma}, \mu_1, \mu_2).
\end{align}
\end{lemma}

\begin{proof}
We are now working exclusively with the choice $\chi = \frac{2}{\gamma}$.  There will be several steps that will successively require to choose different ranges of parameters. We first place ourselves in the following range of parameters:
\begin{equation}\label{eq:para_range_11}
t \in \mathbb{H}, \: \: \epsilon \in (0, \beta_0), \: \:  \beta_1 = \beta_2 = Q -\epsilon, \:\: \beta_3 = \frac{2}{\gamma} + \epsilon,  \: \: \mu_1 \in (0,+\infty), \:\: \mu_2,\mu_3 \in (-\infty,0).
\end{equation}
In the above $\epsilon $ is chosen small enough, smaller than the constant $\beta_0$ required to apply Lemma \ref{lem_reflection_ope2}. Notice also that in this range  $ q < \frac{4}{\gamma^2} \wedge \min_i \frac{2}{\gamma}(Q - \beta_i)$. Furthermore in the above the choice of $\mu_i$ is such that we can apply Proposition \ref{BPZ_eq_3pt} giving that $H_{\frac{2}{\gamma}}(t)$ obeys the hypergeometric equation. We can thus expand $H_{\frac{2}{\gamma}}(t)$ on the basis,
\begin{align}
H_{\frac{2}{\gamma}}(t) &= C_1 F(A,B,C,t) + C_2 t^{1 - C} F(1 + A-C, 1 +B - C, 2 -C, t) \\ \nonumber
 &= B_1 F(A,B,1+A+B- C, 1 -t) + B_2 (1-t)^{C -A -B} F(C- A, C- B, 1 + C - A - B , 1 -t) \\ \nonumber
 &= D_1 e^{i\pi A} t^{-A} F(A,1+A-C,1+A-B,t^{-1}) + D_2 e^{i\pi B}t^{-B} F(B, 1 +B - C, 1 +B - A, t^{-1}),
\end{align}
where again $C_1, C_2, B_1, B_2, D_1, D_2 $ are parametrizing the solution space around the points $0$, $1$, and $\infty$. As before by sending $t$ to $0$ and to $1$ one obtains:
\begin{equation}\label{eq:C_1_B_!}
C_1 = H_{\frac{2}{\gamma}}(0) = \overline{H}^{(\beta_1 - \frac{2}{\gamma}, \beta_2, \beta_3)}_{( \mu_1, e^{i \pi} \mu_2, e^{i \pi} \mu_3)}, \quad B_1 = H_{\frac{2}{\gamma}}(1) = \overline{H}^{(\beta_1 , \beta_2 - \frac{2}{\gamma}, \beta_3)}_{( \mu_1,  \mu_2, e^{i \pi} \mu_3)}.
\end{equation}
Let us make some comments on the values of the $\mu_i$ appearing in $C_1$ and $B_1$. For $C_1$ since $\mu_2, \mu_3$ are negative,  $\mu_1$, $e^{i \pi} \mu_2$, $e^{i \pi} \mu_3$ are all positive numbers and the function $\overline{H}$ appearing is thus well-defined as a GMC quantity. For $B_1$ the parameters  $\mu_1$ and $e^{i \pi} \mu_3$ are positive while $\mu_2$ is negative, so we are no longer under the constraint of Definition \ref{half-space}, but rather in a limiting case. Since the moment of the GMC of the $\overline{H}$ appearing in $B_1$ is positive, i.e. the moment is equal to $ \frac{\epsilon}{\gamma} $, we can still make sense of this GMC by a simple continuity argument. See the remark below Proposition \ref{lem:GMC-moment} in appendix for an explanation.
Since the condition required for Lemma \ref{lem_reflection_ope2}, $\beta_1 \in (Q-\beta_0, Q)$, is satisfied one then derives:
\begin{equation}
 C_2 = \frac{2(Q - \beta_1)}{\gamma} \frac{\Gamma(\frac{2}{\gamma}(\beta_1 -Q)) \Gamma(\frac{2}{\gamma}(Q -\beta_1) -q ) }{\Gamma(-q)} \overline{R}(\beta_1, \mu_1, \mu_2) \overline{H}^{(2 Q - \beta_1 - \frac{2}{\gamma} , \beta_2 , \beta_3)}_{( \mu_1, e^{i \pi} \mu_2,  e^{i \pi} \mu_3)}.
\end{equation}
Now we write the connection formula \eqref{connection1} linking $C_1, B_1, C_2$, setting $\chi = \frac{2}{\gamma}$ in the equation below:
\begin{equation}
B_1 = \frac{\Gamma(\chi(\beta_1-\chi)) \Gamma(1 - \chi\beta_2 + \chi^2)}{\Gamma(\chi(\beta_1-\chi+q\frac{\gamma}{2}) \Gamma(1 - \chi\beta_2 + \chi^2 -q\frac{\gamma\chi}{2}) } C_1 + \frac{\Gamma(2  - \chi\beta_1+\chi^2) \Gamma(1 - \chi\beta_2 + \chi^2)}{\Gamma(1 + \frac{q \gamma\chi}{2}) \Gamma(2 - \chi(\beta_1 + \beta_2-2\chi+q\frac{\gamma}{2}))   } C_2.
\end{equation}
In the range of parameters we have been working with, all three constants $C_1, B_1, C_2$ are well-defined probabilistic quantities involving GMC through a function $\overline{H}$. But now by the analytic continuation result of Lemma \ref{analycity_H} we can view the above identity as an identity of the analytic function $\overline{H}$, and thus it holds in the whole range of parameters where $\overline{H}$ has been analytically continued. By repeating the above strategy in the range of parameters,
\begin{equation}\label{eq:para_range_12}
t \in \mathbb{H}, \: \: \epsilon \in (0, \beta_0), \: \:  \beta_1 = \beta_2 = Q -\epsilon, \:\: \beta_3 = \frac{2}{\gamma} + \epsilon,  \: \: \mu_1, \mu_2 \in (0,+\infty), \:\: \mu_3 \in (-\infty,0),
\end{equation}
one can identify $B_1$, $B_2$, $C_1$. Then again we can write the connection formula \eqref{connection1} linking $B_1$, $B_2$, $C_1$ and extend the identity to an identity of analytic functions. We can proceed similarly for all the triples $(B_1,B_2, D_1)$, $(B_1, D_1, D_2)$, $(C_1,C_2, D_1)$, and $(C_1, D_1, D_2)$. One can deduce the appropriate parameter ranges of $\beta_i, \mu_i$ for each of these four cases by performing a cyclic permutation of the indices $1,2,3$ in the parameters ranges \eqref{eq:para_range_11} and \eqref{eq:para_range_12} used above. Viewing $\overline{H}$ as its analytic extension given by Lemma \ref{analycity_H}, the values of these remaining coefficients are as follows:
\begin{align}
D_1 &=  \overline{H}^{(\beta_1  , \beta_2, \beta_3 - \frac{2}{\gamma})}_{( e^{i \pi} \mu_1, e^{i \pi} \mu_2,  e^{i \pi}\mu_3)}, \\
 B_2 &= \frac{2(Q - \beta_2)}{\gamma} \frac{\Gamma(\frac{2}{\gamma}(\beta_2 -Q)) \Gamma(\frac{2}{\gamma}(Q -\beta_2) -q ) }{\Gamma(-q)} \overline{R}(\beta_2, e^{i \pi} \mu_2, e^{i \pi} \mu_3) \overline{H}^{(\beta_1,2 Q - \beta_2 - \frac{2}{\gamma} , \beta_3)}_{( \mu_1, \mu_2,  e^{i \pi} \mu_3)},\\
 D_2 &= \frac{2(Q - \beta_3)}{\gamma} \frac{\Gamma(\frac{2}{\gamma}(\beta_3 -Q)) \Gamma(\frac{2}{\gamma}(Q -\beta_3) -q ) }{\Gamma(-q)} \overline{R}(\beta_3, \mu_1, e^{i \pi} \mu_3) \overline{H}^{(\beta_1 , \beta_2 , 2Q-\beta_3 -\frac{2}{\gamma})}_{( e^{i \pi} \mu_1, e^{i \pi} \mu_2,  e^{i \pi} \mu_3)}.
\end{align}

With this at hand we apply the connection formulas coming from the hypergeometric equation in the following way. We use the relation \eqref{connection1} expressing $B_2$ in terms of $C_1$ and $C_2$, as well as 
\begin{align}
D_2 = \frac{\Gamma(C)\Gamma(A-B)}{\Gamma(A)\Gamma(C-B)}C_1 + e^{i \pi (1-C)}\frac{\Gamma(2-C)\Gamma(A-B)}{\Gamma(1-B)\Gamma(A-C+1)}C_2
\end{align}
coming from \eqref{hpy1} to eliminate $C_1$ and obtain the following relation:
\begin{align}\label{equation 3.21}
\frac{\Gamma(B)}{\Gamma(A+B-C)}B_2-\frac{\Gamma(C-B)}{\Gamma(A-B)}D_2 = \frac{\Gamma(2-C)}{\Gamma(A-C+1)}\left(  \frac{\Gamma(B)}{\Gamma(B-C+1)} -\frac{e^{i \pi (1-C)}\Gamma(C-B)}{\Gamma(1-B)} \right) C_2.
\end{align}
Let us now take $\beta_1 = \beta \in (\frac{\gamma}{2}, \frac{2}{\gamma})$, $\beta_2 = \frac{\gamma}{2}+\eta$, $\beta_3 = Q-\beta$, $\mu_3 = 0$ and $\mu_1, \mu_2$ such that the pairs $(\mu_1, \mu_2)$, $(\mu_1, e^{i \pi} \mu_2)$, $(e^{i \pi} \mu_1, e^{i \pi} \mu_2)$ all obey the condition of Definition \ref{half-space}.
We will study the asymptotic as $\eta\to 0$. Equation \eqref{equation 3.21} gives a relation between $C_2, B_2, D_2$ which respectively contain the $\overline{H}$ functions
\begin{align}
\overline{H}^{(2 Q - \beta_1 - \frac{2}{\gamma} , \beta_2 , \beta_3)}_{( \mu_1, e^{i \pi} \mu_2,  0)},   \quad \overline{H}^{(\beta_1,2 Q - \beta_2 - \frac{2}{\gamma} , \beta_3)}_{( \mu_1, \mu_2,  0)}, \quad \overline{H}^{(\beta_1 , \beta_2 , 2Q-\beta_3 -\frac{2}{\gamma})}_{( e^{i \pi} \mu_1, e^{i \pi} \mu_2,  0)}.
\end{align}

One can check that the parameter range we have specified is such all three of the above $\overline{H}$ functions obey the constraint \eqref{para_H} for the GMC definition to hold. We compute
\begin{equation}
q = \frac{4}{\gamma^2}-\frac{\eta}{\gamma}, \quad A = -\frac{4}{\gamma^2}+\frac{\eta}{\gamma}, \quad  B = \frac{2\beta}{\gamma}-\frac{4}{\gamma^2}+\frac{\eta}{\gamma}, \quad C = \frac{2\beta}{\gamma}-\frac{4}{\gamma^2},
\end{equation}
and 
\begin{align}
\lim_{\eta \to 0} \eta D_2 &=- 2(\frac{2}{\gamma}-\beta)\frac{\Gamma(\frac{2\beta}{\gamma}-\frac{4}{\gamma^2})\Gamma(1-\frac{2\beta}{\gamma})}{\Gamma(-\frac{4}{\gamma^2})} \overline{R}(Q-\beta, \mu_1, 0)\overline{R}(\beta+\frac{\gamma}{2}, e^{i \pi}\mu_1, 0),  \\
\lim_{\eta \to 0} \eta C_2 &= -2(\beta-\frac{\gamma}{2})\frac{\Gamma(\frac{2\beta}{\gamma}-\frac{4}{\gamma^2})\Gamma(1-\frac{2\beta}{\gamma})}{\Gamma(-\frac{4}{\gamma^2})}  \overline{R}(\beta, \mu_1, \mu_2)  \overline{R}(2Q-\beta-\frac{2}{\gamma}, \mu_1, e^{i \pi} \mu_2),\\
\lim_{\eta \to 0}\eta^2 B_2 &= -\frac{8}{\gamma} \lim_{\eta \to 0}\eta  \overline{R}(\frac{\gamma}{2}+\eta, e^{i \pi}\mu_2, 0).
\end{align}
To obtain the first two limits above one simply needs to apply Lemma \ref{lim_H_R}. For the limit of $B_2$ we use a limit calculated in \cite[Lemma 10.6] {DOZZ2}, which is straightforward to adapt to our case\footnote{In \cite{DOZZ2} one needs to multiply further by $\eta$ and the limit is $-4$. This difference come from the fact \cite{DOZZ2} states the result for the sphere correlation which is related to the GMC moment by an explicit prefactor.}
\begin{equation}
\lim_{\eta \to 0} \overline{H}^{(\beta , Q-\eta , Q-\beta)}_{(\mu_1, \mu_2,  0)} = 2.
\end{equation}
The moment of the GMC defining $\overline{H}$ in this limit is $\frac{\eta}{\gamma}$ and tends to $0$, this gives a contribution $1$ to the limit. But in this case there is also a concentration behavior at the insertion with parameter $Q-\eta$, this adds $1$ to the final limit.

At this stage we need to import the result from the interval case \cite[Proposition 1.5]{interval} where we have found the reflection coefficient $\overline{R}_1^{\partial}(\beta)$ with one of the $\mu_i$ set to $0$:
\begin{equation}
\overline{R}(\beta, 1, 0) =  \overline{R}_1^{\partial}(\beta) =  \frac{ (2 \pi)^{ \frac{2}{\gamma}(Q -\beta ) -\frac{1}{2}} (\frac{2}{\gamma})^{ \frac{\gamma}{2}(Q -\beta ) -\frac{1}{2} }  }{(Q-\beta) \Gamma(1 -\frac{\gamma^2}{4}  )^{ \frac{2}{\gamma}(Q -\beta ) } } \frac{ \Gamma_{\frac{\gamma}{2}}(\beta - \frac{\gamma}{2}  )}{\Gamma_{\frac{\gamma}{2}}(Q- \beta )}.
\end{equation}
The rest of the proof is now direct algebraic computations. Together with \eqref{equation R reflect} we have:
\begin{align}
\lim_{\eta \to 0} \eta D_2 &= -2(\frac{2}{\gamma}-\beta) \mu_1^{\frac{4}{\gamma^2}} \frac{\Gamma(\frac{2\beta}{\gamma}-\frac{4}{\gamma^2})\Gamma(1-\frac{2\beta}{\gamma})}{\Gamma(-\frac{4}{\gamma^2})} \frac{e^{i \pi(\frac{4}{\gamma^2} - \frac{2\beta}{\gamma})}}{\Gamma(1-\frac{2\beta}{\gamma})\Gamma(1+\frac{2\beta}{\gamma})} \frac{\overline{R}(\beta+\frac{\gamma}{2}, 1, 0)}{\overline{R}(\beta+Q, 1, 0)}, \nonumber \\ 
& =  \frac{4}{\gamma} (2\pi)^{\frac{4}{\gamma^2}-1} \mu_1^{\frac{4}{\gamma^2}}e^{i \pi(\frac{4}{\gamma^2} - \frac{2\beta}{\gamma})} \frac{\Gamma(\frac{2\beta}{\gamma}-\frac{4}{\gamma^2})\Gamma(-\frac{2\beta}{\gamma})}{\Gamma(-\frac{4}{\gamma^2}) \Gamma(1-\frac{\gamma^2}{4})^{\frac{4}{\gamma^2}}}, \\
\lim_{\eta \to 0} \eta C_2 &= \gamma \frac{\Gamma(\frac{2\beta}{\gamma}-\frac{4}{\gamma^2})}{\Gamma(-\frac{4}{\gamma^2})\Gamma(\frac{2\beta}{\gamma})}  \frac{\overline{R}(\beta, \mu_1, \mu_2) }{\overline{R}(\beta+\frac{2}{\gamma}, \mu_1, e^{i \pi } \mu_2)},\\
\lim_{\eta \to 0}\eta^2 B_2 & = -4 (2\pi)^{\frac{4}{\gamma^2}-1}\mu_2^{\frac{4}{\gamma^2}} e^{i\pi \frac{4}{\gamma^2}} \frac{1}{\Gamma(1-\frac{\gamma^2}{4})^{\frac{4}{\gamma^2}}}.
\end{align}
Putting all these into \eqref{equation 3.21}, we get:
\begin{align}
&4 (2\pi)^{\frac{4}{\gamma^2}-1} e^{i\pi \frac{4}{\gamma^2}} \frac{\Gamma(\frac{2\beta}{\gamma}-\frac{4}{\gamma^2})}{\Gamma(-\frac{4}{\gamma^2})\Gamma(1-\frac{\gamma^2}{4})^{\frac{4}{\gamma^2}}}(\mu_1^{\frac{4}{\gamma^2}}e^{-i\pi \frac{2\beta}{\gamma}}-\mu_2^{\frac{4}{\gamma^2}}) \\ \nonumber
&= -2\gamma(Q-\beta)e^{i\pi(-\frac{2\beta}{\gamma}+\frac{4}{\gamma^2})}\frac{\Gamma(\frac{2\beta}{\gamma}-\frac{4}{\gamma^2})}{\Gamma(-\frac{4}{\gamma^2})\Gamma(\frac{2\beta}{\gamma})\Gamma(1-\frac{2\beta}{\gamma})}  \frac{\overline{R}(\beta, \mu_1, \mu_2) }{\overline{R}(\beta+\frac{2}{\gamma}, \mu_1, e^{i \pi} \mu_2)}.
\end{align}
After simplification:
\begin{align}\label{eq:shift_r_inproof}
\frac{\overline{R}(\beta, \mu_1, \mu_2) }{\overline{R}(\beta+\frac{2}{\gamma}, \mu_1, e^{i \pi} \mu_2)} = -\frac{2}{\gamma(Q-\beta)} (2\pi)^{\frac{4}{\gamma^2}-1}  \frac{\Gamma(\frac{2\beta}{\gamma}) \Gamma(1-\frac{2\beta}{\gamma})}{\Gamma(1-\frac{\gamma^2}{4})^{\frac{4}{\gamma^2}}}(\mu_1^{\frac{4}{\gamma^2}}-\mu_2^{\frac{4}{\gamma^2}}e^{i\pi \frac{2\beta}{\gamma}}).
\end{align}
Similarly by repeating the same steps using the auxiliary function $\tilde{H}_{\chi}(t)$ one obtains the shift equation:
\begin{align}\label{eq:shift_r_inproof2}
\frac{\overline{R}(\beta+\frac{2}{\gamma}, \mu_1, e^{i \pi } \mu_2) }{\overline{R}(\beta+\frac{4}{\gamma}, \mu_1, \mu_2)} = -\frac{2}{\gamma(\frac{\gamma}{2}-\beta)} (2\pi)^{\frac{4}{\gamma^2}-1}  \frac{\Gamma(\frac{2\beta}{\gamma}+\frac{4}{\gamma^2}) \Gamma(1-\frac{2\beta}{\gamma}-\frac{4}{\gamma^2})}{\Gamma(1-\frac{\gamma^2}{4})^{\frac{4}{\gamma^2}}}(\mu_1^{\frac{4}{\gamma^2}}-\mu_2^{\frac{4}{\gamma^2}}e^{-i\pi \frac{2\beta}{\gamma}}).
\end{align}
Again we have omitted the details for the tilde case, as the computations are exactly the same using $\tilde{H}_{\chi}(t)$. One can deduce equation \eqref{eq:shift_r_inproof2} from \eqref{eq:shift_r_inproof} by first performing in \eqref{eq:shift_r_inproof} the parameter substitution $\beta \rightarrow \beta +\frac{2}{\gamma}$, $\mu_2 \rightarrow e^{-i \pi} \mu_2 $, and then replacing every complex unit $i$ by $-i$. While this operation formally corresponds to complex conjugation if $\beta, \mu_1, \mu_2$ are chosen real, the constraints on the parameters $\mu_i$ do not allow one to obtain \eqref{eq:shift_r_inproof2} rigorously in this way. It is necessary to use $\tilde{H}_{\chi}(t)$.
Hence we arrive at \eqref{R_second_shift}.
\end{proof}

\subsubsection{Solution of the shift equations on $\overline{R}$}

\begin{proof}[Proof of Theorem \ref{main_th2}, equation \eqref{main_th2_R}]
We introduce $\sigma_1, \sigma_2$ defined through the relation $\mu_{i} := e^{i \pi \gamma (\sigma_i-\frac{Q}{2})}$ with the convention that for positive $\mu_i$ one has $\mathrm{Re}(\sigma_i) = \frac{Q}{2}$. We can thus write for $\chi = \frac{\gamma}{2}$ or $\frac{2}{\gamma}$ that:
\begin{equation}
\left| \mu_1^{\frac{2\chi}{\gamma}} - \mu_2^{\frac{2\chi}{\gamma}} e^{i \pi \chi \beta} \right|^2 = 4 e^{2 i \pi \chi(\sigma_1+\sigma_2 -Q)}\sin \left( \frac{\pi \chi}{2}(\beta +2(\sigma_1 -\sigma_2)) \right) \sin \left( \frac{\pi \chi}{2}(\beta +2(\sigma_2 -\sigma_1)) \right).
\end{equation}
One can then rewrite the two shift equations under the following form,
\begin{align}
\frac{\overline{R}(\beta, \mu_1, \mu_2)}{\overline{R}(\beta + \gamma, \mu_1, \mu_2)} =& -\frac{\Gamma(-1+\frac{\gamma \beta}{2} - \frac{\gamma^2}{4} )\Gamma(2 - \frac{\gamma \beta}{2}-\frac{\gamma^2}{4})  }{ \Gamma(1- \frac{\gamma^2}{4})^2 } \frac{\pi}{\sin(\pi \frac{\gamma \beta}{2})}\\ \nonumber
&\times 4 e^{i \pi \gamma(\sigma_1+\sigma_2 -Q)} \sin \left( \frac{\pi \gamma}{4}(\beta +2(\sigma_1 -\sigma_2)) \right) \sin \left( \frac{\pi \gamma}{4}(\beta +2(\sigma_2 -\sigma_1)) \right),
\end{align}
\begin{align}
\frac{\overline{R}(\beta, \mu_1, \mu_2)}{\overline{R}(\beta +\frac{4}{\gamma}, \mu_1, \mu_2)} =& \frac{(2\pi)^{\frac{8}{\gamma^2}}}{\gamma^2(Q-\beta)(\frac{\gamma}{2}-\beta)}  \frac{1}{\Gamma(1-\frac{\gamma^2}{4})^{\frac{8}{\gamma^2}} \sin(\pi\frac{2\beta}{\gamma}) \sin(\pi (\frac{2\beta}{\gamma}+\frac{4}{\gamma^2}))}\\ \nonumber
&\times 4 e^{ \frac{4 i \pi }{\gamma}(\sigma_1+\sigma_2 -Q)} \sin \left( \frac{\pi }{\gamma}(\beta +2(\sigma_1 -\sigma_2)) \right) \sin \left( \frac{\pi }{\gamma}(\beta +2(\sigma_2 -\sigma_1)) \right).
\end{align}
These two shift equation completely specify the function $\overline{R}(\beta, \mu_1, \mu_2)$ as a function of the parameter $\beta$ up to one value. Since we know that $\overline{R}(Q, \mu_1, \mu_2) = 1$, the function $\overline{R}$ is thus uniquely specified and can be identified to be the function written in equation \eqref{main_th2_R} since it obeys the same two shift equations in $\beta$ and has the same value at $\beta = Q$.
\end{proof}

\subsection{Solving the boundary three-point function}\label{subsec:three_point}

With the value of $\overline{R}$ completely specified, we complete the proof of the expression for $\overline{H}$. The first step is to derive the additional shift equation in $\frac{2}{\gamma}$.

\subsubsection{The shift equations for $\overline{H}$}

\begin{proposition}[Shift equations for $\overline{H}$]\label{full_shift_H} Let $\chi = \frac{\gamma}{2}$ or $\frac{2}{\gamma}$. We have the following functional equations for $\overline{H}^{(\beta_1 , \beta_2, \beta_3)}_{( \mu_1, \mu_2,   \mu_3)}$, viewed as a meromorphic function of all of its parameters
\begin{align}\label{shift three point 1}
&\overline{H}^{(\beta_1 , \beta_2 - \chi, \beta_3)}_{( \mu_1, \mu_2,   \mu_3)} = \frac{\Gamma(\chi(\beta_1-\chi)) \Gamma(1 - \chi\beta_2 + \chi^2)}{\Gamma(\chi(\beta_1-\chi+q\frac{\gamma}{2})) \Gamma(1 - \chi\beta_2 + \chi^2 -q\frac{\gamma\chi}{2}) } \overline{H}^{(\beta_1 - \chi, \beta_2, \beta_3)}_{( \mu_1, e^{\frac{i \pi \gamma \chi}{2}} \mu_2,  \mu_3)} \\ \nonumber 
&- \frac{ \chi^2 (2\pi)^{\frac{2\chi}{\gamma}-1}}{\Gamma(1-\frac{\gamma^2}{4})^{\frac{2\chi}{\gamma}}} \frac{\pi\Gamma(-q+\frac{2\chi}{\gamma})\Gamma(1-\chi\beta_1) \Gamma(1 - \chi\beta_2 + \chi^2) (\mu_1^{\frac{2\chi}{\gamma}}-\mu_2^{\frac{2\chi}{\gamma}}e^{2i\pi\chi\beta_1})}{\sin(\pi\chi(\beta_1-\chi))\Gamma(-q)\Gamma(1 + \frac{q \gamma\chi}{2}) \Gamma(2 - \chi(\beta_1 + \beta_2-2\chi+q\frac{\gamma}{2})) }  \overline{H}^{(\beta_1 + \chi, \beta_2, \beta_3)}_{( \mu_1, e^{\frac{i \pi  \gamma\chi}{2}} \mu_2,   \mu_3)} ,
\end{align}
and:
\begin{align}\label{shift three point 2}
& \frac{\chi^2(2\pi)^{\frac{2\chi}{\gamma}-1}}{\Gamma(1-\frac{\gamma^2}{4})^{\frac{2\chi}{\gamma}}}\frac{\Gamma(-q+\frac{2\chi}{\gamma}) \Gamma(1-\chi\beta_2) }{\Gamma(-q)}(\mu_3^{\frac{2\chi}{\gamma}} -\mu_2^{\frac{2\chi}{\gamma}}e^{i\pi\chi\beta_2}) \overline{H}^{(\beta_1 , \beta_2 +\chi , \beta_3)}_{(\mu_1, e^{i\pi\frac{\gamma\chi}{2}}\mu_2,   \mu_3)} \\ \nonumber
&=\frac{\Gamma(\chi(\beta_1-\chi)) }{\Gamma( - q \frac{\gamma\chi}{2} ) \Gamma(-1 + \chi(\beta_1 + \beta_2-2\chi+q\frac{\gamma}{2}) ) }  \overline{H}^{(\beta_1 -\chi , \beta_2, \beta_3)}_{(  \mu_1,  \mu_2,   \mu_3)}  \nonumber\\
&+   \frac{\chi^2(2\pi)^{\frac{2\chi}{\gamma}-1}}{\Gamma(1-\frac{\gamma^2}{4})^{\frac{2\chi}{\gamma}}}\frac{\Gamma(2  - \chi \beta_1+\chi^2) \Gamma(-q+\frac{2\chi}{\gamma}) \Gamma(1-\chi\beta_1) \Gamma(\chi(\beta_1-Q))}{\Gamma(-q)\Gamma(1- \chi(\beta_1-\chi+q\frac{\gamma}{2} ) \Gamma( \chi\beta_2 - \chi^2 +q\frac{\gamma\chi}{2})}(\mu_1^{\frac{2\chi}{\gamma}} -\mu_2^{\frac{2\chi}{\gamma}}e^{i\pi\chi(\chi-\beta_1)}) \overline{H}^{(\beta_1 + \chi , \beta_2 , \beta_3)}_{(\mu_1, \mu_2,   \mu_3)}. \nonumber
\end{align}
\end{proposition}

\begin{proof}
These shift equations all come from applying \eqref{connection1}. The first comes from the relation,
\begin{align}\label{eq_shift_H1}
B_1 &= \frac{\Gamma(\chi(\beta_1-\chi)) \Gamma(1 - \chi\beta_2 + \chi^2)}{\Gamma(\chi(\beta_1-\chi+q\frac{\gamma}{2}) \Gamma(1 - \chi\beta_2 + \chi^2 -q\frac{\gamma\chi}{2}) } C_1 + \frac{\Gamma(2  - \chi\beta_1+\chi^2) \Gamma(1 - \chi\beta_2 + \chi^2)}{\Gamma(1 + \frac{q \gamma\chi}{2}) \Gamma(2 - \chi(\beta_1 + \beta_2-2\chi+q\frac{\gamma}{2}))   } C_2,
\end{align}
and the second can be deduced from the following relation:
\begin{align}\label{eq_shift_H2}
\tilde{B}_2 = \frac{\Gamma(\chi(\beta_1-\chi)) \Gamma(-1 + \chi\beta_2 - \chi^2)}{\Gamma( - q \frac{\gamma\chi}{2} ) \Gamma(-1 + \chi(\beta_1 + \beta_2-2\chi+q\frac{\gamma}{2}) ) } \tilde{C}_1 + \frac{\Gamma(2  - \chi \beta_1+\chi^2) \Gamma(-1 + \chi\beta_2 - \chi^2)}{\Gamma(1- \chi(\beta_1-\chi+q\frac{\gamma}{2} ) \Gamma( \chi\beta_2 - \chi^2 +q\frac{\gamma\chi}{2}) } \tilde{C}_2.
\end{align}
We need to prove the case $\chi =\frac{2}{\gamma}$. It requires a little bit more effort than for deriving the $\chi =\frac{\gamma}{2}$ shift equations because the $\overline{R}$ function appears in the expressions of $C_2, \tilde{C}_2, \tilde{B}_2$ due to the OPE with reflection given by applying Lemma \ref{lem_reflection_ope2}. For instance the expression of $C_2$ is expressed as $\overline{R}(\beta_1, \mu_1, \mu_2) \overline{H}^{(2 Q - \beta_1 - \frac{2}{\gamma} , \beta_2 , \beta_3)}_{( \mu_1, e^{i\pi}\mu_2,  e^{i\pi} \mu_3)}$. To transform it into $ \overline{H}^{(\beta_1 + \frac{2}{\gamma} , \beta_2 , \beta_3)}_{( \mu_1, e^{i\pi} \mu_2, e^{i\pi} \mu_3)}$ we will need to apply the shift equation of Lemma \ref{lem_shift_eq_R} relating $\overline{R}(\beta_1, \mu_1, \mu_2)$ and $\overline{R}(\beta_1+\frac{2}{\gamma}, \mu_1, e^{i\pi} \mu_2)$ and then the reflection principle of Lemma \ref{reflection_H}. The same strategy has to be applied to derive the expressions for $\tilde{C}_2$ and $\tilde{B}_2$. This allows us to write:
\begin{align}
C_2 &= \frac{\chi^2(2\pi)^{\frac{2\chi}{\gamma}-1}}{\Gamma(1-\frac{\gamma^2}{4})^{\frac{2\chi}{\gamma}}}\frac{\Gamma(-q+\frac{2\chi}{\gamma}) \Gamma(1-\chi\beta_1) \Gamma(\chi(\beta_1-Q))}{\Gamma(-q)}(\mu_1^{\frac{2\chi}{\gamma}} -\mu_2^{\frac{2\chi}{\gamma}}e^{i\pi\chi\beta_1}) \overline{H}^{(\beta_1 + \chi , \beta_2 , \beta_3)}_{(\mu_1, e^{i\pi\frac{\gamma\chi}{2}}\mu_2,  e^{i\pi\frac{\gamma\chi}{2}}  \mu_3)},\\
\tilde{C}_2 &= \frac{\chi^2(2\pi)^{\frac{2\chi}{\gamma}-1}}{\Gamma(1-\frac{\gamma^2}{4})^{\frac{2\chi}{\gamma}}}\frac{\Gamma(-q+\frac{2\chi}{\gamma}) \Gamma(1-\chi\beta_1) \Gamma(\chi(\beta_1-Q))}{\Gamma(-q)}(\mu_1^{\frac{2\chi}{\gamma}} -\mu_2^{\frac{2\chi}{\gamma}}e^{i\pi\chi(\chi-\beta_1)}) \overline{H}^{(\beta_1 + \chi , \beta_2 , \beta_3)}_{(e^{i\pi\frac{\gamma\chi}{2}} \mu_1, \mu_2,   \mu_3)},\\
\tilde{B}_2 &= \frac{\chi^2(2\pi)^{\frac{2\chi}{\gamma}-1}}{\Gamma(1-\frac{\gamma^2}{4})^{\frac{2\chi}{\gamma}}}\frac{\Gamma(-q+\frac{2\chi}{\gamma}) \Gamma(1-\chi\beta_2) \Gamma(\chi(\beta_2-Q))}{\Gamma(-q)}(\mu_3^{\frac{2\chi}{\gamma}} -\mu_2^{\frac{2\chi}{\gamma}}e^{i\pi\chi\beta_2}) \overline{H}^{(\beta_1 , \beta_2 +\chi , \beta_3)}_{(e^{i\pi\frac{\gamma\chi}{2}} \mu_1, e^{i\pi\frac{\gamma\chi}{2}}\mu_2,   \mu_3)}.
\end{align}

We recall also the values of $C_1, \tilde{C}_1, B_1$, stated here for both $\chi = \frac{\gamma}{2}$ and $\frac{2}{\gamma}$:
\begin{align}
C_1 = \overline{H}^{(\beta_1 - \chi, \beta_2, \beta_3)}_{( \mu_1, e^{i \pi \frac{\gamma \chi}{2}} \mu_2, e^{i \pi  \frac{\gamma \chi}{2}} \mu_3)}, \quad B_1 = \overline{H}^{(\beta_1 , \beta_2 - \chi, \beta_3)}_{( \mu_1,  \mu_2, e^{i \pi  \frac{\gamma \chi}{2}} \mu_3)}, \quad  \tilde{C}_1 = \overline{H}^{(\beta_1 - \chi, \beta_2 , \beta_3)}_{(  e^{i \pi   \frac{\gamma \chi}{2}} \mu_1,  \mu_2,   \mu_3)}.
\end{align}

Putting all these into \eqref{eq_shift_H1} and \eqref{eq_shift_H2} proves the shift equations stated in the proposition. 
\end{proof}

\subsubsection{The exact formula $\mathcal{I}$ satisfies the shift equations and the reflection principle}
Take again $\mu_{i} := e^{i \pi \gamma (\sigma_i-\frac{Q}{2})}$  with the convention that $\mathrm{Re}(\sigma_i) = \frac{Q}{2}$ when $\mu_i>0$. Recall also that $\overline{\beta} = \beta_1 + \beta_2 + \beta_3$.

To show that $\overline{H}$ is equal to the exact formula $\mathcal{I}$ given by \eqref{formule_PT}, there are three steps that remain to be shown. 1) The function $\mathcal{I}$ satisfies the shift equations of Lemma \ref{full_shift_H} and the reflection principle of Lemma \ref{reflection_H}. 2) A solution of the shift equations of Lemma \ref{full_shift_H} satisfying also the reflection principle of Lemma \ref{reflection_H} is completely specified up to one global constant. 3) $\mathcal{I}$ and $\overline{H}$ are equal at one particular value of parameters. In the following we will show these three claims. We introduce the following notation:
\begin{equation}
\mathcal{I}
\begin{pmatrix}
\beta_1 , \beta_2, \beta_3 \\
\sigma_1,  \sigma_2,   \sigma_3 
\end{pmatrix} :=\int_{\mathcal{C}} \varphi^{(\beta_1 , \beta_2, \beta_3)}_{(  \sigma_1,  \sigma_2,   \sigma_3)}(r) dr.
\end{equation}
We always work with the parameter constraint $\mathrm{Re}\left( Q - \sigma_3 + \sigma_2 -\frac{\beta_2}{2} \right) > 0$ required for the integral over $\mathcal{C}$ in the definition of $\mathcal{I}$ to converge. Let us start by showing the lemma:
\begin{lemma}\label{lem:shift_I}
The function $\mathcal{I}
\begin{pmatrix}
\beta_1 , \beta_2, \beta_3 \\
\sigma_1,  \sigma_2,   \sigma_3 
\end{pmatrix}$ satisfies the shift equations satisfied by $\overline{H}$.
\end{lemma}

\begin{proof}

For the purpose of this proof we must specify carefully how the contour $\mathcal{C}$ is chosen, this is linked to what is written in the proof of Lemma \ref{lem:J}. If we look at the poles in $r$ of $\varphi$, there are three lattices of poles starting from $-(Q-\frac{\beta_2}{2}+\sigma_3-\sigma_2)$, $-(\frac{\beta_3}{2}+\sigma_3-\sigma_1)$, $-(Q-\frac{\beta_3}{2}+\sigma_3-\sigma_1)$ and extending to $-\infty$ by increments of $\frac{\gamma}{2}$ and $\frac{2}{\gamma}$. We call these the left lattices. Similarly we have three lattices of poles starting from $0$, $-(\frac{\beta_1}{2}-\frac{\beta_2}{2}+\sigma_3-\sigma_1)$, $-(Q-\frac{\beta_1}{2}-\frac{\beta_2}{2}+\sigma_3-\sigma_1) $ and extending to $+\infty$ by similar increments, we call them the right lattices. We are going to work under the assumption that  the poles of the six different latices all have different imaginary parts. This constraint can be easily lifted at the end by analyticity. For each function $\mathcal{I}$ appearing below, this assumption allows to choose the contour $\mathcal{C}$ starting from $- i \infty$, passing to the right of the left lattices of poles by a distance of at least $2 \chi$, to the left of the right lattices of poles by a distance of at least $2 \chi$, and finally continuing to $+ i \infty$. This allows us to shift the contour $\mathcal{C}$ by $\pm \chi $ without crossing any poles of the integrand $\varphi$.

Now checking that $\mathcal{I}$ satisfies the shift equations of Lemma \ref{full_shift_H} is equivalent to checking the following shift equations,

\begin{align}\label{shift three point 3}
& \mathcal{I}
\begin{pmatrix}
\beta_1 , \beta_2, \beta_3 \\
\sigma_1,  \sigma_2,   \sigma_3 
\end{pmatrix}  = \frac{\Gamma(\chi(\beta_1-\chi)) \Gamma(1 - \chi\beta_2 )}{\Gamma(1-\frac{\chi}{2}(\beta_2+ \beta_3 -\beta_1)) \Gamma(\frac{\chi}{2}(\beta_1+\beta_3-\beta_2-2\chi)) } \mathcal{I}
\begin{pmatrix}
\beta_1 -\chi , \beta_2 +\chi, \beta_3 \\
\sigma_1,  \sigma_2 + \frac{\chi}{2},   \sigma_3 
\end{pmatrix} \\ \nonumber
&- \frac{ \chi^2 (2\pi)^{\frac{2\chi}{\gamma}-1}}{\Gamma(1-\frac{\gamma^2}{4})^{\frac{2\chi}{\gamma}}} \frac{\pi\Gamma(\frac{1}{\gamma}(\overline{\beta}-\frac{2}{\chi}))\Gamma(1-\chi\beta_1) \Gamma(1 - \chi\beta_2 )  2i e^{i\pi \chi(\frac{\beta_1}{2}-\chi+\sigma_1+\sigma_2)}\sin(\pi\chi(\frac{\beta_1}{2}-\sigma_1+\sigma_2))}{\sin(\pi\chi(\beta_1-\chi))\Gamma(\frac{1}{\gamma}(\overline{\beta}-2Q))\Gamma(1 + \chi(Q-\frac{\overline{\beta}}{2})) \Gamma(1 - \frac{\chi}{2}(\beta_1 + \beta_2-\beta_3)) }
\mathcal{I}
\begin{pmatrix}
\beta_1 +\chi, \beta_2 + \chi, \beta_3 \\
\sigma_1,  \sigma_2 + \frac{\chi}{2},   \sigma_3 
\end{pmatrix},
\end{align}
and:
\begin{align}\label{shift three point 4}
& \frac{\chi^2(2\pi)^{\frac{2\chi}{\gamma}-1}}{\Gamma(1-\frac{\gamma^2}{4})^{\frac{2\chi}{\gamma}}}\frac{\Gamma(\frac{1}{\gamma}(\overline{\beta}-\frac{2}{\chi}
)) \Gamma(1-\chi\beta_2) }{\Gamma(\frac{1}{\gamma}(\overline{\beta}-2Q))}2 i e^{i\frac{\pi\gamma}{2}(\frac{\beta_2}{2}-\chi+\sigma_2+\sigma_3)}\sin(\pi\chi(\frac{\beta_2}{2}+\sigma_2-\sigma_3)) \mathcal{I}
\begin{pmatrix}
\beta_1 +\chi , \beta_2 +\chi, \beta_3 \\
\sigma_1,  \sigma_2 + \frac{\chi}{2},   \sigma_3 
\end{pmatrix} \\
& =\frac{\Gamma(\chi \beta_1) }{\Gamma( \frac{\chi}{2}(\overline{\beta}-2Q) ) \Gamma(\frac{\chi}{2}(\beta_1+\beta_2-\beta_3)) } \mathcal{I}
\begin{pmatrix}
\beta_1 , \beta_2, \beta_3 \\
\sigma_1,  \sigma_2,   \sigma_3 
\end{pmatrix} - \frac{\chi^2(2\pi)^{\frac{2\chi}{\gamma}-1}}{\Gamma(1-\frac{\gamma^2}{4})^{\frac{2\chi}{\gamma}}}\frac{\pi\Gamma(\frac{1}{\gamma}(\overline{\beta}-\frac{2}{\chi}))  \Gamma(1-\chi\beta_1-\chi^2))}{\sin(\pi\chi\beta_1) \Gamma(\frac{1}{\gamma}(\overline{\beta}-2Q))} \nonumber\\
&\quad \times  \frac{2 i e^{i\pi\chi(-\frac{\beta_1}{2}-\chi+\sigma_1+\sigma_2)}\sin(\pi\chi(-\frac{\beta_1}{2}-\sigma_1+\sigma_2))}{\Gamma(\frac{\chi}{2}(\beta_2+\beta_3-\beta_1-2\chi) ) \Gamma( 1-\frac{\chi}{2}(\beta_1+\beta_3-\beta_2)) } \mathcal{I}
\begin{pmatrix}
\beta_1 + 2\chi, \beta_2, \beta_3 \\
\sigma_1,  \sigma_2,   \sigma_3 
\end{pmatrix}. \nonumber
\end{align}
We calculate the ratios of the integrands,
\begin{align}
\frac{\varphi^{(\beta_1-\chi , \beta_2+\chi, \beta_3)}_{(  \sigma_1,  \sigma_2 +\frac{\chi}{2},   \sigma_3)}(r)}{\varphi^{(\beta_1 , \beta_2, \beta_3)}_{(  \sigma_1,  \sigma_2,   \sigma_3)}(r)} =& \frac{\Gamma(\frac{\chi}{2}(\beta_1+\beta_3-\beta_2-2\chi)) \Gamma(1-\frac{\chi}{2}(\beta_2+\beta_3-\beta_1)) \Gamma(1-\chi\beta_1+\chi^2) }{ \pi \Gamma(1-\chi \beta_2) }\\ \nonumber
&\times \sin(\pi \chi (\frac{\beta_1}{2}-\chi+\sigma_1-\sigma_2))\frac{ \sin(\pi \chi (-\frac{\beta_1}{2}+\frac{\beta_2}{2}+\sigma_1-\sigma_3-r))}{\sin(\pi \chi (\frac{\beta_2}{2}+\sigma_2-\sigma_3-r) )},\\
\frac{\varphi^{(\beta_1+\chi , \beta_2+\chi, \beta_3)}_{(  \sigma_1,  \sigma_2 +\frac{\chi}{2},   \sigma_3)}(r)}{\varphi^{(\beta_1 , \beta_2, \beta_3)}_{(  \sigma_1,  \sigma_2,   \sigma_3)}(r)} =& -\frac{\Gamma(1-\frac{\gamma^2}{4})^{\frac{2\chi}{\gamma}} \Gamma(\frac{1}{\gamma}(\overline{\beta}-2Q))\Gamma(1+\chi(Q-\frac{\overline{\beta}}{2})) \Gamma(1-\frac{\chi}{2}(\beta_1+\beta_2-\beta_3)) ie^{i\pi \chi (Q-\frac{\beta_1}{2}-\sigma_1-\sigma_2)} }{\chi^2 (2\pi)^{\frac{2\chi}{\gamma}} \Gamma(\frac{1}{\gamma}(\overline{\beta}-\frac{2}{\chi})) \Gamma(1-\chi \beta_1) \Gamma(1-\chi \beta_2)} \nonumber \\ 
&\times \frac{\sin(\pi\chi(\frac{\beta_1}{2}+\frac{\beta_2}{2}-\chi+\sigma_1-\sigma_3-r))}{\sin(\pi\chi(\frac{\beta_2}{2}+\sigma_2-\sigma_3-r))}.
\end{align}
If we plug $\mathcal{I}$ into equation \eqref{shift three point 3} and regroup terms on one side, we will get:
\begin{align}
\int_{\mathcal{C}} dr \, \varphi^{(\beta_1 , \beta_2, \beta_3)}_{(  \sigma_1,  \sigma_2,   \sigma_3)}(r)\Bigg[\frac{ \sin(\pi\chi(\frac{\beta_1}{2}-\chi+\sigma_1-\sigma_2))\sin(\pi\chi (-\frac{\beta_1}{2}+\frac{\beta_2}{2}+\sigma_1-\sigma_3-r))}{\sin(\pi\chi(\beta_1-\chi))\sin(\pi \chi(\frac{\beta_2}{2}+\sigma_2-\sigma_3-r) )}-1\\ \nonumber
+\frac{\sin(\pi\chi(\frac{\beta_1}{2}-\sigma_1+\sigma_2))\sin(\pi\chi(\frac{\beta_1}{2}+\frac{\beta_2}{2}-\chi+\sigma_1-\sigma_3-r))}{\sin(\pi\chi(\beta_1-\chi))\sin(\pi\chi(\frac{\beta_2}{2}+\sigma_2-\sigma_3-r))} \Bigg].
\end{align}
We can verify with some algebra that the integrand of the above integral equals $0$, hence $\mathcal{I}$ satisfies \eqref{shift three point 3}. To check the second shift equation, we will need additionally the ratio:

\begin{align}
\frac{\varphi^{(\beta_1+2\chi , \beta_2, \beta_3)}_{(  \sigma_1,  \sigma_2 ,   \sigma_3)}(r)}{\varphi^{(\beta_1 , \beta_2, \beta_3)}_{(  \sigma_1,  \sigma_2,   \sigma_3)}(r)} = &-\frac{\pi \Gamma(1-\frac{\gamma^2}{4})^{\frac{2\chi}{\gamma}} \Gamma(\frac{1}{\gamma}(\overline{\beta}-2Q)) \Gamma(1+\chi(Q-\frac{\overline{\beta}}{2})) \Gamma(1-\frac{\chi}{2}(\beta_1+\beta_2-\beta_3)) }{\chi^2 (2\pi)^{\frac{2\chi}{\gamma}} \Gamma(\frac{1}{\gamma}(\overline{\beta}-\frac{2}{\chi})) \Gamma(\frac{\chi}{2}(\beta_1+\beta_3-\beta_2)) \Gamma(1-\frac{\chi}{2}(\beta_2+\beta_3-\beta_1-2\chi))}\\ \nonumber
& \frac{ie^{i\pi \chi(Q-\frac{\beta_1}{2}-\sigma_1-\sigma_2)}\sin(\pi\chi(\frac{\beta_1}{2}+\frac{\beta_2}{2}-\chi+\sigma_1-\sigma_3-r))}{\Gamma(1-\chi\beta_1-\chi^2) \Gamma(1-\chi\beta_1)\sin(\pi\chi(\frac{\beta_1}{2}+\sigma_1-\sigma_2)) \sin(\pi\chi(\frac{\beta_1}{2}-\frac{\beta_2}{2}+\chi+\sigma_3-\sigma_1+r))}.
\end{align}
If we plug $\mathcal{I}$ into equation \eqref{shift three point 4} and regroup things on one side, we will get:
\begin{align}
&\frac{\Gamma(\chi \beta_1) }{\Gamma( \frac{\chi}{2}(\overline{\beta}-2Q) ) \Gamma(\frac{\chi}{2}(\beta_1+\beta_2-\beta_3)) } \int_{\mathcal{C}} dr \,  \varphi^{(\beta_1 , \beta_2, \beta_3)}_{(  \sigma_1,  \sigma_2,   \sigma_3)}(r)\\ \nonumber
 &\Bigg[ \frac{\sin(\pi\chi\beta_1)\sin(\pi\chi(\frac{\beta_2}{2}+\sigma_2-\sigma_3)) e^{i\pi\chi(-\frac{\beta_1}{2}+\frac{\beta_2}{2}-\sigma_1+\sigma_3)} }{\sin(\pi \chi(\frac{\overline{\beta}}{2}-\chi)) \sin(\frac{\pi \chi}{2}(\beta_1+\beta_2-\beta_3))} \frac{\sin(\pi\chi(\frac{\beta_1}{2}+\frac{\beta_2}{2}-\chi+\sigma_1-\sigma_3-r))}{\sin(\pi\chi(\frac{\beta_2}{2}+\sigma_2-\sigma_3-r))} \\ \nonumber
 & -1 + \frac{\sin(\frac{\pi\chi}{2}(\beta_1+\beta_3-\beta_2)) \sin(\frac{\pi\chi}{2}(\beta_2+\beta_3-\beta_1-2\chi)) e^{-i\pi\chi\beta_1 }}{\sin(\pi\chi(\frac{\overline{\beta}}{2}-\chi)) \sin(\frac{\pi\chi}{2}(\beta_1+\beta_2-\beta_3))} \frac{\sin(\pi\chi(\frac{\beta_1}{2}+\frac{\beta_2}{2}-\chi+\sigma_1-\sigma_3-r))}{ \sin(\pi\chi(\frac{\beta_1}{2}-\frac{\beta_2}{2}+\chi+\sigma_3-\sigma_1+r))} \Bigg].
\end{align}
After some algebra we will be able to write it in the form,

\begin{align}
&\frac{\Gamma(\chi \beta_1) }{\Gamma( \frac{\chi}{2}(\overline{\beta}-2Q) ) \Gamma(\frac{\chi}{2}(\beta_1+\beta_2-\beta_3)) } \frac{\sin(\pi\chi\beta_1)e^{i\pi\chi(-\frac{\beta_1}{2}-\frac{\beta_2}{2}+\chi-\sigma_1+\sigma_3)}}{\sin(\pi\chi(\frac{\overline{\beta}}{2}-\chi)) \sin(\frac{\pi \chi}{2}(\beta_1+\beta_2-\beta_3))} \int_{\mathcal{C}} dr \,  \varphi^{(\beta_1 , \beta_2, \beta_3)}_{(  \sigma_1,  \sigma_2,   \sigma_3)}(r)e^{i\pi\chi r} \\ \nonumber
&\Bigg[\frac{\sin(\pi\chi (\frac{\beta_1}{2}+\frac{\beta_2}{2}-\chi+\sigma_1-\sigma_3-r)) \sin(\pi\chi r)}{\sin(\pi\chi(\frac{\beta_2}{2}+\sigma_2-\sigma_3-r))}e^{i\pi\chi (\frac{\beta_2}{2}-\chi-\sigma_2+\sigma_3)}\\ \nonumber
&+\frac{\sin(\pi\chi(\frac{\beta_3}{2}-\sigma_1+\sigma_3+r)) \sin(\pi\chi(\frac{\beta_3}{2}-\chi+\sigma_1-\sigma_3-r)) }{\sin(\pi\chi(\frac{\beta_1}{2}-\frac{\beta_2}{2}+\chi+\sigma_3-\sigma_1+r))}
 \Bigg]\\ \nonumber
 =& \frac{\Gamma(\chi \beta_1) }{\Gamma( \frac{\chi}{2}(\overline{\beta}-2Q) ) \Gamma(\frac{\chi}{2}(\beta_1+\beta_2-\beta_3)) } \frac{\sin(\pi\chi\beta_1) \sin(\pi\chi(\frac{\beta_3}{2}-\sigma_1+\sigma_3)) e^{i\pi\chi(\frac{\beta_2}{2}+\frac{\beta_3}{2}-\sigma_1-\sigma_2+2\sigma_3)}}{\sin(\pi \chi(\frac{\overline{\beta}}{2}-\chi)) \sin(\frac{\pi\chi}{2}(\beta_1+\beta_2-\beta_3)) \sin(\pi\chi(\frac{\beta_1}{2}-\chi+\sigma_1-\sigma_2))}\\ \nonumber
 & \times \int_{\mathcal{C}}dr \, (T_{-\chi}-1) \left( \sin(\pi\chi(\frac{\beta_1}{2}+\frac{\beta_2}{2}-\chi+\sigma_1-\sigma_3-r))\varphi^{(\beta_1 , \beta_2, \beta_3)}_{(  \sigma_1,  \sigma_2+\chi,   \sigma_3+\chi)}(r) e^{i\pi\chi r} \right),
\end{align}
where $T_{-\chi} f(r) = f(r-\chi)$ for any function $f$. As explained in the first paragraph of the proof, the parameters can be chosen in such a way that it is possible to shift the contour $\mathcal{C}$ by $-\chi$ without crossing any poles. Hence the term written above vanishes and this finishes the proof.
\end{proof}
Next we move on to showing:
\begin{lemma}\label{lem:spe_val_I}
The function $\mathcal{I} $ satisfies the following two properties,
\begin{equation}
\mathcal{I}
\begin{pmatrix}
2Q-\beta_2-\beta_3 , \beta_2, \beta_3 \\
\sigma_1,  \sigma_2,   \sigma_3 
\end{pmatrix} = 1,
\end{equation}
and the reflection principle of Lemma \ref{reflection_H}.
\end{lemma}

\begin{proof}
It is rather direct to observe that it satisfies the reflection principle, since the integrand of the contour integral of $\mathcal{I}$ is not changed when applying the transform $\beta_1 \to 2Q-\beta_1$. To conclude one just needs to use the shift equations of $\Gamma_{\frac{\gamma}{2}}$ and $S_{\frac{\gamma}{2}}$ given in Section \ref{sec:double_gamma}. To see the value at $\beta_1 = 2Q-\beta_2-\beta_3 $ equals $1$, we will need to apply the residue theorem. When $\beta_1$ approaches $2Q-\beta_2-\beta_3$ from the right hand side, we have in front of the contour integral a term $\Gamma(\frac{\overline{\beta}-2Q}{\gamma})^{-1}$ that goes to $0$. Additionally in the contour integral, the two poles at $r=-(\frac{\beta_3}{2}+\sigma_3-\sigma_1) $ and $r=-(Q-\frac{\beta_1}{2}-\frac{\beta_2}{2}+\sigma_3-\sigma_1)$ will collapse. To extract the divergent term, we can slightly modify the contour to let it go from the right hand side of $r=-(Q-\frac{\beta_1}{2}-\frac{\beta_2}{2}+\sigma_3-\sigma_1)$, this allows us to pick up the divergent term by residue theorem:
\begin{align}
\int_{\mathcal{C}} \frac{S_{\frac{\gamma}{2}}(Q-\frac{\beta_2}{2}+\sigma_3-\sigma_2+r) S_{\frac{\gamma}{2}}(\frac{\beta_3}{2}+\sigma_3-\sigma_1+r) S_{\frac{\gamma}{2}}(Q-\frac{\beta_3}{2}+\sigma_3-\sigma_1+r)}{S_{\frac{\gamma}{2}}(Q+\frac{\beta_1}{2}-\frac{\beta_2}{2}+\sigma_3-\sigma_1 + r) S_{\frac{\gamma}{2}}(2Q-\frac{\beta_1}{2}-\frac{\beta_2}{2}+\sigma_3-\sigma_1+r) S_{\frac{\gamma}{2}}(Q+r)}e^{i\pi(-\frac{\beta_2}{2}+\sigma_2-\sigma_3)r} \frac{dr}{i}\\ \nonumber
\underset{\beta_1\to 2Q-\beta_2-\beta_3 }{\sim} \frac{1}{2\pi (\frac{\overline{\beta}}{2}-Q)}  \frac{S_{\frac{\gamma}{2}}(\frac{\beta_1}{2}+\sigma_1-\sigma_2) S_{\frac{\gamma}{2}}(Q-\beta_3)}{ S_{\frac{\gamma}{2}}(\beta_1) S_{\frac{\gamma}{2}}(Q-\frac{\beta_3}{2}+\sigma_1-\sigma_3)}e^{i\pi(\frac{\beta_2}{2}-\sigma_2+\sigma_3)(\frac{\beta_3}{2}+\sigma_3-\sigma_1)}.  
\end{align}
We can check that when $\beta_1\to 2Q-\beta_2-\beta_3$, the term in front of the contour integral is equivalent to
\begin{align}
2\pi (\frac{\overline{\beta}}{2}-Q) S_{\frac{\gamma}{2}}(\beta_1)S_{\frac{\gamma}{2}}(\beta_3)\frac{e^{-i\pi(\frac{\beta_2}{2}-\sigma_2+\sigma_3)(\frac{\beta_3}{2}+\sigma_3-\sigma_1)} }{S_{\frac{\gamma}{2}}(\frac{\beta_1}{2}+\sigma_1-\sigma_2)  S_{\frac{\gamma}{2}}(\frac{\beta_3}{2}+\sigma_3-\sigma_1) }. 
\end{align}
This proves that $\mathcal{I}
\begin{pmatrix}
2Q-\beta_2-\beta_3 , \beta_2, \beta_3 \\
\sigma_1,  \sigma_2,   \sigma_3 
\end{pmatrix} = 1$.

\end{proof}

\subsubsection{Uniqueness of the shift equations on $\overline{H}$}
We will now finish the proof of Theorem \ref{main_th2}. There was a gap in the previous version of this proof, a correct argument can be found in \cite[Section 4.2]{ARSZ}. We give this corrected proof below, see \cite{ARSZ} for more details.
\begin{proof}[Proof of Theorem \ref{main_th2}, equation \eqref{main_th2_H}]
We need to show that the solution space of the system comprised of the shift equations of $\overline{H}$ combined with the reflection principle is of dimension at most one $1$. First assume $\gamma^2 \not \in \mathbb Q$. Instead of working directly with $\mathcal{I}$ and $\overline{H}$, we will instead match the following two functions:
\begin{align*}
&\mathcal{J}\begin{pmatrix}
\beta_1 , \beta_2, \beta_3 \\
\sigma_1,  \sigma_2,   \sigma_3 
\end{pmatrix}  :=  \int_{\mathcal{C}} \frac{S_{\frac{\gamma}{2}}(Q-\frac{\beta_2}{2}+\sigma_3-\sigma_2+r) S_{\frac{\gamma}{2}}(\frac{\beta_3}{2}+\sigma_3-\sigma_1+r) S_{\frac{\gamma}{2}}(Q-\frac{\beta_3}{2}+\sigma_3-\sigma_1+r)}{S_{\frac{\gamma}{2}}(Q+\frac{\beta_1}{2}-\frac{\beta_2}{2}+\sigma_3-\sigma_1+r) S_{\frac{\gamma}{2}}(2Q-\frac{\beta_1}{2}-\frac{\beta_2}{2}+\sigma_3-\sigma_1+r) S_{\frac{\gamma}{2}}(Q+r)}e^{i\pi(-\frac{\beta_2}{2}+\sigma_2-\sigma_3)r} \frac{dr}{i},\\
& \mathcal{J}_{\overline{H}}
\begin{pmatrix}
\beta_1 , \beta_2, \beta_3 \\
\sigma_1,  \sigma_2,   \sigma_3 
\end{pmatrix} := \overline{H}
\times \frac{\Gamma(1-\frac{\gamma^2}{4})^{\frac{2Q-\overline{\beta}}{\gamma}}\Gamma(\frac{\overline{\beta}-2Q}{\gamma})}{(2\pi)^{\frac{2Q-\overline{\beta}}{\gamma}+1}(\frac{2}{\gamma})^{(\frac{\gamma}{2}-\frac{2}{\gamma})(Q-\frac{\overline{\beta}}{2})-1}} \\
& \times  \frac{\Gamma_{\frac{\gamma}{2}}(Q) \Gamma_{\frac{\gamma}{2}}(Q-\beta_1) \Gamma_{\frac{\gamma}{2}}(Q-\beta_2) \Gamma_{\frac{\gamma}{2}}(Q-\beta_3) S_{\frac{\gamma}{2}}(\frac{\beta_1}{2}+\sigma_1-\sigma_2)  S_{\frac{\gamma}{2}}(\frac{\beta_3}{2}+\sigma_3-\sigma_1)}{\Gamma_{\frac{\gamma}{2}}(2Q-\frac{\overline{\beta}}{2})\Gamma_{\frac{\gamma}{2}}(\frac{\beta_1+\beta_3-\beta_2}{2})\Gamma_{\frac{\gamma}{2}}(Q-\frac{\beta_1+\beta_2-\beta_3}{2})\Gamma_{\frac{\gamma}{2}}(Q-\frac{\beta_2+\beta_3-\beta_1}{2})} \\
& \times
  e^{- i\frac{\pi}{2}(-(2Q-\frac{\beta_1}{2}-\sigma_1-\sigma_2)(Q-\frac{\beta_1}{2}-\sigma_1-\sigma_2) + (Q+\frac{\beta_2}{2}-\sigma_2-\sigma_3)(\frac{\beta_2}{2}-\sigma_2-\sigma_3)+(Q+\frac{\beta_3}{2}-\sigma_1-\sigma_3)(\frac{\beta_3}{2}-\sigma_1-\sigma_3) -2\sigma_3(2\sigma_3-Q)        )}.
\end{align*}
Namely we have removed the prefactor in front of the contour integral of $\mathcal{I}$ and done the same multiplication for $\overline{H}$. Now thanks to Proposition~\ref{full_shift_H} and Lemma~\ref{lem:shift_I}, both $\mathcal{J}$ and $\mathcal{J}_{\overline{H}}$ obey the same shift equations which have the following form, still for $\chi = \frac{\gamma}{2}$ or $\frac{2}{\gamma}$,
\begin{align}
& \mathcal{J}
\begin{pmatrix}
\beta_1 , \beta_2 - \chi, \beta_3 \\
\sigma_1,  \sigma_2,   \sigma_3 
\end{pmatrix} = f_1(\beta_1, \beta_2) \mathcal{J}
\begin{pmatrix}
\beta_1 - \chi , \beta_2, \beta_3 \\
\sigma_1,  \sigma_2 + \frac{\chi}{2},   \sigma_3 
\end{pmatrix} + f_2(\beta_1, \beta_2) \mathcal{J}
\begin{pmatrix}
\beta_1 +\chi, \beta_2, \beta_3 \\
\sigma_1,  \sigma_2 + \frac{\chi}{2},   \sigma_3 
\end{pmatrix}, \label{shift_3pt_sec4_1} \\ 
 & \mathcal{J}
\begin{pmatrix}
\beta_1 , \beta_2 +\chi, \beta_3 \\
\sigma_1,  \sigma_2 + \frac{\chi}{2},   \sigma_3 
\end{pmatrix} = f_3(\beta_1, \beta_2) \mathcal{J}
\begin{pmatrix}
\beta_1 - \chi, \beta_2, \beta_3 \\
\sigma_1,  \sigma_2,   \sigma_3 
\end{pmatrix}   + f_4(\beta_1, \beta_2) \mathcal{J}
\begin{pmatrix}
\beta_1 +\chi, \beta_2, \beta_3 \\
\sigma_1,  \sigma_2,   \sigma_3 
\end{pmatrix}, \label{shift_3pt_sec4_2}
\end{align}
where the functions $f_1, f_2, f_3, f_4$ are explicit meromorphic functions.
The advantage of this form of the shift equations is that the four functions $f_i$ all contain only functions which are $2/\chi$-periodic in $\beta_1$ for both values of $\chi$ (and no more gamma functions).
Now consider the shift equation \eqref{shift_3pt_sec4_1} with the parameter replacement $\beta_1 \rightarrow \beta_1 +  \chi $,  $\beta_2 \rightarrow \beta_2 + \chi $:
\begin{align*}
&\mathcal{J}
\begin{pmatrix}
\beta_1 + \chi, \beta_2, \beta_3 \\
\sigma_1,  \sigma_2,   \sigma_3 
\end{pmatrix} = f_1(\beta_1 + \chi, \beta_2 + \chi) \mathcal{J}
\begin{pmatrix}
\beta_1 , \beta_2 +\chi, \beta_3 \\
\sigma_1,  \sigma_2 + \frac{\chi}{2},   \sigma_3 
\end{pmatrix} + f_2(\beta_1 + \chi, \beta_2 + \chi) \mathcal{J}
\begin{pmatrix}
\beta_1 +2 \chi, \beta_2 +\chi, \beta_3 \\
\sigma_1,  \sigma_2 + \frac{\chi}{2},   \sigma_3 
\end{pmatrix}.
\end{align*}

In this equation the two $\mathcal{J}$ functions appearing on the right hand side can be expressed in terms of $\mathcal{J}$ functions involving only shifts on $\beta_1$ using twice equation \eqref{shift_3pt_sec4_2}, once as it is and once with the parameter replacement $\beta_1 \rightarrow \beta_1 + 2 \chi $. Performing one more global parameter replacement $\beta_1$ to $\beta_1 + \chi$, the conclusion is that we land on the following shift equation:
\begin{align}
0 &=   f_1(\beta_1 + 2 \chi, \beta_2 + \chi )  f_3(\beta_1 + \chi, \beta_2) \mathcal{J}
\begin{pmatrix}
\beta_1 , \beta_2, \beta_3 \\
\sigma_1,  \sigma_2,   \sigma_3 
\end{pmatrix} \\ \nonumber
&  + \left( -1+  f_1(\beta_1 + 2 \chi, \beta_2 + \chi )  f_4(\beta_1 + \chi, \beta_2) +  f_2(\beta_1 + 2 \chi, \beta_2 + \chi )  f_3(\beta_1 + 3 \chi, \beta_2) \right) \mathcal{J}
\begin{pmatrix}
\beta_1 +2 \chi, \beta_2, \beta_3 \\
\sigma_1,  \sigma_2,   \sigma_3 
\end{pmatrix}   \\ \nonumber
& + f_2(\beta_1 + 2 \chi , \beta_2 + \chi )  f_4(\beta_1 + 3 \chi, \beta_2) \mathcal{J}
\begin{pmatrix}
\beta_1 +4 \chi, \beta_2, \beta_3 \\
\sigma_1,  \sigma_2,   \sigma_3 
\end{pmatrix}.
\end{align}

Now that we have a shift equation only on the variable $\beta_1$, we can fix the five parameters $\beta_2, \beta_3, \sigma_1, \sigma_2, \sigma_3$ to some generic values and view all of the functions below as functions of the single parameter $\beta_1$. From now on we write simply $\mathcal{J}(\beta_1)$ to lighten notations. 
The above shift equations can be put into the form
\begin{align}\label{eq:shift_J}
 \mathcal{J}(\beta_1 + 4 \chi) + a_{\chi}(\beta_1) \mathcal{J}(\beta_1 + 2\chi) + b_{\chi}(\beta_1)  \mathcal{J}(\beta_1) =0,
\end{align}
where $a_{\chi}(\beta_1), b_{\chi}(\beta_1)$ have a simple expression in terms of $f_1, f_2, f_3, f_4$. To argue uniqueness we will introduce the matrices:
\begin{align*}
M_{1}(\beta_1) = \begin{pmatrix}
\mathcal{J}(\beta_1) & \mathcal{J}(\beta_1 - \frac{4}{\gamma})  \\
\mathcal{J}(\beta_1 - \gamma )& \mathcal{J}(\beta_1 - 2Q)
\end{pmatrix}, \quad M_{2}(\beta_1) = \begin{pmatrix}
\mathcal{J}_{\overline{H}}(\beta_1) & \mathcal{J}_{\overline{H}}(\beta_1 - \frac{4}{\gamma})  \\
\mathcal{J}_{\overline{H}}(\beta_1 - \gamma )& \mathcal{J}_{\overline{H}}(\beta_1 - 2Q)
\end{pmatrix}. 
\end{align*}
Set $\beta_0 := 2Q - \beta_2 - \beta_3$. We will show these matrices match at $\beta_1 = \beta_0$, or more precisely that:
\begin{equation}\label{eq:match_res}
\lim_{\beta_1\to \beta_0}(\beta_1 - \beta_0) M_1(\beta_1) = \lim_{\beta_1\to \beta_0}(\beta_1 - \beta_0) M_2(\beta_1).
\end{equation}
From Lemma \ref{lem:spe_val_I} we know that $\mathcal{I} =1$ when $\beta_0 = \beta_1$ and similarly for $\overline{H}$ which implies $\lim_{\beta_1\to \beta_0}(\beta_1 - \beta_0)\mathcal{J}(\beta_1) = \lim_{\beta_1\to \beta_0}(\beta_1 - \beta_0)\mathcal{J}_{\overline{H}}(\beta_1)$. Similarly one can show $\lim_{\beta_1\to \beta_0}(\beta_1 - \beta_0)\mathcal{J}(\beta_1 - \gamma) = \lim_{\beta_1\to \beta_0}(\beta_1 - \beta_0)\mathcal{J}_{\overline{H}}(\beta_1 - \gamma)$ by repeating the proof of Lemma \ref{lem:spe_val_I} for the case of $\mathcal{I}$ and by computing explicitly a first moment of GMC in the case of $\overline{H}$. Then by applying the reflection principle, one obtains that the remaining two limits match and thus the claim \eqref{eq:match_res} holds. See \cite{ARSZ} for more details.

Now let us write down the shift equations satisfied by the matrices $M_1, M_2$. One has
\begin{align*}
M_1(\beta_1 + \gamma) = A_{\frac{\gamma}{2}}(\beta_1) M_1(\beta_1), \quad M_1(\beta_1 + \frac{4}{\gamma}) =  M_1(\beta_1) A_{\frac{2}{\gamma}}(\beta_1)^\top, \quad A_{\chi}(\beta_1) := \begin{pmatrix}
- a_{\chi}(\beta_1) & - b_{\chi}(\beta_1)  \\
1 & 0
\end{pmatrix},
\end{align*}
and the same relations for $M_2$.
The fact that these first order shift equations on the matrices $M_1, M_2$ hold uses \eqref{eq:shift_J} combined with the fact that the function $a_{\chi}(\beta_1), b_{\chi}(\beta_1)$ are both $2/\chi$-periodic for both values of $\chi$.
Consider the determinant: $$D_1(\beta_1) = \det M_1(\beta_1) = \mathcal{J}(\beta_1) \mathcal{J}(\beta_1- 2 Q)- \mathcal{J}(\beta_1 -\gamma) \mathcal{J}(\beta_1 - \frac{4}{\gamma}).$$
It is easy to see that the residue of $D_1(\beta_1)$ at $\beta_1 = \beta_0$ is not zero. Since $D_1(\beta_1)$ is a nonzero meromorphic function, it must have isolated and at most countably many zeros. Therefore the matrix inverse $M_1(\beta_1)^{-1}$ is a well-defined 2-by-2 matrix whose entries are meromorphic functions of $\beta_1$.
Now consider $M(\beta_1) = M_2(\beta_1) M_1(\beta_1)^{-1}$. Then the matrix $M(\beta_1)$ satisfies:
\begin{equation}\label{eq:M-shift}
    M(\beta_1 + \gamma ) = A_{\frac{\gamma}{2}}(\beta_1) M(\beta_1) A_{\frac{\gamma}{2}}(\beta_1)^{-1}, \quad  M(\beta_1 + \frac{4}{\gamma}) =  M(\beta_1).
\end{equation}
Similarly as for $M_1$, since $\det A_{\frac{\gamma}{2}}(\beta_1) = b_{\frac{\gamma}{2}}(\beta_1)$   is nonzero  meromorphic function,  the matrix $A_{\frac{\gamma}{2}}(\beta_1)^{-1}$ is a well-defined 2-by-2 matrix with meromorphic entries. Thanks to equation \eqref{eq:match_res}, we known that $M(\beta_1)$ is the identity matrix for $\beta_1 = \beta_0$. By the standard argument, when $ \gamma^2 \notin \mathbb{Q}$, the two shift equations in~\eqref{eq:M-shift} imply $M(\beta_1)$ is the identity matrix for all $\beta_1$. The same holds for $ \gamma^2 \in \mathbb{Q}$ by continuity. Therefore $\mathcal{J} = \mathcal{J}_{\overline{H}}$ and hence $\overline{H} = \mathcal{I}$.
\end{proof}

\subsubsection{Consistency with the interval GMC}\label{sec:consistency_int}

Finally we include here the consistency check that the formula $\mathcal{I}$ matches the one of \cite{interval} in the special case $\mu_2 = 1$, $\mu_1 = \mu_3 =0$, see also the discussion of Section \ref{app:link_int}.
In terms of the $\sigma_i$ variables, we look at the limit $\textnormal{Im}(\sigma_1),\textnormal{Im}(\sigma_3) \to +\infty$ and $\sigma_2 = \frac{Q}{2}$ and use the asymptotic of the $S_{\frac{\gamma}{2}}$ function given by \eqref{eq:lim_S}. First let us do a change of variable $r \to r -\frac{\beta_3}{2}+\sigma_1-\sigma_3$, the contour integral will become
\begin{align}
\int_{\mathcal{C}} \frac{S_{\frac{\gamma}{2}}(Q-\frac{\beta_2}{2}-\frac{\beta_3}{2}+\sigma_1-\sigma_2+r) S_{\frac{\gamma}{2}}(r) S_{\frac{\gamma}{2}}(Q-\beta_3+r) e^{i\pi(-\frac{\beta_2}{2}+\sigma_2-\sigma_3)(-\frac{\beta_3}{2}+\sigma_1-\sigma_3)}}{S_{\frac{\gamma}{2}}(Q+\frac{\beta_1}{2}-\frac{\beta_2}{2}-\frac{\beta_3}{2}+r) S_{\frac{\gamma}{2}}(2Q-\frac{\overline{\beta}}{2}+r) S_{\frac{\gamma}{2}}(Q-\frac{\beta_3}{2}+\sigma_1-\sigma_3+r)}e^{i\pi(-\frac{\beta_2}{2}+\sigma_2-\sigma_3)r} \frac{dr}{i}.
\end{align}
We send $\textnormal{Im}(\sigma_1) \to +\infty$ and do the change $r\to -r$:
\begin{align}
e^{i\pi(-\frac{\beta_2}{2}+\sigma_2-\sigma_3)(-\frac{\beta_3}{2}+\sigma_1-\sigma_3)} e^{i\frac{\pi}{2}(-(Q-\frac{\beta_2}{2}-\frac{\beta_3}{2}+\sigma_1-\sigma_2)(-\frac{\beta_2}{2}-\frac{\beta_3}{2}+\sigma_1-\sigma_2)+(Q-\frac{\beta_3}{2}+\sigma_1-\sigma_3)(-\frac{\beta_3}{2}+\sigma_1-\sigma_3))} \nonumber \\
\times \int_{\mathcal{C}} \frac{S_{\frac{\gamma}{2}}(\frac{\beta_2}{2}+\frac{\beta_3}{2}-\frac{\beta_1}{2}+r) S_{\frac{\gamma}{2}}(\frac{\overline{\beta}}{2}-Q+r) }{ S_{\frac{\gamma}{2}}(Q+r) S_{\frac{\gamma}{2}}(\beta_3+r)}e^{-i2\pi(\sigma_2-\sigma_3)r} \frac{dr}{i}.
\end{align}
From the result of \cite{Ponsot 2000} on the b-hypergeometric functions, when $\textnormal{Im}(\sigma_3) \to +\infty$, the above contour integral (excluding the prefactor) converges to:
\begin{align}
\frac{S_{\frac{\gamma}{2}}(\frac{\beta_2}{2}+\frac{\beta_3}{2}-\frac{\beta_1}{2}) S_{\frac{\gamma}{2}}(\frac{\overline{ \beta}}{2}-Q) }{  S_{\frac{\gamma}{2}}(\beta_3)}.
\end{align}
The rest of the terms in $\mathcal{I}$ are much easier to analyse. Putting everything together and taking $\textnormal{Im}(\sigma_1),\textnormal{Im}(\sigma_3) \to +\infty$, $\sigma_2=\frac{Q}{2}$ will yield after simplification:
\begin{align}
\overline{H}^{(\beta_1,\beta_2,\beta_3)}_{(0,1,0)}=  \frac{(2\pi)^{\frac{2Q-\overline{\beta}}{\gamma}+1}(\frac{2}{\gamma})^{(\frac{\gamma}{2}-\frac{2}{\gamma})(Q-\frac{\overline{\beta}}{2})-1}}{\Gamma(1-\frac{\gamma^2}{4})^{\frac{2Q-\overline{\beta}}{\gamma}}\Gamma(\frac{\overline{\beta}-2Q}{\gamma})} 
 \frac{\Gamma_{\frac{\gamma}{2}}(\frac{\overline{\beta}}{2}-Q)\Gamma_{\frac{\gamma}{2}}(\frac{\beta_1+\beta_3-\beta_2}{2})\Gamma_{\frac{\gamma}{2}}(\frac{\beta_2+\beta_3-\beta_1}{2})\Gamma_{\frac{\gamma}{2}}(Q-\frac{\beta_1+\beta_2-\beta_3}{2})}{\Gamma_{\frac{\gamma}{2}}(Q) \Gamma_{\frac{\gamma}{2}}(Q-\beta_1) \Gamma_{\frac{\gamma}{2}}(Q-\beta_2) \Gamma_{\frac{\gamma}{2}}(\beta_3)}.
\end{align} 
It can be easily checked that this formula is exactly the same as what the authors have found in \cite{interval}.

\section{Proof of the BPZ differential equations}\label{sec_BPZ}

The goal of this section is to check the BPZ differential equations - reducing in our case to the standard hypergeometric equations - that have been used extensively in Sections \ref{sec_bulk_boundary} and \ref{sec_3pt}.

\subsection{Bulk-boundary case}\label{sec_BPZ1}

We prove here the differential equation used in Section \ref{sec_bulk_boundary}.

\begin{proposition}
Let $\chi = \frac{\gamma}{2}$ or $\frac{2}{\gamma}$, $p = \frac{2}{\gamma}(Q-\alpha-\frac{\beta}{2}+\frac{\chi}{2})$. Consider in the following parameter range,
\begin{equation}\label{para_G_BPZ}
 \beta < Q, \quad p < \frac{4}{\gamma^2} \wedge \frac{2}{\gamma}(Q -\beta), \quad t \in \mathbb{H},
\end{equation}
the auxiliary function,
\begin{equation}
G_{\chi}(t)  =  \mathbb{E} \left[ \left(\int_{\mathbb{R}} \frac{(t-x)^{\frac{\gamma\chi}{2}}}{|x-i|^{\gamma\alpha  } }   g(x)^{\frac{\gamma^2}{8}(p-1)} e^{\frac{\gamma}{2} X(x) } d x \right)^{p} \right].
\end{equation}
Consider furthermore for $t \in \{r e^{i \theta} \: | \: r> 0, \:  \theta \in (0 , \frac{\pi}{2})\} $ the change of variable $s=\frac{1}{1+t^2}$ and $\tilde{G}_{\chi}(s) = s^{p\frac{\gamma \chi}{4}}G_{\chi}(t)$. Then the function $\tilde{G}_{\chi}(s)$ obeys the hypergeometric equation,
\begin{align}
s(1-s)\partial_s^2 \tilde{G}_{\chi}(s) +(C-(A+B+1)s)\partial_s \tilde{G}_{\chi}(s) - AB\tilde{G}_{\chi}(s) = 0,
\end{align}
with the parameters defined by:
\begin{align}
A = -p\frac{\gamma \chi}{4}, B = 1+\chi(\chi - \alpha -p\frac{\gamma}{4}), C = \frac{3}{2}+\chi(\chi -\alpha - p\frac{\gamma}{2}).
\end{align}
\end{proposition}
\begin{remark}
As explained in Section \ref{sec_bulk_boundary}, in the change of variable from $t$ to $s$ the argument of $s$ is in $(-\pi,0)$ and one has $\sqrt{1-s }= t \sqrt{s}$. Furthermore, the exact same hypergeometric equation holds for the dual function $\hat{G}_{\chi}(s)$ introduced in Section \ref{sec_bulk_boundary}.
\end{remark}

\begin{proof}
For simplicity, we introduce the notations,
\begin{equation}
V_1 (x_1;t) = \mathbb{E} \left[ \left(\int_{\mathbb{R}} \frac{(t-x)^{\frac{\gamma\chi}{2}}}{|x-i|^{\gamma\alpha  }   |x-x_1|^{\frac{\gamma^2}{2}}}  e^{\frac{\gamma}{2} X(x) } g(x)^{\frac{\gamma^2}{8}(p-2)} d x \right)^{p-1} \right],
\end{equation}
\begin{equation}
V_2 (x_1,x_2;t) = \mathbb{E} \left[ \left(\int_{\mathbb{R}} \frac{(t-x)^{\frac{\gamma\chi}{2}}}{|x-i|^{\gamma\alpha  }  |x-x_1|^{\frac{\gamma^2}{2}} |x-x_2|^{\frac{\gamma^2}{2}}}  e^{\frac{\gamma}{2} X(x) } g(x)^{\frac{\gamma^2}{8}(p-3)} d x \right)^{p-2} \right].
\end{equation}
We will not be bothered here with the regularization procedure of the log-correlated field $X$ that must in principle be used to perform the computations. A fully rigorous proof implementing the regularization method can be found in \cite{interval}.
Let us compute the derivatives of the function $G_{\chi}(t)$ with the help of the Girsanov Theorem \ref{girsanov} as was done in equation \eqref{OPE_no_reflection}. By direct differentiation:
\begin{align}\label{eq:proof_bpz1}
\partial_t G_{\chi} = p \frac{\gamma \chi}{2} \intr dx_1 \frac{(t-x_1)^{\frac{\gamma\chi}{2} -1 }}{|x_1-i|^{\gamma\alpha  } } V_1(x_1;t).
\end{align}

We need to distinguish between the two values of $\chi$. First we set $\chi = \frac{\gamma}{2}$. By differentiating \eqref{eq:proof_bpz1} with respect to $t$ and then performing integration by parts we obtain:
\begin{align}
&\partial_t^2 G_{\frac{\gamma}{2}} = -p \frac{\gamma^2}{4} \intr dx_1 \frac{\partial_{x_1}(t-x_1)^{\frac{\gamma^2}{4} -1 }}{|x_1-i|^{\gamma\alpha  }  } V_1(x_1;t) + p \frac{\gamma^2}{4} \intr dx_1 \frac{(t-x_1)^{\frac{\gamma^2}{4} -1 }}{|x_1-i|^{\gamma\alpha  }  } \partial_{t}V_1(x_1;t)\\ \nonumber
&= -p\frac{\gamma^3}{8} \intr dx_1 \left( \frac{\alpha}{x_1-i} + \frac{\alpha}{x_1+i}  \right) \frac{(t-x_1)^{\frac{\gamma^2}{4}-1} }{|x_1-i|^{\gamma\alpha  }  } V_1(x_1;t) + p\frac{\gamma^2}{4} \intr dx_1 \frac{(t-x_1)^{\frac{\gamma^2}{4}-1}}{|x_1-i|^{\gamma\alpha  } } (\partial_{x_1} +\partial_{t})V_1(x_1;t).
\end{align}
We can compute the last term by using the symmetry between $x_1$ and $x_2$:
\begin{align}
&\intr dx_1 \frac{(t-x_1)^{\frac{\gamma^2}{4}-1}}{|x_1-i|^{\gamma\alpha  } } (\partial_{x_1}+\partial_{t})V_1(x_1;t)\\ \nonumber
&=(p-1)\frac{\gamma}{2}\intr \intr dx_1 dx_2 \left(-\frac{\gamma}{x_1-x_2} +\frac{\gamma}{2(t-x_2)}\right) \frac{(t-x_1)^{\frac{\gamma^2}{4}-1}(t-x_2)^{\frac{\gamma^2}{4}}}{|x_1-i|^{\gamma\alpha  } |x_2-i|^{\gamma\alpha  } |x_1-x_2|^{\frac{\gamma^2}{2}}} V_2(x_1,x_2;t)\\ \nonumber
&=(p-1)\frac{\gamma}{2}\intr \intr dx_1 dx_2 \left(- \frac{\gamma}{2} \frac{t- x_2}{x_1-x_2} + \frac{\gamma}{2} \frac{t- x_1}{x_1-x_2} +\frac{\gamma}{2} \right) \frac{(t-x_1)^{\frac{\gamma^2}{4}-1}(t-x_2)^{\frac{\gamma^2}{4}-1}}{|x_1-i|^{\gamma\alpha  } |x_2-i|^{\gamma\alpha  } |x_1-x_2|^{\frac{\gamma^2}{2}}} V_2(x_1,x_2;t) =0.
\end{align}

Note that the function $\frac{1}{x_1-x_2}$ is actually not integrable, but this argument can be made fully rigorous by working with a regularized version of $X$, in which case one would land on a smoothed version of $\frac{1}{x_1-x_2}$. See equation (3.2) in \cite{interval} for a precise regularization and the computations below equation (3.3) in \cite{interval} for the same steps using the regularization.
Putting these two steps together and performing simple algebra:
\begin{align}
&\partial_t^2 G_{\frac{\gamma}{2}} = -p\frac{\gamma^3}{8} \intr dx_1 \left( \frac{1}{t-x_1}\left(\frac{\alpha}{t-i} +\frac{\alpha}{t+i}\right) +\frac{\alpha}{(t-i)(x_1-i)}+ \frac{\alpha}{(t+i)(x_1+i)} \right) \frac{(t-x_1)^{\frac{\gamma^2}{4}} }{|x_1-i|^{\gamma\alpha  } } V_1(x_1;t) \nonumber \\ 
&= -\frac{\gamma}{2} \left(\frac{\alpha}{t-i} +\frac{\alpha}{t+i}\right) \partial_t G_{\frac{\gamma}{2}} -p\frac{\gamma^3}{8} \intr dx_1 \left( \frac{\alpha}{(t-i)(x_1-i)}+ \frac{\alpha}{(t+i)(x_1+i)} \right)\frac{(t-x_1)^{\frac{\gamma^2}{4}} }{|x_1-i|^{\gamma\alpha  }  } V_1(x_1;t).
\end{align}
Now moving to the case $\chi =\frac{2}{\gamma}$, first by using integration by parts:
\begin{align}
 & -\frac{\gamma \alpha}{2} \intr dx_1 \left( \frac{1}{x_1 - i} + \frac{1}{x_1 + i}  \right) \frac{1}{|x_1-i|^{\gamma\alpha  } } V_1(x_1;t) =  \intr dx_1 \partial_{x_1} \frac{1}{|x_1-i|^{\gamma\alpha  } } V_1(x_1;t)\\ \nonumber
 & = - \intr dx_1 \frac{1}{|x_1-i|^{\gamma\alpha  } } \partial_{x_1}  V_1(x_1;t)  = (p-1)\frac{\gamma^2}{4} \intr \intr dx_1 dx_2 \frac{1}{|x_1-i|^{\gamma\alpha  } |x_2-i|^{\gamma\alpha  }  |x_1-x_2|^{\frac{\gamma^2}{2}}} V_2(x_1,x_2;t).
\end{align}
Applying a derivative to \eqref{eq:proof_bpz1} and using the above equation we obtain:
\begin{align}
\partial_t^2 G_{\frac{2}{\gamma}} = p  \intr dx_1 \frac{1}{|x_1-i|^{\gamma\alpha  } } \partial_{t}V_1(x_1;t) &= p(p-1) \intr \intr dx_1 dx_2 \frac{1}{|x_1-i|^{\gamma\alpha  } |x_2-i|^{\gamma\alpha  }  |x_1-x_2|^{\frac{\gamma^2}{2}}} V_2(x_1,x_2;t) \nonumber \\
&= -\frac{2}{\gamma} p \alpha \intr dx_1 \left( \frac{1}{x_1 - i} + \frac{1}{x_1 + i}  \right) \frac{1}{|x_1-i|^{\gamma\alpha  } } V_1(x_1;t).
\end{align}
Using again \eqref{eq:proof_bpz1} this can be rewritten as:

\begin{align}
\partial_t^2 G_{\frac{2}{\gamma}} &= -\frac{2}{\gamma}\left(\frac{\alpha}{t-i} +\frac{\alpha}{t+i}\right) \partial_t G_{\frac{2}{\gamma}} - p\frac{2}{\gamma}\intr dx_1 \left( \frac{\alpha}{(t-i)(x_1-i)}+ \frac{\alpha}{(t+i)(x_1+i)} \right)\frac{t-x_1 }{|x_1-i|^{\gamma\alpha  } } V_1(x_1;t).
\end{align}

We can also write $G_{\chi}$ in a similar form. For both values of $\chi$, an integration by parts together with the symmetry shows that:
\begin{align}
(\frac{\gamma\chi}{2}+1) G_{\chi} =& (\frac{\gamma\chi}{2}+1) \mathbb{E} \left[ \int_{\mathbb{R}} \frac{(t-x_1)^{\frac{\gamma\chi}{2}}}{|x_1-i|^{\gamma\alpha  } }   g(x_1)^{\frac{\gamma^2}{8}(p-1)} e^{\frac{\gamma}{2} X(x_1) } d x_1 \left(\int_{\mathbb{R}} \frac{(t-x)^{\frac{\gamma\chi}{2}}}{|x-i|^{\gamma\alpha  } }   g(x)^{\frac{\gamma^2}{8}(p-1)} e^{\frac{\gamma}{2} X(x) } d x \right)^{p-1} \right] \nonumber  \\ \nonumber
=& - \intr dx_1 \frac{\partial_{x_1} (t-x_1)^{\frac{\gamma\chi}{2}+1}}{|x_1-i|^{\gamma\alpha  } } V_1(x_1;t)\\ \nonumber
=& -\frac{\gamma}{2}\intr dx_1 \left( \frac{\alpha}{x_1-i}+\frac{\alpha}{x_1+i}  \right)\frac{ (t-x_1)^{\frac{\gamma\chi}{2}+1}}{|x_1-i|^{\gamma\alpha  } } V_1(x_1;t)  + \intr dx_1 \frac{ (t-x_1)^{\frac{\gamma\chi}{2}+1}}{|x_1-i|^{\gamma\alpha  } } \partial_{x_1} V_1(x_1;t) \\ \nonumber
=& -\frac{\gamma}{2}\intr dx_1 \left( \frac{\alpha}{x_1-i}+\frac{\alpha}{x_1+i}  \right)\frac{ (t-x_1)^{\frac{\gamma\chi}{2}+1}}{|x_1-i|^{\gamma\alpha  } } V_1(x_1;t) +(p-1)\frac{\gamma^2}{4} G_{\chi}\\ 
=&-\frac{\gamma}{2}\intr dx_1 \left( \frac{\alpha(t-i)}{x_1-i}+\frac{\alpha(t+i)}{x_1+i} \right)\frac{ (t-x_1)^{\frac{\gamma\chi}{2}}}{|x_1-i|^{\gamma\alpha  }  } V_1(x_1;t) +((p-1)\frac{\gamma^2}{4}+\gamma\alpha) G_{\chi}.
\end{align}
Above to go from the third to the fourth line we have used symmetrization in the following way:
\begin{align}
\intr &dx_1 \frac{(t-x_1)^{\frac{\gamma\chi}{2}+1}}{|x_1-i|^{\gamma\alpha  }  } \partial_{x_1}V_1(x_1;t)\\ \nonumber
 &= -(p-1)\frac{\gamma^2}{2}\intr \intr dx_1 dx_2\frac{1}{x_1-x_2} \frac{(t-x_1)^{\frac{\gamma\chi}{2} +1 }(t-x_2)^{\frac{\gamma\chi}{2}}}{|x_1-i|^{\gamma\alpha  }  |x_2-i|^{\gamma\alpha  }  |x_1-x_2|^{\frac{\gamma^2}{2}}} V_2(x_1,x_2;t) \\ \nonumber
 &= (p-1) \frac{\gamma^2}{4} G_{\chi},  \qquad \text{by the symmetry} \quad  x_1 \leftrightarrow x_2.
\end{align}

Now we summarize the expressions of the derivatives,
\begin{align}
(2\chi+\frac{1}{\chi} - p\frac{\gamma}{2}-2\alpha)G_{\chi} =& -\intr dx_1 \left( \frac{\alpha(t-i)}{x_1-i}+\frac{\alpha(t+i)}{x_1+i}  \right)\frac{ (t-x_1)^{\frac{\gamma\chi}{2}}}{|x_1-i|^{\gamma\alpha  }  } V_1(x_1;t),\\
\partial_t G_{\chi} = &-p\frac{\gamma}{2} \intr dx_1 \left( \frac{\alpha}{x_1-i} + \frac{\alpha}{x_1+i} \right) \frac{(t-x_1)^{\frac{\gamma\chi}{2}}}{|x_1-i|^{\gamma\alpha  } } V_1(x_1;t),
\end{align}
and when $\chi = \frac{\gamma}{2}$ or $\frac{2}{\gamma}$,
\begin{small}
\begin{align}
\partial_t^2 G_{\chi} =-\chi\left(\frac{\alpha}{t-i} +\frac{\alpha}{t+i}\right) \partial_t G_{\chi} -p\frac{\gamma\chi^2}{2} \intr dx_1 \left( \frac{\alpha}{(t-i)(x_1-i)} + \frac{\alpha}{(t+i)(x_1+i)} \right)\frac{(t-x_1)^{\frac{\gamma\chi}{2}}}{|x_1-i|^{\gamma\alpha  } }V_1(x_1;t).
\end{align}
\end{small}
Combining everything implies that $G_{\chi}$ satisfies a differential equation:
\begin{align}
(t^2+1)\partial_t^2 G_{\chi} + 2\chi(\alpha-\chi)t\partial_t G_{\chi} + p\chi^2(\gamma\chi+\frac{\gamma}{2\chi} - p\frac{\gamma^2}{4}-\gamma\alpha)G_{\chi}=0.
\end{align}
Now consider $s=\frac{1}{1+t^2}$ and take the function,
\begin{equation}
\tilde{G}_{\chi}(s) = s^{p\frac{\gamma \chi}{4}}G_{\chi}(t).
\end{equation}
One then has:
\begin{align}
\partial_s \tilde{G}_{\chi}(s) =& - \frac{1}{2}s^{-\frac{3}{2}}(1-s)^{-\frac{1}{2}} s^{p\frac{\gamma \chi}{4}}\partial_t G_{\chi}(t)+p\frac{\gamma\chi}{4}s^{-1}\tilde{G}_{\chi}(s),\\
\partial_{s}^2 \tilde{G}_{\chi}(s) =& \frac{1}{4}\left( (3-p\frac{\gamma\chi}{2})s^{-1} -(1-s)^{-1} \right)s^{-\frac{3}{2}}(1-s)^{-\frac{1}{2}} s^{p\frac{\gamma \chi}{4}}\partial_t G_{\chi}(t)\\ \nonumber
& + \frac{1}{4}s^{-3}(1-s)^{-1} s^{p\frac{\gamma \chi}{4}}\partial_t^2 G_{\chi}(t) +p\frac{\gamma\chi}{4}s^{-1}\partial_s \tilde{G}_{\chi}(s)-p\frac{\gamma\chi}{4}s^{-2}\tilde{G}_{\chi}(s).
\end{align}
Then,
\begin{align}
t \partial_t G_{\chi} = &-2s(1-s)\partial_s \tilde{G}_{\chi} +p\frac{\gamma\chi}{2} (1-s)\tilde{G}_{\chi},\\
(t^2+1)\partial_t^2 G_{\chi} = & 4s^2(1-s)\partial_s^2\tilde{G}_{\chi} - p\gamma\chi s(1-s)\partial_s \tilde{G}_{\chi} +p\gamma\chi(1-s)\tilde{G}_{\chi}\\ \nonumber
&+2s \left((p\frac{\gamma\chi}{2}-4)s + 3-p\frac{\gamma\chi}{2}\right)\partial_s \tilde{G}_{\chi} -p\frac{\gamma\chi}{2} \left((p\frac{\gamma\chi}{2}-4)s + 3-p\frac{\gamma\chi}{2}\right)\tilde{G}_{\chi}\\ \nonumber
=& 4s^2(1-s)\partial_s^2\tilde{G}_{\chi} +2s((p\gamma\chi-4)s+3-p\gamma\chi)\partial_s \tilde{G}_{\chi} + p\frac{\gamma\chi}{2}\left((2-p\frac{\gamma\chi}{2}) s +p\frac{\gamma\chi}{2} -1 \right)\tilde{G}_{\chi}.
\end{align}
This allows to transform the equation of $G_{\chi}$ into a hypergeometric equation of $\tilde{G}_{\chi}$,
\begin{align}
s(1-s)\partial_s^2 \tilde{G}_{\chi} +(C-(A+B+1)s)\partial_s \tilde{G}_{\chi} - AB\tilde{G}_{\chi} = 0,
\end{align}
with the parameters defined by
\begin{align}
A = -p\frac{\gamma \chi}{4}, B = 1+\chi(\chi - \alpha -p\frac{\gamma}{4}), C = \frac{3}{2}+\chi(\chi -\alpha - p\frac{\gamma}{2}).
\end{align}

\end{proof}

\subsection{Boundary three-point case}\label{sec_BPZ2}
Moving on to the equation used in Section \ref{sec_3pt}.
\begin{proposition}\label{BPZ_eq_3pt}
Let $\chi = \frac{\gamma}{2} \text{ or }\frac{2}{\gamma}$ and $q = \frac{1}{\gamma}(2 Q  - \beta_1 - \beta_2 - \beta_3 + \chi)$. In the parameter range,
\begin{equation}
 \beta_i < Q, \quad \mu_1 \in (0, \infty), \: \: \mu_2, \mu_3 \in - \overline{\H}, \quad  q < \frac{4}{\gamma^2} \wedge \min_i \frac{2}{\gamma}(Q - \beta_i), \quad t \in \mathbb{H},
\end{equation}
we define the function,
\begin{equation}
H_{\chi}(t)  = \mathbb{E} \left[ \left( \intr \frac{(t-x)^{\frac{\gamma\chi}{2}}}{|x|^{\frac{\gamma \beta_1 }{2}} |x-1|^{\frac{\gamma \beta_2 }{2}}}  g(x)^{\frac{\gamma^2}{8}(q-1)} e^{\frac{\gamma}{2} X(x)} d \mu(x) \right)^q \right].
\end{equation}
Then $H_{\chi}(t)$ obeys the hypergeometric equation,
\begin{equation}
t(1-t)\partial_t^2 H_{\chi} + (C-(A+B+1)t)\partial_t H_{\chi}-ABH_{\chi} = 0,
\end{equation}
with parameters:
\begin{equation}
A =-q\frac{\gamma\chi}{2}, B = -1+ \chi(\beta_1+\beta_2-2\chi + q\frac{\gamma}{2}), C = \chi(\beta_1 - \chi).
\end{equation}
Furthermore the exact same hypergeometric equation holds for the dual function,
\begin{equation}
\tilde{H}_{\chi}(t)  = \mathbb{E} \left[ \left( \intr \frac{(x-t)^{\frac{\gamma\chi}{2}}}{|x|^{\frac{\gamma \beta_1 }{2}} |x-1|^{\frac{\gamma \beta_2 }{2}}}  g(x)^{\frac{\gamma^2}{8}(q-1)} e^{\frac{\gamma}{2} X(x)} d {\mu}(x) \right)^q \right],
\end{equation}
this time in the range of parameters:
\begin{equation}
t \in -\H, \: \: \beta_i < Q, \: \:   \mu_1, \mu_2 \in - \overline{\H}, \: \: \mu_3 \in  (0,\infty), \quad \text{and} \quad   q < \frac{4}{\gamma^2} \wedge \min_i \frac{2}{\gamma}(Q - \beta_i).
\end{equation}
\end{proposition}

\begin{proof}
We denote for a small $\epsilon>0$,
\begin{equation}
R_{\epsilon} = \mathbb{R} \backslash \{(-\epsilon, \epsilon) \cup (1-\epsilon,1+\epsilon)\}.
\end{equation}
Consider
\begin{equation}
H_{\chi, \epsilon}(t)  =  \mathbb{E} \left[ \left( \int_{R_{\epsilon}} \frac{(t-x)^{\frac{\gamma\chi}{2}}}{|x|^{\frac{\gamma \beta_1 }{2}} |x-1|^{\frac{\gamma \beta_2 }{2}}}  g(x)^{\frac{\gamma^2}{8}(q-1)} e^{\frac{\gamma}{2} X(x)} d \mu(x) \right)^{q} \right].
\end{equation}
For simplicity, we introduce the notations,
\begin{equation}
V_{\epsilon}(x_1;t) = \mathbb{E} \left[ \left( \int_{R_{\epsilon}} \frac{(t-x)^{\frac{\gamma\chi}{2}}}{|x|^{\frac{\gamma \beta_1 }{2}} |x-1|^{\frac{\gamma \beta_2 }{2}}|x-x_1|^{\frac{\gamma^2}{2}}}  g(x)^{\frac{\gamma^2}{8}(q-2)} e^{\frac{\gamma}{2} X(x)} d \mu(x) \right)^{q-1} \right],
\end{equation}
\begin{align}
E_{0,\epsilon}^-(t) &= \mu_1\frac{(t+\epsilon)^{\frac{\gamma\chi}{2}}}{\epsilon^{\frac{\gamma \beta_1 }{2}} (1+\epsilon)^{\frac{\gamma \beta_2 }{2}}}V_{\epsilon}(-\epsilon;t), \quad \quad E_{0,\epsilon}^+(t) = \mu_2\frac{(t-\epsilon)^{\frac{\gamma\chi}{2}}}{\epsilon^{\frac{\gamma \beta_1 }{2}} (1-\epsilon)^{\frac{\gamma \beta_2 }{2}}}V_{\epsilon}(\epsilon;t),\\
E_{1,\epsilon}^{-}(t) &= \mu_2\frac{(t-1+\epsilon)^{\frac{\gamma\chi}{2}}}{(1-\epsilon)^{\frac{\gamma \beta_1 }{2}} \epsilon^{\frac{\gamma \beta_2 }{2}}}V_{\epsilon}(1-\epsilon;t), \quad E_{1,\epsilon}^{+}(t) = \mu_3\frac{(t-1-\epsilon)^{\frac{\gamma\chi}{2}}}{(1+\epsilon)^{\frac{\gamma \beta_1 }{2}} \epsilon^{\frac{\gamma \beta_2 }{2}}}V_{\epsilon}(1+\epsilon;t).
\end{align}

The proof follows the same step as the previous case, the only difference is that here we have additional boundary terms when performing integration by parts due to the presence of the insertions at $0$ and $1$. Similarly we compute,
\begin{align}
(2\chi+\frac{1}{\chi}-q\frac{\gamma}{2}-\beta_1-\beta_2)H_{\chi, \epsilon} = - \int_{R_{\epsilon}}d \mu(x_1) \left(\frac{\beta_1 t}{x_1}+\frac{\beta_2 (t-1)}{x_1-1}\right)\frac{(t-x_1)^{\frac{\gamma\chi}{2}}}{|x_1|^{\frac{\gamma \beta_1 }{2}} |x_1-1|^{\frac{\gamma \beta_2 }{2}}}V_{\epsilon}(x_1;t) \nonumber\\
+\frac{2}{\gamma} \left( -(t+\epsilon)E_{0,\epsilon}^-(t) + (t-\epsilon)E_{0,\epsilon}^+(t) - (t-1+\epsilon)E_{1,\epsilon}^-(t) + (t-1-\epsilon)E_{1,\epsilon}^+(t) \right),
\end{align}
\begin{align}
\partial_t H_{\chi, \epsilon} = -q\frac{\gamma}{2}\int_{R_{\epsilon}}d\mu(x_1) \left(\frac{\beta_1 }{x_1}+\frac{\beta_2}{x_1-1}\right)\frac{(t-x_1)^{\frac{\gamma\chi}{2}}}{|x_1|^{\frac{\gamma \beta_1 }{2}} |x_1-1|^{\frac{\gamma \beta_2 }{2}}}V_{\epsilon}(x_1;t) \nonumber\\
+ q\left( -E_{0,\epsilon}^-(t) + E_{0,\epsilon}^+(t) - E_{1,\epsilon}^-(t) + E_{1,\epsilon}^+(t) \right),
\end{align}
\begin{align}
\partial_t^2 H_{\chi, \epsilon} = -\chi\left(\frac{\beta_1}{t}+\frac{\beta_2}{t-1}\right)\partial_t H_{\chi, \epsilon}-q\frac{\gamma \chi^2}{2}\int_{R_{\epsilon}}d\mu(x_1) \left(\frac{\beta_1 }{tx_1}+\frac{\beta_2}{(t-1)(x_1-1)}\right)\frac{(t-x_1)^{\frac{\gamma\chi}{2}}}{|x_1|^{\frac{\gamma \beta_1 }{2}} |x_1-1|^{\frac{\gamma \beta_2 }{2}}}V_{\epsilon}(x_1;t) \nonumber\\
+ q\chi^2 \left( -\frac{1}{t+\epsilon} E_{0,\epsilon}^-(t) + \frac{1}{t-\epsilon}E_{0,\epsilon}^+(t) - \frac{1}{t-1+\epsilon}E_{1,\epsilon}^-(t) + \frac{1}{t-1-\epsilon}E_{1,\epsilon}^+(t) \right).
\end{align}
Then we have
\begin{align}
&t(1-t)\partial_t^2 H_{\chi, \epsilon} + (C-(A+B+1)t)\partial_t H_{\chi, \epsilon}-ABH_{\chi, \epsilon} \nonumber\\
& \quad =  q\chi^2 \left( \frac{\epsilon(1+\epsilon)}{t+\epsilon} E_{0,\epsilon}^-(t) + \frac{\epsilon(1-\epsilon)}{t-\epsilon}E_{0,\epsilon}^+(t) - \frac{\epsilon(1-\epsilon)}{t-1+\epsilon}E_{1,\epsilon}^-(t) - \frac{\epsilon(1+\epsilon)}{t-1-\epsilon}E_{1,\epsilon}^+(t) \right),
\end{align}
with the parameters given by:
\begin{equation}
A = -q\frac{\gamma\chi}{2}, B = -1+ \chi(\beta_1+\beta_2-2\chi + q\frac{\gamma}{2}), C = \chi(\beta_1 - \chi).
\end{equation}
To complete the proof the only thing left is to argue that the boundary terms $\epsilon E_{\cdot, \epsilon}^{\pm}(t)$ converge to $0$ as $\epsilon $ goes to $0$ locally uniformly in $t$. This has been done in \cite{interval}, in the arguments detailed below equation (3.8) of \cite{interval}. Finally the exact same argument works for $\tilde{H}_{\chi}(t)$.
\end{proof}

\section{Appendix}\label{sec_appendix}

\subsection{Useful facts in probability}\label{sec:useful_facts}
We start by explaining how to construct our GFF $X$ from the standard Neumann boundary GFF $X_\mathbb{D}$ on $\mathbb{D}$, also called free boundary GFF. This field has a covariance given by, for $x,y \in \mathbb{D}$,
\begin{equation}\label{GFF_D}
\E[X_\mathbb{D}(x)X_\mathbb{D}(y)] = \ln \frac{1}{|x-y||1- x\bar{y}|}.
\end{equation}
The field $X_\mathbb{D}$ has zero average on the unit circle. Notice also that if $x$ or $y$ is on the unit circle, the covariance \eqref{GFF_D} reduces to $-2 \ln |x-y|$. One can then conformally map the disk $\mathbb{D}$ equipped with the Euclidean metric to the upper-half plane $\mathbb{H}$ equipped with the metric $\hat{g}(x) = \frac{4}{|x+i|^4}$. By this map from the field $X_\mathbb{D}$ we obtain the field $X_{\hat{g}}$ defined on $\mathbb{H}$ which has covariance
\begin{equation}\label{GFF_hg}
 \E[X_{\hat{g}}(x)X_{\hat{g}}(y)] = \ln \frac{1}{|x-y||x-\bar{y}|} -\frac{1}{2}\ln \hat{g}(x) - \frac{1}{2} \ln \hat{g}(y),
 \end{equation}
and zero average on $\mathbb{R}$ in the metric $\hat{g}$. Finally the field $X$ can be obtained from the field $X_{\hat{g}}$ by simply setting:
\begin{equation}
X(x) = X_{\hat{g}}(x) - \frac{1}{\pi} \int_0^{\pi} X_{\hat{g}}(e^{i \theta}) d \theta.
\end{equation}

Next we state a result of finiteness of GMC moments covering all situations encountered in the main text. This follows from \cite[Corollary 3.10]{Disk} except that we need to consider complex $\mu_i$.
\begin{proposition}[Finiteness of moments of GMC]\label{lem:GMC-moment}
	Fix $\gamma \in (0, 2)$. The following claims hold.
	\begin{itemize}
		\item For $\beta < Q $ and $\frac{\gamma}{2} - \alpha < \frac{\beta}{2} < \alpha $ we have 
		\begin{equation}
		  \mathbb{E} \left[ \left(\int_{\mathbb{R}} \frac{g(x)^{\frac{\gamma}{4}(\frac{2}{\gamma}-\alpha-\frac{\beta}{2})}}{|x-i|^{\gamma\alpha  } }  e^{\frac{\gamma}{2} X(x) }  d x \right)^{\frac{2}{\gamma}(Q-\alpha-\frac{\beta}{2})} \right] < \infty.
		  \end{equation}
		\item  For $(\mu_i)_{i=1,2,3}$ satisfying Definition \ref{half-space}, $\beta_i < Q$,  $  {\frac{1}{\gamma}(2 Q  - \sum_{i=1}^3 \beta_i)} < \frac{4}{\gamma^2} \wedge \min_i \frac{2}{\gamma}(Q - \beta_i)$ we have

		\begin{equation}
		 \mathbb{E} \left[ \left| \intr \frac{g(x)^{\frac{\gamma}{8}(\frac{4}{\gamma}- \sum_{i=1}^3 \beta_i )} }{|x|^{\frac{\gamma \beta_1 }{2}}|x-1|^{\frac{\gamma \beta_2 }{2}}}  e^{\frac{\gamma}{2} X(x)} d \mu(x) \right|^{\frac{1}{\gamma}(2 Q  - \sum_{i=1}^3 \beta_i)} \right] < +\infty.
		\end{equation}
		\item For   $\chi \in \{\frac{\gamma}{2}, \frac{2}{\gamma}\}$,  
		 $p = \frac{2}{\gamma}(Q-\alpha-\frac{\beta}{2}+\frac{\chi}{2})$, $\beta < Q$, $ p < \frac{4}{\gamma^2} \wedge \frac{2}{\gamma}(Q - \beta)$, $t \in \mathbb{H}$
 we have 
		\begin{equation}
		\mathbb{E} \left[ \left| \int_{\mathbb{R}} \frac{(t-x)^{\frac{\gamma\chi}{2}}}{|x-i|^{\gamma\alpha  } }   g(x)^{\frac{\gamma^2}{8}(p-1)}  e^{\frac{\gamma}{2} X(x) } d x \right|^{p} \right]   <\infty.
		\end{equation}
		\item For  $\chi \in \{\frac{\gamma}{2}, \frac{2}{\gamma}\}$,  
		 $q = \frac{1}{\gamma}(2 Q  - \beta_1 - \beta_2 - \beta_3 + \chi)$,  $\beta_i < Q$,   $ \mu_1 \in (0,\infty)$,  $\mu_2, \mu_3 \in - \overline{\H}$, $  q < \frac{4}{\gamma^2} \wedge \min_i \frac{2}{\gamma}(Q - \beta_i)$, $t \in \mathbb{H}$
 we have 

		\begin{equation}
		\mathbb{E} \left[ \left| \intr \frac{(t-x)^{\frac{\gamma\chi}{2}}}{|x|^{\frac{\gamma \beta_1 }{2}} |x-1|^{\frac{\gamma \beta_2 }{2}}}  g(x)^{\frac{\gamma^2}{8}(q-1)} e^{\frac{\gamma}{2} X(x)} d \mu(x) \right|^{q} \right]  <\infty.
		\end{equation}
	\end{itemize}
\end{proposition}
\begin{proof}

	For the first claim, since a positive function is integrated against the GMC measure, we are in the classical
	case of the existence of moments of GMC with an insertion of weight $\beta$ (here the insertion is placed at infinity, but this changes nothing to the result). Following \cite[Lemma~3.10]{Sphere},
	adapted to the case of one-dimensional GMC, the condition is thus $ \beta <Q$ and
	$\frac{2}{\gamma}(Q-\alpha-\frac{\beta}{2}) < \frac{4}{\gamma^2} \wedge \frac{2}{\gamma} (Q -\beta)$. One can check this last condition translates into $\frac{\gamma}{2} - \alpha < \frac{\beta}{2} < \alpha $. See also \cite[Corollary 3.10]{Disk}.
	
	The second claim does not fit exactly into the framework of \cite{Sphere} since the $\mu_i$ can be complex and therefore we have a complex valued quantity integrated against the GMC. Let $q_0 := \frac{1}{\gamma}(2 Q  - \sum_{i=1}^3 \beta_i)$.
	 For the case of positive moments
	$q_0 \geq 0$ one can simply use the bound
	\begin{align}
		\E\left[ \left|  \intr \frac{g(x)^{\frac{\gamma}{8}(\frac{4}{\gamma}- \sum_{i=1}^3 \beta_i )} }{|x|^{\frac{\gamma \beta_1 }{2}}|x-1|^{\frac{\gamma \beta_2 }{2}}}  e^{\frac{\gamma}{2} X(x)} d \mu(x)  \right|^{q_0}\right]  \leq M \E\left[\left(\intr \frac{g(x)^{\frac{\gamma}{8}(\frac{4}{\gamma}- \sum_{i=1}^3 \beta_i )} }{|x|^{\frac{\gamma \beta_1 }{2}}|x-1|^{\frac{\gamma \beta_2 }{2}}}  e^{\frac{\gamma}{2} X(x)} dx  \right)^{q_0}\right],
	\end{align}
	which is valid for $M = \max_i |\mu_i| > 0$.  The claim then reduces to the first case. 
	
	Now for negative moments corresponding to $q_0< 0$, this is precisely where we are going to use the half-space condition of Definition \ref{half-space} on the $\mu_i$ parameters.
	The condition implies $\mu_1, \mu_2, \mu_3$ are contained in a half-space $\mathcal{H}$. Let $v_1 \in \mathbb{C}$ be a normal
	vector contained in the half-space, and let $v_2 \in \mathbb{C}$ be perpendicular to $v_1$. For each $i =1,2,3$, write $\mu_i = \lambda_i v_1+ \lambda_i' v_2$ where $\lambda_i  \geq 0$ and $\lambda_i' \in \mathbb{R}$. By Definition \ref{half-space} at least one $\lambda_i$ is non zero. From this one can deduce the upper bound
	\begin{align}
		\E\left[ \left|  \intr \frac{g(x)^{\frac{\gamma}{8}(\frac{4}{\gamma}- \sum_{i=1}^3 \beta_i )} }{|x|^{\frac{\gamma \beta_1 }{2}}|x-1|^{\frac{\gamma \beta_2 }{2}}}  e^{\frac{\gamma}{2} X(x)} d \mu(x)  \right|^{q_0}\right] 
		\leq M' \E\left[ \left(  \intr \frac{g(x)^{\frac{\gamma}{8}(\frac{4}{\gamma}- \sum_{i=1}^3 \beta_i )} }{|x|^{\frac{\gamma \beta_1 }{2}}|x-1|^{\frac{\gamma \beta_2 }{2}}}  e^{\frac{\gamma}{2} X(x)} d \lambda(x)  \right)^{q_0}\right] 
	\end{align}
	for some $M' >0$ and where $d \lambda(x)$ is defined in the same as $d \mu(x)$ but with $\mu_i$ replaced by $\lambda_i$ . We can now apply again the first case to show finiteness. Lastly the third case and fourth cases are treated similarly to the second one.
\end{proof}

\begin{remark}
Let us discuss what happens in the above proposition if we are in the limiting case of the half-space condition of Definition \ref{half-space}, meaning the $\mu_i$ are lying on the boundary of the half-space. For instance assume $\mu_1, \mu_2 $ are real and positive and $\mu_3$ is $e^{i \pi}$ times a positive number. In this case the second claim of the above proposition will remain true if $q_0 := \frac{1}{\gamma}(2 Q  - \sum_{i=1}^3 \beta_i) \geq 0$, the proof being exactly the same. On the other hand if $q_0 <0$ then there are cancellations that appear and the claim is no longer true. It seems the condition for finiteness then becomes $q_0 >-1$ but this seems a little technical to show and we will not require it in the present paper.
\end{remark}

Finally we recall some theorems in probability that we will use without further justification. In the following $D$ is a compact subset of $\mathbb{R}^d$. 
\begin{theorem}[Girsanov theorem]\label{girsanov}
Let $(Z(x))_{x\in D}$ be a continuous centered Gaussian process and $Z$ a Gaussian variable which belongs to the $L^2$ closure of the vector space spanned by $(Z(x))_{x\in D}$. Let $F$ be a real continuous bounded function from $\mathcal{C}(D,\mathbb{R})$ to $\mathbb{R}$. Then we have the following identity:
\begin{equation}\label{equation Girsanov}
\E\left[e^{Z-\frac{\E[Z^2]}{2}}F((Z(x))_{x\in D})\right] = \E[F((Z(x)+\E[Z(x)Z])_{x\in D})].
\end{equation}
\end{theorem}
When applied to our case, although the log-correlated field $X$ is not a continuous Gaussian process, we can still make the arguments rigorous by using a regularization procedure (see \cite[Appendix A.1]{interval} for a more detailed explanation). Next we recall Kahane's inequality:
\begin{theorem}[Kahane's inequality]\label{theo Kahane inequality}
Let $(Z_0(x))_{x\in D}$, $(Z_1(x))_{x\in D}$ be two continuous centered Gaussian processes such that for all $x,y\in D$:
\begin{equation}
\left|\E[Z_0(x)Z_0(y)] -\E[Z_1(x)Z_1(y)]\right| \le C.
\end{equation}
Define for $u\in [0,1]$:
\begin{align}
Z_u = \sqrt{1-u} Z_0 + \sqrt{u} Z_1, \quad W_u = \int_D e^{Z_u(x)-\frac{1}{2}\E[Z_u(x)^2]} \sigma(dx).
\end{align}
Then for all smooth function $F$ with at most polynomial growth at infinity, and $\sigma$ a complex Radon measure over $D$,
\begin{equation}
\left| \E\left[F\left(\int_D e^{Z_0(x)-\frac{1}{2}\E[Z_0(x)^2]} \sigma(dx)\right)\right]-\E\left[F\left(\int_D e^{Z_1(x)-\frac{1}{2}\E[Z_1(x)^2]} \sigma(dx)\right)\right]\right| \le \sup_{u \in [0,1]}\frac{C}{2}\E[|W_u|^2 |F''(W_u)|].
\end{equation}
\end{theorem}
The same remark as the one below Theorem \ref{girsanov} is valid to justify one can use this inequality in the case where $Z_0$ and $Z_1$ are log-correlated fields. Finally we provide the Williams decomposition theorem, see for instance \cite{Williams} for a reference.
\begin{theorem}\label{theo Williams}
Let $(B_s-vs)_{s\ge 0}$ be a Brownian motion with negative drift, i.e. $v>0$ and let $M = \sup_{s\ge 0}(B_s-vs)$. Then conditionally on $M$ the law of the path $(B_s-vs)_{s\ge 0}$ is given by the joining of two independent paths:\\
1) A Brownian motion $(B^1_s+vs)_{0\le s \le \tau_M}$ with positive drift $v$ run until its hitting time $\tau_M$ of $M$.\\
2) $(M+B^2_t-vt)_{t\ge 0}$ where $(B^2_t-vt)_{t\ge 0}$ is a Brownian motion with negative drift conditioned to stay negative.
\par Moreover, one has the following time reversal property for all $C>0$ (where $\tau_C$ denotes the hitting time of $C$),
\begin{equation}\label{equation 4.3}
(B^1_{\tau_C-s}+v(\tau_C-s)-C)_{0\le s\le \tau_C} \overset{\text{law}}{=} (\tilde{B}_s-vs)_{0\le s \le L_{-C}},
\end{equation}
where $(\tilde{B}_s-vs)_{s\ge 0}$ is a Brownian motion with drift $-v$ conditioned to stay negative and $L_{-C}$ is the last time $(\tilde{B}_s-vs)_{s\ge 0}$ hits $-C$.
\end{theorem}

\subsection{Technical estimates on GMC}

We repeat here several proofs found in \cite{DOZZ2, interval} that must be adapted because our objects are complex valued.

\subsubsection{OPE with reflection}

We want to compute the asymptotic expansion of the functions $\tilde{G}_{\chi}$ and $H_{\chi}$ in the case where there will be reflection. This has been performed in the previous works \cite{DOZZ2,interval} but it is not straightforward to adapt the proofs as we are working with complex valued quantities so there are many inequalities that need to be adapted. We will treat separately the cases where $\chi = \frac{\gamma}{2}$ and $ \chi = \frac{2}{\gamma}$. Starting with the case where $\chi = \frac{2}{\gamma}$:

\begin{lemma}\label{lem_reflection_ope2} (OPE with reflection for $\chi = \frac{2}{\gamma}$) Recall $p = \frac{2}{\gamma}(Q-\alpha-\frac{\beta}{2}+\frac{1}{\gamma})$ and consider $s \in (-1,0)$. Recall also the functions $\tilde{G}_{\frac{2}{\gamma}}$ and $H_{\frac{2}{\gamma}}$ given by \eqref{eq:def_G_tilde} and \eqref{eq:def_H_chi} for $\chi = \frac{2}{\gamma}$. There exists a small parameter $\beta_0>0$ such that for $\beta \in (Q -\beta_0, Q)$ and $\alpha$ such that $  p < \frac{4}{\gamma^2} \wedge \frac{2}{\gamma}(Q - \beta)$, the following asymptotic expansion holds:
\begin{align}
 \tilde{G}_{\frac{2}{\gamma}}(s) - \tilde{G}_{\frac{2}{\gamma}}(0) = -s^{\frac{1}{2} + \frac{2}{\gamma^2} - \frac{ \beta}{\gamma}}\frac{\Gamma(1-\frac{2}{\gamma}(Q-\beta)) \Gamma(\frac{2}{\gamma}(Q-\beta) -p)}{\Gamma(-p)}\overline{R}(\beta, 1, -1) \overline{G}(\alpha, 2Q-\beta -\frac{2}{\gamma})  + o(|s|^{\frac{1}{2} + \frac{2}{\gamma^2} - \frac{ \beta}{\gamma}}).
\end{align}
Similarly, recall $q = \frac{1}{\gamma}(2 Q  - \beta_1 - \beta_2 - \beta_3 + \frac{2}{\gamma})$ and consider $t \in (0,1)$. Then in the following parameter range,
\begin{equation}
\beta_1 \in (Q - \beta_0,Q), \quad q < \frac{4}{\gamma^2} \wedge \min_i \frac{2}{\gamma}(Q - \beta_i), \quad \mu_1 \in (0,+\infty), \quad \mu_2, \mu_3 \in(-\infty,0),
\end{equation}
the following asymptotic also holds:
\begin{align}
H_{\frac{2}{\gamma}}(it) - H_{\frac{2}{\gamma}}(0) &= - (it)^{1 - \frac{ 2 \beta_1}{\gamma} + \frac{4}{\gamma^2}}  \frac{\Gamma(1 - \frac{2}{\gamma}(Q - \beta_1)) \Gamma(\frac{2}{\gamma}(Q -\beta_1) -q ) }{\Gamma(-q)} \overline{R}(\beta_1, \mu_1, \mu_2) \overline{H}^{(2 Q - \beta_1 - \frac{2}{\gamma} , \beta_2 , \beta_3)}_{( \mu_1, -\mu_2,  -\mu_3)}  \\ \nonumber
&+ o(|t|^{1 - \frac{ 2 \beta_1}{\gamma} + \frac{4}{\gamma^2}}).
\end{align}
\end{lemma}

\begin{proof} We will prove only the case of $H_{\frac{2}{\gamma}}$, the case of $\tilde{G}_{\frac{2}{\gamma}}$ can be treated in a similar fashion. For a Borel set $I \subseteq \mathbb{R}$, we introduce the notation,
\begin{equation}
K_{I}(it) : =  \int_I \frac{it-x}{|x|^{\frac{\gamma \beta_1 }{2}} |x-1|^{\frac{\gamma \beta_2 }{2}}}  g(x)^{\frac{\gamma^2}{8}(q-1)} e^{\frac{\gamma}{2} X(x)} d \mu(x),
\end{equation}
where as always $d\mu(x) = \mu_1\mathbf{1}_{(-\infty,0)}(x)dx + \mu_2 \mathbf{1}_{(0,1)}(x)dx + \mu_3 \mathbf{1}_{(1,\infty)}(x)dx$. In the following it is convenient to use $d|\mu|(x)$ to denote the measure $\mu_1\mathbf{1}_{(-\infty,0)}(x)dx - \mu_2 \mathbf{1}_{(0,1)}(x)dx- \mu_3 \mathbf{1}_{(1,\infty)}(x)dx$ which is a positive measure thanks to our choice $\mu_1 \in (0,+\infty)$,  $\mu_2, \mu_3 \in(-\infty,0)$. The signs of the parameters $\mu_i$ allows to separate $K_{I}(it)$ into a positive real part $K_{I}(0)$ and an imaginary part. This remark is used to bound $|K_{I}(it)|^{q-1}$ by $|K_{I}(0)|^{q-1}$ and in several other similar cases (remark that necessarily $q-1<0$). Now we want to study the asymptotic of,
\begin{equation}
 \E[K_{\mathbb{R}}(it)^{q}]-\E[K_{\mathbb{R}}(0)^{q}] =: T_1+T_2,
\end{equation}
where we defined:
\begin{equation}
T_1: = \E[K_{(-t, t)^c}(it)^{q}]-\E[K_{\mathbb{R}}(0)^{q}], \quad T_2:= \E[K_{\mathbb{R}}(it)^{q}]- \E[K_{(-t, t)^c}(it)^{q}].
\end{equation}

\noindent
$\Diamond $ First we consider $T_1$. The goal is to show that $T_1 =o(|t|^{1 - \frac{ 2 \beta_1}{\gamma} + \frac{4}{\gamma^2}}) = o(|t|^{\frac{2}{\gamma}(Q-\beta_1)})$. By interpolation,
\begin{align}
|T_1| \le &|q| \int_0^1 du \E\left[|K_{(-t, t)^c}(it)-K_{\mathbb{R}}(0)||uK_{(-t, t)^c}(it)+(1-u)K_{\mathbb{R}}(0)|^{q-1}\right] \\ \nonumber
\le &|q| \E\left[\left(|K_{(-t, t)^c}(it)-K_{(-t,t)^c}(0)| + |K_{(-t,t)^c}(0) - K_{\mathbb{R}}(0)| \right)|K_{(-t,t)^c}(0)|^{q-1}\right]\\ \nonumber
=&|q| ( A_1 + A_2),
\end{align}
with:
\begin{equation}
A_1 := \E\left[|K_{(-t, t)^c}(it)-K_{(-t,t)^c}(0)| |K_{(-t,t)^c}(0)|^{q-1}\right], \quad A_2 := \E\left[ |K_{(-t,t)^c}(0) - K_{\mathbb{R}}(0)||K_{(-t,t)^c}(0)|^{q-1}\right].
\end{equation}
We have
\begin{align}
A_1 &\le t\int_{(-t,t)^c}d|\mu|(x_1)\, \frac{1}{|x_1|^{\frac{\gamma \beta_1 }{2}} |x_1-1|^{\frac{\gamma \beta_2 }{2}}} \E\left[\left(\int_{(-t,t)^c}\frac{g(x)^{\frac{\gamma^2}{8}(q-2)}e^{\frac{\gamma}{2} X(x)} d |\mu|(x)}{|x|^{\frac{\gamma\beta_1 }{2}-1} |x-1|^{\frac{\gamma \beta_2 }{2}}|x-x_1|^{\frac{\gamma^2}{2}}} \right)^{q-1}\right]\\ \nonumber
&\le t\int_{\mathbb{R}}d|\mu|(x_1)\, \frac{2|x_1|\mathbf{1}_{(-\frac{1}{2},\frac{1}{2})^c} + \mathbf{1}_{(-\frac{1}{2},-t)\cup (t,\frac{1}{2})}}{|x_1|^{\frac{\gamma \beta_1 }{2}} |x_1-1|^{\frac{\gamma \beta_2 }{2}}} \E\left[\left(\int_{(-t,t)^c}\frac{g(x)^{\frac{\gamma^2}{8}(q-2)}e^{\frac{\gamma}{2} X(x)} d |\mu|(x)}{|x|^{\frac{\gamma\beta_1 }{2}-1} |x-1|^{\frac{\gamma \beta_2 }{2}}|x-x_1|^{\frac{\gamma^2}{2}}} \right)^{q-1}\right] \\ \nonumber
&\le 2t \E\left[\left(\int_{(-t,t)^c}\frac{g(x)^{\frac{\gamma^2}{8}(q-1)}e^{\frac{\gamma}{2} X(x)} d |\mu|(x)}{|x|^{\frac{\gamma\beta_1 }{2}-1} |x-1|^{\frac{\gamma \beta_2 }{2}}} \right)^{q}\right] + O(t^{2-\frac{\gamma\beta_1}{2}}) = O(t^{2-\frac{\gamma\beta_1}{2}}).
\end{align}
In the last equality we have ignored the first term since it is a $O(t)$ and we will take $\beta_1>\frac{2}{\gamma}$. On the other hand,
\begin{align}
A_2 \le c_1 \int_{(-t,t)}d|\mu|(x_1)\, \frac{1}{|x_1|^{\frac{\gamma \beta_1 }{2}-1} |x_1-1|^{\frac{\gamma \beta_2 }{2}}} = O(t^{2-\frac{\gamma\beta_1}{2}}),
\end{align}
for some constant $c_1 > 0$. When $\beta_1 > \frac{\frac{4}{\gamma^2}-1}{\frac{\gamma}{2} - \frac{2}{\gamma}}$ is satisfied, i.e., $\beta_0 < \frac{1-\frac{\gamma^2}{4}}{\frac{\gamma}{2}-\frac{2}{\gamma}}$, we have $O(t^{2-\frac{\gamma\beta_1}{2}}) = o(t^{\frac{2}{\gamma}(Q-\beta_1)})$. This proves that
\begin{equation}
T_1 = o(t^{\frac{2}{\gamma}(Q-\beta_1)}).
\end{equation}

$\Diamond$ Now we focus on $T_2$. The goal is to restrict $K$ to $(-\infty,-t)\cup(-t^{1+h}, t^{1+h})\cup(t,\infty)$, with $h>0$ a small positive constant to be fixed, and then the GMC measures on the three disjoint parts will be weakly correlated. We have by interpolation and by dropping the imaginary part,
\begin{align}
&\left|\E[K_{\mathbb{R}}(it)^{q}]- \E[K_{(-\infty,-t)\cup(-t^{1+h}, t^{1+h})\cup(t,\infty)}(it)^{q}] \right|  \\ \nonumber
&\le |q| \int_0^1 du \mathbb{E}\left[ |K_{ (-t, -t^{1 + h} ) \cup (t^{1 + h},t )}(it)|  | u K_{\mathbb{R}}(0) + (1 - u) K_{(-\infty,-t)\cup(-t^{1+h}, t^{1+h})\cup(t,\infty)}(0)  |^{q-1} \right]\\ \nonumber
&\le  c_2|q| \int_{(-t, -t^{1 + h} ) \cup (t^{1 + h},t )}d|\mu|(x_1)\, \frac{t+|x_1|}{|x_1|^{\frac{\gamma \beta_1 }{2}} |x_1-1|^{\frac{\gamma \beta_2 }{2}}} \\ \nonumber
&= O(t^{1+(1+h)(1-\frac{\gamma\beta_1}{2})}),
\end{align}
for some constant $c_2>0$. By taking $h$ satisfying the condition,
\begin{equation}
h < \frac{1+(\frac{2}{\gamma}-\frac{\gamma}{2})\beta_1 -\frac{4}{\gamma^2}}{\frac{\gamma\beta_1}{2}-1}, 
\end{equation}
we have:
\begin{equation}
\E[K_{\mathbb{R}}(it)^{q}]- \E[K_{(-\infty,-t)\cup(-t^{1+h}, t^{1+h})\cup(t,\infty)}(it)^{q}] = o(t^{\frac{2}{\gamma}(Q-\beta_1)}).
\end{equation}
It remains to evaluate $\E[K_{(-\infty,-t)\cup(-t^{1+h}, t^{1+h})\cup(t,\infty)}(it)^{q}]-\E[K_{(-t,t)^c}(it)^{q}]$. We now introduce the radial decomposition of the field $X$,
\begin{equation}\label{equation radial decomposition}
X(x) = B_{-2\ln|x|} + Y(x),
\end{equation}
where $B$, $Y$ are independent Gaussian processes with $(B_{s})_{s \in \mathbb{R}}$ a Brownian motion starting from $0$ for $s \geq 0$, $B_s= 0$ when $s<0$, and $Y$ is a centered Gaussian process with covariance,
\begin{equation}\label{equation correlation Y}
\E[Y(x)Y(y)] = 
\begin{cases}
2\ln \frac{|x|\lor |y|}{|x-y|}, \quad & |x|,|y| \le 1,\\
2\ln \frac{1}{|x-y|} -\frac{1}{2}\ln g(x) -\frac{1}{2}\ln g(y), \quad &\text{else.}
\end{cases}
\end{equation}
One can wonder why the process $Y$ with the above covariance is well-defined. To construct $Y$, starting from $X$ set:
\begin{equation}
Y(x)=
\begin{cases}
X(x) - \frac{1}{\pi} \int_0^{\pi} X(|x| e^{i \theta}) d \theta, \quad & |x| \le 1,\\
X(x), \quad & |x| \geq 1.
\end{cases}
\end{equation}
Now with this decomposition one can write:
\begin{align}
K_{I}(it) =  \int_{I} \frac{it-x}{|x|^{\frac{\gamma \beta_1 }{2}-\frac{\gamma^2}{4}} |x-1|^{\frac{\gamma \beta_2 }{2}}}  g(x)^{\frac{\gamma^2}{8}(q-1)} e^{\frac{\gamma}{2} B_{-2\ln|x|}} e^{\frac{\gamma}{2}Y(x)} d\mu(x).
\end{align}
From \eqref{equation correlation Y}, we deduce that for $|x'| \le t^{1+h}$ and $|x| \ge t$,
\begin{equation}
\left|\E[Y(x)Y(x')]\right| = 
\left|2\ln \left|1-\frac{x'}{x}\right| \right| \le 4t^h,
\end{equation}
where we used the inequality $|\ln |1-x|| \le 2|x|$ for $x\in [-\frac{1}{2},\frac{1}{2}]$. Define the processes,
\begin{align}
P(x) &:= Y(x)\mathbf{1}_{|x| \le t^{1+h}} + Y(x)\mathbf{1}_{|x| \ge t},\\
\tilde{P}(x) &:= \tilde{Y}(x)\mathbf{1}_{|x| \le t^{1+h}} + Y(x)\mathbf{1}_{|x| \ge t},
\end{align}
where $\tilde{Y}$ is an independent copy of $Y$. Then we have the inequality over the covariance:
\begin{equation}
\left|\E[P(x)P(y)] - \E[\tilde{P}(x) \tilde{P}(y)]\right| \le 4t^h.
\end{equation}
Consider now for $u\in [0,1]$:
\begin{align}
P_u(x) &= \sqrt{1-u} P(x) + \sqrt{u} \tilde{P}(x),\\
K_{I}(it,u) &=  \int_{I} \frac{it-x}{|x|^{\frac{\gamma \beta_1 }{2}-\frac{\gamma^2}{4}} |x-1|^{\frac{\gamma \beta_2 }{2}}}  g(x)^{\frac{\gamma^2}{8}(q-1)} e^{\frac{\gamma}{2} B_{-2\ln|x|}} e^{\frac{\gamma}{2}P_u(x)} d\mu(x).
\end{align}
By applying Kahane's inequality of Theorem \ref{theo Kahane inequality}, 
\begin{align}
&\left|\E\left[K_{(-\infty,-t)\cup(-t^{1+h}, t^{1+h})\cup(t,\infty)}(it)^{q}\right] - \E\left[\left(K_{(-\infty,-t)\cup(t,\infty)}(it)+\tilde{K}_{(-t^{1+h}, t^{1+h})}(it)\right)^{q}\right]\right| \\
&\le  2|q(q-1)|t^h  \sup_{u\in [0,1]}\E\left[\left| K_{I}(it,u) \right|^{q}\right]\nonumber\\ \nonumber
&\le c_3\, t^h,
\end{align}
for some constant $c_3 >0$, and where in $\tilde{K}_{(-t^{1+h}, t^{1+h})}(it)$ we simply use the field $\tilde{Y}$ instead of $Y$. When $h > \frac{2}{\gamma}(Q-\beta_1)$, we can bound the previous term by $o(t^{\frac{2}{\gamma}(Q-\beta_1)})$.

Consider now the change of variable $x = t^{1+h}e^{-s/2}$ for the field $\tilde{K}_{(-t^{1+h}, t^{1+h})}(it)$. By the Markov property of the Brownian motion and stationarity of 
\begin{equation}
d\mu_{\tilde{Y}}(s): = \mu_1 e^{ \frac{\gamma}{2} \tilde{Y}(-e^{-s/2}) }ds + \mu_2 e^{\frac{\gamma}{2} \tilde{Y}(e^{-s/2}) }ds,
\end{equation}
we have
\begin{equation}
\tilde{K}_{(-t^{1+h}, t^{1+h})}(it) = \frac{1}{2}i t^{1+(1+h)(1-\frac{\gamma\beta_1}{2}+\frac{\gamma^2}{4})}e^{\frac{\gamma}{2}B_{2(1+h)\ln (1/t)}}\int_{0}^{\infty} \frac{(1 + i t^h e^{-s/2})}{|t^{1+h}e^{-s/2}-1|^{\frac{\gamma\beta_2}{2}}}  e^{ \frac{\gamma}{2}(\tilde{B}_s - \frac{s}{2}(Q-\beta_1) ) } d\mu_{\tilde{Y}}(s),
\end{equation}
with $\tilde{B}$ an independent Brownian motion. We denote
\begin{equation}
\sigma_t:= t^{1+(1+h)(1-\frac{\gamma\beta_1}{2}+\frac{\gamma^2}{4})}e^{\frac{\gamma}{2}B_{2(1+h)\ln (1/t)}},\quad V: =\frac{1}{2}\int_{0}^{\infty} e^{ \frac{\gamma}{2}(\tilde{B}_s - \frac{s}{2}(Q-\beta_1) ) } d\mu_{\tilde{Y}}(s).
\end{equation}
By interpolation, we can prove that for some constant $c_4>0$:
\begin{align}
&\left|\E[(K_{(-t,t)^c}(it) + \tilde{K}_{(-t^{1+h}, t^{1+h})}(it))^q] - \E\left[\left(K_{(-t,t)^c}(it) + i\sigma_t V\right)^q\right]\right| \\ \nonumber
& \le  c_4|q|t^{1+h+(1+h)(1-\frac{\gamma\beta_1}{2}+\frac{\gamma^2}{4})}\E\left[e^{\frac{\gamma}{2}B_{2(1+h)\ln (1/t)}}\int_{0}^{\infty} e^{ \frac{\gamma}{2}(\tilde{B}_s - \frac{s}{2}(Q-\beta_1) ) } d\mu_{\tilde{Y}}(s)|K_{(1,2)}(0)|^{q-1}\right].
\end{align}
Since $B_{2(1+h)\ln (1/t)}$, $(\tilde{B}_s)_{s\ge 0}$, $(\tilde{Y}(x))_{|x|\le 1}$, and $K_{(1,2)}(0)$ are independent, we can easily bound the last term by, for some $c_5>0$,
\begin{align}
c_5\,t^{(1+h)(2-\frac{\gamma\beta_1}{2})} = o(t^{\frac{2}{\gamma}(Q-\beta_1)}).
\end{align}
By the Williams path decomposition of Theorem \ref{theo Williams} we can write,
\begin{equation}
V =  e^{\frac{\gamma}{2}M}\frac{1}{2}\int_{-L_M}^{\infty}  e^{ \frac{\gamma}{2}\mathcal{B}^{\frac{Q-\beta_1}{2}}_s} \mu_{\tilde{Y}}(ds),
\end{equation}
where $ M = \sup_{s \geqslant 0} (\tilde{B}_s -\frac{Q-\beta_1}{2} s) $ and $L_M$ is the last time $\left( \mathcal{B}^{\frac{Q-\beta_1}{2}}_{-s} \right)_{s\ge 0} $ hits $-M$. Recall that the law of $M$ is known, for $v\ge 1$,
\begin{equation}\label{equation 5.18}
\mathbb{P}( e^{\frac{\gamma}{2} M} >v ) = \frac{1}{v^{ \frac{2}{\gamma}(Q-\beta_1)}}.
\end{equation}
For simplicity, we introduce the notation:
\begin{equation}\label{equation rho}
\rho(\beta_1) :=\frac{1}{2}\int_{-\infty}^{\infty}  e^{ \frac{\gamma}{2}\mathcal{B}^{\frac{Q-\beta_1}{2}}_s} \mu_{\tilde{Y}}(ds).
\end{equation}
Again by interpolation and then independence we can show that
\begin{align}
&\left|\E\left[\left(K_{(-t,t)^c}(it) + i\sigma_t V\right)^q\right] - \E\left[\left(K_{(-t,t)^c}(it) + i\sigma_t e^{\frac{\gamma}{2}M} \rho(\beta_1)\right)^q\right]\right|\\ \nonumber
&\le \frac{1}{2}|q|t^{1+(1+h)(1-\frac{\gamma\beta_1}{2}+\frac{\gamma^2}{4})}\E\left[e^{\frac{\gamma}{2}B_{2(1+h)\ln (1/t)}}\int_{-\infty}^{0} e^{ \frac{\gamma}{2}\mathcal{B}^{\frac{Q-\beta_1}{2}}_s} \mu_{\tilde{Y}}(ds)\left|K_{(1,2)}(0)\right|^{q-1}\right]\\ \nonumber
 &= O(t^{1+(1+h)(1-\frac{\gamma\beta_1}{2})}) = o(t^{\frac{2}{\gamma}(Q-\beta_1)}).
\end{align}
In summary, 
\begin{equation}
T_2 = \E[(K_{(-t,t)^c}(it) + i\sigma_t e^{\frac{\gamma}{2}M} \rho(\beta_1))^q] - \E[K_{(-t,t)^c}(it)^q] + o(t^{\frac{2}{\gamma}(Q-\beta_1)}).\\
\end{equation}
Finally, we evaluate the above difference at first order explicitly using the fact that density of $e^{\frac{\gamma}{2}M}$ is known:
\begin{align}
&\quad \E[(K_{(-t,t)^c}(it) + i\sigma_t e^{\frac{\gamma}{2}M} \rho(\beta_1))^q] - \E[K_{(-t,t)^c}(it)^q]\\ \nonumber
&\quad = \frac{2}{\gamma}(Q-\beta_1) \E\left[\int_1^{\infty} \frac{dv}{v^{\frac{2}{\gamma}(Q-\beta_1)+1}}\left(\left(K_{(-t,t)^c}(it) + i\sigma_t \rho(\beta_1) v\right)^q - K_{(-t,t)^c}(it)^q \right) \right]\\ \nonumber
&\quad = t^{\frac{2}{\gamma}(Q-\beta_1)} \frac{2}{\gamma}(Q-\beta_1)\E\left[\int_{\frac{\hat{\sigma}_t \rho(\beta_1)}{\hat{K}_{(-t,t)^c}(it)}}^{\infty} \frac{du}{u^{\frac{2}{\gamma}(Q-\beta_1)+1}}((iu+1)^q-1) \rho(\beta_1)^{\frac{2}{\gamma}(Q-\beta_1)}\hat{K}_{(-t,t)^c}(it)^{q-\frac{2}{\gamma}(Q-\beta_1)} \right].
\end{align}
In the last equality we have applied Theorem \ref{girsanov}. Next,
\begin{align}
\hat{K}_{(-t,t)^c}(it) = \int_{(-t,t)^c} \frac{it-x}{|x|^{\frac{\gamma}{2}(2Q-\beta_1-\frac{2}{\gamma})} |x-1|^{\frac{\gamma \beta_2 }{2}}}  g(x)^{\frac{\gamma^2}{8}(q-1)} e^{\frac{\gamma}{2} X(x)} d \mu(x) \underset{\text{a.s.}}{\overset{t\to 0_+}{\longrightarrow}} \hat{K}_{\mathbb{R}}(0),
\end{align}
and for $h < \frac{2}{\gamma(Q-\beta_1)}-1$,
\begin{align}
\hat{\sigma}_t = t^{1-(1+h)(1-\frac{\gamma\beta_1}{2}+\frac{\gamma^2}{4})}e^{\frac{\gamma}{2}B_{2(1+h)\ln (1/t)}} \underset{\text{a.s.}}{\overset{t\to 0_+}{\longrightarrow}} 0.
\end{align}
With some simple arguments of uniform integrability, we conclude that:
\begin{align}
&\quad \E\left[\left(K_{(-t,t)^c}(it) + i\sigma_t e^{\frac{\gamma}{2}M} \rho(\beta_1)\right)^q\right] - \E[K_{(-t,t)^c}(it)^q] \\ \nonumber
&\quad \overset{t \to 0_+}{\sim}t^{\frac{2}{\gamma}(Q-\beta_1)} \frac{2}{\gamma}(Q-\beta_1) \left( \int_{0}^{\infty} \frac{du}{u^{\frac{2}{\gamma}(Q-\beta_1)+1}}((iu+1)^q-1) \right)  \E\left[\rho(\beta_1)^{\frac{2}{\gamma}(Q-\beta_1)}\right] \E\left[\hat{K}_{\mathbb{R}}(0)^{q-\frac{2}{\gamma}(Q-\beta_1)} \right]\\ \nonumber
&\quad = (it)^{\frac{2}{\gamma}(Q-\beta_1)}\frac{2}{\gamma}(Q - \beta_1) \frac{\Gamma(\frac{2}{\gamma}(\beta_1 -Q)) \Gamma(\frac{2}{\gamma}(Q -\beta_1) -q ) }{\Gamma(-q)} \overline{R}(\beta_1, \mu_1, \mu_2) \overline{H}^{(2 Q - \beta_1 - \frac{2}{\gamma}, \beta_2 , \beta_3)}_{( \mu_1, -\mu_2,  -\mu_3)}.
\end{align}
The power of $i$ comes from the evaluation of the integral. Inspecting the proof we see that the conditions on $\beta_0$ and $h$ indeed allow us to find small values of these parameters that make the arguments work. Therefore we have proved the claim.
\end{proof}

Now the analogue result for $\chi = \frac{\gamma}{2}$:

\begin{lemma}\label{lem_reflection_ope1} (OPE with reflection for $\chi = \frac{\gamma}{2}$) Recall $p = \frac{2}{\gamma}(Q-\alpha-\frac{\beta}{2}+\frac{\gamma}{4})$ and consider $s \in (0,1)$. Recall also the functions $\tilde{G}_{\frac{\gamma}{2}}$ and $H_{\frac{\gamma}{2}}$ given by \eqref{eq:def_G_tilde} and \eqref{eq:def_H_chi} for $\chi = \frac{\gamma}{2}$. There exists a small parameter $\beta_0>0$ such that for $\beta \in (Q -\beta_0, Q)$ and $\alpha$ such that $  p < \frac{4}{\gamma^2} \wedge \frac{2}{\gamma}(Q - \beta)$, the following asymptotic expansion holds:
\begin{align}
 \tilde{G}_{\frac{\gamma}{2}}(s) - \tilde{G}_{\frac{\gamma}{2}}(0) &= -s^{\frac{1}{2} + \frac{\gamma^2}{8} - \frac{\gamma \beta}{4}}\frac{\Gamma(1-\frac{2(Q-\beta)}{\gamma}) \Gamma(-p+\frac{2}{\gamma}(Q-\beta))}{\Gamma(-p)}\overline{R}(\beta, 1, e^{i\pi\frac{\gamma^2}{4}}) \overline{G}(\alpha, 2Q-\beta -\frac{\gamma}{2}) \nonumber \\  &+ o(|s|^{{\frac{1}{2} + \frac{\gamma^2}{8} - \frac{\gamma \beta}{4}}}).
\end{align}
Similarly, recall $q = \frac{1}{\gamma}(2 Q  - \beta_1 - \beta_2 - \beta_3 + \frac{\gamma}{2})$ and consider $t \in (0,1)$. Then for $\mu_1, \mu_2, \mu_3 \in (0,+\infty)$, $\beta_1 \in (Q - \beta_0,Q)$ and $\beta_2, \beta_3$ chosen so that $  q < \frac{4}{\gamma^2} \wedge \min_i \frac{2}{\gamma}(Q - \beta_i) $, the following asymptotic also holds:
\begin{align}
H_{\frac{\gamma}{2}}(t) - H_{\frac{\gamma}{2}}(0) &= t^{1 -\frac{\gamma \beta_1 }{2} + \frac{\gamma^2}{4}} \frac{2(Q - \beta_1)}{\gamma} \frac{\Gamma(\frac{2}{\gamma}(\beta_1 -Q)) \Gamma(\frac{2}{\gamma}(Q -\beta_1) -q ) }{\Gamma(-q)} \overline{R}(\beta_1, \mu_1, \mu_2) \overline{H}^{(2 Q - \beta_1 - \frac{\gamma}{2} , \beta_2 , \beta_3)}_{( \mu_1, e^{i\pi\frac{\gamma^2}{4}} \mu_2,  e^{i\pi\frac{\gamma^2}{4}}  \mu_3)} \nonumber \\ 
  &+ o(|t|^{1 -\frac{\gamma \beta_1 }{2} + \frac{\gamma^2}{4}}).
\end{align}
\end{lemma}

\begin{proof}
We will keep the notations in the proof of Lemma \ref{lem_reflection_ope2} although there are some slight differences. This time $K$ is defined with the $\chi = \frac{\gamma}{2}$ insertion:
\begin{equation}
K_{I}(t) : =  \int_I \frac{(t-x)^{\frac{\gamma^2}{4}}}{|x|^{\frac{\gamma \beta_1 }{2}} |x-1|^{\frac{\gamma \beta_2 }{2}}}  g(x)^{\frac{\gamma^2}{8}(q-1)} e^{\frac{\gamma}{2} X(x)} d \mu(x).
\end{equation}
To deal with the complex phases we will simply use the following inequality. For a fixed $p<1$ and $\varphi \in [0, \pi)$, there exists a constant $c>0$ such that for all $x_1, x_2, y_1, y_2 \in (0, +\infty)$:
\begin{equation}\label{complex inequality}
|(x_1+e^{i\varphi}y_1)^p - (x_2 + e^{i\varphi}y_2)^p| \le c( |x_1^p - x_2^p|+ |y_1^p - y_2^p|).
\end{equation}
This inequality can be proved by studying the derivative of the function $(x,y) \mapsto (x^{1/p}+e^{i\varphi}y^{1/p})^p$. With the help of this inequality we will be able to perform the same proof as in the case of the previous lemma. 
Following the same steps as in \cite{interval}, we have:
\begin{align}
\E[ |K_{(-\infty,-t)}(t)^q - K_{(-\infty,0)}(0)^q|] &= o(t^{\frac{\gamma}{2}(Q-\beta_1)}),\\
\E[ ||K_{(t, \infty)}(t)|^q - |K_{(0, \infty)}(0)|^q|] &= o(t^{\frac{\gamma}{2}(Q-\beta_1)}).
\end{align}
Applying \eqref{complex inequality} implies that:
\begin{align}
\E[ |K_{(-t,t)^c}(t)^q - K_{\mathbb{R}}(0)^q|]  &= \E[ |(K_{(-\infty,-t)}(t)+e^{i\pi\frac{\gamma^2}{4}}|K_{(t,\infty)}(t)|)^q - (K_{(-\infty, 0)}(0)+e^{i\pi\frac{\gamma^2}{4}}|K_{(0,\infty)}(0)|)^q|] \nonumber \\ \nonumber
&\le c \E[ |K_{(-\infty,-t)}(t)^q - K_{(-\infty,0)}(0)^q|] + c \E[||K_{(t,\infty)}(t)|^q - |K_{(t,\infty)}(0)|^q|]\\ 
 & \le o(t^{\frac{\gamma}{2}(Q-\beta_1)}).
\end{align}
Next we repeat the step where we introduce a small $h>0$ and want to compare $K_{\mathbb{R}}(t)$ and\\ $K_{(-\infty,-t)\cup(-t^{1+h},t^{1+h})\cup(t,\infty)}(t)$. Following again the steps of \cite{interval}, under the constraint on $h$,
\begin{equation}
h<\frac{\frac{\gamma\beta_1}{2}-1}{1-\frac{\gamma\beta_1}{2}+\gamma^2},
\end{equation}
one can show that:
\begin{align}
\E[ |K_{(-\infty,t)}(t)^q - K_{(-\infty,-t)\cup(-t^{1+h},t^{1+h})}(t)^q|] &= o(t^{\frac{\gamma}{2}(Q-\beta_1)}).
\end{align} 
By applying again \eqref{complex inequality} one obtains,
\begin{align}
\E[ |K_{\mathbb{R}}(t)^q - K_{(-\infty,-t)\cup(-t^{1+h},t^{1+h})\cup(t,\infty)}(t)^q|] \le c\E[ |K_{(-\infty,t)}(t)^q - K_{(-\infty,-t)\cup(-t^{1+h},t^{1+h})}(t)^q|] = o(t^{\frac{\gamma}{2}(Q-\beta_1)}).
\end{align}
Therefore as in the previous lemma we have now reduced the problem to studying the difference:
\begin{align}
\E[K_{(-\infty,-t)\cup(-t^{1+h},t^{1+h})\cup(t,\infty)}(t)^q] - \E[K_{(-t,t)^c}(t)^q].
\end{align}

We proceed exactly in the same way as the case $\chi=\frac{2}{\gamma}$, using Kahane's inequality of Theorem \ref{theo Kahane inequality} to obtain:
\begin{align}
\E[K_{(-\infty,-t)\cup(-t^{1+h}, t^{1+h})\cup(t,\infty)}(t)^{q}] - \E[(K_{(-t,t)^c}(t) + \sigma_t V)^q]  = O(t^{h})+O(t^{(1+h)(1-\frac{\gamma\beta_1}{2}+\frac{\gamma^2}{4})}).
\end{align}
When $h > \frac{\gamma}{2}(Q-\beta_1)$ this term is also a $o(t^{\frac{\gamma}{2}(Q-\beta_1)})$. Here the expression of $\sigma_t$ is slightly different:
\begin{align}
\sigma_t = t^{\frac{\gamma^2}{4}+(1+h)(1-\frac{\gamma\beta_1}{2}+\frac{\gamma^2}{4})}e^{\frac{\gamma}{2}B_{2(1+h)\ln (1/t)}}.
\end{align}
As in our previous work \cite{interval}, we can show that 
\begin{align}
&\E[(K_{(-\infty,-t)}(t) + \sigma_t V)^q] - \E[K_{(-\infty,-t)}(t)^q] \\ \nonumber
 &= \E[(K_{(-\infty,-t)}(t) + \sigma_t e^{\frac{\gamma}{2}M} \rho(\beta_1))^q] - \E[K_{(-\infty,-t)}(t)^q] + o(t^{\frac{\gamma}{2}(Q-\beta_1)}).
\end{align}
This result is proved using inequalities, the lower bound and upper bound are then equivalent to a term with order $t^{\frac{\gamma}{2}(Q-\beta_1)}$. As a consequence,
\begin{align}
\E[(K_{(-\infty,-t)}(t) + \sigma_t V)^q] - \E[(K_{(-\infty,-t)}(t) + \sigma_t e^{\frac{\gamma}{2}M} \rho(\beta_1))^q] = o(t^{\frac{\gamma}{2}(Q-\beta_1)}).
\end{align}
Furthermore, we can write $V$ as 
\begin{align}
\sigma_t V = \sigma_t e^{\frac{\gamma}{2}M}\frac{1}{2}\int_{-L_M}^{\infty}  e^{ \frac{\gamma}{2}\mathcal{B}^{\frac{Q-\beta_1}{2}}_s} \mu_{\tilde{Y}}(ds) \le \sigma_t e^{\frac{\gamma}{2}M} \rho(\beta_1).
\end{align}
This allows us to put an absolute value in expectation:
\begin{align}
\E[|(K_{(-\infty,-t)}(t) + \sigma_t V)^q - (K_{(-\infty,-t)}(t) + \sigma_t e^{\frac{\gamma}{2}M} \rho(\beta_1))^q|] = o(t^{\frac{\gamma}{2}(Q-\beta_1)}).
\end{align}
We can conclude by using \eqref{complex inequality} that:
\begin{equation}
\E[(K_{(-t,t)^c}(t) + \sigma_t V)^q] - \E[(K_{(-t,t)^c}(t) + \sigma_t e^{\frac{\gamma}{2}M} \rho(\beta_1))^q] = o(t^{\frac{\gamma}{2}(Q-\beta_1)}).
\end{equation}
We estimate as in the case $\chi = \frac{2}{\gamma}$:
\begin{align}
&\E[(K_{(-t,t)^c}(t) + \sigma_t e^{\frac{\gamma}{2}M} \rho(\beta_1))^q] - \E[K_{(-t,t)^c}(t)^q]\\ \nonumber
&=  t^{\frac{\gamma}{2}(Q-\beta_1)}\frac{2(Q - \beta_1)}{\gamma} \frac{\Gamma(\frac{2}{\gamma}(\beta_1 -Q)) \Gamma(\frac{2}{\gamma}(Q -\beta_1) -q ) }{\Gamma(-q)} \overline{R}(\beta_1, \mu_1, \mu_2 ) \overline{H}^{(2 Q - \beta_1 - \frac{\gamma}{2}, \beta_2 , \beta_3)}_{( \mu_1, \mu_2e^{i\pi\frac{\gamma^2}{4}},  \mu_3e^{i\pi\frac{\gamma^2}{4}})} + o(t^{\frac{\gamma}{2}(Q-\beta_1)}).
\end{align}
Finally it is again possible to choose suitable small $h>0 $ and $\beta_0>0$ which make the argument work. This concludes the proof of the lemma.
\end{proof}

\subsubsection{Analytic continuation}

In this section we prove the lemma of analyticity of the moments of GMC that we have used repetitively throughout the paper. This fact has been first shown in \cite[Theorem 6.1]{DOZZ2} in the case of the correlation functions on the sphere. The main idea is that starting from the range of real parameters of $\beta_i$ or $\alpha$ where a given GMC expression is defined, one can find a small neighborhood in $\mathbb{C}$ of the parameter range where the quantity will still be well-defined and is complex analytic on this neighborhood. We also use in Section \ref{sec_3pt} the fact that the three-point function is complex analytic in the $\mu_i$. This fact is obtained directly by differentiating with respect to $\mu_i$.

\begin{lemma}\label{lem_analycity} (Analyticity in insertions weights and in $\mu_i$ of moments of GMC)
Consider the following functions defined in the given parameter range:
\begin{itemize}
\item $(\alpha, \beta) \mapsto \overline{G}(\alpha, \beta)$ for $\beta < Q $, $ \frac{\gamma}{2} - \alpha < \frac{\beta}{2} < \alpha$.
\item $(\alpha, \beta) \mapsto G_{\chi}(t) $ for $t \in \H$, $ \beta < Q$,  $ \frac{2}{\gamma} \left(Q -\alpha - \frac{\beta}{2} + \frac{\chi}{2} \right) < \frac{4}{\gamma^2} \wedge \frac{2}{\gamma}(Q - \beta)$.

\item $(\beta_1, \beta_2, \beta_3) \mapsto \overline{H}^{(\beta_1 , \beta_2 , \beta_3)}_{( \mu_1, \mu_2,  \mu_3 )}$ for:
 $$(\mu_i)_{i=1,2,3} \: \: \: \text{satisfies Definition \ref{half-space},} \quad \beta_i < Q,  \quad   {\frac{1}{\gamma}(2 Q  - \sum_{i=1}^3 \beta_i)} < \frac{4}{\gamma^2} \wedge \min_i \frac{2}{\gamma}(Q - \beta_i).$$
\item $(\beta_1, \beta_2, \beta_3) \mapsto H_{\chi}(t)$ for:
$$ \beta_i < Q, \quad \mu_1 \in (0, \infty), \: \: \mu_2, \mu_3 \in - \overline{\H}, \quad  q < \frac{4}{\gamma^2} \wedge \min_i \frac{2}{\gamma}(Q - \beta_i), \quad t \in \mathbb{H}.$$
\end{itemize} 
Then for each function above, and for each of the function's variables, it is complex analytic in a small complex neighborhood of any compact set $K$ contained in the domain of definition of the function for real parameters. Furthermore the function  $\overline{H}$ now viewed as a function of $\mu_1, \mu_2, \mu_3$ is complex analytic in any compact $\tilde{K}$ contained in the range of parameters written above.
\end{lemma}

\begin{proof}
We briefly adapt the proof of \cite[Theorem 6.1]{DOZZ2} for the function $H^{(\beta_1, \beta_2, \beta_3)}_{(\mu_1, \mu_2, \mu_3)}$ as the other cases can be treated in a similar manner. The first step performed in \cite{DOZZ2} is to apply the Girsanov Theorem \ref{girsanov} to pull out the insertions outside of the GMC expectation. It will be convenient to assume the three insertions are not located at $0$, $1$ and $\infty$ but rather at three points $s_1$, $s_2$, $s_3$ all in $\mathbb{R}$ and obeying the extra constraints  $|s_i|>2$ and $|s_i-s_{i'}|>2$ respectively for all $i \in \{1,2,3 \}$ and for all $i \neq i'$. The reason it is possible to assume this is that the Liouville correlations are conformally invariant in the sense of \cite[Theorem~3.5.]{Disk}. It will be convenient to use the notations $\bm{\beta} = (\beta_1, \beta_2, \beta_3)$ and $\s = (s_1, s_2, s_3)$. Our starting point is thus that it is possible to write,
\begin{align}
\overline{H}^{(\beta_1, \beta_2, \beta_3)}_{(\mu_1, \mu_2, \mu_3)} = C \overline{H}^{(\beta_1, \beta_2, \beta_3)}_{(\mu_1, \mu_2, \mu_3)}(\s),
\end{align}
for $C$ an explicit prefactor that is analytic in the $\beta_i$ and hence can be ignored and where we have introduced:
\begin{align}
\overline{H}^{(\beta_1, \beta_2, \beta_3)}_{(\mu_1, \mu_2, \mu_3)}(\s) = \mathbb{E} \left[ \left( \intr \frac{g(x)^{\frac{\gamma}{8}(\frac{4}{\gamma}- \sum_{i=1}^3 \beta_i )} }{\prod_{i=1}^3 |x-s_i|^{\frac{\gamma \beta_i }{2}}}  e^{\frac{\gamma}{2} X(x)} d \mu(x) \right)^{\frac{1}{\gamma}(2 Q  - \sum_{i=1}^3 \beta_i)} \right].
\end{align}
Now by applying Theorem \ref{girsanov} we can obtain $\overline{H}^{(\beta_1, \beta_2, \beta_3)}_{(\mu_1, \mu_2, \mu_3)}(\s)$ from the following limit,
\begin{equation}
H^{(\beta_1, \beta_2, \beta_3)}_{(\mu_1, \mu_2, \mu_3)}(\s) = \lim_{r \to \infty} \mathcal{F}_r(\bm{\beta}),
\end{equation}
where we have introduced,
\begin{equation}
\mathcal{F}_r(\bm{\beta}) = \mathbb{E} \left[ \prod_{i=1}^3 e^{\beta_i X_{r}(s_i) - \frac{\beta_i^2}{2} \mathbb{E}[X_r(s_i)^2] } \left( \int_{\mathbb{R}_r} g(x)^{\frac{1}{2}}  e^{\frac{\gamma}{2} X(x)} d \mu(x) \right)^{p_0} \right],
\end{equation}
$p_0 = \frac{1}{\gamma}(2 Q  - \sum_{i=1}^3 \beta_i)$ and:
\begin{equation}
\mathbb{R}_r := \mathbb{R} \backslash \cup_{i=1}^3 (s_i-e^{-r/2}, s_i+e^{-r/2}).
\end{equation}
The fields $X_r(s_i)$ are the radial parts of $X(s_i)$ obtained by taking the mean of $X(s_i)$ over the upper-half circles of radius $e^{-r/2}$, $\partial B(s_i,e^{-r/2})_+$.

Now when $\beta_i$ are complex numbers, we write $\beta_i = a_i + i b_i$. We want to prove there exists a complex neighborhood $V$ in $\mathbb{C}^3$ containing the domain of definition for real $\beta_i$ such that for all compact sets contained in $ V$, $\mathcal{F}_r(\bm{\beta})$ converges uniformly as $r \rightarrow + \infty$ over the compact set. It is known that $X_{r+t}(s_i)-X_{r}(s_i)$ are independent Brownian motions for different $s_i$. Hence,
\begin{align}
&|\mathcal{F}_{r+1}(\bm{\beta}) - \mathcal{F}_{r}(\bm{\beta})|\\ \nonumber
&= \left|\mathbb{E} \left[ \prod_{i=1}^3 e^{i b_i X_{r+1}(s_i) + \frac{b_i^2}{2} \mathbb{E}[X_{r+1}(s_i)^2] } \left( \left( \int_{\mathbb{R}_{r+1}}  e^{\frac{\gamma}{2} X(x)} f(x) d \mu(x) \right)^{p_0}  - \left( \int_{\mathbb{R}_r}  e^{\frac{\gamma}{2} X(x)} f(x) d \mu(x) \right)^{p_0}\right) \right] \right|\\ \nonumber
&\le  c \,e^{\frac{r+1}{2}\sum_{i=1}^3 b_i^2}\mathbb{E} \left[ \left| \left( \int_{\mathbb{R}_{r+1}}  e^{\frac{\gamma}{2} X(x)} f(x) d \mu(x) \right)^{p_0}  - \left( \int_{\mathbb{R}_r}  e^{\frac{\gamma}{2} X(x)}  f(x) d \mu(x) \right)^{p_0}\right| \right],
\end{align}
where we denote $f(x) = \frac{g(x)^{\frac{\gamma^2}{8}(p_0-1)}}{\prod_{i=1}^3 |x-s_i|^{\frac{\gamma a_i}{2}}}  $. Set $Z_r := \int_{\mathbb{R}_{r}}  e^{\frac{\gamma}{2} X(x)} f(x) d \mu(x)$ and $Y_r := Z_{r+1} - Z_{r}$. We want to estimate 
\begin{align}
\E[|(Z_r + Y_r)^{p_0} - Z_r^{p_0}|] \le \E[\mathbf{1}_{|Y_r| < \epsilon}|(Z_r + Y_r)^{p_0} - Z_r^{p_0}|] + \E[\mathbf{1}_{|Y_r| \ge \epsilon}|(Z_r + Y_r)^{p_0} - Z_r^{p_0}|],
\end{align}
where $\epsilon>0$ will be fixed later. By interpolation,
\begin{align}
\E[\mathbf{1}_{|Y_r| < \epsilon}|(Z_r + Y_r)^{p_0} - Z_r^{p_0}|] \le |p_0|\epsilon \sup_{u\in [0,1]} \E[|(1-u)Z_r + uY_r|^{\textnormal{Re}(p_0)-1}] \le c \, \epsilon.
\end{align}
For the other term, we use the H\"older inequality with $\lambda > 1$ such that $\frac{\lambda}{\lambda-1} \textnormal{Re}(p_0) < \min_{i=1}^3 \frac{2}{\gamma}(Q-a_i) \wedge \frac{4}{\gamma^2}$, and $0< m < \frac{4}{\gamma^2}$,
\begin{align}
\E[\mathbf{1}_{|Y_r| \ge \epsilon}|(Z_r + Y_r)^{p_0} - Z_r^{p_0}|] &\le c\, \P(|Y_r| \ge \epsilon)^{\frac{1}{\lambda}} \le c \epsilon^{-\frac{m}{\lambda}}\E[|Y_r|^m]^{\frac{1}{\lambda}}\\ \nonumber
&\le c\epsilon^{-\frac{m}{\lambda}} \E \left[ \left| \sum_{i=1}^3 \int_{(s_i - e^{-r/2}, s_i + e^{r/2})}  e^{\frac{\gamma}{2}X(x)} f(x) d\mu(x) \right|^m \right]^{\frac{1}{\lambda}}\\ \nonumber
&\le c' \epsilon^{-\frac{m}{\lambda}} \left(\max_i e^{-\frac{r}{2}((1+\frac{\gamma^2}{2}-\frac{\gamma a_i}{2})m - \frac{\gamma^2 m^2}{2})}\right)^{\frac{1}{\lambda}} =: c'\epsilon^{-\frac{m}{\lambda}} e^{-\frac{\theta}{\lambda}r},
\end{align}
where in the last step $\theta \in \mathbb{R} $ is defined by the last equality and we have used the multifractal scaling property of the GMC, see e.g.~\cite[Section 3.6]{BP} or \cite[Section 4]{review}. We can choose a suitable $m$ such that $\theta>0$. Now take $\epsilon = e^{-\eta r}$ with $\eta = \frac{\theta}{\lambda + m}$, then:
\begin{align}
\E[|(Z_r + Y_r)^{p_0} - Z_r^{p_0}|] \le c\,e^{\frac{r+1}{2}\sum_{i=1}^3 b_i^2}(\epsilon + \epsilon^{-\frac{m}{\lambda}} e^{-\frac{\theta}{\lambda}r}) \le c'\, e^{-(\eta - \frac{1}{2}\sum_{i=1}^3 b_i^2)r}.
\end{align}
Hence if one chooses the open set $V$ in such a way that $\frac{1}{2}\sum_{i=1}^3 b_i^2 < \eta$ always holds, all the inequalities we have done before hold true and thus we have shown that $\mathcal{F}_r(\bm{\beta})$ converges locally uniformly. This proves the analyticity result. 

Lastly we very briefly justify all the other cases. The analyticity of $\overline{G}(\alpha, \beta)$ can be proved in the exact same way as done above for $\overline{H}^{(\beta_1 , \beta_2 , \beta_3)}_{( \mu_1, \mu_2,  \mu_3 )}$. Furthermore adding the dependence $t$ to get the functions $G_{\chi}(t)$ and $H_{\chi}(t)$ also changes nothing to the above argument and so the same claim also holds in this case. Lastly for the analyticity in $\mu_i$ of $\overline{H}^{(\beta_1 , \beta_2 , \beta_3)}_{( \mu_1, \mu_2,  \mu_3 )}$, one simply needs to notice the complex derivatives are well-defined. For instance for $\mu_1$ one can write,
\begin{align}
\partial_{\mu_1} &\overline{H}^{(\beta_1, \beta_2, \beta_3)}_{(\mu_1, \mu_2, \mu_3)}  = \partial_{\mu_1}\mathbb{E} \left[ \left( \intr \frac{g(x)^{\frac{\gamma}{8}(\frac{4}{\gamma}- \sum_{i=1}^3 \beta_i )} }{|x|^{\frac{\gamma \beta_1 }{2}}|x-1|^{\frac{\gamma \beta_2 }{2}}}  e^{\frac{\gamma}{2} X(x)} d \mu(x) \right)^{\frac{1}{\gamma}(2 Q  - \sum_{i=1}^3 \beta_i)} \right]\\ \nonumber
&= \int_{-\infty}^0 dx_1 \frac{g(x_1)^{\frac{\gamma}{8}(\frac{4}{\gamma}- \sum_{i=1}^3 \beta_i )} }{|x|^{\frac{\gamma \beta_1 }{2}}|x-1|^{\frac{\gamma \beta_2 }{2}}}  \E \left[ e^{\frac{\gamma}{2} X(x_1)} \left( \intr \frac{g(x)^{\frac{\gamma}{8}(\frac{4}{\gamma}- \sum_{i=1}^3 \beta_i )} }{|x|^{\frac{\gamma \beta_1 }{2}}|x-1|^{\frac{\gamma \beta_2 }{2}}}  e^{\frac{\gamma}{2} X(x)} d \mu(x) \right)^{\frac{1}{\gamma}(2 Q  - \sum_{i=1}^3 \beta_i)-1} \right],
\end{align}
where the last expression is clearly well-defined. Furthermore one can check that $\partial_{\overline{\mu}_1} \overline{H}^{(\beta_1, \beta_2, \beta_3)}_{(\mu_1, \mu_2, \mu_3)} = 0$. Therefore $\mu_1 \mapsto \overline{H}^{(\beta_1 , \beta_2 , \beta_3)}_{( \mu_1, \mu_2,  \mu_3 )}$ is complex analytic in the claimed domain.
\end{proof}

\subsubsection{The limit of $\overline{H}$ recovers $\overline{R}$}\label{sec_H_R}

Here we will prove Lemma \ref{lim_H_R}.
With our choice of $\mu_i $ satisfying Definition \ref{half-space} this is an easy adaptation of the positive case as performed in \cite{DOZZ2, interval}.
\begin{proof}
We prove the lemma in the first case where $\beta_2 < \beta_1$. 
Let us denote $\epsilon = \frac{\beta_3 - (\beta_1 - \beta_2)}{\gamma}$, $p_1 = \frac{2}{\gamma}(Q-\beta_1)$. For $I \subseteq \mathbb{R}$ a Borel set, we introduce the notation:
\begin{align}
K_{I} = \int_I \frac{1}{|x|^{\frac{\gamma \beta_1 }{2}} |x-1|^{\frac{\gamma \beta_2 }{2}}}  g(x)^{\frac{\gamma^2}{8}(p-1-\epsilon)} e^{\frac{\gamma}{2} X(x)} d x.
\end{align}
In our previous paper \cite{interval} it is proved that:
\begin{align}
\epsilon\E[K_{[0,1]}^{p_1-\epsilon}] \overset{\epsilon \to 0}{\longrightarrow} p_1 \overline{R}(\beta_1, 0, 1).
\end{align}
Using the density of $e^{\frac{\gamma}{2}M}$, we have by definition of the reflection coefficient,
\begin{align}
\epsilon\E\left[\left( e^{\frac{\gamma}{2}M} \rho_+(\beta_1) \right)^{p_1-\epsilon} \right] \overset{\epsilon \to 0}{\longrightarrow} p_1 \overline{R}(\beta_1, 0, 1),
\end{align}
where:
\begin{align}
\rho_{\pm}(\beta_1) := \frac{1}{2} \int_{-\infty}^{\infty}  e^{ \frac{\gamma}{2}\mathcal{B}^{\frac{Q-\beta_1}{2}}_s} e^{\frac{\gamma}{2}Y(\pm e^{-s/2})} ds.
\end{align}
On the other hand, by the William's path decomposition of Theorem \ref{theo Williams} we can write:
\begin{align}
K_{[0,1]} &=  e^{\frac{\gamma}{2}M} \frac{1}{2}  \int_{-L_M}^{\infty}  e^{ \frac{\gamma}{2}\mathcal{B}^{\frac{Q-\beta_1}{2}}_s} e^{\frac{\gamma}{2}Y(e^{-s/2})} ds \le e^{\frac{\gamma}{2}M}\rho_+(\beta_1).
\end{align}
Therefore, the result from \cite{interval} implies that:
\begin{align}
\E \left[ \left|K_{[0,1]}^{p_1-\epsilon} - (e^{\frac{\gamma}{2}M}\rho_+(\beta_1))^{p_1-\epsilon}\right| \right] = o(\epsilon^{-1}).
\end{align}
Similarly we also have
\begin{align}
\E \left[ \left|K_{[-1,0)}^{p_1-\epsilon} - \left(e^{\frac{\gamma}{2}M}\rho_-(\beta_1)\right)^{p_1-\epsilon}\right| \right] = o(\epsilon^{-1}).
\end{align}
We will use these results to prove the complex $\mu_i$ case. Consider first the case $p_1>1$. Using interpolation and H\"older's inequality, for $\lambda > 1$,
\begin{align}
&\E \left[ \left|(\mu_1 K_{(-\infty, 0)} + \mu_2 K_{[0,1]} + \mu_3 K_{(1,\infty)} )^{p_1-\epsilon} -  (\mu_1 K_{[-1, 0)} + \mu_2 K_{[0,1]} )^{p_1-\epsilon} \right| \right]\\ \nonumber
& \le  \E \left[ |\mu_1 K_{(-\infty, -1)} + \mu_3 K_{(1,\infty)} |^{\lambda}\right]^{\frac{1}{\lambda}} \\
& \qquad \times \sup_{u\in [0,1]}\E\left[\left|(1-u)(\mu_1 K_{(-\infty, 0)} + \mu_2 K_{[0,1]} + \mu_3 K_{(1,\infty)} ) + u(\mu_1 K_{[-1, 0)} + \mu_2 K_{[0,1]}) \right|^{(p_1-1-\epsilon)\frac{\lambda}{\lambda-1}} \right]^{\frac{\lambda-1}{\lambda}}. \nonumber
\end{align}
Take $p_1 < \lambda < \min\{\frac{4}{\gamma^2},\frac{2}{\gamma}(Q-\beta_2\lor \beta_3)\}$, then both expectations can be bounded by $O(1)$. By the same techniques with $\lambda = p-\epsilon$ we prove:
\begin{align}
&\E \left[ \left|\left(\mu_1 K_{[-1, 0)} + \mu_2 K_{[0,1]}\right)^{p_1-\epsilon}  - \left(\mu_1 e^{\frac{\gamma}{2}M}\rho_-(\beta_1)+  \mu_2 e^{\frac{\gamma}{2}M}\rho_+(\beta_1) \right)^{p_1-\epsilon} \right| \right]\\ \nonumber
& \le  \E\left[\left(e^{\frac{\gamma}{2}M} \frac{1}{2}  \int_{-\infty}^{-L_M}  e^{ \frac{\gamma}{2}\mathcal{B}^{\frac{Q-\beta_1}{2}}_s} \left(|\mu_1| e^{\frac{\gamma}{2}Y(-e^{-s/2})} + |\mu_2| e^{\frac{\gamma}{2}Y(e^{-s/2})} \right)ds\right)^{p_1-\epsilon}\right]^{\frac{1}{p_1-\epsilon}}\\ \nonumber
&\times \E\left[\left(|\mu_1| e^{\frac{\gamma}{2}M}\rho_-(\beta_1)+  |\mu_2| e^{\frac{\gamma}{2}M}\rho_+(\beta_1) \right)^{p_1-\epsilon}\right]^{\frac{p_1-1-\epsilon}{p_1-\epsilon}}.
\end{align}
The second expectation is a $O(\epsilon^{-1})$. For the first expectation, we use the inequality that for $x,y>0$ one has $x^{p_1-\epsilon} + y^{p_1-\epsilon} < (x+y)^{p_1-\epsilon}$. This shows that:
\begin{align}
&\E\left[\left(e^{\frac{\gamma}{2}M} \frac{1}{2}  \int_{-\infty}^{-L_M}  e^{ \frac{\gamma}{2}\mathcal{B}^{\frac{Q-\beta_1}{2}}_s} \left(|\mu_1| e^{\frac{\gamma}{2}Y(-e^{-s/2})} + |\mu_2| e^{\frac{\gamma}{2}Y(e^{-s/2})} \right)ds\right)^{p_1-\epsilon}\right]\\ \nonumber
& \le \E\left[\left(|\mu_1| e^{\frac{\gamma}{2}M}\rho_-(\beta_1)+  |\mu_2| e^{\frac{\gamma}{2}M}\rho_+(\beta_1) \right)^{p-\epsilon}\right]- \E\left[\left (|\mu_1| K_{[-1, 0)} + |\mu_2| K_{[0,1]}\right)^{p_1-\epsilon}\right] = o(\epsilon^{-1}).
\end{align}
The last inequality comes from the fact that the two expectations are equivalent when $\epsilon \to 0$ to a term $O(\epsilon^{-1})$. Therefore:
\begin{align}
\E \left[ \left|\left(\mu_1 K_{[-1, 0)} + \mu_2 K_{[0,1]}\right)^{p_1-\epsilon}  - \left(\mu_1 e^{\frac{\gamma}{2}M}\rho_-(\beta_1)+  \mu_2 e^{\frac{\gamma}{2}M}\rho_+(\beta_1) \right)^{p_1-\epsilon} \right| \right] = o(\epsilon^{-1}).
\end{align}
Now consider the case $p_1 \le 1$. Since $p_1 = \frac{2}{\gamma}(Q-\beta_1)>0$, we are in the case $0 < p_1 \le 1$. By studying the first order derivatives of the function,
\begin{align}
(\mathbb{R}_+^*)^3 \ni (x_1,x_2,x_3) \mapsto \left(\mu_1 x_1^{\frac{1}{p_1}} + \mu_2 x_2^{\frac{1}{p_1}} + \mu_3 x_3^{\frac{1}{p_1}} \right)^{p_1},
\end{align}
we can prove the following inequality with a constant $c>0$ depending only on the $\mu_i$. For $x_i, x_i' >0$,
\begin{align}
\left|(\sum_{i=1}^3 \mu_i x_i)^{p_1} - (\sum_{i=1}^3\mu_i x_i')^{p_1}\right| \le c \sum_{i=1}^3 |x_i^{p_1}-x_i'^{p_1}|.
\end{align}
Applying the inequality,
\begin{align}
&\E \left[ \left|\left(\mu_1 K_{[-\infty, 0)} + \mu_2 K_{[0,1]} + \mu_3 K_{(1,\infty)}\right)^{p_1-\epsilon}  - \left(\mu_1 e^{\frac{\gamma}{2}M}\rho_-(\beta_1)+  \mu_2 e^{\frac{\gamma}{2}M}\rho_+(\beta_1) \right)^{p_1-\epsilon} \right| \right]\\ \nonumber
& \le  c \E\left[\left|K_{(-\infty,0)}^{p_1-\epsilon} - (e^{\frac{\gamma}{2}M} \rho_-(\beta_1))^{p_1-\epsilon}\right|\right] + c \E\left[\left|K_{[0,1]}^{p_1-\epsilon} - (e^{\frac{\gamma}{2}M} \rho_+(\beta_1))^{p_1-\epsilon}\right|\right] + O(1)\\ \nonumber
&\le c \E\left[\left|K_{(-\infty,0)}^{p_1-\epsilon} - (e^{\frac{\gamma}{2}M} \rho_-(\beta_1))^{p_1-\epsilon}\right|\right] + o(\epsilon^{-1}).
\end{align}
Moreover, by sub-additivity,
\begin{align}
\E[|K_{(-\infty,0)}^{p_1-\epsilon}- K_{(-1,0)}^{p_1-\epsilon}|] =\E[K_{(-\infty,0)}^{p_1-\epsilon}- K_{(-1,0)}^{p_1-\epsilon}] \le \E[K_{(-\infty,-1)}^{p_1-\epsilon}] = O(1).
\end{align}
Therefore we can bound
\begin{align}
\E\left[\left|K_{(-\infty,0)}^{p_1-\epsilon} - (e^{\frac{\gamma}{2}M} \rho_-(\beta_1))^{p_1-\epsilon}\right|\right] \le \E\left[\left|K_{(-1,0)}^{p_1-\epsilon} - (e^{\frac{\gamma}{2}M} \rho_-(\beta_1))^{p_1-\epsilon}\right|\right] + o(\epsilon^{-1}) = o(\epsilon^{-1}).
\end{align}
In conclusion,
\begin{align}
\lim_{\epsilon \to 0} \epsilon \E\left[\left(\mu_1 K_{(-\infty, 0)} + \mu_2 K_{[0,1]} + \mu_3 K_{(1,\infty)} \right)^{p_1-\epsilon} \right] &= \lim_{\epsilon \to 0} \epsilon \E\left[\left(\mu_1 e^{\frac{\gamma}{2}M}\rho_-(\beta_1)+  \mu_2 e^{\frac{\gamma}{2}M}\rho_+(\beta_1) \right)^{p_1-\epsilon}\right]\\ \nonumber
&=p_1 \overline{R}(\beta_1,\mu_1,\mu_2).
\end{align}
This finishes the proof of the lemma.
\end{proof}

\subsection{Mapping GMC moments from $\partial \mathbb{D}$ to $\mathbb{R}$}\label{sec_path_proba}

We prove here a lemma providing a very concrete computation linking the moment of GMC on $\partial \mathbb{D}$ to the moment on $\mathbb{R}$. This will be used to relate the moment formula for GMC on the circle of \cite{remy} to the $\overline{U}(\alpha)$ defined in our paper. This lemma can be deduced from \cite[Proposition 3.7]{Disk}, but doing so requires several straightforward but tedious computational steps. Therefore for clarity we include a self-contained proof.
\begin{lemma}\label{link_D_H}
Consider $\beta < Q$ and $\frac{\gamma}{2} - \alpha < \frac{\beta}{2} < \alpha $ and let $X$ and $X_{\mathbb{D}}$ be the GFF respectively on $\mathbb{H}$ and $\mathbb{D}$ with covariance given by equations \eqref{covariance} and \eqref{GFF_D}. Then the following equality holds,
\begin{equation}
\E\left[ \left( \int_0^{2\pi} \frac{1}{|e^{i \theta} -1|^{\frac{\gamma \beta}{2}}} e^{\frac{\gamma}{2} X_{\mathbb{D}}(e^{i\theta})}  d\theta\right)^{\frac{2Q-2\alpha - \beta}{\gamma}}\right] =  2^{(\alpha - \frac{\beta}{2})(Q-\alpha - \frac{\beta}{2})}   \E \left[   \left( \int_{\mathbb{R}} \frac{e^{\frac{\gamma}{2}  X(x)   }}{|x - i|^{\gamma \alpha}} g(x)^{\frac{1}{2} - \frac{\alpha \gamma }{4} - \frac{\beta\gamma}{8}} dx \right)^{\frac{2Q -2 \alpha - \beta}{\gamma}} \right],
\end{equation}

where both GMC measures are defined by a renormalization according to variance as performed in Definition \ref{def_GMC}. Setting $\beta = 0$, we obtain for $ \alpha > \frac{\gamma}{2}$ the equation 
\begin{equation}
\E\left[ \left( \int_0^{2\pi} e^{\frac{\gamma}{2} X_{\mathbb{D}}(e^{i\theta})}  d\theta\right)^{\frac{2Q-2\alpha}{\gamma}}\right] = 2^{\alpha(Q-\alpha)} \E \left[   \left( \int_{\mathbb{R}} \frac{e^{\frac{\gamma}{2}  X(x)   }}{|x - i|^{\gamma \alpha}} g(x)^{\frac{1}{2} - \frac{\alpha \gamma }{4}} dx \right)^{\frac{2Q -2 \alpha}{\gamma}} \right].
\end{equation}
\end{lemma}

\begin{proof}

Take $\psi : z\mapsto i \frac{1 +z}{1 - z}$ the conformal map that maps the unit disk $\mathbb{D}$ equipped with the Euclidean metric to the upper-half plane $\mathbb{H}$ equipped with the metric $\hg(x) = \frac{4}{|x+i|^4}$. This also maps the field $X_{\mathbb{D}}$ to the field $X_{\hg}$ with covariance given by \eqref{GFF_hg}. Record the following:
\begin{equation}
\hat{g}(x)^{- \frac{\gamma \alpha}{4}} = 2^{- \frac{\gamma \alpha}{2}} |x-i|^{\gamma \alpha}, \quad \frac{1}{|e^{i \theta} -1|^{\frac{\gamma \beta}{2}}} =  2^{- \frac{\gamma \beta}{4}} \hat{g}(x)^{- \frac{\gamma \beta}{8}}, \quad d \theta = \hat{g}(x)^{1/2} dx. 
\end{equation}
This change of coordinates applied to the GMC implies the following relation:
\begin{align}\label{Ualpha_D_H}
&\E\left[ \left( \int_0^{2\pi} \frac{1}{|e^{i \theta } -1 |^{\frac{\gamma \beta}{2}}} e^{\frac{\gamma}{2} X_{\mathbb{D}}(e^{i\theta}) - \frac{\gamma^2}{8} \E[X_{\mathbb{D}}(e^{i\theta})^2]}  d\theta\right)^{\frac{2Q-2\alpha - \beta}{\gamma}}\right]\\
& = 2^{(\alpha - \frac{\beta}{2})(Q-\alpha - \frac{\beta}{2})}  \E\left[ \left( \intr \frac{e^{\frac{\gamma}{2} X_{\hg}(x)   -\frac{\gamma^2}{8} \E[X_{\hg}(x)^2]}}{|x-i|^{\gamma\alpha}} \hg(x)^{\frac{\gamma}{4}(\frac{2}{\gamma}-\alpha -\frac{\beta}{2})}d x \right)^{\frac{2Q-2\alpha - \beta}{\gamma}}\right]. \nonumber
\end{align}
Notice in the above expression we explicitly wrote the renormalization of the GMC to emphasize the formula holds when the GMC is renormalized by variance. Now lets momentarily assume $\alpha +\frac{\beta}{2}>Q$ and introduce the following integral over $c$:

\begin{align}
&\E\left[ \left( \intr \frac{e^{\frac{\gamma}{2} X_{\hg}(x)   -\frac{\gamma^2}{8} \E[X_{\hg}(x)^2]}}{|x-i|^{\gamma\alpha}} \hg(x)^{\frac{\gamma}{4}(\frac{2}{\gamma}-\alpha - \frac{\beta}{2})}d x \right)^{\frac{2Q-2\alpha -\beta}{\gamma}}\right] \\ \nonumber
&= \frac{\gamma}{2} \frac{e^{\frac{\alpha}{2}(Q-\alpha - \frac{\beta}{2}) \ln \hg(i)}}{ \Gamma(\frac{2}{\gamma}(\alpha + \frac{\beta}{2} - Q))}   \int_{\mathbb{R}} dc e^{(\alpha + \frac{\beta}{2} - Q)c} \E \left[ e^{\alpha  X_{\hg}(i) - \frac{\alpha^2}{2} \E[ X_{\hg}(i)^2] }e^{- e^{\frac{\gamma c}{2}} \int_{\mathbb{R}} e^{\frac{\gamma}{2}  X_{\hg}(x) - \frac{\gamma^2}{8} \E[ X_{\hg}(x)^2] } \hat{g}(x)^{\frac{1}{2} - \frac{\gamma \beta}{8}} dx } \right].
\end{align}
To go from the field $ X_{\hg}$ to the field $X$ we must perform the change of variable $X = X_{\hg} - Y $ with $ Y = \frac{1}{\pi} \int_0^{\pi} X_{\hg}(e^{i \theta}) d \theta$. We perform this replacement and at the same time shift the integration over $c$ by $-Y$ to obtain:
\begin{align}
&\int_{\mathbb{R}} dc e^{(\alpha + \frac{\beta}{2} - Q)c} \E \left[ e^{(Q - \frac{\beta}{2})Y} e^{\alpha  X(i) - \frac{\alpha^2}{2} \E[ X_{\hg}(i)^2] }e^{- e^{\frac{\gamma c}{2}} \int_{\mathbb{R}} e^{\frac{\gamma}{2}  X(x) - \frac{\gamma^2}{8} \E[ X_{\hg}(x)^2] } \hat{g}(x)^{\frac{1}{2} - \frac{\gamma \beta}{8}} dx } \right] \\ \nonumber
&= \int_{\mathbb{R}} dc e^{(\alpha + \frac{\beta}{2} - Q)c} \E \Bigg[ e^{\frac{(Q - \frac{\beta}{2})^2}{2} \E[Y^2]} e^{\alpha  X(i) + \alpha (Q - \frac{\beta}{2})  \E[X(i)Y] - \frac{\alpha^2}{2} \E[ X_{\hg}(i)^2] } \\
&\qquad \qquad \qquad \qquad \qquad   \times e^{- e^{\frac{\gamma c}{2}} \int_{\mathbb{R}} e^{\frac{\gamma}{2}  X(x) + \frac{\gamma }{2}(Q -\frac{\beta}{2}) \E[X(x)Y] - \frac{\gamma^2}{8} \E[ X_{\hg}(x)^2] } \hat{g}(x)^{\frac{1}{2} - \frac{\gamma \beta}{8}} dx } \Bigg]. \nonumber
\end{align}
In the last line we have applied the Girsanov Theorem \ref{girsanov} to $e^{(Q - \frac{\beta}{2})Y}$. Record the following easy computations:
\begin{align}
&\E[Y^2] = - \frac{1}{\pi} \int_0^{\pi} \ln \hg(e^{i \theta}) d \theta, \quad \quad  \E[Y X_{\hg}(x)] = \frac{1}{2} \ln \frac{g(x)}{\hg(x)}  + \frac{1}{2}\E[Y^2], \\ \nonumber
&\E[Y X(x)] =   \frac{1}{2} \ln \frac{g(x)}{\hg(x)}  - \frac{1}{2}\E[Y^2], \quad \quad   \E[ X_{\hg}(x)^2] =  \E[ X(x)^2] +  \ln \frac{g(x)}{\hg(x)}.  
\end{align}
Then we get:
\begin{align}
 & \int_{\mathbb{R}} dc e^{(\alpha + \frac{\beta}{2} - Q)c} \E \Bigg[ e^{ \frac{1}{2} (Q- \frac{\beta}{2})(Q-\alpha - \frac{\beta}{2}) \E[Y^2]} e^{\frac{\alpha}{2}(Q -\alpha - \frac{\beta}{2}) \ln \frac{g(i)}{\hg(i)}} e^{\alpha  X(i)  - \frac{\alpha^2}{2} \E[ X(i)^2] }\\ \nonumber
  &  \qquad \qquad \qquad \qquad \qquad   \times  e^{- e^{\frac{\gamma c}{2} - \frac{\gamma}{4} (Q -\frac{\beta}{2}) \E[Y^2]} \int_{\mathbb{R}} e^{\frac{\gamma}{2}  X(x)  - \frac{\gamma^2}{8} \E[ X(x)^2] } g(x)^{\frac{1}{2} - \frac{\gamma \beta}{8}} dx } \Bigg] \\ \nonumber
 &=e^{\frac{\alpha}{2}(Q -\alpha - \frac{\beta}{2}) \ln \frac{g(i)}{\hg(i)}} \int_{\mathbb{R}} dc e^{(\alpha + \frac{\beta}{2} - Q)c} \E \left[   e^{\alpha  X(i)  - \frac{\alpha^2}{2} \E[ X(i)^2] }e^{- e^{\frac{\gamma c}{2} } \int_{\mathbb{R}} e^{\frac{\gamma}{2}  X(x)  - \frac{\gamma^2}{8} \E[ X(x)^2] } g(x)^{\frac{1}{2} - \frac{\gamma \beta}{8}} dx } \right]\\ \nonumber
  &= \frac{2}{\gamma} \Gamma\left(\frac{2}{\gamma}(\alpha + \frac{\beta}{2} -Q) \right) e^{\frac{\alpha}{2}(\alpha + \frac{\beta}{2} -Q) \ln \hg(i) }  \E \left[   \left( \int_{\mathbb{R}} \frac{e^{\frac{\gamma}{2}  X(x)  - \frac{\gamma^2}{8} \E[ X(x)^2] }}{|x - i|^{\gamma \alpha}} g(x)^{\frac{1}{2} - \frac{\alpha \gamma }{4}  - \frac{\gamma \beta}{8}} dx \right)^{\frac{2Q -2 \alpha - \beta}{\gamma}} \right].
\end{align}
To obtain the second line we have shifted the integral over $c$ by $\frac{1}{2}(Q - \frac{\beta}{2}) \E[Y^2]$ and to obtain the last one we have computed the integral over $c$.
The conclusion of the above is thus that:

\begin{equation}
\E\left[ \left( \intr \frac{e^{\frac{\gamma}{2} X_{\hg}(x)   -\frac{\gamma^2}{8} \E[X_{\hg}(x)^2]}}{|x-i|^{\gamma\alpha}} \hg(x)^{\frac{\gamma}{4}(\frac{2}{\gamma}-\alpha - \frac{\beta}{2})}d x \right)^{\frac{2Q-2\alpha - \beta}{\gamma}}\right] = \E \left[   \left( \int_{\mathbb{R}} \frac{e^{\frac{\gamma}{2}  X(x)  - \frac{\gamma^2}{8} \E[ X(x)^2] }}{|x - i|^{\gamma \alpha}} g(x)^{\frac{1}{2} - \frac{\alpha \gamma }{4} - \frac{\beta \gamma}{8}} dx \right)^{\frac{2Q -2 \alpha - \beta}{\gamma}} \right].
\end{equation}
To lift the constraint $\alpha + \frac{\beta}{2} > Q $ we have introduced to write the integral over $c$ we can simply use analyticity in $\alpha$ of both sides of the above equation given by Lemma \ref{lem_analycity}. Then combining this equation with \eqref{Ualpha_D_H} implies the claim of the lemma.
\end{proof}

\subsection{Special functions}\label{sec_special_func}

\subsubsection{Hypergeometric equations}\label{sec_hypergeo}
Here we recall some facts we have used on the hypergeometric equation and its solution space. We always use $\mathbb{N}$ as the set of non-negative integers.
For $A>0$ let $\Gamma(A) = \int_0^{\infty} t^{A-1} e^{-t} dt $ denote the standard Gamma function which can then be analytically extended to $\mathbb{C} \setminus \{ - \mathbb{N} \} $. We recall the following useful properties:
\begin{align}\label{prop_gamma}
\Gamma(A+1) = A \Gamma(A), \quad \Gamma(A) \Gamma(1-A) = \frac{\pi}{\sin(\pi A)}, \quad \Gamma(A) \Gamma(A + \frac{1}{2}) =  \sqrt{\pi} 2^{1 - 2A}  \Gamma(2A).
\end{align}
Let $(A)_n : = \frac{\Gamma(A +n)}{\Gamma(A)}$. For $A, B, C,$ and $t$ real numbers we define the hypergeometric function $F$ by:
\begin{equation}
F(A,B,C,t) :=  \sum_{n=0}^{\infty} \frac{(A)_n (B)_n}{n! (C)_n} t^n.
\end{equation}
This function can be used to solve the following hypergeometric equation:
\begin{equation}
\left( t (1-t) \frac{d^2}{d t^2} + ( C - (A +B +1)t) \frac{d}{dt} - AB\right) f(t) =0.
\end{equation}
We can give the following three bases of solutions corresponding respectively to a power series expansion around $t=0$, $t=1$, and $t = \infty$. Under the assumption that $C$ is not an integer:
\begin{align}
f(t) &= C_1 F(A,B,C,t) + C_2 t^{1 -C} F(1 + A-C, 1 +B - C, 2 -C, t).
\end{align}
Under the assumption that $ C - A - B $ is not an integer:
\begin{align}
f(t) &= B_1 F(A,B,1+A+B- C, 1 -t)\\ \nonumber
& + B_2 (1-t)^{C- A - B} F(C- A, C- B, 1 + C - A - B , 1 -t).
\end{align}
Under the assumption that $ A - B $ is not an integer:
\begin{align}
f(t) &= D_1 t^{-A}F(A,1+A-C,1+A-B,t^{-1})\\  \nonumber
& + D_2 t^{-B} F(B, 1 +B - C, 1 +B - A, t^{-1}).
\end{align}
For each basis we have two real constants that parametrize the solution space, $C_1, C_2$, $B_1, B_2$, and $D_1, D_2$. We thus expect to have an explicit change of basis formula that will give a link between $C_1, C_2$, $B_1, B_2$, and $D_1, D_2$. This is precisely what gives the so-called connection formulas,
\begin{equation}\label{hpy1}
\begin{pmatrix}
C_1  \\
C_2
\end{pmatrix} 
= 
 \begin{pmatrix}
\frac{\Gamma(1 -C ) \Gamma( A - B +1  )  }{\Gamma(A - C +1 )\Gamma(1- B   )  } & \frac{\Gamma(1 -C) \Gamma( B- A +1 )  }{\Gamma(B - C +1 )\Gamma( 1- A )   } \\
 \frac{\Gamma(C-1 ) \Gamma( A - B +1  )  }{\Gamma( A  )  \Gamma( C - B  )} & \frac{\Gamma(C-1 )  \Gamma(B - A +1 ) }{ \Gamma(  B ) \Gamma( C - A )  }
\end{pmatrix} 
\begin{pmatrix}
D_1  \\
D_2 
\end{pmatrix},
\end{equation}

\begin{equation}\label{connection1}
\begin{pmatrix}
B_1  \\
B_2
\end{pmatrix} 
= 
 \begin{pmatrix}
\frac{\Gamma(C)\Gamma(C-A-B)}{\Gamma(C-A)\Gamma(C-B)} & \frac{\Gamma(2-C)\Gamma(C-A-B)}{\Gamma(1-A)\Gamma(1-B)}  \\
\frac{\Gamma(C)\Gamma(A+B-C)}{\Gamma(A)\Gamma(B)} &  \frac{\Gamma(2-C)\Gamma(A+B-C)}{\Gamma(A-C+1)\Gamma(B-C+1)}
\end{pmatrix} 
\begin{pmatrix}
C_1  \\
C_2 
\end{pmatrix}.
\end{equation}

These relations come from the theory of hypergeometric equations and we will extensively use them in Section \ref{sec_bulk_boundary} and Section \ref{sec_3pt} to deduce our shift equations.

\subsubsection{The double gamma function}\label{sec:double_gamma}

We will now provide some explanations on the functions $\Gamma_{\frac{\gamma}{2}}(x)$ and $S_{\frac{\gamma}{2}}(x)$ that we have introduced. For all $\gamma \in (0,2) $ and for $\mathrm{Re}(x) >0$, $\Gamma_{\frac{\gamma}{2}}(x)$ is defined by the integral formula,
\begin{equation}
\ln \Gamma_{\frac{\gamma}{2}}(x) = \int_0^{\infty} \frac{dt}{t} \left[ \frac{ e^{-xt} -e^{- \frac{Qt}{2}}   }{(1 - e^{- \frac{\gamma t}{2}})(1 - e^{- \frac{2t}{\gamma}})} - \frac{( \frac{Q}{2} -x)^2 }{2}e^{-t} + \frac{ x -\frac{Q}{2}  }{t} \right],
\end{equation} 
where we have $Q = \frac{\gamma}{2} +\frac{2}{\gamma}$. Since the function $\Gamma_{\frac{\gamma}{2}}(x)$ is continuous it is completely determined by the following two shift equations
\begin{align}
\frac{\Gamma_{\frac{\gamma}{2}}(x)}{\Gamma_{\frac{\gamma}{2}}(x + \frac{\gamma}{2}) }&= \frac{1}{\sqrt{2 \pi}}
\Gamma(\frac{\gamma x}{2}) ( \frac{\gamma}{2} )^{ -\frac{\gamma x}{2} + \frac{1}{2}
}, \\
\frac{\Gamma_{\frac{\gamma}{2}}(x)}{\Gamma_{\frac{\gamma}{2}}(x + \frac{2}{\gamma}) }&= \frac{1}{\sqrt{2 \pi}} \Gamma(\frac{2
x}{\gamma}) ( \frac{\gamma}{2} )^{ \frac{2 x}{\gamma} - \frac{1}{2} },
\end{align}
and by its value in $\frac{Q}{2}$, $\Gamma_{\frac{\gamma}{2}}(\frac{Q}{2} ) =1$. Furthermore $x \mapsto \Gamma_{\frac{\gamma}{2}}(x)$ admits a meromorphic extension to all of $\mathbb{C}$ with single poles at $x = -n\frac{\gamma}{2}-m\frac{2}{\gamma}$  for any $n,m \in \mathbb{N}$ and $\Gamma_{\frac{\gamma}{2}}(x)$ is never equal to $0$.  We have also used the double sine function defined by:
\begin{equation} 
S_{\frac{\gamma}{2}}(x) = \frac{\Gamma_{\frac{\gamma}{2}}(x)}{\Gamma_{\frac{\gamma}{2}}(Q -x)}.
\end{equation}
It obeys the following two shift equations:
\begin{align}\label{eq:shift_S}
\frac{S_{\frac{\gamma}{2}}(x+\frac{\gamma}{2})}{S_{\frac{\gamma}{2}}(x)} = 2\sin(\frac{\gamma\pi}{2}x), \quad \frac{S_{\frac{\gamma}{2}}(x+\frac{2}{\gamma})}{S_{\frac{\gamma}{2}}(x)} = 2\sin(\frac{2\pi}{\gamma}x).
\end{align}
The double sine function admits a meromorphic extension to $\mathbb{C}$ with poles at $x = -n\frac{\gamma}{2}-m\frac{2}{\gamma}$ and with zeros at $x = Q+n\frac{\gamma}{2}+m\frac{2}{\gamma}$ for any $n,m \in \mathbb{N}$. We also record the following asymptotic for $S_{\frac{\gamma}{2}}(x) $ which can be found in \cite[Equation (B.52)]{Eber}:
\begin{align}\label{eq:lim_S}
S_{\frac{\gamma}{2}}(x) \sim 
\begin{cases}
e^{-i\frac{\pi}{2}x(x-Q)} e^{- i \frac{\pi}{12}(Q^2 +1)} & \text{as} \quad \textnormal{Im}(x) \to \infty,\\
e^{i\frac{\pi}{2}x(x-Q)} e^{ i \frac{\pi}{12}(Q^2 +1)} &\text{as} \quad \textnormal{Im}(x) \to -\infty.
\end{cases}
\end{align}

Finally in Section \ref{app:laws_GMC} in order to state Corollaries \ref{cor_law2} and \ref{cor_law3} on the law of the GMC measures we have used the following random variable $\beta_{2,2}$ defined in \cite{Ostro_review}. Its moments involve the function $\Gamma_{\frac{\gamma}{2}}$.
\begin{definition} [Existence theorem] $\\$
The distribution $-\ln \beta_{2,2}(a_1,a_2;b_0,b_1,b_2)$ is infinitely divisible on $[0,\infty)$ and has the L\'evy-Khintchine decomposition for $\operatorname{Re} (p) >-b_0$:
\begin{align}
&\E[\exp(p \ln \beta_{2,2}(a_1, a_2;b_0,b_1,b_2))] = \exp \Big(\int_0^{\infty} (e^{-pt}-1)e^{-b_0t} \frac{(1-e^{-b_1t})(1-e^{-b_2t})}{(1-e^{-a_1t})(1-e^{-a_2t})}  \frac{dt}{t}\Big).
\end{align}
Furthermore, the distribution $\ln \beta_{2,2}(a_1,a_2;b_0,b_1,b_2)$ is absolutely continuous with respect to the Lebesgue measure.
\end{definition}
We only work with the case $(a_1,a_2)= (1,\frac{4}{\gamma^2})$. Then  $\beta_{2,2}(1,\frac{4}{\gamma^2};b_0,b_1,b_2)$ depends on 4 parameters $\gamma, b_0, b_1, b_2$ and its real moments $p>-b_0$ are given by the formula:
\begin{align}
&\mathbb{E}[ \beta_{2,2}(1, \frac{4}{\gamma^2}; b_0, b_1, b_2)^p]  = \frac{\Gamma_{\frac{\gamma}{2}}(\frac{\gamma}{2}(p + b_0)) \Gamma_{\frac{\gamma}{2}}( \frac{\gamma}{2}(b_0 + b_1)) \Gamma_{\frac{\gamma}{2}}(\frac{\gamma}{2}(b_0 + b_2)) \Gamma_{\frac{\gamma}{2}}(\frac{\gamma}{2}(p + b_0 + b_1 +b_2))}{\Gamma_{\frac{\gamma}{2}}(\frac{\gamma}{2}b_0)\Gamma_{\frac{\gamma}{2}}(\frac{\gamma}{2}(p + b_0 + b_1))\Gamma_{\frac{\gamma}{2}}(\frac{\gamma}{2}(p + b_0 +b_2)) \Gamma_{\frac{\gamma}{2}}( \frac{\gamma}{2}(b_0 +b_1 +b_2))}.
\end{align}
Of course we have $\gamma \in (0,2)$ and the real numbers $p, b_0, b_1, b_2$ must be chosen so that the arguments of all the $\Gamma_{\frac{\gamma}{2}}$ are positive.

\subsubsection{The exact formula $\mathcal{I}$}\label{formula_I}
We provide here an analysis of the formula $\mathcal{I}$ we have written to give the expression for $\overline{H}$. This formula comes from taking the limit $\mu \rightarrow 0$ in formula for the boundary three-point function proposed in \cite{three_point}. We denote the integral it contains by $\mathcal{J}$ and first give a condition of convergence for $\mathcal{J}$.

\begin{lemma}\label{lem:J}
Consider parameters $\beta_1, \beta_2, \beta_3 \in \mathbb{C}$ and $\sigma_1, \sigma_2, \sigma_3 \in \mathbb{C}$ such that the inequality 
\begin{equation}\label{eq:cond_J}
Q > \mathrm{Re}\left( \sigma_3 - \sigma_2  + \frac{\beta_2}{2} \right)
\end{equation}
holds. Then the following integral is well-defined as meromorphic function of all its parameters

\begin{align}
&\mathcal{J} :=  \int_{\mathcal{C}} \frac{S_{\frac{\gamma}{2}}(Q-\frac{\beta_2}{2}+\sigma_3-\sigma_2+r) S_{\frac{\gamma}{2}}(\frac{\beta_3}{2}+\sigma_3-\sigma_1+r) S_{\frac{\gamma}{2}}(Q-\frac{\beta_3}{2}+\sigma_3-\sigma_1+r)}{S_{\frac{\gamma}{2}}(Q+\frac{\beta_1}{2}-\frac{\beta_2}{2}+\sigma_3-\sigma_1+r) S_{\frac{\gamma}{2}}(2Q-\frac{\beta_1}{2}-\frac{\beta_2}{2}+\sigma_3-\sigma_1+r) S_{\frac{\gamma}{2}}(Q+r)}e^{i\pi(-\frac{\beta_2}{2}+\sigma_2-\sigma_3)r} \frac{dr}{i},
\end{align}
where the contour $\mathcal{C}$ of the integral goes from $-i \infty$ to $ i \infty$ passing to the right of the poles at $r = -(Q-\frac{\beta_2}{2}+\sigma_3-\sigma_2) -n\frac{\gamma}{2}-m\frac{2}{\gamma}$, $r = -(\frac{\beta_3}{2}+\sigma_3-\sigma_1)-n\frac{\gamma}{2}-m\frac{2}{\gamma}$, $r= -(Q-\frac{\beta_3}{2}+\sigma_3-\sigma_1)-n\frac{\gamma}{2}-m\frac{2}{\gamma}$ and to the left of the poles at $r=-(\frac{\beta_1}{2}-\frac{\beta_2}{2}+\sigma_3-\sigma_1)+n\frac{\gamma}{2}+m\frac{2}{\gamma}$, $r = -(Q-\frac{\beta_1}{2}-\frac{\beta_2}{2}+\sigma_3-\sigma_1)+n\frac{\gamma}{2}+m\frac{2}{\gamma} $, $r=n\frac{\gamma}{2}+m\frac{2}{\gamma} $ with $m,n \in \mathbb{N}^2$. Furthermore the poles of the function $\mathcal{J}$ occur when $\zeta = n \frac{\gamma}{2} + m \frac{2}{\gamma}$ where $n, m \in \mathbb{N}$ and $\zeta$ is equal to any of the following

$$
\begin{array}{lll}
- Q + \sigma_2 + \frac{\beta_1}{2} - \sigma_1, & \quad \sigma_2  - \sigma_1 - \frac{\beta_1}{2}, & \quad -Q + \frac{\beta_2}{2} - \sigma_3 + \sigma_2, \\[0.2cm]
\frac{\beta_1}{2} - \frac{\beta_2}{2} - \frac{\beta_3}{2}, & \quad Q - \frac{\beta_1}{2} - \frac{\beta_2}{2} - \frac{\beta_3}{2}, & \quad - \frac{\beta_3}{2} - \sigma_3 + \sigma_1, \\[0.2cm]
- Q + \frac{\beta_1}{2} - \frac{\beta_2}{2} + \frac{\beta_3}{2}, & \quad \frac{\beta_3}{2} - \frac{\beta_1}{2} - \frac{\beta_2}{2}, & \quad - Q + \frac{\beta_3}{2} - \sigma_3 + \sigma_1.
 \end{array}
     $$
\end{lemma}

\begin{proof} In the process of proving the above claims, we will also explain in detail the way the contour $\mathcal{C}$ is chosen. Notice first how the poles of the integrand of $\mathcal{J}$ are located. There are three lattices of poles starting from $-(Q-\frac{\beta_2}{2}+\sigma_3-\sigma_2)$, $-(\frac{\beta_3}{2}+\sigma_3-\sigma_1)$, $-(Q-\frac{\beta_3}{2}+\sigma_3-\sigma_1)$ and extending to $-\infty$ by increments of $\frac{\gamma}{2}$ and $\frac{2}{\gamma}$. We call these the left lattices. Similarly we have three lattices of poles starting from $0$, $-(\frac{\beta_1}{2}-\frac{\beta_2}{2}+\sigma_3-\sigma_1)$, $-(Q-\frac{\beta_1}{2}-\frac{\beta_2}{2}+\sigma_3-\sigma_1) $ and extending to $+\infty$ by similar increments, we call them the right lattices. The situation where it is the easiest to draw the correct contour $\mathcal{C}$ is when the parameters are chosen such that the poles of the six different latices all have different imaginary parts. Then it is clear how to draw a line starting from $-i \infty$, passing to the right of the left lattices of poles, to the left of the right lattices of poles, and finally continuing to $+ i \infty$.

Lets us check why the condition \eqref{eq:cond_J} implies the convergence of such a contour at $-i \infty$ and $+ i \infty$. Using the asymptotic of $S_{\frac{\gamma}{2}}$ given by \eqref{eq:lim_S} one obtains the integrand of $\mathcal{J}$ as $r \rightarrow i \infty $ is equivalent to $c_1 e^{2 i \pi  Q r} e^{-i \pi r(2 \sigma_3 - 2 \sigma_2 +\beta_2)}$ for $c_1 \in \mathbb{C}$ a constant independent of $r$. In the other direction, as $r \rightarrow -i \infty $, one finds similarly that the integrand is equivalent to $c_2 e^{-2 i \pi  Q r}$, which is always convergent. The asymptotic as $r \rightarrow i \infty $ thus tells us the integral converges if $ Q > \mathrm{Re}(\sigma_3 - \sigma_2  + \frac{\beta_2}{2})$.

Lastly let us discuss the poles of $\mathcal{J}$ as function of its parameters. This problem is related to being able to choose correctly the contour $\mathcal{C}$ in any situation. It turns out the poles of $\mathcal{J}$ occur when the parameters $\beta_i, \sigma_i$ are chosen such that there is pole from one of the left lattices that coincides with a pole from of one of the right lattices. In such a situation it is clearly not possible to choose the contour $\mathcal{C}$ as required. To solve this issue one can for instance deform the contour $\mathcal{C}$ in a small neighborhood of the collapsing poles so that it crosses one of the two poles before they collapse. By the residue theorem this adds a meromorphic function along side the contour integral in the expression of $\mathcal{J}$. This meromorphic function then has a pole precisely in situation we described, where the parameters are such that poles from the left and right lattices collapse. Lastly one can see by drawing a picture of all the poles that in any other situation one can always draw the contour $\mathcal{C}$. There is one tricky situation where a pole from a left lattice is to the right of a pole from the right lattice, and both poles have the same imaginary part. But in the setup the condition on $\mathcal{C}$ can be satisfied by choosing an eight shaped contour around the two poles.

The conclusion is thus that it is possible to choose consistently the contour $\mathcal{C}$ passing to the left and right of the right and left lattices of poles, except when the parameters are such that two poles of a right and a left lattice coincide. These special points then correspond to poles of the meromorphic function $\mathcal{J}$ and are given by the values taken by $\zeta$ in the statement of the lemma.

\end{proof}
Building upon the previous lemma we can easily deduce the following result about $\mathcal{I}$.
\begin{lemma}
Recall the expression \begin{align}\label{formule_PT}
&\mathcal{I}
\begin{pmatrix}
\beta_1 , \beta_2, \beta_3 \\
\sigma_1,  \sigma_2,   \sigma_3 
\end{pmatrix}\\
   & = \mathcal{J} \times  \frac{(2\pi)^{\frac{2Q-\overline{\beta}}{\gamma}+1}(\frac{2}{\gamma})^{(\frac{\gamma}{2}-\frac{2}{\gamma})(Q-\frac{\overline{\beta}}{2})-1}}{\Gamma(1-\frac{\gamma^2}{4})^{\frac{2Q-\overline{\beta}}{\gamma}}\Gamma(\frac{\overline{\beta}-2Q}{\gamma})} 
 \frac{\Gamma_{\frac{\gamma}{2}}(2Q-\frac{\overline{\beta}}{2})\Gamma_{\frac{\gamma}{2}}(\frac{\beta_1+\beta_3-\beta_2}{2})\Gamma_{\frac{\gamma}{2}}(Q-\frac{\beta_1+\beta_2-\beta_3}{2})\Gamma_{\frac{\gamma}{2}}(Q-\frac{\beta_2+\beta_3-\beta_1}{2})}{\Gamma_{\frac{\gamma}{2}}(Q) \Gamma_{\frac{\gamma}{2}}(Q-\beta_1) \Gamma_{\frac{\gamma}{2}}(Q-\beta_2) \Gamma_{\frac{\gamma}{2}}(Q-\beta_3)} \nonumber \\ 
&\times \frac{e^{i\frac{\pi}{2}(-(2Q-\frac{\beta_1}{2}-\sigma_1-\sigma_2)(Q-\frac{\beta_1}{2}-\sigma_1-\sigma_2) + (Q+\frac{\beta_2}{2}-\sigma_2-\sigma_3)(\frac{\beta_2}{2}-\sigma_2-\sigma_3)+(Q+\frac{\beta_3}{2}-\sigma_1-\sigma_3)(\frac{\beta_3}{2}-\sigma_1-\sigma_3) -2\sigma_3(2\sigma_3-Q)        )}}{S_{\frac{\gamma}{2}}(\frac{\beta_1}{2}+\sigma_1-\sigma_2)  S_{\frac{\gamma}{2}}(\frac{\beta_3}{2}+\sigma_3-\sigma_1) }. \nonumber 
\end{align}
$\mathcal{I}$ is defined as a meromorphic function of all its parameters on the domain given by $ Q > \mathrm{Re} \left( \sigma_3 - \sigma_2  + \frac{\beta_2}{2} \right)$. The poles are a subset of the values of parameters for which there exists $n, m \in \mathbb{N}$ such that $\zeta = n \frac{\gamma}{2} + m \frac{2}{\gamma}$  is equal to one of the following:
$$
\begin{array}{lll}
- Q + \sigma_2 + \frac{\beta_1}{2} - \sigma_1, & \quad \sigma_2  - \sigma_1 - \frac{\beta_1}{2}, & \quad -Q + \frac{\beta_2}{2} - \sigma_3 + \sigma_2, \\[0.2cm]
\frac{\beta_1}{2} - \frac{\beta_2}{2} - \frac{\beta_3}{2}, & \quad Q - \frac{\beta_1}{2} - \frac{\beta_2}{2} - \frac{\beta_3}{2}, & \quad - \frac{\beta_3}{2} - \sigma_3 + \sigma_1, \\[0.2cm]
- Q + \frac{\beta_1}{2} - \frac{\beta_2}{2} + \frac{\beta_3}{2}, & \quad \frac{\beta_3}{2} - \frac{\beta_1}{2} - \frac{\beta_2}{2}, & \quad - Q + \frac{\beta_3}{2} - \sigma_3 + \sigma_1, \\[0.2cm]
\frac{\beta_1}{2} + \frac{\beta_2}{2}+ \frac{\beta_3}{2} - 2Q, & \quad \frac{\beta_2}{2} - \frac{\beta_1}{2} - \frac{\beta_3}{2}, & \quad - Q + \frac{\beta_1}{2} + \frac{\beta_2}{2} - \frac{\beta_3}{2}, \\[0.2cm]
  - Q + \frac{\beta_2}{2} + \frac{\beta_3}{2} - \frac{\beta_1}{2}, & \quad - Q + \frac{\beta_1}{2} + \sigma_1 - \sigma_2, & \quad -Q + \frac{\beta_3}{2} + \sigma_3 -\sigma_1.
 \end{array}
     $$
\end{lemma}

\begin{proof}
The proof of this claim is a direct consequence of Lemma \ref{lem:J}, since $\mathcal{I}$ is obtained from $\mathcal{J}$ by multiplying by an explicit meromorphic function with a known pole structure. One simply adds in the list of poles of $\mathcal{J}$ the poles coming from this function.
\end{proof}

The obvious drawback of the expression of $\mathcal{I}$ is that the integral $\mathcal{J}$ it contains forces one to work under the condition $ Q > \mathrm{Re} \left(  \sigma_3 - \sigma_2  + \frac{\beta_2}{2} \right)$. Luckily thanks to the result of the main text we can propose a meromorphic extension of $\mathcal{I}$ to all of $\mathbb{C}^6$. The logic is as follows. First by the result of Lemma \ref{analycity_H}, we know the function $\overline{H}$ defined using GMC admits a meromorphic extension to a subset of $\mathbb{C}^6$, where the $\beta_i$ are in complex neighborhood of $\mathbb{R}$ and the $\sigma_i$ in $\mathbb{C}$. Then thanks to the results of Section \ref{subsec:three_point}, we know that the function $\mathcal{I}$ matches on its domain of definition with the meromorphic extension of $\overline{H}$. Therefore $\mathcal{I}$ admits a meromorphic extension to the same subset of  $\mathbb{C}^6$. Furthermore by using simple symmetries of the function $\overline{H}$ proved in the main text it is actually possible to deduce an analytic continuation of $\mathcal{I}$ to $\mathbb{C}^6$. The two symmetries we will use correspond to performing a cyclic permutation of the parameters and to using the reflection principle link $\beta_i$ and $2Q - \beta_i$. For this purpose consider the following list of domains with the associated function $\mathcal{I}$.
$$
\begin{array}{ll}
\mathcal{I}
\begin{pmatrix}
\beta_1 , \beta_2, \beta_3 \\
\sigma_1,  \sigma_2,   \sigma_3 
\end{pmatrix}, & \quad \quad \mathcal{D}_1 =  \left\{ (\beta_i, \sigma_i)_{i = 1,2,3} \in \mathbb{C}^6 \: \vert \:  \mathrm{Re} \left( Q - \sigma_3 + \sigma_2 - \frac{\beta_2}{2} \right) >0 \right \}, \\
\mathcal{I}
\begin{pmatrix}
\beta_1 , 2Q - \beta_2, \beta_3 \\
\sigma_1,  \sigma_2,   \sigma_3 
\end{pmatrix}, & \quad \quad  \mathcal{D}_2 = \left\{ (\beta_i, \sigma_i)_{i = 1,2,3} \in \mathbb{C}^6 \: \vert \:  \mathrm{Re} \left( - \sigma_3 + \sigma_2 + \frac{\beta_2}{2} \right) >0 \right \},     \\
\mathcal{I}
\begin{pmatrix}
\beta_2 , \beta_3, \beta_1 \\
\sigma_2,  \sigma_3,   \sigma_1 
\end{pmatrix}, & \quad \quad \mathcal{D}_3 = \left\{ (\beta_i, \sigma_i)_{i = 1,2,3} \in \mathbb{C}^6 \: \vert \:  \mathrm{Re} \left( Q - \sigma_1 + \sigma_3 - \frac{\beta_3}{2} \right) >0 \right \}, \\
\mathcal{I}
\begin{pmatrix}
\beta_2 , 2Q - \beta_3, \beta_1 \\
\sigma_2,  \sigma_3,   \sigma_1 
\end{pmatrix}, & \quad \quad \mathcal{D}_4 = \left\{ (\beta_i, \sigma_i)_{i = 1,2,3} \in \mathbb{C}^6 \: \vert \:  \mathrm{Re} \left( - \sigma_1 + \sigma_3 + \frac{\beta_3}{2} \right) >0 \right \},     \\
\mathcal{I}
\begin{pmatrix}
\beta_3 , \beta_1, \beta_2 \\
\sigma_3,  \sigma_1,   \sigma_2 
\end{pmatrix}, & \quad \quad \mathcal{D}_5 = \left\{ (\beta_i, \sigma_i)_{i = 1,2,3} \in \mathbb{C}^6 \: \vert \:   \mathrm{Re} \left( Q - \sigma_2 + \sigma_1 - \frac{\beta_1}{2} \right) >0 \right \}, \\
\mathcal{I}
\begin{pmatrix}
\beta_3 , 2Q - \beta_1, \beta_2 \\
\sigma_3,  \sigma_1,   \sigma_2 
\end{pmatrix}, & \quad \quad \mathcal{D}_6 = \left\{ (\beta_i, \sigma_i)_{i = 1,2,3} \in \mathbb{C}^6 \: \vert \:  \mathrm{Re} \left( - \sigma_2 + \sigma_1 + \frac{\beta_1}{2} \right) >0 \right \}.     
\end{array}
$$
We first give the following lemma about the domains $\mathcal{D}_i$.

\begin{lemma}\label{lem:D_i}
Consider the six domains $\mathcal{D}_i$, $i=1, \dots, 6$, defined above. Then one has $\cup_{i=1}^6 \mathcal{D}_i = \mathbb{C}^6$. Furthermore for any $i, j$ the set $\mathcal{D}_i \cap \mathcal{D}_j $ is non-empty and contains an open ball of $\mathbb{C}^6$.
\end{lemma}
\begin{proof}
To see why the first claim is true, assume a point $(\beta_i, \sigma_i)_{i = 1,2,3} \in \mathbb{C}^6$ is contained in none of the domains $\mathcal{D}_i$. Then it satisfies the reverse inequalities of those defining the domains $\mathcal{D}_i$. By summing all those inequalities one obtains $3Q < 0$. Hence the contradiction. The second claim of the lemma is straightforward to check from the inequalities.
\end{proof}

Using the above the lemma it is now possible to express the analytic continuation of $\mathcal{I}$ to $\mathbb{C}^6$ simply by patching the different domains $\mathcal{D}_i$ since their overlap always contains an open set.

\begin{lemma}
The function $\mathcal{I}$ of the six parameters $\beta_1, \beta_2, \beta_3, \sigma_1, \sigma_2, \sigma_3$ originally defined on the domain $\mathcal{D}_1$ admits a meromorphic extension to $\mathbb{C}^6$. Furthermore its expression on any of the domains $\mathcal{D}_i$ can be given up to an explicit prefactor in terms of the function $\mathcal{I}$ written to the left of the definition of the corresponding $\mathcal{D}_i$ in the above table. 
\end{lemma}

\begin{proof} From the result of the main text we know $\mathcal{I}$ matches with $\overline{H}$ on the intersection of $\mathcal{D}_1$ and of the subset of $\mathbb{C}^6$ where $\overline{H}$ has been analytically extended. Next using the probabilistic definition of $\overline{H}$ one can check that this function is invariant under a cyclic permutation of the indices $i =1,2,3$, applied simultaneously to both set of variables $\beta_i$ and $\sigma_i$. Similarly under the reflection $\beta_i \rightarrow 2Q -\beta_i$ the function $\overline{H}$ gets multiplied by an explicit meromorphic function given in Lemma \ref{reflection_H}. These same properties must then hold for the function $\mathcal{I}$. But by performing the cyclic permutation or the reflection, the domain of valid of $\mathcal{I}$ changes from $\mathcal{D}_1$ to one of the other $\mathcal{D}_i$. By the result of Lemma \ref{lem:D_i} the different expressions are the meromorphic extension of one another, and jointly they cover all of $\mathbb{C}^6$.
\end{proof}

Lastly we include here a sanity check on the formula of $\mathcal{I}$. We check that it obeys the scaling property verified by the probabilistic formula for $\overline{H}$ that was used in the proof of Lemma \ref{analycity_H}. 

\begin{lemma}\label{lem:scaling_I} Let $A \in \mathbb{C}$. Then the following holds as an equality of meromorphic functions of $\mathbb{C}^6$
\begin{equation}
\mathcal{I}
\begin{pmatrix}
\beta_1 , \beta_2, \beta_3 \\
\sigma_1 +A ,  \sigma_2 + A,   \sigma_3 +A 
\end{pmatrix}
= e^{i \pi A(2Q - \beta_1 - \beta_2 - \beta_3)}\mathcal{I}
\begin{pmatrix}
\beta_1 , \beta_2, \beta_3 \\
\sigma_1,  \sigma_2,   \sigma_3 
\end{pmatrix}.
\end{equation}

\end{lemma}

\begin{proof}
Notice that when replacing all the $\sigma_i$ by $\sigma_i +A$, the function $\mathcal{J}$ does not change. The only thing that changes in $\mathcal{I}$ is the function 
$$e^{i\frac{\pi}{2}(-(2Q-\frac{\beta_1}{2}-\sigma_1-\sigma_2)(Q-\frac{\beta_1}{2}-\sigma_1-\sigma_2) + (Q+\frac{\beta_2}{2}-\sigma_2-\sigma_3)(\frac{\beta_2}{2}-\sigma_2-\sigma_3)+(Q+\frac{\beta_3}{2}-\sigma_1-\sigma_3)(\frac{\beta_3}{2}-\sigma_1-\sigma_3) -2\sigma_3(2\sigma_3-Q)        )}.$$
A direct computation then shows the change gives precisely a factor $e^{i \pi A(2Q - \beta_1 - \beta_2 - \beta_3)}$.
\end{proof}

\subsubsection{Some useful integrals} \label{section useful integrals}
\begin{lemma}
For $\theta_0 \in [-\pi,\pi]$, $-1<g< 1$ and $1 \lor (1+g)<b<2$ we have the identity:
\begin{equation}\label{equation integral 1}
\int_{\mathbb{R}_+ e^{i\theta_0}} \frac{ (1+u)^{ g } - 1}{ u^b} du=\frac{\Gamma(1-b)\Gamma(-1+b-g)}{\Gamma(-g)}.
\end{equation}
By $\mathbb{R}_+ e^{i\theta_0}$ we mean a complex contour that is obtained by rotating the half-line $(0,+\infty)$ by an angle $e^{i \theta_0}$. In particular for $\theta_0 = \pi $ it is passing above $-1$ and for $\theta_0 = - \pi $ it is passing below.
\end{lemma}
\begin{proof}
Denote by $(x)_n:= x(x+1)\dots (x+n-1)$. We start by the case $\theta_0 = 0$:
\begin{align}
\int_0^{\infty} \frac{ (1+u)^{ g } - 1}{ u^b} du =& \sum_{n=0}^{\infty} \frac{(-1)^n}{n!} (-g)_n \frac{1}{n +1-b } - \sum_{n=0}^{\infty} \frac{(-1)^n}{n!} (-g)_n \frac{1}{ 1-b  +g-n } \\ \nonumber
=&\frac{1}{1-b}\sum_{n=0}^{\infty} \frac{(-1)^n}{n!} \frac{(-g)_n (1-b)_n}{(2-b)_n} - \frac{1}{1-b + g}\sum_{n=0}^{\infty} \frac{(-1)^n}{n!} \frac{(-g)_n(-1+b - g)_n}{ (b - g )_n } \\ \nonumber
=&\frac{1}{1-b}\,F(-g,1-b,2-b,-1)-\frac{1}{1-b+g}\,F(-g,-1+b-g,b-g,-1)\\ \nonumber
=&\frac{\Gamma(1-b)\Gamma(-1+b-g)}{\Gamma(-g)},
\end{align}
where in the last line we used the formula, for suitable $\overline{a}, \overline{b} \in \mathbb{R}$,
\begin{equation}
\bar{b}\,F(\bar{a}+\bar{b},\bar{a},\bar{a}+1,-1)+\bar{a}\,F(\bar{a}+\bar{b},\bar{b},\bar{b}+1,-1)=\frac{\Gamma(\bar{a}+1)\Gamma(\bar{b}+1)}{\Gamma(\bar{a}+\bar{b})}.
\end{equation}
Then by rotating the contour, it is easy to observe that the value of the integral is the same for all $\theta_0 \in [-\pi,\pi]$, which finishes the proof.
\end{proof}

A direct consequence by a change of variable is the following identity:

\begin{lemma}
For $\theta_0 \in [-\pi,\pi]$, $-1<g< 1$ and $g<b<1 \land (1+g)$ we have the identity:
\begin{equation}\label{equation integral 2}
\int_{\mathbb{R}_+ e^{i\theta_0}} \frac{ (1+u)^{ g } - u^g}{ u^b} du=\frac{\Gamma(1-b)\Gamma(-1+b-g)}{\Gamma(-g)}.
\end{equation}
\end{lemma}

\end{document}